\documentclass{article} 
\usepackage{amsmath} 
\usepackage{amssymb} 
\usepackage{array} 
\usepackage{bbm} 
\usepackage{hyperref} 
\hypersetup{
    colorlinks,
    linkcolor={blue!80!black},
    citecolor={green!50!black},
}
\usepackage{amsthm} 
\usepackage{xcolor} 
\usepackage[]{algorithm2e} 
\usepackage[shortlabels]{enumitem} 
\usepackage{booktabs}
 
\usepackage{mathtools} 
\mathtoolsset{showonlyrefs}

\def\whp{for $\epsilon < c'$ with probability at least $1 - \sC(\epsilon)e^{-\sc(\epsilon)p}$}

\def\MM{\mathbb{M}}

\def\PP{\mathbb{P}}

\def\SS{\mathbb{S}}

\def\ZZ{\mathbb{Z}} 
 
\def\cB{\mathcal{B}} 
 
\def\cG{\mathcal{D}} 
\def\cE{\mathcal{E}} 
 
\def\cG{\mathcal{G}} 
 
\def\cI{\mathcal{I}}

\def\cS{\mathcal{S}} 
\def\cT{\mathcal{T}}

\def\bA{\boldsymbol{A}} 
\def\bB{\boldsymbol{B}} 
 
\def\bD{\boldsymbol{D}} 
\def\bE{\boldsymbol{E}} 
 
\def\bG{\boldsymbol{G}} 
\def\bh{\boldsymbol{H}}

\def\bK{\boldsymbol{K}} 
 
\def\bM{\boldsymbol{M}} 
\def\bN{\boldsymbol{N}}

\def\bQ{\boldsymbol{Q}} 
\def\bR{\boldsymbol{R}} 
\def\bS{\boldsymbol{S}} 
\def\bT{\boldsymbol{T}} 
\def\bU{\boldsymbol{U}} 
\def\bV{\boldsymbol{V}} 
 
\def\bX{\boldsymbol{X}} 
\def\bY{\boldsymbol{Y}} 
\def\bZ{\boldsymbol{Z}} 
 
\def\ba{\boldsymbol{a}} 
\def\bb{\boldsymbol{b}} 
 
\def\bd{\boldsymbol{d}} 
\def\be{\boldsymbol{e}} 
\def\bmf{\boldsymbol{f}} 
\def\bg{\boldsymbol{g}} 
\def\bh{\boldsymbol{h}}

\def\br{\boldsymbol{r}} 
\def\bs{\boldsymbol{s}} 
\def\bt{\boldsymbol{t}} 
\def\bu{\boldsymbol{u}} 
\def\bv{\boldsymbol{v}} 
\def\bw{\boldsymbol{w}} 
\def\bx{\boldsymbol{x}} 
\def\by{\boldsymbol{y}} 
\def\bz{\boldsymbol{z}} 
 
\def\bA{\boldsymbol{A}} 
\def\bB{\boldsymbol{B}} 
 
\def\bD{\boldsymbol{D}}
\def\bE{\boldsymbol{E}}

\def\bG{\boldsymbol{G}}
\def\bH{\boldsymbol{H}}

\def\bK{\boldsymbol{K}}

\def\bM{\boldsymbol{M}}
\def\bN{\boldsymbol{N}}

\def\bQ{\boldsymbol{Q}}
\def\bR{\boldsymbol{R}}
\def\bS{\boldsymbol{S}}
\def\bT{\boldsymbol{T}}
\def\bU{\boldsymbol{U}}
\def\bV{\boldsymbol{V}}

\def\bX{\boldsymbol{X}}
\def\bY{\boldsymbol{Y}}
\def\bZ{\boldsymbol{Z}}

\def\bxi{\boldsymbol{\xi}}
\def\bpi{\boldsymbol{\pi}}

\def\bDf{\boldsymbol{D\hspace{-.1cm}f}}

\def\sc{\mathsf{c}}

\def\sx{\mathsf{x}}

\def\sC{\mathsf{C}}
\def\sD{\mathsf{D}}

\def\sF{\mathsf{F}}

\def\sM{\mathsf{M}}
\def\sN{\mathsf{N}}

\def\sP{\mathsf{P}}

\def\sR{\mathsf{R}}

\def\sT{\mathsf{T}}

\def\btheta{\boldsymbol{\theta}}
\def\bSigma{\boldsymbol{\Sigma}}

\def\bxi{\boldsymbol{\xi}}
\def\bgamma{\boldsymbol{\gamma}}

\def\bdelta{\boldsymbol{\delta}}
\def\bzero{\boldsymbol{0}}

\def\bTheta{\boldsymbol{\Theta}}
\def\bphi{\boldsymbol{\phi}}

\def\op{\overline{p}}

\def\cad{\mathrm{cad}}

\newcommand\proj{\mathsf{P}}

\newcommand\indic[1]{\mathbf{1}\{#1\}}

\renewcommand{\P}{\mathbb{P}}
\newcommand{\E}{\mathbb{E}}

\newcommand{\eps}{\varepsilon}

\DeclareMathOperator*{\argmax}{arg\,max}
\DeclareMathOperator*{\argmin}{arg\,min}

\newcommand{\Cov}{\operatorname{Cov}}
\newcommand{\sign}{\operatorname{sign}}

\newcommand{\diag}{\operatorname{diag}}

\newcommand{\Tr}{\operatorname{Tr}}

\newcommand{\normal}{\mathsf{N}}

\newcommand{\reals}{\mathbb{R}}
\newcommand{\hbeta}{\widehat \beta}
\newcommand{\de}{\mathrm{d}}
\renewcommand{\div}{\operatorname{div}}

\def\cPmodel{\mathcal{P}_{\mathrm{model}}}
\def\cPregr{\mathcal{P}_{\mathrm{regr}}}

\def\bias{{\sf bias}}
\def\std{{\sf std}}
\def\KS{{\sf KS}}
\def\naive{{\mathrm{nv}}}

\def\btheta{{\boldsymbol \theta}}
\def\bgamma{{\boldsymbol \gamma}}
\def\hbtheta{\hat{\boldsymbol \theta}}
\def\lb{\ell_{\textrm{b}}}
\def\bSigma{{\boldsymbol \Sigma}}
\def\bx{{\boldsymbol x}}
\def\reals{{\mathbb R}}
\def\orac{\mbox{\tiny\rm orac}}
\def\MM{\mbox{\tiny\rm MM}}
\def\by{{\boldsymbol y}}
\def\bX{{\boldsymbol X}}
\def\hlambda{\hat{\lambda}}

\def\<{\langle}
\def\>{\rangle}

\newcommand{\id}{\mathbf{I}}

\newcommand\df{\mathsf{df}}

\newtheoremstyle{myremark} 
    {\topsep}                    
    {\topsep}                    
    {\rm}                        
    {}                           
    {\bf}                        
    {.}                          
    {.5em}                       
    {}  

\newtheorem{theorem}{Theorem}[section]
\newtheorem{lemma}[theorem]{Lemma}
\newtheorem{corollary}[theorem]{Corollary}

\newtheorem{definition}{Definition}[section]
\theoremstyle{myremark}
\newtheorem{remark}[theorem]{Remark}

\allowdisplaybreaks
\usepackage[top=1in,bottom=1in,left=1in,right=1in]{geometry}
\usepackage[utf8]{inputenc}

\title{CAD: Debiasing the Lasso with inaccurate covariate model}
\author{Michael Celentano\,\, and\,\, Andrea Montanari}

\begin{document}

\maketitle

\newcommand{\barF}{\bar F}
\newcommand{\barG}{\bar G}
\newcommand{\barH}{\bar H}
\newcommand{\barpsi}{\bar \psi}
\newcommand{\sfP}{\mathsf{P}}

\newcommand{\loo}{\mathrm{loo}}

\begin{abstract}
  We consider the problem of estimating a low-dimensional parameter in high-dimensional linear regression. Constructing an approximately unbiased
  estimate of the parameter of interest is a crucial step towards performing statistical inference.
  Several authors suggest to orthogonalize both the variable of interest and the outcome with respect to the nuisance variables,
  and then regress the residual outcome with respect to the residual variable.
  This is possible if the covariance structure of the regressors is perfectly known, or is sufficiently  structured that it
  can be estimated accurately from data (e.g., the precision matrix is sufficiently sparse).

  Here we consider a regime in which the covariate model can only be estimated inaccurately, and hence existing debiasing approaches
  are not guaranteed to work. When errors in estimating the 
  covariate model are correlated with errors in estimating the linear model parameter, an incomplete elimination of the bias occurs.
  We propose the \emph{Correlation Adjusted Debiased Lasso (CAD)}, which nearly eliminates this bias in some cases, including cases in which the estimation errors are neither negligible nor orthogonal.

  We consider a setting in which some unlabeled samples
  might be available to the statistician alongside labeled ones (semi-supervised learning), and our
  guarantees hold under the assumption of jointly Gaussian covariates.
  The new debiased estimator is guaranteed to cancel the bias in two cases: $(1)$ when the total number of samples (labeled and unlabeled)
  is larger than the number of parameters, or $(2)$ when the covariance of
  the nuisance (but not the effect of the nuisance on the variable of interest) is known. Neither of these cases
  is treated by state-of-the-art methods.
\end{abstract}
\tableofcontents

\section{Introduction}

An important task in high-dimensional and semi-parametric statistics is to estimate and perform inference on a low-dimensional parameter in the presence of a high-dimensional nuisance.
We study this problem in the context of a random-design linear model
\begin{equation}\label{eq:lin-model}
\begin{gathered}
    \by = \bw\beta + \bX \btheta + \sigma \bz,
\end{gathered}
\end{equation}
where $\bX \in \reals^{n \times p}$ has rows $\bx_i \stackrel{\mathrm{iid}}\sim \normal(0,\bSigma)$,
$\bw = \bX \bgamma + \kappa \bw^\perp$,
and $\bw^\perp,\bz \stackrel{\mathrm{iid}}\sim \normal(0,\id_n)$ independent of $\bX$.
The target of estimation and inference is $\beta$.
We consider a setting in which $n$ is comparable to and possibly smaller than $p$,
but a larger unlabelled data set might be available for estimating $\bgamma$.

Our ability to estimate and perform inference on $\beta$ is substantially improved by knowledge of the distribution of the features $(\bw,\bX)$, parameterized by $\kappa,\bgamma,\bSigma$.
Consider, for example, the debiased Lasso.
The debiased Lasso is a one-step correction to the Lasso estimate defined by
\begin{equation}
\begin{gathered}
    (\hbeta,\hat \btheta) 
        := 
        \argmin_{(b,\bt)} 
        \Big\{ 
            \frac1{2n} \| \by - \bw b - \bX \bt \|_2^2 + \frac{\lambda}{\sqrt{n}}(|b| + \| \bt \|_1)
        \Big\},
    \\
    \hbeta^\de 
        := 
        \hbeta 
            + \frac{1}{n\hat \tau^2 }
            (\bw - \bX \hat \bgamma)^\top(\by - \bw \hbeta - \bX \hat \btheta)\,,
\end{gathered}\label{eq:FirstDebiased}
\end{equation}
where $\hat \bgamma$ is an estimate of $\bgamma$ and $\hat \tau^2$ is a normalization coefficient.
Both of these are estimated from data or on the basis of some knowledge of the data distributiomn.
A substantial line of research develops debiasing procedure which differ in how $\hat \bgamma,\hat \tau$ are defined, see
\cite{zhang2014,vanDeGeer2014,javanmardMontanari2014,javanmard2013nearly,javanmard2014,javanmard2018,bellec2020debiasing}
for a few entrypoints in this literature.

Under appropriate conditions and up to constant and logarithmic factors, 
the Lasso achieves estimation error in $\ell_1$ and $\ell_2$ comparable to what is achievable when the support of $\btheta$ is known \cite{bickelEtAl2009,buhlmann2011statistics}.
These guarantees apply only to estimation of the high-dimensional parameter $(\beta,\btheta)$ as a whole.
For any given coordinate (in particular, for the parameter of interest $\beta$) the Lasso error may be substantially
larger than that achievable were the support of $\btheta$ known.\footnote{For example, an explicit calculation with design-matrix given in Example 7.1 of \cite{buhlmann2011statistics} shows that the error in a single coordinate can be on the order $\Theta(\sigma \sqrt{s_\theta \log(p/s_\theta)/n})$ if we choose $\lambda = \Theta(\sigma\sqrt{\log(p/s_\theta)})$ in Eq.~\eqref{eq:FirstDebiased}.}

The debiased Lasso \eqref{eq:FirstDebiased} targets estimation of a single coordinate and can improve error guarantees for it.
Moreover, in certain regimes $\hbeta^\de$ is approximately normal,
permitting the construction of $p$-values and confidence intervals.

As we have already alluded to, 
the debiased Lasso benefits substantially from knowledge of $\bgamma$.
When $\bgamma$ is unknown,
the best bounds establish $\hbeta^\de - \beta = O_p\big( (\sigma/\kappa)\big(1/\sqrt{n} +  (s_\theta \wedge s_\gamma) \log(p) / n \big) \big)$
\cite{javanmard2018,bellec2020debiasing}, where $s_{\theta}:=\|\btheta\|_0$ and $s_{\gamma}:=\|\bgamma\|_0$
are the numbers of nonzero coefficients (these bounds can be improved to cover the case of approximately sparse vectors).
The $1/\sqrt{n}$ term is a parametric rate and would be present even if $\bgamma,\btheta$ were known.
The term $(s_\theta \wedge s_\gamma)\log(p)/n$ captures the cost of not knowing the nuisance parameters $\bgamma$ and $\btheta$.
It dominates the estimation error unless $s_\theta \wedge s_\gamma = O(\sqrt{n}/\log(p))$, 
a much stronger requirement than that required for consistency of the Lasso for either $\btheta$ or $\bgamma$.
In contrast, if $\bgamma$ is known,
then by setting $\hat \bgamma = \bgamma$ and appropriately choosing $\hat \tau,\lambda$,
the debiased Lasso achieves the parametric rate even if $\bgamma$ is fully dense and $s_\theta \log p / n = \Theta(1)$
\cite{javanmardMontanari2014,javanmard2013nearly,miolane2018distribution,bellec2020second,celentano2020lasso}.
A related procedure achieves paramteric rates provided the bound $\| \btheta \|_0 \leq s_\theta$ is replaced by  the bound $\btheta^\top \bSigma \btheta = O(\sigma^2)$ \cite{zhuBradic2018a,zhuBraidc2018b,Bradic2019TestabilityOH}.

A key motivation of our work is that approaches to dealing with unknown $\bgamma$
\cite{zhang2014,vanDeGeer2014,javanmard2014,javanmard2018,bellec2020debiasing}
operate essentially by reduction to the case of known $\bgamma$.
Roughly speaking, these papers assume either that $\bgamma$ is sparse enough that it can be accurately estimated from the data,
or that $\btheta$ is sparse enough that the uncertainty on $\bgamma$ has negligible effect.
Under the condition  $s_\theta \wedge s_\gamma = O(\sqrt{n}/\log(p))$, the uncertainty on $\bgamma$ is negligible and
does not need to be accounted for in the resulting inference procedures.

The current paper considers a regime in which $\bgamma$ cannot be estimated accurately, and ---for this reason---
standard debiasing procedures are unsuccessful.
We show that ---in some cases--- these procedures can be corrected to cancel the residual bias,
and propose the \emph{Correlation Adjusted Debiased Lasso (CAD)} to implement the correction. 
Under the assumption of
(correlated) Gaussian covariates, we prove that CAD yields nearly unbiased estimates.
An illustration is provided Figure \ref{fig:SweepLamb} (see Sections \ref{sec:Debiasing} and \ref{sec:Simulations} for further explanation).
Although the Lasso is the primary motivation for the current work, our theory will also establish the success of correlation adjusted debiasing for ridge regression.

We work in a semisupervised model in which the statistician has access to $N$ samples,
some of which are labeled and some of which are unlabeled.
A subset $\cI_\theta \subset [N]$ of units, all of which are labelled, will be used for the estimation of $ \btheta$, which we call the
\emph{outcome regression}.
A subset $\cI_\gamma \subset[N]$ of units, which can contain both labelled and unlabelled units,
will be used for the estimation of $\bgamma$, which we call the \emph{precision regression}
(see Section \ref{sec:Debiasing} for a complete description).
In other words $\{ (y_i,w_i,\bx_i)\}_{i \in \cI_\theta}$ are used for the outcome regression
and $\{ (w_i,\bx_i)\}_{i \in \cI_\gamma}$  are used for precision regression. The two are then combined
to construct the CAD estimate of $\beta$.
We assume $\cI_\theta \cup \cI_\gamma = [N]$ (so that all units are used in either the outcome or precision regression),
and we allow for $\cI_\theta$ and $\cI_\gamma$ to be overlapping.
In particular, we may use both labeled and unlabeled data to estimate $\bgamma$.

Part of earlier work  assumes no unlabeled samples and hence $\bgamma$ is estimated from the same labeled samples as $\btheta$,
which is a special case of our setting.
Other earlier results assume $\bgamma$ to be perfectly known.
The latter can be justified by assuming  access to a very large number of unlabeled data  (see, e.g., \cite{javanmard2018}).
By considering a semisupervised  model, our analysis covers these two extremes as special cases,
and allows us to quantify how many additional samples are needed to
estimate  $\bgamma$ accurately enough.

We prove that CAD is nearly unbiased in at least two cases:
$(1)$~When sufficientlly many samples are available for the precision regression to be underparametrized, namely
$n_{\gamma}>(1+\eps)p$, with $\eps$ bounded away from $0$;
$(2)$~When the covariance structure of the nuisance $\bSigma$ is known.
We emphasize that none of these cases can be treated using existing methodology.

\subsection{Related literature}

Our results cover the proportional asymptotics in which $n,p$ are both large with $n\asymp p$. This is the most challenging regime,
since the effect of the uncertainty on $\btheta$ and $\bgamma$ is comparable to or larger than the intrinsic parametric error $(\sigma/\kappa)/\sqrt{n}$.
However, we emphasize that: $(i)$~Our analysis is fully non-asymptotic and we establish concentration bounds
on the newly developed estimator that are valid at finite $p,n$; $(ii)$~These concentration bounds do not degrade as $n/p\to\infty$,
and hence cover a broad range of high-dimensional regimes.

Over the last few years, several mathematical techniques have been developed to characterize high-dimensional statistical problems
in the proportional regime \cite{bayati2011lasso,amelunxen2013,donoho2016high,el2018impact,reeves2016replica, barbier2019optimal}.
Gordon's Gaussian comparison inequality \cite{gordon1985some,Gordon1988} is particularly useful in the study of regularized
 convex M-estimators.  
 As first pointed out in  \cite{stojnic2013framework,thrampoulidis2015regularized},
 such estimators can be written as saddle points of convex-concave cost functions, and Gordon's inequality can be used to
 sharply characterize their asymptotic properties, see e.g.
 \cite{thrampoulidis2015regularized,thrampoulidis2018,Deng2019AMO,Montanari2019,Liang2020APH}.
 
 The most closely related pieces of work are \cite{miolane2018distribution}, which provides a non-asymptotic characterization of the Lasso
 in the proportional regime, for uncorrelated Gaussian designs, and \cite{celentano2020lasso} which generalizes it to correlated designs.
 The paper \cite{celentano2020lasso}  also develops a procedure for statistical inference on low-dimensional parameters
 under the assumption that effect of the nuisance variables (encoded by $\bgamma$) is known.
 
 Our focus is on the case in which $\bgamma$ needs to be
 estimated from data. This leads to a simultaneous regression problem, and a direct generalization of the
 Gaussian comparison approach of \cite{gordon1985some,Gordon1988,stojnic2013framework,thrampoulidis2015regularized}
 simply fails. Namely, the `correct' dominating Gaussian process (the one that yields the right characterization) does not satisfy the
 conditions of Gordon's inequality.

 We overcome this problem by introducing a new proof techniques, based on a two-step argument.
 Given two simultaneous regression problems, we analyze the first one using the Gaussian comparison approach.
 We then show that we can condition on the result of the first regression and develop a `conditional Gordon inequality'
 to characterize the conditional behavior of the second. Putting everything together, we obtain a joint characterization of the
 two regressions. The only earlier example of a joint characterization of two estimators in the proportional asymptotics is the one of \cite{mondelli2020},
 which is limited however to the special case of linear regression and spectral estimation.
 
In recent years it has become popular to base
inference in possibly non-linear regression models on complete knowledge of the distribution of covariates. 
This approach is often referred to as \emph{model-X inference.}
Prominent examples include knockoffs and conditional randomization testing \cite{candesFan2018,katsevich2020theoretical,liu2020fast}.
Although model-X assumptions are sometimes based on scientific domain knowledge (see, e.g., \cite{Bates24117,sesia2018}),
they are often also justifed by informal appeals to estimates coming from large unlabelled data sets \cite{katsevich2020theoretical}.
In the context of linear models, our work provides rigorous justification for these claims, and quantifies the amount of
additional data needed to estimate the covariates model.

In particular, we clarify that  the success of the procedure relies on more than just the size of the error in the model for the features; 
it also relies on how these models were estimated and how these estimates were incorporated into our analysis of the labeled data.
Although random-design linear models are outside of the model-X paradigm---we posit both a model for the covariates and the outcome variable---we have no reason to expect model misspecification in model-X procedures to be any more benign.
In fact, recent work has shown close relationships between debiasing estimators and certain  conditional randomization tests \cite{celentano2020lasso}.

The scope of debiasing procedures was generalized in \cite{chernozhukov2018double} and follow-up work by considering a broader set of
(not necessarily linear) models for the nuisance. As in the original debiasing literature
\cite{zhang2014,vanDeGeer2014,javanmard2014}, these authors assume conditions under which the error in estimating the nuisance model can be safely neglected.
Further, they deal with nonlinear nuisance models by using sample splitting or crossfitting, which avoids mathematical difficulties but may result in inefficiencies in a proportional regime.

As discussed below, one of the elements of the CAD estimator
$\hbeta_{\cad}$ is given  by the so-called \emph{degrees of freedom adjustement}, which was introduced in \cite{javanmardMontanari2014} and 
analyzed for general random designs in \cite{bellec2020debiasing,bellec2020second}. As further discussed in the next section, the 
degrees of freedom adjustement is not sufficient in the case of unknown precision model. In Section \ref{sec:MainResult} we further compare our results
with the ones of \cite{bellec2020debiasing,bellec2020second}.

\section{Correlation adjusted debiasing}
\label{sec:Debiasing}

\subsection{Description of correlation adjusted debiasing}
\begin{figure}[!t]
\centering
\includegraphics[width = 0.99\linewidth]{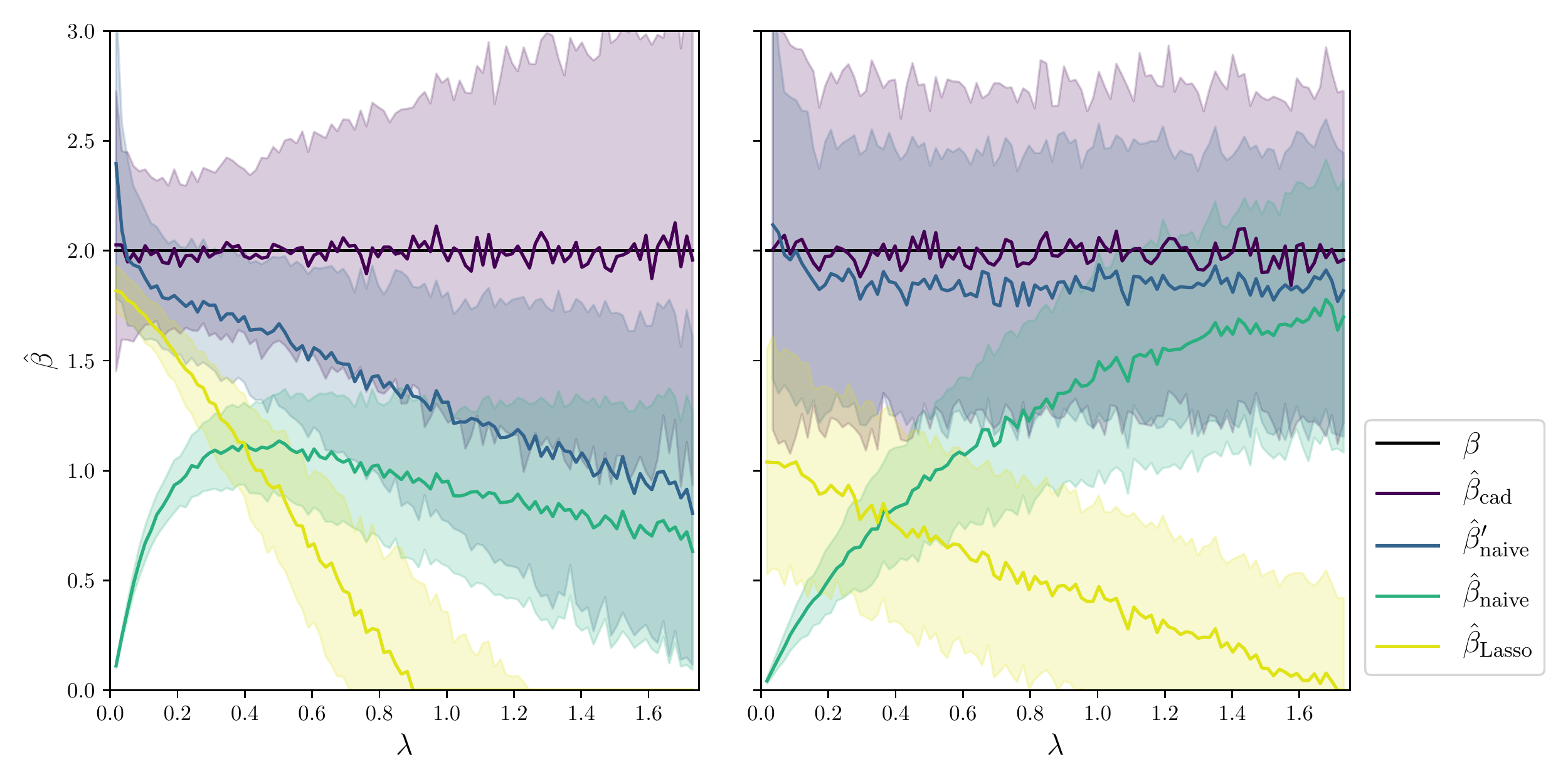}
\caption{Comparing different debiasing methods, for different data distributions (see Section \ref{sec:Simulations} for definitions).
Left: block model  with $\lb=10$, $q=1$, $\mu=2$, $\sigma=1$.
  Right: circulant model with $\mu=2$, $\sigma=1$ and inverse covariance specified in Eq.~\eqref{eq:CircSpecial}.
  In both cases, $n_{\theta} = 200$, $n_{\gamma} = 600$, $p = 400$.
  Bands correspond to $20\%$ and $80\%$ percentiles of various debiased estimates over $500$ repetitions. }
\label{fig:SweepLamb}
\end{figure}
In order to motivate the correlation adjusted debiasing (CAD) procedure,
it is worth recalling some important steps in standard analyses of the debiased Lasso.
As mentioned above, various debiasing procedures differ in the way the estimate of the nuisance parameters $\hat \bgamma$
and normalization coefficients $\hat \tau$ appearing in Eq.~\eqref{eq:FirstDebiased}  are defined.
For our discussion,  we assume  $\hat \tau^2 = \< \bw - \bX \hat \bgamma, \bw \>/n$.
This choice captures many of the important intuitions, and avoids some inessential complications.
With this definition, some straightforward algebra gives
\begin{equation}\label{eq:db-cor-form}
    \hbeta^\de 
        - \beta 
        = 
        \frac{\< \bw - \bX \hat \bgamma , \by - \bX \hat \btheta \>}{n \hat \tau^2} - \beta 
        =
        \frac{\< \bw - \bX \hat \bgamma , \sigma \bz - \bX(\hat \btheta - \btheta) \>}{n \hat \tau^2}.
\end{equation}
Writing the debiased Lasso in this form reveals that the debiased estimate is a rescaled correlation of the residuals of two regressions: 
the feature of interest $\bw$ on the remaining covariates $\bX$ and the outcome $\by$ on the remaining covariates.
This is not surprising.
In the population, $\beta$ is identified in precisely this way.
Indeed, 
we may write the linear model \eqref{eq:lin-model}
in the form 
\begin{equation}
    \by = \bX \bar \btheta + \sigma \bz + \kappa \beta \bw^\perp,
\end{equation}
where $\bar \btheta := \btheta + \beta \bgamma$.
The parameter $\bar \btheta$ results from regressing $\by$ onto $\bX$ in the population.
Then,
\begin{equation}\label{eq:pop-cor}
    \beta = \frac{\E[(\bw - \bX \bgamma)^\top (\by - \bX \bar \btheta)]}{n\kappa^2}.
\end{equation}
Equation~\eqref{eq:db-cor-form} reveals that the debiased Lasso---at least in the case that $\hat \tau^2 = \< \bw - \bX \hat \bgamma, \bw \>/n$---is an empirical form of this identifying equation.

If $\hat \bgamma = \bgamma$ and $\hat \btheta = \btheta$, 
then it is easy to see that $\hat \tau^2 = \kappa^2 + O_p(n^{-1/2})$,
whence $\sqrt{n}(\hbeta^\de - \beta) = \< \kappa \bw^\perp , \sigma \bz \> / (\sqrt{n} \hat \tau^2) \approx \normal(0,\sigma^2/\kappa^2) $.
This is exactly what we would get in the low dimensional setting in which we regress $\beta \kappa \bw^\perp + \sigma \bz$ on $\kappa \bw^\perp$.
In the high-dimensional setting,
performance is degraded relative to this benchmark due to uncertainty in $\hat \bgamma$ and $\hat \btheta$.
In fact,
we may decompose
\begin{equation}
    \sqrt{n}
        (\hbeta^\de 
        - 
        \beta)
        = 
        \frac{\< \kappa \bw^\perp , \sigma \bz \>}{\sqrt{n}\hat \tau^2}
        -
        \frac{\< \kappa \bw^\perp , \bX(\hat \btheta - \btheta) \>}{\sqrt{n}\hat \tau^2}
        -
        \frac{\< \bX(\hat \bgamma - \bgamma) , \sigma \bz \>}{\sqrt{n}\hat \tau^2}
        + 
        \frac{\< \bX(\hat \bgamma - \bgamma) , \bX(\hat \btheta - \btheta) \>}{\sqrt{n} \hat \tau^2}.
\end{equation}
The final three terms are errors due to uncertainty in $\bgamma,\btheta$,
and are the source of the non-parametric rate $(s_\theta \wedge s_\gamma) \log p / n$.

Rather than construct an estimate of $\beta$ as a one-step correction to the Lasso estimate $(\hbeta,\hat \btheta)$,
we use the identifying equation \eqref{eq:pop-cor} as our starting point.
Recall that the statistician uses the samples indexed by $\cI_\theta \subset [N]$ (all labeled) to estimate $\bar\btheta$,
and the ones indexed by $\cI_\gamma \subset[N]$ (which can be either labeled or unlabeled) to estimate $\bgamma$.
These two sets have sizes $n_\theta := |\cI_\theta|$, $n_\gamma := |\cI_\gamma|$ and may overlap or even coincide..

A natural approach to estimating $\beta$ is to use an empirical version of the identifying equation Eq.~\eqref{eq:pop-cor} with plug-in estimates for $\hat \btheta,\hat \bgamma$: 
\begin{equation}\label{eq:DefNaive}
    \hbeta_{\mathrm{naive}} = \frac1{\hat \kappa^2} \frac1{ n_\theta } \sum_{i \in \cI_\theta } (y_i - \bx_i^\top \hat \btheta )(w_i - \bx_i^\top \hat \bgamma).
\end{equation}
Here $\hat \kappa$ is a consistent estimator of $\kappa$ (we will leave this unspecified for the moment),
and $\hat \btheta$, $\hat \bgamma$ are estimates of $\bar \btheta$, $\bgamma$ coming from samples $\cI_\theta$, $\cI_\gamma$, respectively.
Although our discussion up to this point has focused on the Lasso estimator,
which has received the most focus in the debiasing literature,
we will also develop theory for ridge regression.
In particular, we consider using estimators
\begin{equation}\label{eq:regressions}
\begin{gathered}
    \hat \btheta 
        :=
        \argmin_{\bpi}
            \Big\{
                \frac1{2n_\theta}\| \by_{ \cI_\theta } - \bX_{ \cI_\theta } \bpi \|_2^2 + \Omega_\theta(\bpi)
            \Big\},
    \\
    \hat \bgamma
        :=
        \argmin_{\bpi}
            \Big\{
                \frac1{2 n_\gamma }\| \by_{ \cI_\gamma } - \bX_{ \cI_\gamma } \bpi \|_2^2 
                + 
                \Omega_\gamma(\bpi)
            \Big\},
\end{gathered}
\end{equation}
where the penalty is, for $\sx\in\{\theta,\gamma\}$,
\begin{equation}\label{eq:penalty}
    \Omega_{\sx}(\bpi)
        :=
        \begin{cases}
            \sqrt{\frac{p}{n_\sx}} \frac{\lambda_{\sx}}{2}  \| \bpi \|_2^2 \quad & \text{for ridge regression},\\
            \frac{\lambda_{\sx}}{\sqrt{n_\sx}} \| \bpi \|_1 \quad & \text{for Lasso},
        \end{cases}
\end{equation}
and $\by_{\cI} \in \reals^{|\cI|}$ is the vector of outcomes corresponding to those units in a set $\cI$,
and $\bX_{\cI} \in \reals^{|\cI| \times p}$ is a matrix with rows corresponding to those units in $\cI$.
When $n_\gamma > p$,
we may set $\lambda_\gamma = 0$, which gives the least-squares estimator for $\bgamma$.

The estimator $\hbeta_{\mathrm{naive}}$ can be biased due to errors in estimating $\bar \btheta$ and $\bgamma$.
For instance,  when both $n_{\theta} ,n_{\gamma}>p$, and $\bar\btheta,\bgamma$ are estimated using least squares,
elementary linear regression theory implies that $\hbeta_{\mathrm{naive}}$ has a bias that is of order
$p/n_{\theta}$.  In the case of the Lasso, the bias is of order $\tilde{O}\big((s_\theta/n_\theta) \vee (s_\gamma/n_\gamma)\big)$
(neglecting logarithmic factors). 
Figure \ref{fig:SweepLamb} illustrates this phenomenon on two synthetic data distributions (see Section \ref{sec:Simulations} for definitions
and more extensive simulations).
In this case $\cI_{\theta}\subseteq \cI_{\gamma}$: we use all the available samples for the precision regression, and all the labeled samples
for the outcome regression. Further, $n_{\theta}<p$ while $n_{\gamma}>p$, and we use the Lasso  for the outcome regression and least
squares for the precision regression.  We notice that both the Lasso estimate
$\hbeta_{\mathrm{Lasso}}$ and the naive debiased Lasso $\hbeta_{\mathrm{naive}}$ are strongly biased. Indeed the bias is
of the same order as the true coefficient size.
In Section \ref{sec:why-bias}, we dicuss the sources of bias for the naive estimator, and its relation to the correlation in the errors in estimating $\bar \btheta$ and $\bgamma$.

It convenient to focus on the proportional regime in which $n_{\theta} ,n_{\gamma}\asymp p$, and (in the case of sparse vectors)
$s_{\theta} ,s_{\gamma}\asymp p$ as well. In this case, the naive debiased estimator $\hbeta_{\mathrm{naive}}$  has a
bias that does not vanish asymptotically.

The CAD estimator is a correction to the na\"ive estimator which removes the major sources of bias of $\hbeta_{\mathrm{naive}}$.
The correction to the na\"ive estimator involves the \emph{degrees of freedom} of each regression method, 
defined for $\sx \in \{\theta,\gamma\}$
\begin{equation}
    \Tr\Big(\bX_{ \cI_\sx }\frac{\de \hat \btheta(\by_{ \cI_\sx })}{\de \by_{ \cI_\sx }}\Big).
\end{equation}
For the regression methods above, this takes the explicit form
\begin{equation}
\label{eq:hat-df}
    \hat \df_{\sx} 
        :=
        \begin{cases}
            p \quad & \text{for least squares,}\\
            \Tr\big(\big(\frac1{n_{\sx}}\bX_{\cI_\sx}^\top\bX_{\cI_\sx} + \sqrt{\frac{p}{n_\sx}}\,\lambda_{\sx} \id_p \big)^{-1}\frac1{n_{\sx}}\bX_{\cI_\sx}^\top \bX_{\cI_\sx}\big) \quad & \text{for ridge regression,}\\
            \| \hat \btheta \|_0 \quad &\text{for Lasso.}
        \end{cases}
\end{equation}
The na\"ive estimator $\hbeta_{\mathrm{naive}}$ involves the empirical correlation between unadjusted regression residuals.
It is intuitively clear that these residuals are smaller than in the population when $i\in\cI_{\theta}$ (for the outcome regression) or
$i\in\cI_{\gamma}$ (for the precision regression). 
This effect can be compensated by defining the following degrees of freedom adjusted residuals.
\begin{equation}\label{eq:dof-adjusted-residual}
    r^\theta_i 
        =
        \begin{cases}
            y_i - \bx_i^\top \hat \btheta \quad & \text{if } i \not \in  \cI_\theta ,\\
            \frac{y_i - \bx_i^\top \hat \btheta}{1 - \hat \df_\theta/n_\theta} \quad & \text{if } i \in  \cI_\theta,
        \end{cases}
    \;\;\;\;\;\;\;\;
    r^\gamma_i 
        =
        \begin{cases}
            w_i - \bx_i^\top \hat \bgamma \quad & \text{if } i \not \in  \cI_\gamma ,\\
            \frac{y_i - \bx_i^\top \hat \bgamma}{1 - \hat \df_\gamma/ n_\gamma } \quad & \text{if } i \in  \cI_\gamma .
        \end{cases}
      \end{equation}
      Define
\begin{equation}\label{eq:DefNaive2}
    \hbeta_{\mathrm{naive}}'
        = 
        \frac1{\hat \kappa^2} \frac1{ n_\theta }
        \sum_{i\in \cI_\theta} r_i^\theta r_i^\gamma.
      \end{equation}
      The suggestion of replacing $y_i - \bx_i^\top \hat \btheta$ by  $r_i^\theta$ has already appeared in the literature
      \cite{javanmardMontanari2014,bellec2020debiasing,celentano2020lasso}, under the name of `degrees of freedom adjustement'.
      However, earlier work assumes $\bgamma$ known, and adjusted only the outcome residuals.
      If only the outcome residuals were adjusted in the definition of $\hbeta_{\mathrm{naive}}'$,
      we would arrive at certain earlier forms of the debiased Lasso with degrees-of-freedom adjustment, as indicated by Eq.~\eqref{eq:db-cor-form}.
        As anticipated, we instead treat
      the uncertainty in the outcome and precision model on the same footing, applying the degrees of freedom adjustment to the residuals from both.

The estimate $\hbeta_{\mathrm{naive}}'$ corrects some---but not all---sources of bias in $\hbeta_{\mathrm{naive}}$.
This is confirmed by Figure \ref{fig:SweepLamb}: $\hbeta_{\mathrm{naive}}'$ has a much smaller bias than
$\hbeta_{\mathrm{naive}}$ or the Lasso estimate $\hbeta$ in both data generating models.
However, the bias is non-vanishing in both cases.
The debiased estimator $\hbeta_{\cad}$, which is the main focus on this paper, will correct the sources of bias in $\hbeta_{\mathrm{naive}}'$ as well.

In order to correct for the bias of $\hbeta_{\mathrm{naive}}'$, it is useful to introduce the debiased outcome model
\begin{equation}
\label{eq:db-def}
    \hat \btheta^\de
        :=
        \hat \btheta + \frac{\bSigma^{-1}\bX_{ \cI_\theta }^\top(\by_{ \cI_\theta } - \bX_{ \cI_\theta } \hat \btheta)}{n_\theta - \hat \df_\theta},
\end{equation}
and similarly define $\hat \bgamma^\de$. Note that, if $\hat \btheta$ is the least squares estimator, then
  $\hat \btheta^\de=  \hat \btheta$ (and similarly for $\hat \bgamma$).
  Finally, we can provide the correlation adjusted debiased (CAD) estimator, which gives a nearly-unbiased estimate of $\beta$:
\begin{equation}\label{eq:cad-estimate}
    \hbeta_{\cad} 
        =
        \hbeta_{\mathrm{naive}}'
        +
        \frac{ n_{\theta\gamma} /p}{(n_\theta/p)( n_\gamma /p)}
        \Big(
            1 - \frac{\hat \df_\gamma}{p} - \frac{\hat \df_\theta}{p}
        \Big) \hbeta_{\mathrm{naive}}' 
        - 
        \frac{\< \hat \btheta^\de - \hat \btheta , \hat \bgamma^\de - \hat \bgamma \>_{\bSigma}}{\hat \kappa^2 },
      \end{equation}
      where $n_{\theta\gamma} = |\cI_\theta \cap \cI_\gamma|$,
      and we use the noise estimate (which we will see is consistent)
    \begin{equation}\label{eq:kappa-est}
        \hat \kappa^2 
            := 
            \Big(1 + \frac{p}{n_\gamma} - 2\frac{\hat \df_\gamma}{n_\gamma}\Big)\frac{\| \by_{\cI_\gamma} - \bX_{\cI_\gamma} \hat \bgamma \|_2^2}{n_\gamma(1 - \hat \df_\gamma/n_\gamma)^2}
            -
            \| \hat \bgamma^{\de} - \hat \bgamma \|_{\bSigma}^2.
    \end{equation}
      The second and third terms in Eq.~\eqref{eq:cad-estimate} do not appear in earlier work and are entirely due to the inaccurate estimation of $\bgamma$:
      indeed, they vanish if we let $n_{\gamma}\to\infty$, in which case $\hat\bgamma, \hat \bgamma^\de \to\bgamma$.
  Importantly, these terms do not only depend on the size of the error in estimating $\bgamma$ but also on its correlation with the error in estimating $\btheta$.
    This depends in a non-trivial way on the estimation method used and the extent to which $\cI_\theta$ and $\cI_\gamma$ overlap.
    Because these terms correct for the bias induced by the correlated errors in estimating $\bgamma$ and $\bar \btheta$, we refer to them as \emph{correlation adjustments}. (See Section \ref{sec:why-bias} for more discussion on this point).

      Although any consistent estimate $\hat \kappa$ of $\kappa$ can be used in Eqs.~\eqref{eq:DefNaive2} and \eqref{eq:db-def} without impacting the consistency of $\hbeta_{\cad}$,
      we will use the estimate $\hat \kappa$ provided in Eq.~\eqref{eq:kappa-est} throughout our analysis.
      It is a generalization of the noise estimate provided in \cite{bayatiErdogdu2013} (see discussion in Section \ref{sec:gen-thm}).
      In the case that $\hat \bgamma$ is the least-squares estimator,
      this estimate takes a simple form.
      One can then check that $\hat \bgamma^\de = \hat \bgamma$.
        Recalling that in this case $\hat \df_\gamma = p$,
        we see that $\hat \kappa^2 = \| \by_{\cI_\gamma} - \bX_{\cI_\gamma} \hat \bgamma\|_2^2 / (n_\gamma(1-p/n_\gamma)^2)$, which is the standard noise variance estimate for least-squares.
        In this case, computing $\hat \kappa^2$ does not require knowledge of $\bSigma$.

      Note also that the first two terms in Eq.~\eqref{eq:cad-estimate} and the first term in Eq.~\eqref{eq:kappa-est} can be evaluated from data, while for the third term in Eq.~\eqref{eq:cad-estimate} and second term in Eq.~\eqref{eq:kappa-est} this is not always the case,
      since they require knowledge of the correlation of the nuisance $\bSigma$. 
      Nevertheless, there are at least two important
      cases in which   $\hbeta_{\cad}$ is a practical statistical estimate: $(1)$ When least squares is used for the precision model in which case $\hat \bgamma^\de = \hat \bgamma$; $(2)$ When $\bSigma$
      is known or can be estimated, in which case the third term in Eq.~\eqref{eq:cad-estimate} and second term in Eq.~\eqref{eq:kappa-est} can be estimated as well. 
      While the second scenario might appear to
      bring us back to a similar assumption as the one we want to avoid (knowledge of $\bgamma$), it is not quite the same,
      as discussed below.
      Also, we expect that an inaccurate estimate of $\bSigma$  in Eq.~\eqref{eq:cad-estimate} will have a smaller effect than an inaccurate estimate of $\bgamma$, but leave this point to future investigation.

      \subsection{Main result and discussion}
      \label{sec:MainResult}

We postpone a complete statement of the assumptions in our main result to Section \ref{sec:assumptions}.
Our main result is the following.
\begin{theorem}\label{thm:debiased-estimate}
    Under assumptions \textsf{A1} and \textsf{A2},
    there exist constants $c_0,C_0>0$ depending uniquely on the constants in those assumptions (and, in particular, independent of
    $\beta,\btheta,\bgamma$), such that
    \begin{equation}
    \begin{gathered}
      \P\Big(| \hbeta_{\cad} - \beta| > \sqrt{\frac{p^{1-c_0}}{ n_\theta  \wedge n_\gamma }}\Big) \leq \frac{C_0}{p^2}\, .
    \end{gathered}
  \end{equation}
  \end{theorem}
  \begin{remark}
    This theorem implies in particular that $\hbeta_{\cad}$  is a consistent estimate of $\beta$ in the proportional asymptotics,
    despite $\bgamma$ being unknown.
    Namely, for any sequence of problems with $p\to\infty$, and both $n_\theta,s_\theta \ge \delta_0 p$ and $n_\gamma \ge \delta_0 p$ (for some constant $\delta_0>0$), we have
  \begin{align}
    | \hbeta_{\cad} - \beta|\stackrel{\text{a.s.}}{\longrightarrow} 0\, .
  \end{align}
  This occurs even when $s_\theta = \Omega(n_\theta)$ and $s_\gamma = \Omega(n_\gamma)$ as well.
  As discussed above,  the same is not achieved by earlier debiasing procedures.
\end{remark}

\begin{remark}
  We do not expect the upper bound on the size of $| \hbeta_{\cad} - \beta|$ in Theorem \ref{thm:debiased-estimate}
  to be optimal: namely we expect $\hbeta_{\cad} - \beta$ to be significantly smaller than  $O(\sqrt{p^{1-c_0}/n})$.
  As a consequence, we do not propose to use Theorem \ref{thm:debiased-estimate} as a basis to construct confidence intervals, which would be overly conservative.

  Nevertheless, in Section \ref{sec:Simulations} we provide simulation evidence that the CAD estimator is approximately normal and centered on the true parameter. This suggests that developing inference procedures based on $\hbeta_{\cad}$
  is a promising avenue for future research.
\end{remark}

\begin{remark}
It is instructive to compare Theorem \ref{thm:debiased-estimate} with earlier results for the Lasso or debiased Lasso.
 Tight  coordinatewise bounds for the Lasso estimator $\hbeta$ were proved in \cite{Lounici2008}. These imply 
$|\hbeta-\beta| = O(\sqrt{\log p/n})$ but assume an incoherence condition that can only hold in the very sparse regime $s_\theta\ll \sqrt{n}$. 
The best bounds for the debiased Lasso (with degrees of freedom adjustement) yield \cite{bellec2020debiasing},
\begin{align}
|\hbeta^{\de}-\beta| \lesssim \frac{(s_\theta \wedge s_\gamma)\log p}{n}\wedge \|\bgamma\|_1\sqrt{\frac{\log p}{n}}+\frac{1}{\sqrt{n}}\, .
\end{align}
Considering for instance the case in which $\bgamma$ has $s_{\gamma}$ non-zero entries of magnitude $1/\sqrt{s_{\gamma}}$, the second term dominates.
The bound then reduces to $|\hbeta^{\de}-\beta| \lesssim \sqrt{s_{\gamma}\frac{\log p}{n}}$ which is larger than one if $s_{\gamma}\asymp p\asymp n$ (even if $n>p$).
In contrast, our approach yields  $|\hbeta_{\cad} - \beta|=o_p(1)$ with an estimator that can be computed from data as soon as $n_{\gamma}>p$.
\end{remark}

We finally compare our upper bound with lower bounds available in the literature, for the case of unknown precision matrix.
As proven in \cite{caiGuo2017,javanmard2018}, 
in the fully supervised setting $\cI_{\gamma}=\cI_{\theta}$ ($n_{\theta}=n_{\gamma}=:n$), and if
$s_\theta \lesssim  p^\eta \wedge (n/\log p)$ for some $\eta < 1/2$,
no estimator of $\beta$ achieves worst-case error $o_p\big( (\sigma/\kappa)\big(1/\sqrt{n} +  (s_\theta \wedge s_\gamma) \log p / n \big) \big)$ over the parameter class (see \cite[Proposition 4.2]{javanmard2018})
\begin{equation}
    \Big\{
        (\kappa,\sigma,\beta,\bgamma,\btheta,\bSigma)
        \Bigm|
        \kappa \geq c_1,\, 
        \sigma \leq c_2,\,
        \| \bgamma\|_2 \leq c_3,\,
        \|\bgamma\|_0 \leq s_\gamma,\,
        \|\btheta\|_0 \leq s_\theta,\,
        c_3\id \preceq \bSigma \preceq c_4\id
    \Big\}.
  \end{equation}
  Theorem  \ref{thm:debiased-estimate} guarantees that $\hbeta_{\cad}$ achieves an error that is significantly smaller than this
  lower bound, provided
  $ (s_\theta \wedge s_\gamma)^2  \gg n p^{1-c'_0}$, which is for instance the case in the proportional asymptotics.

  One might think that the improvement of our estimator relative to this minimax lower bound results from our use of
  $n_{\gamma}>p$ unlabeled samples for the precision model (if we use least squares for the latter), or from our assumption that
  $\bSigma$ is known (if Lasso or ridge are used for the precision model).
However, this is not the case.
Indeed, inspecting the  proofs of \cite{caiGuo2017} and \cite[Proposition 4.2]{javanmard2018}  reveals that the minimax lower bound
applies also to the parameter space with restriction $\bSigma = \id$.

The reason why our upper bound does not contradict the lower bounds of \cite{caiGuo2017,javanmard2018}  is
that the latter only apply to $s_\theta = o(p^{1/2})$. 
As far as we know, proving sharp lower bounds for $s_\theta \gtrsim p^{1/2}$ remains an open problem.

\section{Why are the na\"ive estimators biased?}
\label{sec:why-bias}

To build intuition about reason the na\"ive estimators $\hbeta_{\mathrm{naive}}$ and $\hbeta_{\mathrm{naive}}'$ 
are biased, it is useful to consider a simpler estimator in which the correlation between residuals is computed on different units
from the ones used to estimate $\hat \btheta$ and $\hat \bgamma$.
Namely, consider $\cI_\theta,\cI_\gamma \subset [N]$ and $\cI\subseteq [N]\setminus (\cI_\theta\cup\cI_\gamma)$ (where, unlike above, $\cI_\theta \cup \cI_\gamma$ is a strict subset of $[N]$).
Define
\begin{equation}
    \hbeta_{\mathrm{split}}
        := 
        \frac1{\kappa^2}\frac1{|\cI|}
        \sum_{i \in \cI}(y_i - \bx_i^\top \hat \btheta)(w_i - \bx_i^\top \hat \bgamma).
      \end{equation}
For simplicity (and because we are only interested here in building intuition), we set $\hat \kappa^2 = \kappa^2$.
We compute
\begin{equation}
\begin{aligned}
    \E[\hbeta_{\mathrm{split}}\bigm| \hat \btheta,\hat \bgamma]
        &=
        \frac1{\kappa^2}\E[(\sigma z_i + \kappa \beta w_i^\perp)\kappa w_i^\perp\bigm| \hat \btheta,\hat \bgamma]
        +
        \frac1{\kappa^2}\E[(\hat \btheta - \bar \btheta)^\top \bx_i \bx_i^\top(\hat \bgamma - \bgamma)\bigm| \hat \btheta,\hat \bgamma]
    \\
        &=
        \beta 
        +
        \< \hat \btheta - \bar \btheta , \hat \bgamma - \bgamma \>_{\bSigma}/\kappa^2.
\end{aligned}
\end{equation}
The sample splitting estimator is biased when the error in estimating $\bar \btheta$ is aligned (in the sense of having a non-vanishing inner product)
with the error in estimating $\bgamma$.
There are two reasons these errors might be aligned.
The first reason is that the estimators for $\bar \btheta$ and $\bgamma$ are themselves biased---as is the case when using the Lasso or ridge regression---and these biases are non-orthogonal. 
For example, 
if the signs of $\bar \theta_i$ and $\gamma_i$ are the same and both are estimated using the Lasso,
we expect the errors $\hat \theta_i - \bar \theta_i$ and $\hat \gamma_i - \gamma_i$ to also be aligned,
even if $\hat \btheta$ and $\hat \bgamma$ are computed on disjoint samples of data (i.e., $\cI_\theta \cap \cI_\gamma = \emptyset$).
We call this source of bias \emph{correlated shrinkage bias}.
The second reason the error in estimating $\bar \btheta$ may be aligned with the error in estimating $\bgamma$ is that the estimates may come from the same or overlapping samples. 
Indeed, 
in the case $n_\gamma=n_\theta > p$ and both $\hat \btheta$ and $\hat \bgamma$ are computed with least-squares on the same units,
we do not expect the errors $\hat \btheta - \bar \btheta $ and $\hat \bgamma - \bgamma$ to be uncorrelated. 
We call this source of bias \emph{overlapping samples bias}.

Characterizing the correlation $\< \hat \btheta - \bar \btheta , \hat \bgamma - \bgamma \>_{\bSigma}$ theoretically
is a key technical contribution of the present work.
As mentioned above, this requires the joint characterization of two regularized regressions,
a type of problem that goes beyond the scope of previous work \cite{karoui2013,dobriban2018,bayati2011lasso,thrampoulidis2015,miolane2018distribution,celentano2020lasso, mondelli2020}.
More important than characterizing the correlated shrinkage bias and overlapping samples bias theoretically is correcting them empirically.
One proposal would be to use sample splitting again, computing $\hat \btheta$ and $\hat \bgamma$ with different samples.
That is,
we could take $\cI_\theta,\cI_\gamma,\cI$ be mutually disjoint.
This approach would eliminate overlapping samples bias but not correlated shrinkage bias.
Thus, sample splitting does not correct for bias in the estimation of $\beta$. 
Moreover, it comes at the cost of statistical efficiency.

One attempt to improve the estimator $\hbeta_{\mathrm{split}}$ is to instead use a leave-one-out estimator:
\begin{equation}
    \hbeta_{\mathrm{loo}}
        :=
        \frac1{\kappa^2}\frac1{n_\theta}
        \sum_{i \in \cI_\theta}(y_i - \bx_i^\top \hat \btheta^{(-i)})(w_i - \bx_i^\top \hat \bgamma^{(-i)}),
\end{equation}
where for each $i$ the estimate $\hat \btheta^{(-i)}$ is computed using all labeled units except unit $i$ (and similarly for $\hat\bgamma$, except that unlabeled units may also be used).
For a single term in this average, we have $\E[(y_i - \bx_i^\top \hat \btheta^{(-i)})(w_i - \bx_i^\top \hat \bgamma^{(-i)})/\kappa^2\bigm| \by_{-i},\bX_{-i}] = \beta + \< \hat \btheta^{(-i)} - \bar \btheta , \hat \bgamma^{(-i)} - \bgamma \>_{\bSigma}/\kappa^2$,
where $\by_{-i},\bX_{-i}$ contain the data from all units except unit $i$.
Thus, the leave-one-out estimator is also subject to correlated shrinkage bias and overlapping samples bias.
The primary motivation to use it in place of the sample-splitting estimator $\hbeta_{\mathrm{split}}$ is to improve statistical efficiency.
Indeed, the leave-one-out estimator uses more samples to estimate the nuisance parameters $\bar\btheta,\bgamma$, and it averages over more samples to compute the estimate $\hbeta_{\mathrm{loo}}$.

One drawback of the leave-one-out procedure is the computational cost of computing it:
in principle, it requires computing $2n_{\theta}$ regression estimates, which may be prohibitive if $n_\theta$ and $p$ are large.
The estimator $\hbeta_{\mathrm{naive}}'$, introduced in Eq.~\eqref{eq:DefNaive2},
attempts to approximate the leave-one-out estimator while only requiring that we compute two high-dimensional regression estimates.
It is based on the following remarkable approximation which holds for the in-sample residuals (that is, those $i$ which are used to compute $\hat \btheta$):
\begin{equation}
\label{eq:dof-loo-approx}
    y_i - \bx_i^\top \hat \btheta
        \approx
        (1 - \hat \df_\theta/n_\theta)(y_i - \bx_i^\top \hat \btheta^{(-i)}).
\end{equation}
The analogous approximation holds for the precision regression residuals.
We do not quantify the quality of this approximation in the present work, 
though similar approximations have appeared elsewhere \cite{karoui2013,bellec2020debiasing,yadlowsky2021sloe}.
Based on this approximation, the degrees-of-freedom adjustment in Eq.~\eqref{eq:dof-adjusted-residual}
is an attempt to approximate the leave-one-out residuals without computing them directly.
Without being rigorous, we expect $ \hbeta_{\mathrm{naive}}' \approx \hbeta_{\mathrm{loo}}$, where $ \hbeta_{\mathrm{naive}}'$ is much easier to compute.
As we have already mentioned,
if we were to apply the degrees-of-freedom adjustment only to the outcome residuals in the definition of $\hbeta_{\mathrm{naive}}'$,
we would arrive at certain earlier forms of the debiased Lasso with degrees-of-freedom adjustment.
Such methods can be thought of as attempts to approximate a leave-one-out estimator that uses oracle knowledge of (or a very good estimate of) $\bgamma$.

The discussion in the previous paragraph motivates why we use the degrees-of-freedom adjusted  estimator $\hbeta_{\mathrm{naive}}'$ in place of the unadjusted  estimator $\hbeta_{\mathrm{naive}}$ defined in Eq.~\eqref{eq:DefNaive}.
In particular, $\hbeta_{\mathrm{naive}}'$ corrects a bias present in the unadjusted estimator $\hbeta_{\mathrm{naive}}$, which
we call \emph{overfitting bias}.
It results from the fact that in-sample residuals---i.e., the residual $y_i - \bx_i^\top \hat \btheta$ when $(y_i,\bx_i)$ is used in the fitting of $\hat \btheta$---tend to be smaller than the residual $y_i - \bx_i^\top \hat \btheta$ when $(y_i,\bx_i)$ is not used in the fitting of $\hat \btheta$.
This phenomenon will be unsurprising to many statisticians, and is consistent with Eq.~\eqref{eq:dof-loo-approx}.
Although $\hbeta_{\mathrm{naive}}'$ corrects the overfitting bias present in the unadjusted estimator $\hbeta_{\mathrm{naive}}$,
it is still subject to correlated shrinkage bias and overlapping samples bias.
Indeed, it approximates the leave-one-out estimator $\hbeta_{\mathrm{loo}}$, which is subject to these biases as well. 

The primary innovation of the correlation adjusted debiasing estimate ---and the reason for its name--- is to correct for the bias in $\hbeta_{\mathrm{naive}}'$ due to the correlation $\< \hat \btheta - \bar \btheta , \hat \bgamma - \bgamma \>_{\bSigma}$.
Importantly, CAD is a fully empirical procedure that does not require any prior knowledge of either $\bgamma$ or $\bar \btheta$.
In the case that least-squares is used for the precision regression (provided $n_\gamma > p$),
it also does not require any prior knowledge of $\bSigma$. 
We emphasize that even when least-squares is used for the precision regression,
the overlapping samples bias need not be 0,
so that correcting for it in this case is still non-trivial.
Because we do not resort to sample splitting, we correct the bias without sacrificing statistical efficiency.
Although the discussion in this section has been non-rigorous, Theorem \ref{thm:debiased-estimate} shows that the degrees-of-freedom and correlation adjustments to the naive estimate $\hbeta_{\mathrm{naive}}$ together achieve the desired goal of
removing all non-vanishing bias.

We summarize the adjustments which enter the construction of the CAD estimator and the biases they correct for:
\begin{enumerate}

    \item 
    \textbf{Degrees-of-freedom adjustment.} 
    The degrees-of-freedom adjustment scales in-sample residuals by $(1 - \hat \df_{\sx}/n_{\sx})^{-1}$ (see Eq.~\eqref{eq:dof-adjusted-residual}).
    Its purpose is to correct for overfitting bias.
    Overfitting bias occurs because in-sample residuals behave differently than out-of-sample residuals; in particular, they are shrunk towards zero. 
    There is a nearly deterministic relationship between in- and out-of-sample residuals,
    with the former being $(1 - \hat \df_{\sx}/n_{\sx})$ times the latter.
    The degrees-of-freedom adjustment applied to the outcome model has appeared in previous work \cite{javanmardMontanari2014,bellec2020debiasing,bellec2020second}.

    \item 
    \textbf{Correlation adjustment.}
    The correlation adjustment refers to the second and third terms in Eq.~\eqref{eq:cad-estimate},
    and it corrects for bias due to the correlation of the estimation errors $\< \hat \btheta - \bar \btheta , \hat \bgamma - \bgamma \>_{\bSigma}$.
    This correlation results from correlated shrinkage bias and overlapping samples bias.
    Correcting for it is the primary innovation of the current work.

    Correlated shrinkage bias occurs when the bias of $\hat \btheta$ and $\hat \bgamma$ are both non-zero and aligned with each other. 
    Because least-squares regression is unbiased, 
    correlated shrinkage bias only occurs when both $\hat \btheta$ and $\hat \bgamma$ are fit using a penalized procedure.
	Note that when either $\hat \btheta$ or $\hat \bgamma$ is fit with least-squares, the third term in Eq.~\eqref{eq:cad-estimate} disappears.
   This indicates that the final term in Eq.~\eqref{eq:cad-estimate} is the only term involved in correcting correlated shrinkage bias.

    Overlapping samples bias occurs because the same or overlapping samples are used to estimate $\bar \btheta$ and $\bgamma$.
    When there is no overlap between the units used to estimate the outcome model and those used to estimate the precision model, 
    the second term in Eq.~\eqref{eq:cad-estimate} disappears because $n_{\theta\gamma} = |\cI_\theta \cap \cI_\gamma| = 0$.
    This indicates that the second term is primarily involved in correcting overlapping samples bias.

\end{enumerate}

\section{Assumptions}
\label{sec:assumptions}

We state the formal assumptions for our main result, Theorem \ref{thm:debiased-estimate}. In what follows, we denote by
$\mu_1(\bA)\ge \mu_2(\bA)\ge \dots\ge\mu_n(\bA)$ the ordered eigenvalues of the symmetric matrix $\bA\in\reals^{n\times n}$.

The assumptions we will use when studying the Lasso estimator are exactly the assumptions used in \cite{celentano2020lasso}, which we repeat here. 
A vector $\bt \in \reals^p$ is said to $(\bx,M)$-\emph{approximately sparse} for $\bx \in \{-1,0,1\}^p$ and $M > 0$ if there exists $\bar \bt \in \reals^p$ with $\| \bar \bt - \bt \|_1 / \sqrt{p} \leq M$ and $\bx \in \sign(\bar \bt)$, where the sign is taken in an entry-wise manner: $\sign(x) = 1$ for $x > 0$, $\sign(x) = -1$ for $x < 0$, and $\sign(x) = \{-1,0,1\}$ for $x = 0$.

The structural assumption we will make depends on a certain functional notion of Gaussian width provided in \cite{celentano2020lasso}.
We consider the probability space $(\reals^p,\cB,\gamma_p)$ where $\cB$ is the Borel $\sigma$-algebra and $\gamma_p$ is the standard Gaussian measure in $p$ dimensions. 
Let $L^2$ be the space of functions $\bmf : \reals^p \rightarrow \reals^p$ which are square integrable in $(\reals^p,\cB,\gamma_p)$, and denote by $\< \bmf_1,\bmf_2\>_{L^2} = \E[\< \bmf_1(\bxi),\bmf_2(\bxi)\>_{L^2}]$ the inner product on this space, where $\bxi \sim \normal(0,\id_p)$. Let $\| \bmf \|_{L^2}$ be the norm induced by this inner product. 
For $\bx \in \{-1,0,1\}^p$ and $\bSigma \in \SS_+^p$, define
\begin{equation}
    F(\bv;\bx,\bSigma)
        :=
        \< \bx ,\bSigma^{-1/2} \bv\> + \| ( \bSigma^{-1/2} \bv)_{S^c} \|_1
        \quad \text{for }S := \mathsf{supp}(\bx),
\end{equation}
and define the \emph{functional Gaussian width}
\begin{equation}
    \cG(\bx,\bSigma) 
        := 
        \sup
        \Big\{
            \frac1{\sqrt{p}}\<\bv,\bxi\>_{L^2}:
            \;\;
            \bv \in L^2,
            \;\;
            \| \bv\|_{L^2} \leq 1,
            \;\;
            \E[F(\bv;\bx,\bSigma)]\leq 0
        \Big\},
\end{equation}
where $\bxi$ denotes the identity function of $L^2$.
Finally, we can provide the notion of approximate sparsity used in our assumptions for the Lasso.
\begin{definition}
    We say $\bt$ is $(s,W,M)$-approximately sparse for $W > 0$ and $s \in \ZZ_{>0}$ if there exists $\bx \in \{-1,0,1\}$ wuch that $\bt$ is $(\bx,M)$-approximately sparse, $\| \bx \|_0 = s$, and $\cG(\bx,\bSigma)\leq W$.
\end{definition}
\noindent We refer the reader to \cite{celentano2020lasso} for a more detailed discussion of the functional Gaussian width, its relation to other notions of Gaussian width, and references to the literature.

\begin{description}

    \item{\textsf{A1.}} The population covariance has eigenvalues bounded by $0 < \mu_{\min} \leq \mu_j(\bSigma) \leq \mu_{\max} < \infty$, 
    the noise variances are bounded by $ 0 < \kappa_{\min}^2 \leq \kappa^2 \leq \kappa_{\max}^2 < \infty$, $0 < \sigma_{\min}^2 \leq \sigma^2 \leq \sigma_{\max}^2 < \infty$,
    and the parameter of interest is bounded by $|\beta| \leq \beta_{\max} < \infty$.

    \item{\textsf{A2.}} 
    Depending on the regression method used, we make the following assumptions (for either $\sx \in \{\theta,\gamma\}$):
    \begin{itemize}

        \item 
        For least-squares regression, we assume $ n_{\sx} /p \geq \delta_{\mathrm{min}} > 1$.

        \item 
        For ridge regression, we assume the parameter has bounded $\ell_2$-norm $\| \bar \btheta \|_2^2 \leq r_{\max}^2$ or $\| \bgamma \|_2^2 \leq r_{\max}^2$,
        regularization parameter
        $0 \leq \lambda_{\min} \leq \lambda_{\sx} \leq \lambda_{\max} < \infty$, and $ n_{\sx} /p \geq \delta_{\mathrm{min}} > 0$.
        We assume $(\delta_{\mathrm{min}}-1) \vee \lambda_{\min} > 0$.

        \item 
          For the Lasso, we assume the parameter $\bar \btheta$ (resp.\ $\bgamma$) is $(s,\sqrt{ n_{\sx} /p}(1-\Delta_{\min}),M)$-approximately sparse for some $s/p > \nu_{\min} $, $\Delta_{\min} \in (0,1)$, and $ n_{\sx} /p \in [\delta_{\min},\delta_{\max}]$ (for $\sx\in\{\theta,\gamma\}$) for
          $0 < \delta_{\min} \leq \delta_{\max} < \infty$,
        and $0 < \lambda_{\min} < \lambda_{\sx} < \lambda_{\max} < \infty$.

    \end{itemize}
    (We emphasize that each bullet point applies only to the model which uses that regression method. For example, if we use least squares for the precision regression and the Lasso for the outcome regression, we assume $n_\gamma/p \geq \delta_{\min}$ and $n_\theta/p \in [\delta_{\min},\delta_{\max}]$.)
\end{description}
We denote the collection of model parameters appearing in assumptions \textsf{A1} and \textsf{A2} by $\cPmodel$ and $\cPregr$, respectively.
That is, 
\begin{equation}
    \cPmodel = (\mu_{\min},\mu_{\max},\kappa_{\min},\kappa_{\max},\sigma_{\min},\sigma_{\max},\beta_{\max})
    \text{ and } \cPregr = (\delta_{\min},\delta_{\max}, \lambda_{\min},\lambda_{\max}, \Delta_{\min}, M,\nu_{\min}).
\end{equation}

\section{Numerical simulations}
\label{sec:Simulations}

We carried out a numerical comparison of  various debiasing methods for synthetic data under a few data distribution models. We
focus on the case $\cI_{\theta}\subseteq \cI_{\gamma}$, with $n_{\gamma}>p$, $n_{\theta}<p$, so that we can evaluate
$  \hbeta_{\cad}$ using least squares for precision regression.
In this secion, we denote the Lasso regularization parameter in the outcome model as $\lambda := \lambda_\theta$.
Below, we define the data distributions used in simulations in terms of the joint covariance $\bSigma_+\in\reals^{\op\times \op}$,
$\op=p+1$, of $(w_i,\bx_i)$:

\vspace{0.1cm}

\noindent {\bf Circulant}. The precision matrix $\bSigma_+^{-1}$ is a circulant matrix, namely
$(\bSigma_+^{-1})_{ij}= s(i-j)$ where it is understood that $s(-k) = s(k)$ and $s(\op+k) = s(k)$ for all $k$.
(Note that, as a consequence, $\bSigma_+$ is circulant as well.) As a special example, we consider the following  choice:
\begin{align}
  s(0) =1\, ,\;\;\; s(1)=0.4\, ,\;\;\; s(2)=0.3\, ,\;\;\; s(3)=0.2\,\;\;\; s(k)=0\,\;\forall k\in\{4,\dots, \op-1\}\, ,\label{eq:CircSpecial}
\end{align}
together with the following choice for the parameters vector (with $s_0$ even):
\begin{align}
  &  \beta= \mu\,,\;\;\;\; \theta_1=\dots=\theta_{s_0/2-1}=\mu\,,\;\;\;\; \theta_{s_0/2}=\dots=\theta_{s_0-1}=-\mu\, , \;\;\;\; 
  \theta_{i}=0 , \;\;\; \forall i\ge s_0\, .
\end{align}

\vspace{0.1cm}

\noindent {\bf Block model}. The covariance has a block structure with block comprising $\lb$ equicorrelated features.
Namely, we denote by $b(i) = \lfloor i/\lb\rfloor$ the block index for the $i$-th covariate, and set
\begin{align}
  (\bSigma_+)_{ij}=\begin{cases}
    1& \mbox{ if $i=j$,}\\
    q/\lb & \mbox{ if $b(i)=b(j)$ but $i\neq j$,}\\
    0& \mbox{ otherwise.}
  \end{cases}\label{eq:BlockCovariance}
\end{align}
Futher, for $\mu\in \reals$ and $s\in\{+1,-1\}$, we consider the following model for the coefficients vector
\begin{align}
  &  \beta= \mu\,,\;\;\;\; \theta_1=\dots=\theta_{s_0-1}=-\mu\, , \;\;\;\; 
  \theta_{i}=0 , \;\;\; \forall i\ge s_0\, .\label{eq:TruthBlock}
\end{align}
This data distribution is parametrized by the tuple $(q,\lb,\mu,s_0)$.

\vspace{0.1cm}

For each of these data distributions, we compute various estimators of the parameter of interest $\beta$.
Figure \ref{fig:SweepLamb} illustrates the dependence on the regularization parameter $\lambda$. The emerging scenario
is fairly robust with respect to the choice of $\lambda$: the CAD estimator   $\hbeta_{\cad}$ is nearly unbiased over a broad range of values of $\lambda$,
while all other estimators are biased. The dispersion of these estimators on the other hand depends on the choice of $\lambda$.

We then  focus on two specific procedures to choose $\lambda$, described below. These correspond to a semi-realistic scenario in which $\lambda$
is not selected following theoretical prescriptions (e.g.~the `universal' value   $\lambda =\lambda_{\MM}:=\sigma\sqrt{2\log p}$),
and is instead selected to optimize reconstruction on an hold-out set.

Let us emphasize that these procedures are not entirely realistic, in that they makes use of a large hold-out dataset and have access to
the unknown parameter $\btheta$, to select the optimal $\lambda$. As such, these should not be regarded as practical methods for selecting $\lambda$,
but rather as `oracle' choices that we use for comparing different debiasing methods in simulations.

\vspace{0.1cm}

\noindent{\bf Model-size based selection.} We generate $m_{\orac}$ datasets $\by^{(i)},\bX^{(i)}$, $i\le m_{\orac}$ from the
same distribution as the original data. In particular, $\by^{(i)}\in\reals^{N}$, $\bX^{(i)}\in\reals^{N\times p}$. For each of
these datasets,
we fit the Lasso for a grid of values of the regularization parameter $\lambda\in\Lambda:=\{\lambda_k=k\lambda_{\MM}/10:\, k\le 100\}$,
where $\lambda_{\MM}$ is the `universal' regularization value.
We let $\hbtheta^{(i)}(\lambda)$ be the Lasso parameter estimate for regularization $\lambda$, 
\begin{align}
  \hlambda^{(i)} := \min\big\{\lambda\in\Lambda:\;\;\; \|\hbtheta^{(i)}(\lambda)\|_0\le\|\btheta\|_0\big\}\, .
\end{align}
We then set $\lambda$ equal to the (empirical) $90\%$ percentile of the set of values $(\hlambda^{(i)})_{i\le m_{\orac}}$.

This choice of $\lambda$ mimics the behavior of a parsimonious statistician who chooses a value
of $\lambda$ that produces an approximately correct model size.

\vspace{0.1cm}

\noindent{\bf Estimation error-based selection.}
We proceed as above, except that, for each data set $\by^{(i)},\bX^{(i)}$, the corresponding optimal 
value of $\lambda$ is selected via 
\begin{align}
  \hlambda^{(i)} := \max\big\{\lambda\in\Lambda:\;\;\; \|\hbtheta(\lambda)-\btheta\|_2\le
 1.1\min_{\lambda'\in\Lambda} \|\hbtheta(\lambda')-\btheta\|_2\big\}\, .
\end{align}
As above, in our simulations we set $\lambda$ equal to the (empirical) $90\%$ percentile of the set of values
$(\hlambda^{(i)})_{i\le m_{\orac}}$.
This choice  mimicks the behavior of a statistician that tries to optimize the estimation error.

\vspace{0.1cm}

For experiments in which the precision model $\bgamma$ is also fitted using the Lasso, the same procedure is
repeated to select the corresponding $\lambda$ parameter. 

\begin{figure}[!t]
\centering
\includegraphics[width = 0.99\linewidth]{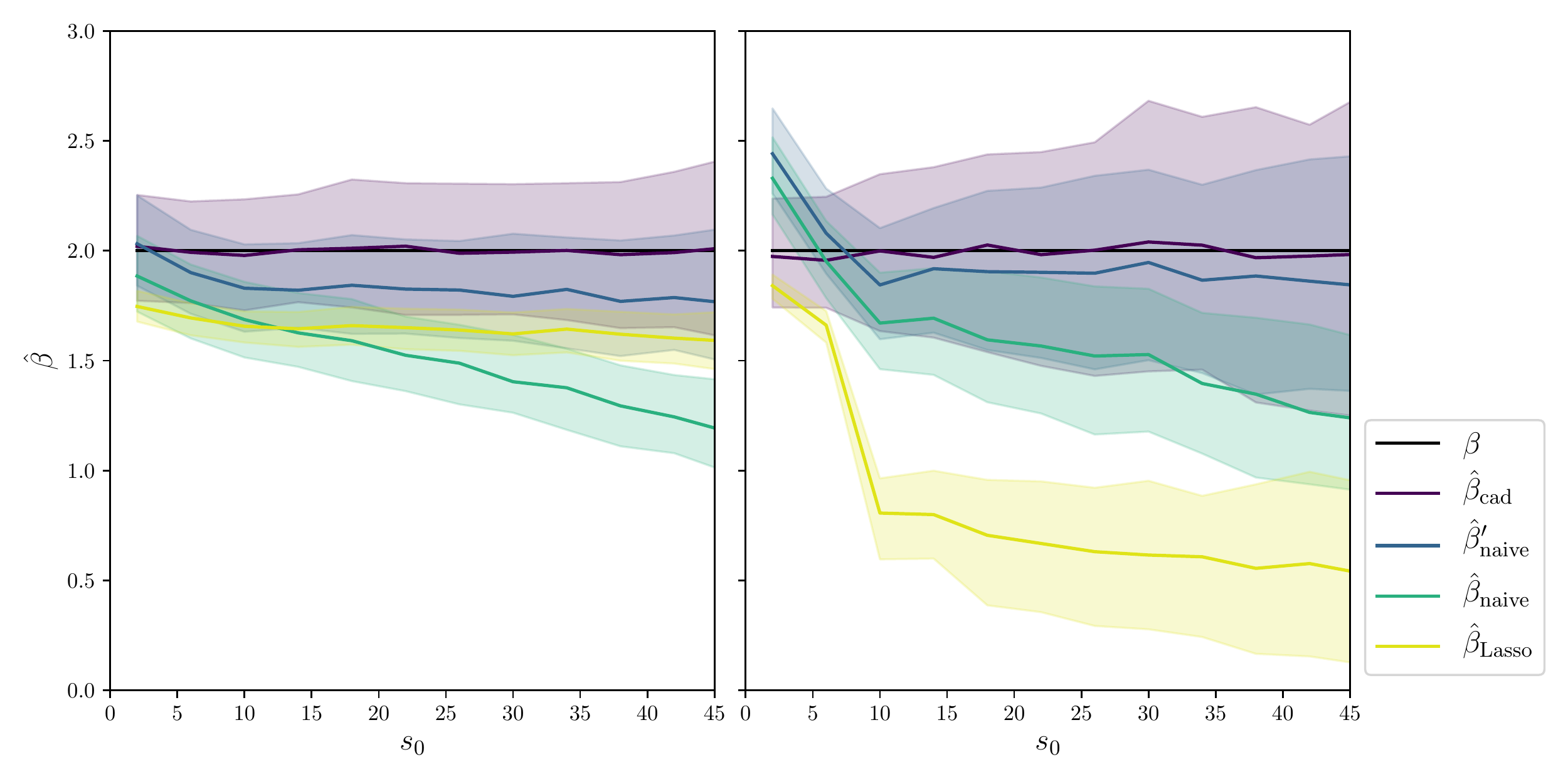}
\caption{Comparing different debiasing methods, for different data distributions. Left: block model
  with $\lb=10$, $q=1$, $\mu=2$, $\sigma=1$.
  Right: circulant model with $\mu=2$, $\sigma=1$ and inverse covariance specified in Eq.~\eqref{eq:CircSpecial}.
  In both cases, $n_{\theta} = 200$, $n_{\gamma} = 600$, $p = 400$,  and we use the  model size based selection method for $\lambda$.
  Bands correspond to $20\%$ and $80\%$ percentiles of various debiased estimates over $500$ repetitions. }
\label{fig:SweepS0_ltype0}
\end{figure}

\begin{figure}[!t]
\centering
\includegraphics[width = 0.99\linewidth]{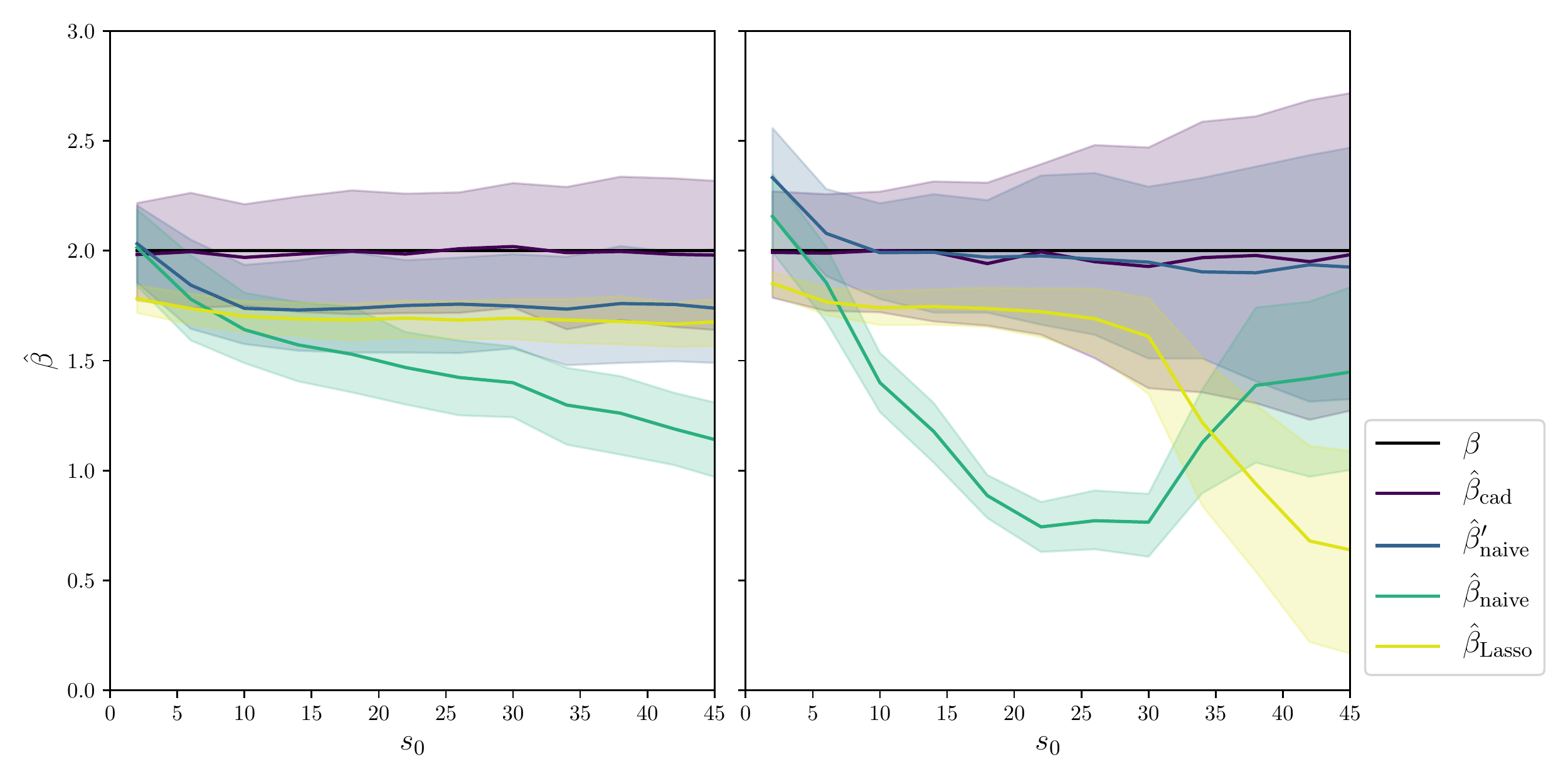}
\caption{Same as in Figure \ref{fig:SweepS0_ltype0}, except that we use the estimation error-based selection
  method for $\lambda$.}
\label{fig:SweepS0_ltype1}
\end{figure}

\begin{table}
{\small
   \begin{tabular}{|c|c|c|c|c|c|c|c|c|c|c|c|c|c|c|c|c|c|}
    \hline
     $n$ & $n_{\gamma}$ & $p$ & $s_0$ & $\bias\, /\, \std (\hbeta)$ & $\bias\, /\, \std (\hbeta_{\naive})$ & $\bias\, /\,\std(\hbeta'_{\naive})$ &
                                                                                                                                                   $\bias\, /\,\std(\hbeta_{\cad})$ 
  & $\KS(\hbeta_{\cad})$ & $\|\hbtheta\|_0$\\
    \hline
     \hline
    $   600 $ & $  1500 $ & $   450 $ & $    18 $ &$ -0.26 \, /\, 0.06$ &$ -0.96\, /\, 0.15$ &$ 0.05\, /\, 0.22$ &$ -0.002\, /\, 0.216$ &$ 0.33 $ & $ 27.4 $\\ 
$   800 $ & $  2000 $ & $   600 $ & $    24 $ &$ -0.23 \, /\, 0.05$ &$ -0.96\, /\, 0.13$ &$ 0.06\, /\, 0.19$ &$ -0.000\, /\, 0.186$ &$ 0.96 $ & $ 37.0 $\\ 
$   100 $ & $  1000 $ & $   200 $ & $    20 $ &$ -2.20 \, /\, 0.63$ &$ -1.24\, /\, 0.55$ &$ 0.12\, /\, 1.02$ &$ 0.024\, /\, 0.985$ &$ 0.10 $ & $ 25.7 $\\ 
$   100 $ & $   300 $ & $   200 $ & $    20 $ &$ -2.24 \, /\, 0.62$ &$ -2.20\, /\, 0.33$ &$ 0.33\, /\, 1.41$ &$ 0.006\, /\, 1.262$ &$ 0.93 $ & $ 23.5 $\\ 
$   100 $ & $   400 $ & $   200 $ & $    20 $ &$ -2.20 \, /\, 0.62$ &$ -1.85\, /\, 0.41$ &$ 0.25\, /\, 1.19$ &$ 0.013\, /\, 1.096$ &$ 0.23 $ & $ 26.4 $\\ 
$   200 $ & $  2000 $ & $   400 $ & $    40 $ &$ -2.11 \, /\, 0.62$ &$ -1.26\, /\, 0.52$ &$ 0.09\, /\, 0.95$ &$ -0.004\, /\, 0.922$ &$ 0.10 $ & $ 53.7 $\\ 
$   200 $ & $   500 $ & $   150 $ & $     6 $ &$ -0.30 \, /\, 0.08$ &$ -0.96\, /\, 0.20$ &$ 0.06\, /\, 0.31$ &$ -0.002\, /\, 0.302$ &$ 0.19 $ & $ 8.2 $\\ 
$   200 $ & $   600 $ & $   400 $ & $    40 $ &$ -2.10 \, /\, 0.63$ &$ -2.20\, /\, 0.32$ &$ 0.33\, /\, 1.32$ &$ 0.020\, /\, 1.192$ &$ 0.42 $ & $ 54.6 $\\ 
$   200 $ & $   800 $ & $   400 $ & $    40 $ &$ -2.12 \, /\, 0.64$ &$ -1.84\, /\, 0.40$ &$ 0.25\, /\, 1.14$ &$ 0.016\, /\, 1.051$ &$ 0.22 $ & $ 52.3 $\\ 
$   300 $ & $  1200 $ & $   600 $ & $    60 $ &$ -2.05 \, /\, 0.64$ &$ -1.84\, /\, 0.39$ &$ 0.23\, /\, 1.10$ &$ 0.003\, /\, 1.019$ &$ 0.97 $ & $ 84.9 $\\ 
$   300 $ & $   900 $ & $   600 $ & $    60 $ &$ -2.07 \, /\, 0.63$ &$ -2.21\, /\, 0.31$ &$ 0.30\, /\, 1.30$ &$ -0.011\, /\, 1.177$ &$ 0.13 $ & $ 85.1 $\\ 
$   400 $ & $  1000 $ & $   300 $ & $    12 $ &$ -0.31 \, /\, 0.07$ &$ -0.96\, /\, 0.16$ &$ 0.06\, /\, 0.24$ &$ 0.002\, /\, 0.239$ &$ 0.32 $ & $ 17.9 $\\ 
$   400 $ & $  1200 $ & $   800 $ & $    80 $ &$ -2.07 \, /\, 0.64$ &$ -2.21\, /\, 0.30$ &$ 0.31\, /\, 1.29$ &$ -0.011\, /\, 1.163$ &$ 0.19 $ & $ 110.4 $\\ 
$   400 $ & $  1600 $ & $   800 $ & $    80 $ &$ -2.05 \, /\, 0.63$ &$ -1.84\, /\, 0.39$ &$ 0.22\, /\, 1.10$ &$ -0.004\, /\, 1.015$ &$ 0.81 $ & $ 112.2 $\\ 
$   400 $ & $  4000 $ & $   800 $ & $    80 $ &$ -2.04 \, /\, 0.64$ &$ -1.24\, /\, 0.51$ &$ 0.09\, /\, 0.91$ &$ -0.002\, /\, 0.886$ &$ 0.32 $ & $ 115.5 $\\ 
$   500 $ & $  1500 $ & $  1000 $ & $   100 $ &$ -2.03 \, /\, 0.64$ &$ -2.21\, /\, 0.31$ &$ 0.33\, /\, 1.29$ &$ 0.006\, /\, 1.167$ &$ 0.91 $ & $ 142.2 $\\ 
$   500 $ & $  2000 $ & $  1000 $ & $   100 $ &$ -2.07 \, /\, 0.63$ &$ -1.84\, /\, 0.39$ &$ 0.21\, /\, 1.09$ &$ -0.017\, /\, 1.011$ &$ 0.07 $ & $ 140.2 $\\ 
$   500 $ & $  5000 $ & $  1000 $ & $   100 $ &$ -2.06 \, /\, 0.63$ &$ -1.23\, /\, 0.51$ &$ 0.09\, /\, 0.90$ &$ -0.001\, /\, 0.873$ &$ 0.89 $ & $ 138.1 $\\ 
     \hline
   \end{tabular}
 }
 \caption{Circulant model with $\mu=3$, $\sigma=1$ and inverse covariance specified in Eq.~\eqref{eq:CircSpecial}.
   We  use the  model size based selection method for $\lambda$  and report statistics computed over $n_{\text{rep}}=10,000$
   repetitions.}\label{table:Circulant0}
 \end{table}

 \begin{table}
   {\small
   \begin{tabular}{|c|c|c|c|c|c|c|c|c|c|c|c|c|c|c|c|c|c|}
    \hline
     $n$ & $n_{\gamma}$ & $p$ & $s_0$ & $\bias\, /\, \std (\hbeta)$ & $\bias\, /\, \std (\hbeta_{\naive})$ & $\bias\, /\,\std(\hbeta'_{\naive})$ &
                                                                                                                                                   $\bias\, /\,\std(\hbeta_{\cad})$ 
  & $\KS(\hbeta_{\cad})$ & $\|\hbtheta\|_0$\\
    \hline
     \hline
     $    50 $ & $   150 $ & $   100 $ & $    10 $ &$ -1.48 \, /\, 0.74$ &$ -2.06\, /\, 0.47$ &$ 0.13\, /\, 1.56$ &$ 0.010\, /\, 1.497$ &$ 0.43 $ & $ 31.0 $\\ 
$    50 $ & $   200 $ & $   100 $ & $    10 $ &$ -1.63 \, /\, 0.73$ &$ -1.59\, /\, 0.60$ &$ 0.10\, /\, 1.34$ &$ 0.024\, /\, 1.301$ &$ 0.44 $ & $ 29.1 $\\ 
$    50 $ & $   500 $ & $   100 $ & $    10 $ &$ -1.83 \, /\, 0.72$ &$ -0.84\, /\, 0.80$ &$ 0.06\, /\, 1.17$ &$ 0.026\, /\, 1.155$ &$ 0.12 $ & $ 26.1 $\\ 
$   100 $ & $  1000 $ & $   200 $ & $    20 $ &$ -1.95 \, /\, 0.66$ &$ -0.81\, /\, 0.72$ &$ 0.04\, /\, 1.02$ &$ 0.011\, /\, 1.008$ &$ 0.03 $ & $ 38.8 $\\ 
$   100 $ & $   300 $ & $   200 $ & $    20 $ &$ -1.81 \, /\, 0.68$ &$ -2.11\, /\, 0.38$ &$ 0.16\, /\, 1.38$ &$ -0.014\, /\, 1.293$ &$ 0.15 $ & $ 46.9 $\\ 
$   100 $ & $   400 $ & $   200 $ & $    20 $ &$ -1.68 \, /\, 0.70$ &$ -1.63\, /\, 0.51$ &$ 0.08\, /\, 1.17$ &$ -0.011\, /\, 1.131$ &$ 0.11 $ & $ 55.8 $\\ 
$   200 $ & $  2000 $ & $   400 $ & $    40 $ &$ -1.61 \, /\, 0.70$ &$ -1.17\, /\, 0.54$ &$ 0.10\, /\, 0.94$ &$ 0.015\, /\, 0.911$ &$ 0.26 $ & $ 118.2 $\\ 
$   200 $ & $   600 $ & $   400 $ & $    40 $ &$ -1.76 \, /\, 0.69$ &$ -2.25\, /\, 0.30$ &$ 0.39\, /\, 1.36$ &$ 0.006\, /\, 1.200$ &$ 0.71 $ & $ 98.2 $\\ 
$   200 $ & $   800 $ & $   400 $ & $    40 $ &$ -1.84 \, /\, 0.67$ &$ -1.90\, /\, 0.38$ &$ 0.29\, /\, 1.15$ &$ 0.013\, /\, 1.047$ &$ 0.73 $ & $ 86.9 $\\ 
$   300 $ & $  1200 $ & $   600 $ & $    60 $ &$ -1.51 \, /\, 0.70$ &$ -2.06\, /\, 0.32$ &$ 0.41\, /\, 1.17$ &$ 0.024\, /\, 1.031$ &$ 0.11 $ & $ 197.1 $\\ 
$   300 $ & $   900 $ & $   600 $ & $    60 $ &$ -1.57 \, /\, 0.70$ &$ -2.35\, /\, 0.25$ &$ 0.52\, /\, 1.40$ &$ -0.008\, /\, 1.186$ &$ 0.43 $ & $ 185.4 $\\ 
$   400 $ & $  1200 $ & $   800 $ & $    80 $ &$ -1.52 \, /\, 0.70$ &$ -2.32\, /\, 0.26$ &$ 0.48\, /\, 1.36$ &$ -0.007\, /\, 1.163$ &$ 0.52 $ & $ 261.8 $\\ 
$   400 $ & $  1600 $ & $   800 $ & $    80 $ &$ -1.66 \, /\, 0.69$ &$ -2.06\, /\, 0.32$ &$ 0.37\, /\, 1.15$ &$ -0.000\, /\, 1.019$ &$ 0.60 $ & $ 221.8 $\\ 
$   400 $ & $  4000 $ & $   800 $ & $    80 $ &$ -1.64 \, /\, 0.70$ &$ -1.84\, /\, 0.34$ &$ 0.18\, /\, 0.95$ &$ 0.004\, /\, 0.898$ &$ 0.72 $ & $ 225.8 $\\ 
$   500 $ & $  1500 $ & $  1000 $ & $   100 $ &$ -1.64 \, /\, 0.70$ &$ -2.37\, /\, 0.24$ &$ 0.56\, /\, 1.38$ &$ -0.000\, /\, 1.163$ &$ 0.69 $ & $ 283.5 $\\ 
$   500 $ & $  2000 $ & $  1000 $ & $   100 $ &$ -1.56 \, /\, 0.70$ &$ -2.16\, /\, 0.29$ &$ 0.43\, /\, 1.19$ &$ -0.006\, /\, 1.034$ &$ 0.71 $ & $ 312.0 $\\ 
$   500 $ & $  5000 $ & $  1000 $ & $   100 $ &$ -1.56 \, /\, 0.70$ &$ -1.74\, /\, 0.37$ &$ 0.17\, /\, 0.94$ &$ 0.013\, /\, 0.891$ &$ 0.66 $ & $ 310.7 $\\ 
     \hline
\end{tabular}
}
\caption{Circulant model: same experiments as in Table \ref{table:Circulant0}
  except that we use the estimation error based selection method for $\lambda$.}\label{table:Circulant1}
 \end{table}

 \begin{table}
   {\small
   \begin{tabular}{|c|c|c|c|c|c|c|c|c|c|c|c|c|c|c|c|c|c|}
    \hline
     $n$ & $n_{\gamma}$ & $p$ & $s_0$ & $\bias\, /\, \std (\hbeta)$ & $\bias\, /\, \std (\hbeta_{\naive})$ & $\bias\, /\,\std(\hbeta'_{\naive})$ &
                                                                                                                                                   $\bias\, /\,\std(\hbeta_{\cad})$ 
  & $\KS(\hbeta_{\cad})$ & $\|\hbtheta\|_0$\\
    \hline
     \hline
     $    50 $ & $   150 $ & $   100 $ & $    10 $ &$ -1.78 \, /\, 0.61$ &$ -2.15\, /\, 0.34$ &$ 0.24\, /\, 1.31$ &$ -0.004\, /\, 1.204$ &$ 0.63 $ & $ 12.6 $\\ 
$    50 $ & $   200 $ & $   100 $ & $    10 $ &$ -1.33 \, /\, 0.52$ &$ -1.76\, /\, 0.44$ &$ 0.19\, /\, 1.15$ &$ 0.008\, /\, 1.077$ &$ 0.72 $ & $ 13.4 $\\ 
$    50 $ & $   500 $ & $   100 $ & $    10 $ &$ -1.43 \, /\, 0.55$ &$ -1.13\, /\, 0.57$ &$ 0.08\, /\, 0.97$ &$ 0.009\, /\, 0.942$ &$ 0.01 $ & $ 13.2 $\\ 
$   100 $ & $  1000 $ & $   200 $ & $    20 $ &$ -0.71 \, /\, 0.25$ &$ -1.19\, /\, 0.39$ &$ 0.07\, /\, 0.68$ &$ -0.007\, /\, 0.659$ &$ 0.01 $ & $ 26.6 $\\ 
$   100 $ & $   300 $ & $   200 $ & $    20 $ &$ -0.62 \, /\, 0.23$ &$ -2.18\, /\, 0.23$ &$ 0.30\, /\, 0.91$ &$ 0.015\, /\, 0.829$ &$ 0.46 $ & $ 28.1 $\\ 
$   100 $ & $   400 $ & $   200 $ & $    20 $ &$ -0.62 \, /\, 0.22$ &$ -1.80\, /\, 0.29$ &$ 0.21\, /\, 0.79$ &$ 0.004\, /\, 0.737$ &$ 0.26 $ & $ 28.1 $\\ 
$   200 $ & $  2000 $ & $   400 $ & $    40 $ &$ -0.38 \, /\, 0.13$ &$ -1.18\, /\, 0.27$ &$ 0.10\, /\, 0.48$ &$ 0.013\, /\, 0.462$ &$ 0.10 $ & $ 57.2 $\\ 
$   200 $ & $   600 $ & $   400 $ & $    40 $ &$ -0.39 \, /\, 0.13$ &$ -2.19\, /\, 0.16$ &$ 0.30\, /\, 0.65$ &$ 0.003\, /\, 0.588$ &$ 0.90 $ & $ 57.3 $\\ 
$   200 $ & $   800 $ & $   400 $ & $    40 $ &$ -0.39 \, /\, 0.14$ &$ -1.84\, /\, 0.19$ &$ 0.22\, /\, 0.55$ &$ -0.004\, /\, 0.505$ &$ 0.09 $ & $ 57.2 $\\ 
$   300 $ & $  1200 $ & $   600 $ & $    60 $ &$ -0.30 \, /\, 0.10$ &$ -1.83\, /\, 0.16$ &$ 0.23\, /\, 0.45$ &$ 0.006\, /\, 0.416$ &$ 0.31 $ & $ 88.0 $\\ 
$   300 $ & $   900 $ & $   600 $ & $    60 $ &$ -0.29 \, /\, 0.10$ &$ -2.20\, /\, 0.12$ &$ 0.32\, /\, 0.52$ &$ 0.008\, /\, 0.466$ &$ 0.28 $ & $ 88.0 $\\ 
$   400 $ & $  1200 $ & $   800 $ & $    80 $ &$ -0.24 \, /\, 0.08$ &$ -2.21\, /\, 0.11$ &$ 0.30\, /\, 0.45$ &$ -0.005\, /\, 0.404$ &$ 0.24 $ & $ 122.2 $\\ 
$   400 $ & $  1600 $ & $   800 $ & $    80 $ &$ -0.24 \, /\, 0.08$ &$ -1.83\, /\, 0.14$ &$ 0.22\, /\, 0.39$ &$ -0.000\, /\, 0.364$ &$ 0.88 $ & $ 123.1 $\\ 
$   400 $ & $  4000 $ & $   800 $ & $    80 $ &$ -0.24 \, /\, 0.09$ &$ -1.22\, /\, 0.19$ &$ 0.09\, /\, 0.34$ &$ 0.001\, /\, 0.326$ &$ 1.00 $ & $ 122.1 $\\ 
$   500 $ & $  1500 $ & $  1000 $ & $   100 $ &$ -0.21 \, /\, 0.08$ &$ -2.21\, /\, 0.09$ &$ 0.31\, /\, 0.40$ &$ -0.005\, /\, 0.359$ &$ 0.28 $ & $ 151.5 $\\ 
$   500 $ & $  2000 $ & $  1000 $ & $   100 $ &$ -0.22 \, /\, 0.08$ &$ -1.85\, /\, 0.12$ &$ 0.22\, /\, 0.34$ &$ -0.006\, /\, 0.318$ &$ 0.01 $ & $ 151.9 $\\ 
   $   500 $ & $  5000 $ & $  1000 $ & $   100 $ &$ -0.21 \, /\, 0.08$ &$ -1.23\, /\, 0.17$ &$ 0.09\, /\, 0.30$ &$ 0.002\, /\, 0.291$ &$ 0.33 $ & $ 151.7 $\\ 
         \hline
\end{tabular}
}
\caption{Block model, with covariance given in Eq.~\eqref{eq:BlockCovariance} with parameters $\lb=10$, $q=1$,
  and true coefficients specified in Eq.~\eqref{eq:TruthBlock} with $\mu=3$, $\sigma=1$.
   We  use the  model size based selection method for $\lambda$  and report statistics computed over $n_{\text{rep}}=10,000$
   repetitions.}\label{table:Block0}
\end{table}

\begin{table}
   {\small
   \begin{tabular}{|c|c|c|c|c|c|c|c|c|c|c|c|c|c|c|c|c|c|}
    \hline
     $n$ & $n_{\gamma}$ & $p$ & $s_0$ & $\bias\, /\, \std (\hbeta)$ & $\bias\, /\, \std (\hbeta_{\naive})$ & $\bias\, /\,\std(\hbeta'_{\naive})$ &
                                                                                                                                                   $\bias\, /\,\std(\hbeta_{\cad})$ 
  & $\KS(\hbeta_{\cad})$ & $\|\hbtheta\|_0$\\
    \hline
     \hline
   $    50 $ & $   150 $ & $   100 $ & $    10 $ &$ -0.70 \, /\, 0.34$ &$ -2.17\, /\, 0.31$ &$ 0.29\, /\, 1.24$ &$ 0.008\, /\, 1.131$ &$ 0.93 $ & $ 17.9 $\\ 
$    50 $ & $   200 $ & $   100 $ & $    10 $ &$ -0.76 \, /\, 0.36$ &$ -1.79\, /\, 0.40$ &$ 0.21\, /\, 1.09$ &$ 0.006\, /\, 1.013$ &$ 0.08 $ & $ 16.9 $\\ 
$    50 $ & $   500 $ & $   100 $ & $    10 $ &$ -0.70 \, /\, 0.34$ &$ -1.14\, /\, 0.55$ &$ 0.09\, /\, 0.95$ &$ 0.017\, /\, 0.923$ &$ 0.21 $ & $ 17.9 $\\ 
$   100 $ & $  1000 $ & $   200 $ & $    20 $ &$ -0.46 \, /\, 0.19$ &$ -1.20\, /\, 0.38$ &$ 0.10\, /\, 0.68$ &$ 0.014\, /\, 0.655$ &$ 0.11 $ & $ 34.5 $\\ 
$   100 $ & $   300 $ & $   200 $ & $    20 $ &$ -0.45 \, /\, 0.19$ &$ -2.19\, /\, 0.21$ &$ 0.31\, /\, 0.88$ &$ 0.008\, /\, 0.795$ &$ 0.46 $ & $ 34.6 $\\ 
$   100 $ & $   400 $ & $   200 $ & $    20 $ &$ -0.49 \, /\, 0.20$ &$ -1.84\, /\, 0.27$ &$ 0.23\, /\, 0.76$ &$ 0.003\, /\, 0.701$ &$ 0.60 $ & $ 32.4 $\\ 
$   200 $ & $  2000 $ & $   400 $ & $    40 $ &$ -0.33 \, /\, 0.12$ &$ -1.24\, /\, 0.26$ &$ 0.08\, /\, 0.48$ &$ -0.005\, /\, 0.461$ &$ 0.00 $ & $ 64.8 $\\ 
$   200 $ & $   600 $ & $   400 $ & $    40 $ &$ -0.33 \, /\, 0.12$ &$ -2.21\, /\, 0.14$ &$ 0.33\, /\, 0.62$ &$ 0.007\, /\, 0.552$ &$ 0.51 $ & $ 64.7 $\\ 
$   200 $ & $   800 $ & $   400 $ & $    40 $ &$ -0.33 \, /\, 0.12$ &$ -1.84\, /\, 0.19$ &$ 0.24\, /\, 0.54$ &$ 0.006\, /\, 0.501$ &$ 0.47 $ & $ 64.7 $\\ 
$   300 $ & $  1200 $ & $   600 $ & $    60 $ &$ -0.27 \, /\, 0.10$ &$ -1.86\, /\, 0.15$ &$ 0.24\, /\, 0.44$ &$ -0.001\, /\, 0.405$ &$ 0.53 $ & $ 94.2 $\\ 
$   300 $ & $   900 $ & $   600 $ & $    60 $ &$ -0.27 \, /\, 0.10$ &$ -2.22\, /\, 0.12$ &$ 0.32\, /\, 0.51$ &$ -0.005\, /\, 0.453$ &$ 0.13 $ & $ 94.1 $\\ 
$   400 $ & $  1200 $ & $   800 $ & $    80 $ &$ -0.24 \, /\, 0.08$ &$ -2.22\, /\, 0.10$ &$ 0.33\, /\, 0.43$ &$ -0.003\, /\, 0.385$ &$ 0.42 $ & $ 123.0 $\\ 
$   400 $ & $  1600 $ & $   800 $ & $    80 $ &$ -0.22 \, /\, 0.08$ &$ -1.86\, /\, 0.13$ &$ 0.24\, /\, 0.38$ &$ 0.000\, /\, 0.351$ &$ 0.97 $ & $ 133.6 $\\ 
$   400 $ & $  4000 $ & $   800 $ & $    80 $ &$ -0.22 \, /\, 0.08$ &$ -1.27\, /\, 0.17$ &$ 0.09\, /\, 0.33$ &$ -0.001\, /\, 0.317$ &$ 0.10 $ & $ 133.7 $\\ 
$   500 $ & $  1500 $ & $  1000 $ & $   100 $ &$ -0.20 \, /\, 0.07$ &$ -2.23\, /\, 0.09$ &$ 0.33\, /\, 0.39$ &$ -0.006\, /\, 0.347$ &$ 0.27 $ & $ 163.1 $\\ 
$   500 $ & $  2000 $ & $  1000 $ & $   100 $ &$ -0.21 \, /\, 0.08$ &$ -1.87\, /\, 0.12$ &$ 0.24\, /\, 0.34$ &$ -0.005\, /\, 0.314$ &$ 0.01 $ & $ 151.7 $\\ 
$   500 $ & $  5000 $ & $  1000 $ & $   100 $ &$ -0.21 \, /\, 0.08$ &$ -1.28\, /\, 0.15$ &$ 0.09\, /\, 0.29$ &$ -0.000\, /\, 0.281$ &$ 0.36 $ & $ 151.7 $\\ 
     \hline
\end{tabular}
}
\caption{Block model: same experiments as in Table \ref{table:Block0}
  except that we use the estimation error-based selection method for $\lambda$.}\label{table:Block1}
\end{table}

The results of our simulations are summarized in Figures \ref{fig:SweepS0_ltype0}, \ref{fig:SweepS0_ltype1} and in Tables \ref{table:Circulant0},
\ref{table:Circulant1}, \ref{table:Block0}, \ref{table:Block1}.
The CAD estimator $\hbeta_{\cad}$ is compared with the
simple Lasso estimator $\hbeta$ (see Eq.~\eqref{eq:FirstDebiased}), and
the naive debiased estimators $\hbeta_{\naive}=\hbeta_{\mathrm{naive}}$ and $\hbeta'_{\naive}=\hbeta'_{\mathrm{naive}}$
(see Eqs.~\eqref{eq:DefNaive}, \eqref{eq:DefNaive2}). 

For each of
these estimators, we report the empirical mean and standard deviation over $n_{\text{rep}} = 10,000$ data realizations. The CAD estimator $\hbeta_{\cad}$  is the only one that has consistently negligible bias. We also report the Kolmogorov-Smirnov 
p-value for testing normality of  $(\hbeta_{\cad}-\beta)/\std(\hbeta_{\cad})$. We observe that $\KS(\hbeta_{\cad})$ is
most of the time large, indicating that the distribution of $\hbeta_{\cad}$ typically cannot be  distiguished from normal using $10,000$
samples.

\section{General theorems on simultaneous regression}
\label{sec:gen-thm}

The consistency of the CAD estimate (Theorem \ref{thm:debiased-estimate}) is a corollary of general theorems on the joint distribution of regression estimators fit using the same (or overlapping) data from two linear models.
These theorems are more powerful than Theorem \ref{thm:debiased-estimate}
and may be of independent interest.
In this section, we present these general results.

\subsection{A symmetric rewriting of the regression models}

Mathematically,
there is no reason to think of $\by$ as an outcome variable and $\bw$ as a covariate,
and not the other way around.
We introduce new notation which makes this symmetry explicit.
We will henceforth denote the outcome variables by $\by_k$, the regression parameters by $\btheta_k$, and the noise variables by $\be_k$,
with $k = 1$ for the precision model and $k = 2$ for the outcome model.
Explicitly, set
    $\by_1 := \bw$, $\by_2 := \by$, $\btheta_1 := \bgamma$, $ \btheta_2 := \bar \btheta$, $ \be_1 := \kappa \bw^\perp$,
    $\be_2 := \kappa \beta \bw^\perp + \sigma \bz$.
The precision and outcome models can be written as 
\begin{equation}
	\by_k = \bX \btheta_k + \be_k,
	\quad
	k = 1,2,
\end{equation}
where $(\be_1,\be_2) \sim \normal\big(\bzero,\bS_e \otimes \id_N \big)$,
with 
$
    \bS_e
        :=
        \begin{pmatrix}
            \tau_{e_1}^2 & \tau_{e_1}\tau_{e_2} \rho_e \\
            \tau_{e_1}\tau_{e_2} \rho_e & \tau_{e_2}^2,
        \end{pmatrix}
$
and $\tau_{e_1}^2 = \kappa^2$, $\tau_{e_2}^2 = \kappa^2 \beta^2 + \sigma^2$, and $\rho_e = \kappa\beta/(\kappa^2 \beta^2 + \sigma^2)^{1/2}$.
For reasons which will become clear, we will refer to the model in Eq.~\eqref{eq:random-design} as the \emph{random-design model}.
Succinctly,
we can write it as
\begin{equation}\label{eq:random-design}
\begin{gathered}
    \textbf{Random-design model}\\
    \bY = \bX \bTheta + \bE,
\end{gathered}
\end{equation}
where $\bY,\bE \in \reals^{N \times 2}$, $\bTheta \in \reals^{p\times 2}$, 
and these matrices have columns $\by_k$, $\btheta_k$, $\be_k$, $k=1,2$.
We will also denote the quantities $\hat \df_\gamma,\hat \df_\theta$ appearing in the definition of $r_i^\gamma,r_i^\theta$ 
and $\hat \bgamma^\de$, $\hat \btheta^\de$ (see Eqs.~\eqref{eq:dof-adjusted-residual} and \eqref{eq:db-def}) as $\hat \df_k$ for $k = 1,2$.
We will see that $\hat \df_k$ concentrates on a deterministic quantity $\df_k$, 
whose definition we postpone Eqs.~\eqref{eq:R-df}, \eqref{eq:fixed-pt-eqns}, and the discussion after Eq.~\eqref{eq:def-R-rhoh}.
The quantity $\df_k$ which we define is non-empirical: although she can estimate it, she cannot compute its exact value using information she know.
Nevertheless, it is most convenient to develop theory for the debiased regression estiamtes and degrees of freedom adjusted residuals using this deterministic quantity in place of $\hat \df_k$, and later show that we may replace $\df_k$ with the estimate $\hat \df_k$ without impacting our results.
Thus, we will study regression estimators in the model \eqref{eq:random-design} after replacing $\hat \df_k$ by $\df_k$.
We write these as
\begin{equation}\label{eq:param-est}
	\hat \btheta_k
		= 
		\argmin_{\bpi}
		\Big\{
			\frac1{2n_k} \| \by_{k,\cI_k} - \bX_{\cI_k}\bpi \|_2^2 + \Omega_k(\bpi)
		\Big\},
    \qquad 
    \hat \btheta_k^{\de}
        =
        \hat \btheta_k
            +
            \frac{\bSigma^{-1}\bX_{\cI_k}^\top(\by_{\cI_k}-\bX_{\cI_k}\hat \btheta_k)}{n_k - \df_k},
\end{equation}
where $n_k,\cI_k,\Omega_k$ are defined in the obvious way.
We also define $n_{12}:=|\cI_1\cap\cI_2|$
and define the matrix $\widehat{\bTheta} \in \reals^{p\times 2}$ to have columns $\hat \btheta_k$,
and $\widehat{\bTheta}^{\de}$ to have columns $\hat \btheta_k^{\de}$.
Note that there is some abuse in terminology in calling $\hat \btheta_k^\de$ an ``estimator'' because it depends on the non-empirical quantity $\df_k$.
We will show that all of our results hold also with $\hat \df_k$ in place of $\df_k$ (see Theorem \ref{thm:joint-characterization}(iii)).

Although it is typical to think of the task in the regression model \eqref{eq:random-design} as parameter estimation,
it is enlightening to think of the task as also including noise estimation. 
The tasks are complementary: the better we estimate the unknown parameter, the better we can estimate the noise, and vice versa.
In the regression procedure \eqref{eq:param-est},
the fitted residuals are natural estimates of the noise,
and
the degrees-of-freedom adjusted residuals can be understood as their debiased counterparts.
In particular,
for noise estimation,
we write
\begin{equation}\label{eq:noise-est}
    \hat \be_{k,\cI} 
        =
        \begin{cases}
            \by_{k,\cI_k} - \bX_{\cI_k} \hat \btheta_k \quad &\text{if } \cI = \cI_k,\\
            \bzero \quad & \text{if } \cI = \cI_k^c,
        \end{cases}
    \qquad 
    \hat \be_k^{\de}
        =
        \frac{\df_k}{n_k - \df_k} \hat \be_k + (\by_k - \bX \hat \btheta_k).
\end{equation}
Although it may seem unnatural to set the out-of-sample noise estimate $\hat \be_{k,\cI_k^c} = \bzero$,
the proofs of our main theorems will make clear why this is---for our purposes---the correct choice.
We define the matrix $\widehat{\bE} \in \reals^{N\times 2}$ to have columns $\hat \be_k$,
and $\widehat{\bE}^{\de}$ to have columns $\hat \be_k^{\de}$.
As with the debiased estimate of the parameter, 
there is some abuse in terminology in calling $\hat \be_k^\de$ an ``estimator'' because it depends on the non-empirical quantity $\hat \df_k$.
As we have already stated, we will show that all of our results hold also with $\hat \df_k$ in place of $\df_k$ (see Theorem \ref{thm:joint-characterization}(iii)).

\subsection{Joint characterization for simultaneous regression}

The proof of Theorem \ref{thm:debiased-estimate} is an application of a more general theorem characterterizing the joint behavior of the regression estimators $\{ \hat \be_k \}_k$, $\{\hat \be_k^{\de}\}_k$, $\{ \hat \btheta_k \}_k$, and $\{ \hat \btheta_k^{\de}\}_k$. 

To state the result, 
we must first introduce two new statistical models---one for parameter estimation and one for noise estimation---which behave, in a certain sense, like the random design model \eqref{eq:random-design}.
The two models are
\begin{equation}\label{eq:fixed-des}
    \begin{gathered}
	\textbf{Fixed-design model}\\
    \textbf{for parameter estimation}\\
    \bY^f
		:= 
		\bSigma^{1/2} \bTheta + \bG^f
        \;\;
        \text{(FD-P)}
    \end{gathered}
    \qquad\qquad \qquad 
    \begin{gathered}
    \textbf{Fixed-design model}\\
    \textbf{for noise estimation}\\
    \bR^f = \bE + \bH^f
    \;\;
    \text{(FD-N)}
    \end{gathered}
\end{equation}
where $\bG^f \in \reals^{p \times 2}$ has columns $(\bg_1^f,\bg_2^f) \sim \normal\big(\bzero,\bS_g \otimes \id_p\big)$,
and $\bH^f \in \reals^{N \times 2}$ has columns $(\bh_1^f,\bh_2^f) \sim \normal\big(\bzero,\bS_h \otimes \id_N\big)$,
for some matrices $\bS_g,\bS_h \in \mathbb{S}_+^2$ to be chosen.
Like model~\eqref{eq:random-design},
the fixed-design model for parameter estimation involves 
linear measurements of the unknown parameters $\btheta_1,\btheta_2$ with correlated Gaussian noise. 
The models differ in that the random design matrix $\bX$ in model \eqref{eq:random-design} is replaced by the fixed design matrix $\bSigma^{1/2}$ in (FD-P). (A consequence is that the dimensionality of the observations is $p$ instead of $N$).
The fixed-design model for noise estimation also has similarities with model \eqref{eq:random-design}.
Indeed, $\bX\bTheta \sim \normal(\bzero, (\bTheta^\top \bSigma\bTheta) \otimes \id_N)$,
so that model \eqref{eq:random-design} can be viewed as specifying correlated Gaussian observations of the noise variables.
In this sense, the model \eqref{eq:random-design} takes the same form as (FD-N).
The models differ in that in model \eqref{eq:random-design} the matrix $\bX$ is observed, 
providing information about the Gaussian corruption which is not available in (FD-N).

Our general characterizaton result establishes a precise form of the following claim:
\begin{center}
    \textit{Parameter estimation in model \eqref{eq:random-design} behaves like parameter estimation in model \emph{(FD-P)}.}
    \\[3pt]
    \textit{Noise estimation in model \eqref{eq:random-design} behaves like noise estimation in model \emph{(FD-N)}.}
\end{center}
The correctness of these statements relies on specifying the correct estimators and covariance structures $\bS_g,\bS_h$  in the fixed-design models.

In analogy  with Eq.~\eqref{eq:param-est},
the parameter estimates in the fixed-design model (FD-P) are
\begin{equation}\label{eq:fixed-design-est}
\begin{gathered}
	\hat \btheta_k^f
		:=
        \eta_k(\by_k^f;\zeta_k)
        :=
		\argmin_{\bpi \in \reals^p}
		\Big\{
			\frac12 \| \by_k^f - \bSigma^{1/2}\bpi\|_2^2 + \frac1{\zeta_k} \Omega_k(\bpi)
		\Big\},
    \\
    \hat \btheta_k^{f,\de}
        :=
        \hat \btheta_k^f
        +
        \bSigma^{-1/2}(\by_k^f - \bSigma^{1/2}\hat \btheta_k^f)
        =
        \bSigma^{-1/2}\by_k^f.
\end{gathered}
\end{equation}
For noise estimation, the analogy to Eq.~\eqref{eq:noise-est} is less clear.
We set the noise estimates in the fixed-design model (FD-N) to be
\begin{equation}\label{eq:fixed-design-err-est}
    \hat \be^f_{k,\cI} 
        =
        \begin{cases}
            \zeta_k\br^f_{\cI_k} \quad &\text{if } \cI = \cI_k,\\
            \bzero \quad & \text{if } \cI = \cI_k^c,
        \end{cases}
    \qquad 
    \hat \be_k^{f,\de}
        =
        \frac{1-\zeta_k}{\zeta_k} \hat \be_k^f  
        +
        \Big(
            \br_k^f - \frac{1-\zeta_k}{\zeta_k} \hat \be_k^f
        \Big)
        =
        \br_k^f.
\end{equation}
The proofs of our main theorems will make clear why this is---for our purposes---the correct choice of noise estimates.

The covariance structures $\bS_g,\bS_h$ and parameters $(\zeta_1,\zeta_2)$ are determined by
a system of equations which we call the \emph{fixed-point equations}.
The system of equations involves functions describing the behavior of parameter estimation in the fixed-design model.
The functions are 
\begin{equation}\label{eq:R-df}
    \begin{gathered}
	\textbf{Second moments of prediction error}
    \\
		\sR(\bS_g,\{\zeta_k\}) = \E[( \widehat{\boldsymbol{\Theta}}^f - \boldsymbol{\Theta} )^\top \bSigma ( \widehat{\boldsymbol{\Theta}}^f - \boldsymbol{\Theta} ) ],
    \\
    \phantom{k=1,2}
	\end{gathered}
    \qquad
    \begin{gathered}
	\textbf{Degrees-of-freedom}
	\\
		\df_k(\bS_g,\{\zeta_k\}) = \E[\div \bSigma^{1/2} \eta_k(\by_k^f ; \zeta_k ) ],
    \\
        \qquad\qquad\qquad\qquad\qquad\qquad\qquad k = 1,2.
    \end{gathered}
\end{equation}
The right-hand sides depend implicitly on $\bS_g$ via $\widehat{\bTheta}^f,\bY^f$. 
We define
$\bS_g,\{\zeta_k\}$ 
as solutions to the system of equations
\begin{equation}\label{eq:fixed-pt-eqns}
\begin{gathered}
    \bS_g = \bN^- \odot \big(\bS_e + \sR(\bS_g,\{\zeta_k\})\big),\\
    \zeta_k = 1 - \frac{\df_k(\bS_g,\{\zeta_k\})}{n_k},\quad k=1,2,
\end{gathered}
\end{equation}   
where $\bN^- := \begin{pmatrix} 1/n_1 & n_{12}/(n_1n_2) \\ n_{12}/(n_1n_2) & 1/n_2\end{pmatrix}$ and $\odot$ denotes coordinate-wise multiplication.
These equations have a solution.
\begin{lemma}[Existence and uniqueness of fixed-point parameters]\label{lem:fixed-pt-soln}
    Assume $\bSigma$ is invertible, $\bS_e \succeq 0$,
    and for each $k$, either $n_k > p$ or $\lambda_k > 0$, where $\lambda_k$ is the regularization parameter of the penalty (see Eq.~\eqref{eq:penalty}).
    Then equations \eqref{eq:fixed-pt-eqns} have a unique solution $\bS_g,\{\zeta_k\}$.
\end{lemma}
\noindent Lemma \ref{lem:fixed-pt-soln} is proved in Section \ref{sec:existence-fix-pt}.

Given a solution $\bS_g,\{\zeta_k\}$ to Eqs.~\eqref{eq:fixed-pt-eqns},
define the parameters
\begin{equation}\label{eq:def-R-rhoh}
	\bS_h
		:= 
		\sR(\bS_g,\{\zeta_k\}),
    \qquad 
    \bDf
        :=
        \begin{pmatrix}
            \df_1(\bS_g,\{\zeta_k\}) & 0 \\
            0 & \df_2(\bS_g,\{\zeta_k\})
        \end{pmatrix}.
\end{equation}
For simplicity, we will often denote $\df_k(\bS_g,\{\zeta_k\})$ by $\df_k$, 
where it is understood that this is evaluated at solutions to the fixed point equations \eqref{eq:fixed-pt-eqns}.

We are ready to state our general characterization result.
\begin{theorem}[Joint characterization]\label{thm:joint-characterization}
  Assume \textsf{A1} and \textsf{A2} hold. Let $\bS_g,\{\zeta_k\}$ be the unique solutions of  the
  fixed-point equations \eqref{eq:fixed-pt-eqns} and $\bS_h$ be defined by Eq.~\eqref{eq:def-R-rhoh}.
  Define the fixed design model and estimators according to Eqs.~\ref{eq:fixed-des}, \ref{eq:fixed-design-est}, \ref{eq:fixed-design-err-est}.

  There exist constants $\nu,c' > 0$ and functions $\sC,\sc: \reals_{>0} \rightarrow \reals_{>0}$ which depend uniquely on
  $\cPmodel$, $\cPregr$ and the regression method such that $\sc(\epsilon)\ge c'(\epsilon^\nu\wedge 1)$ and the
  following occurs.

    \begin{enumerate}[(i)]

        \item 
        \emph{(Noise estimation)}
        Let $\phi_e\Big(\Big\{\frac{\be_{k,\cI_2}}{\sqrt{n_2}}\Big\},\Big\{\frac{\hat \be_{k,\cI_2}}{\sqrt{n_2}}\Big\},\Big\{\frac{\hat \be_{k,\cI_2}^{\de}}{\sqrt{n_2}}\Big\}\Big)$ be a real-valued function which is $M_k$-Lipschitz in its arguments $\be_{k,\cI_2}/\sqrt{n_2},\hat \be_{k,\cI_2}/\sqrt{n_2},\hat \be_{k,\cI_2}^\de/\sqrt{n_2}$ for $k=1,2$.
        Then
        for $\epsilon < c'$ 
        with probability at least $1 - \sC(\epsilon) e^{-\sc(\epsilon) p}$ 
        \begin{equation}
        \begin{gathered}
            \Big|
                \phi_e\Big(\Big\{\frac{\be_{k,\cI_2}}{\sqrt{n_2}}\Big\},\Big\{\frac{\hat \be_{k,\cI_2}}{\sqrt{n_2}}\Big\},\Big\{\frac{\hat \be_{k,\cI_2}^{\de}}{\sqrt{n_2}}\Big\}\Big)
                -
                \E\Big[
                    \phi_e\Big(\Big\{\frac{\be_{k,\cI_2}}{\sqrt{n_2}}\Big\},\Big\{\frac{\hat \be_{k,\cI_2}^f}{\sqrt{n_2}}\Big\},\Big\{\frac{\hat \be_{k,\cI_2}^{\de,f}}{\sqrt{n_2}}\Big\}\Big)
                \Big]
            \Big| <  (M_1\sqrt{p/n_2} + M_2)\, \epsilon.
        \end{gathered}
        \end{equation}
        The result is also true if we replace $\cI_2$ with $\cI_1$, $n_2$ with $n_1$,
        and $M_1\sqrt{p/n_2} + M_2$ with $M_1 + M_2\sqrt{p/n_1}$.

        \item 
        \emph{(Parameter estimation)}
        Let $\phi_\theta : (\reals^p)^4 \rightarrow \reals$ be $M_1$-Lipschitz in its first two arguments and $M_2$-Lipschitz in its second two arguments.
        Then
        for $\epsilon < c'$ with probability at least $1 - \sC(\epsilon) e^{-\sc(\epsilon) p}$ 
        \begin{equation}
        \begin{gathered}
            \big| 
                \phi_\theta\big( \hat \btheta_1 , \hat \btheta_1^{\de} , \hat \btheta_2 , \hat \btheta_2^{\de} \big) 
                    - 
                    \E\big[\phi_\theta\big( \hat \btheta_1^f , \hat \btheta_1^{f,\de} , \hat \btheta_2^f , \hat \btheta_2^{f,\de}  \big)\big] 
            \big| 
                < 
                \min\big\{ M_1 \sqrt{p/n_1} + M_2 , M_1 + M_2 \sqrt{p/n_2} \big\}\,\epsilon.
        \end{gathered}
        \end{equation}

        \item 
        Parts (i) holds with $\hat \df_k$ in place of $\df_k$ in the definition of $\hat \be_k^\de$ (see Eq.~\eqref{eq:noise-est})
        if the right-hand side of the bound is changed to $(M_1(\sqrt{p/n_2} + p/n_1) + M_2)\epsilon$,
        (or if we replace $\cI_2$ with $\cI_1$, the upper bound becomes $(M_1 + M_2(\sqrt{p/n_1} + p/n_2))\epsilon$).

        Parts (ii) holds (with the same bound on the right-hand side) with $\hat \df_k$ in place of $\df_k$ in the definition of $\hat \btheta_k^\de$ (see Eq.~\eqref{eq:param-est}).

    \end{enumerate}
\end{theorem}

\begin{remark}\label{rmk:joint-char-rates}
    We conjecture that Theorem \ref{thm:joint-characterization} holds for $\sC(\epsilon) = C$ and $\sc(\epsilon) = c\epsilon^2$.
    Our proof techniques allow us to establish the result for $\sc(\epsilon) = c\epsilon^\nu$ for some $\nu > 2$. 
    Because the exponent $\nu$ we can provide is not the correct one and tracking it complicates our proofs, 
    we do not track it.
    We further conjecture that the upper bound on the right-hand sides can be replaced by $(M_1+M_2)\sqrt{p/n_2}\,\epsilon$ in the case of noise estimation,
    and by $(M_1 \sqrt{p/n_1} + M_2 \sqrt{p/n_2})\epsilon$ in the case of parameter estimation.
    See the remarks following Lemma \ref{lem:marginal-characterization} for further interpretation of these rates.
\end{remark}

\subsection{Interpretation of the joint characterization}

As anticipated, Theorem \ref{thm:joint-characterization} states that parameter estimation in model \eqref{eq:random-design} behaves like parameter estimation in model (FD-P),
and noise estimation in model \eqref{eq:random-design} behaves like noise estimation in model (FD-N),
where ``behaves like'' refers to the concentration properties of Lipschitz functions.
Because $\{\hat \be_k\},\{\hat \be_k^{f,\de}\}, \{ \hat \btheta_k\}, \{ \hat \btheta_k^{\de}\}$ are $C$-Lipschitz functions of the Gaussian noise $\bG^f,\bH^f$,
in the fixed design models the functions $\phi_e,\phi_\theta$ will concentrate on their expectations.
Theorem \ref{thm:joint-characterization} states that $\phi_e,\phi_\theta$ also concentrate in the random-design model,
and they concentrate on the same values.

Theorem \ref{thm:joint-characterization} gives substantive meaning to the fixed point parameters $\bS_g$ and $\bS_h$:
they are the values on which
the random-design quantities 
\begin{equation}\label{eq:db-regr-err}
    \frac1p(\widehat \bTheta^{\de} - \bTheta)^\top \bSigma (\widehat \bTheta^{\de} - \bTheta)
    \quad \text{and} \quad 
    \frac1{n_k}(\widehat \bE_{\cI_k}^{\de} - \bE_{\cI_k})^\top (\widehat \bE_{\cI_k}^{\de} - \bE_{\cI_k})
\end{equation}
concentrate.
Although the quantities in the previous display are not Lipschitz functions of $\widehat \bTheta^{\de}$ and $\widehat \bE^{\de}_{\cI_k}$,
their concentration is straightforward to establish using a Lipschitz approximation argument,
which we will carry out formally in our proofs when necessary (see Section \ref{sec:marg-char-support}).
Theorem \ref{thm:joint-characterization} can also be used to establish the concentration of several other quantities in the regression models.
An important example is the random design quantity $(\widehat \bTheta - \bTheta)^\top \bSigma (\widehat \bTheta - \bTheta)$, 
which concentrates on $\sR(\bS_g,\{ \zeta_k \}) = \bS_h$. 
The quantity $\sR_{kk}(\bS_g,\{ \zeta_k \}) = \bS_{h,kk}$ can then be interpreted as the typical prediction risk of the estimate from regression $k$, 
and $\sR_{12}(\bS_g,\{ \zeta_k \})$ as a measure of the typical alignment between the errors in the predictions from the two regressions,
which the reader will recall gives the bias of the estimator $\hbeta_{\mathrm{split}}$ (see Section \ref{sec:why-bias}).

An implication of Theorem \ref{thm:joint-characterization} is that the random-design quantities which concentrate on $\bS_g$, $\bS_h$, and $\sR(\bS_g,\{ \zeta_k \})$ will, with high-probability, approximately satisfy Eq.~\eqref{eq:fixed-pt-eqns}. 
This has important statistical consequences.
Indeed, $\bS_g$ and $\bS_h$ are not known a priori to the statistician,
and they describe the behavior of the errors \eqref{eq:db-regr-err} and $(\widehat \bTheta - \bTheta)^\top \bSigma (\widehat \bTheta - \bTheta)$ which the statistician does not observe but may wish to estimate.
As we will see,
the fixed point equations \eqref{lem:fixed-pt-soln} will allow us to write $\bS_g$, $\bS_h$, and other unknown quantities in terms of parameters we can estimate from data.
For example, Theorem \ref{thm:joint-characterization} can be used to show that $\widehat \bE_{\cI_k}^{\de\top}\widehat \bE_{\cI_k}^{\de}/n_k$ concentrates on $\E[\bR_{\cI_k}^{f\top}\bR_{\cI_k}^f]/n_k = \bS_e + \bS_h = \bS_e + \sR(\bS_g,\{\zeta_k\})$.
The fixed point equations \eqref{eq:fixed-pt-eqns} then suggest that
\begin{equation}\label{eq:Sg-hat}
    \widehat \bS_g^{(k)} 
        :=
        \bN^- \odot \Big(\frac1{n_k}\widehat \bE_{\cI_k}^{\de\top}\widehat \bE_{\cI_k}^{\de}\Big),
\end{equation}
is a consistent estimator of $\bS_g$,
and thus of the error structure of the debiased estimates.
We will see that this claim is correct.
Pushing this strategy further, we will develop a consistent estimator of the noise covariance $\bS_e$,
which
will give us Theorem \ref{thm:debiased-estimate}.

For future reference, we have collected in Table \ref{tab:fix-pt} several random design quantities and the ``typical values'' on which they concentrate.
For some of these quantities,
we have provided multiple equivalent ways of writing the typical values,
either using the fixed-point equations or the definitions of the fixed point parameters.
We have also introduced the parameter $\bS_{\hat e^{d}} := \E[\widehat \bE_{\cI_k}^{f,\de\top}\widehat \bE_{\cI_k}^{f,\de}] / n_k$ and $\bS_v := \E[(\widehat \bTheta^f - \bTheta)^\top \bSigma (\widehat \bTheta^f - \bTheta)] $. 
This is consistent with a convention which we adopt by which $\bS$ always denotes the fixed-design second-moment matrix of the quantity in its subscript (possibly with some renormalization).
We will later define $\bV = \bSigma^{1/2}(\widehat \bTheta - \bTheta)$, justifying the notation $\bS_v$.
We do not claim to have yet proved any of the concentration properties or relationships implied by Table \ref{tab:fix-pt} (although they are straightforward to show using Theorem \ref{thm:joint-characterization}).
These will be provided in our proofs as needed.
Table \ref{tab:fix-pt} is not intended as a replacement for a proof, but rather as a useful reference for the reader.

\begin{table}
\begin{center}
\begin{tabular}{@{}llll@{}} \toprule
\multicolumn{2}{c}{Noise estimation} & \multicolumn{2}{c}{Parameter estimation} \\ \cmidrule(r){1-2}  \cmidrule(r){3-4}
Random-design quantity & Typical value & Random-design quantity & Typical value \\ \midrule
$\bE_{\cI_k}^\top \bE_{\cI_k}/n_k$ & $\bS_e$ & $(\widehat \bTheta - \bTheta)^\top \bSigma (\widehat \bTheta - \bTheta)$  & $\bS_v$, $\bS_h$, $\mathsf{R}(\bS_g,\{\zeta_k\})$ \\\addlinespace[.1cm]
$(\widehat \bE_{\cI_k}^{\de} - \bE_{\cI_k})^\top (\widehat \bE_{\cI_k}^{\de} - \bE_{\cI_k})/n_k$ & $\bS_h$ & $(\widehat \bTheta^{\de} - \bTheta)^\top \bSigma (\widehat \bTheta^{\de} - \bTheta)/p$  & $\bS_g$, $\bN^- \odot \bS_{\hat e^{\de}}$ \\\addlinespace[.1cm]
$\widehat \bE_{\cI_k}^{\de\top}\widehat \bE_{\cI_k}^{\de} / n_k$ & $\bS_{\hat e^{\de}}$, $\bS_e + \bS_h$ & $(\widehat \bTheta^{\de} - \bTheta)^\top \bSigma (\widehat \bTheta - \bTheta)$ & $\bS_g \bDf $ \\
\bottomrule
\end{tabular}
\end{center}
\caption{Some random design quantities and the values on which they concentrate.}\label{tab:fix-pt}
\end{table}

\subsection{Estimating noise covariance and the proof of Theorem \ref{thm:debiased-estimate}}

In this section,
we construct a consistent estimate of the noise covariance $\bS_e$,
and then show that the consistency of the CAD estimate (Theorem \ref{thm:debiased-estimate}) follows from this construction.
Our estimate of $\bS_e$ is a generalization of the estimate of the Lasso noise 
level provided by \cite{bayatiErdogdu2013} (see Eq.~(1.4) and the following display in that paper).
However, the setting here is considerably more challenging since we study two coupled
linear models instead of a single one, and correlated Gaussian designs instead of iid designs
($\bSigma \neq \id_p$).

In order to derive our estimate of $\bS_g$,
we follow the strategy described in the previous section:
we use the characterizations listed in Table \ref{tab:fix-pt} 
and the fixed point equations \eqref{eq:fixed-pt-eqns} to find a quantity which concentrates on $\bS_e$.
We will first derive the estimate heuristically, and then will prove its consistency.

Writing $\widehat \bTheta^{\de} - \widehat \bTheta = (\widehat \bTheta^{\de} - \bTheta) - (\widehat \bTheta - \bTheta)$ and using Table \ref{tab:fix-pt},
we expect that $(\widehat \bTheta^{\de} - \widehat \bTheta)^\top \bSigma (\widehat \bTheta^{\de} - \widehat \bTheta)$ will concentrate on
\begin{equation}
    p\bS_g - \bS_g \bDf - \bDf \bS_g + \bS_h 
        = 
        p\bS_g - \bS_g \bDf - \bDf \bS_g + \bS_{\hat e^{\de}} - \bS_e.
\end{equation}
We proposed in Eq.~\eqref{eq:Sg-hat} an estimate $\widehat \bS_g^{(k)}$ of $\bS_g$. 
A natural estimate of $\bS_{\hat e^{\de}}$ is 
\begin{equation}
    \widehat \bS_{\hat e^{\de}}^{(k)}
        :=
        \frac1{n_k}\widehat \bE_{\cI_k}^{\de\top}\widehat \bE_{\cI_k}^{\de}.
\end{equation}
Combining these estimates and identities,
we propose the estimate 
\begin{equation}
    \widehat \bS_e^{(k)}
    :=
    \widehat \bS_{\hat e^{\de}}^{(k)} + p\widehat \bS_g^{(k)} - \widehat \bS_g^{(k)} \widehat \bDf - \widehat \bDf \widehat \bS_g^{(k)} - (\widehat \bTheta^{\de} - \widehat \bTheta)^\top \bSigma (\widehat \bTheta^{\de} - \widehat \bTheta),
\end{equation}
where $\widehat \bDf = \diag(\hat \df_1,\hat \df_2)$.
Our main theorem is as follows:
\begin{theorem}\label{thm:noise-est}
    Consider that $\hat \btheta_k^\de$ and $\hat \be_k^\de$ are computed with $\hat \df_k$ in place of $\df_k$ in Eqs.~\eqref{eq:param-est} and \eqref{eq:noise-est}.
    Under assumptions \textsf{A1} and \textsf{A2},
    there exist constants $\nu, c' > 0$ and $\sC,\sc : \reals_{>0} \rightarrow \reals_{>0}$ such that
    $\sc(\eps)>c'\eps^{\nu}$ and,  for $\epsilon \leq c'$,
    \begin{equation}
    \begin{gathered}
        \P\Big(\big\| \widehat \bS_e^{(k)} - \bS_e\big\|_{\sF} > C' \sqrt{\frac{p}{ n_1 \wedge  n_2 }}\,\epsilon\Big) \leq \sC(\epsilon) e^{-\sc(\epsilon)p}.
    \end{gathered}
    \end{equation}
\end{theorem}
\noindent We prove Theorem \ref{thm:noise-est} in Section \ref{sec:proof-noise-est}.

As a byproduct, Theorem \ref{thm:noise-est} provides an estimate of the noise level
$\sigma$, which generalizes the estimate in \cite{bayatiErdogdu2013}, which was limited
to uncorrelated designs. 
Recall that $S_{e,kk} = \tau_{e_k}^2$ is the noise level in the $k^\text{th}$ linear model.
The estimate for $\tau_{e_k}^2$ provided by $\widehat \bS_e^{(k)}$ is
\begin{equation}
    \hat \tau_{e_k}^2
        := 
        \Big(1 + \frac{p}{n_k} - 2\frac{\hat \df_k}{n_k}\Big)\frac{\|  \hat\be^{ \de}_{k,\cI_k}\|_2^2}{n_k}
        -
        \| \hat \btheta_k^{\de} - \hat \btheta_k \|_{\bSigma}^2.
\end{equation}
In the case of the Lasso,
we estimate $\df_k$ by $\hat \df_k = \| \hat \btheta_k \|_0$.
Further, when $\bSigma = \id_p$, 
we have $\hat \btheta_k^{\de} - \hat \btheta_k = \bX_{\cI_k}^\top(\by_{\cI_k} - \bX_{\cI_k}\hat \btheta_k)/(n_k - \| \hat \btheta_k \|_0)$.
Finally, the degrees of freedom adjusted residuals $\hat \be_{k,\cI_k}^{\de}$
are given by $(\by_{k,\cI_k} - \bX_{\cI_k} \hat \btheta_k)/(1 - \| \hat \btheta_k \|_0/n_k)$ (see Eq.~\eqref{eq:noise-est}).
Thus, in the special case of the Lasso with $\bSigma = \id_p$, we may rewrite the estimate above as 
\begin{equation}
    \hat \tau_{e_k}^2
        = 
        \Big(1 + \frac{p}{n_k} - 2\frac{\| \hat \btheta_k \|_0}{n_k}\Big)\frac{\| \by_{k,\cI_k} - \bX_{\cI_k} \hat \btheta_k \|_2^2}{n_k(1 - \| \hat \btheta_k \|_0/n_k)^2}
        -
        \frac{\|\bX_{\cI_k}^\top(\by_{\cI_k} - \bX_{\cI_k}\hat \btheta_k)\|_2^2}{(n_k - \|\hat \btheta_k\|_0)^2},
\end{equation}
which agrees, under our normalization, with the estimate provided in \cite{bayatiErdogdu2013}.

As a sanity check, it is useful to write the noise estimate $\hat \tau_{e_k}^2$ in the case of OLS when $n_k > p$,
a case to which Theorem \ref{thm:noise-est} also applies.
In this case, $\df_k = p$ and $\bX_{\cI_k}^\top(\by_{\cI_k} - \bX_{\cI_k}\hat \btheta) = 0$.
The noise estimate simplifies to
\begin{equation}
    \hat \tau_{e_k}^2
        = 
        \frac{\| \by_{k,\cI_k} - \bX_{\cI_k} \hat \btheta_k \|_2^2}{n_k - p}
        \quad \text{for estimation by least squares.}
\end{equation}
This is the standard unbiased estimate of noise variance with least squares.

We can now prove Theorem \ref{thm:debiased-estimate}.
\begin{proof}[Proof of Theorem \ref{thm:debiased-estimate}]
    Using that $\tau_{e_1}^2 = \kappa^2$, $\tau_{e_2}^2 = \kappa^2 \beta^2 + \sigma^2$, and $\rho_e = \kappa \beta/(\kappa^2 \beta^2 + \sigma^2)^{1/2}$,
    we see that $\beta = \tau_{e_1}\tau_{e_2}\rho_e / \tau_{e_1}^2 = S_{e,12}/S_{e,11}$.
    Using that $r_i^\gamma = e_{1,i}^\de$, $r_i^\theta = e_{2,i}^\de$, and $\hat \kappa^2 = \widehat S_{e,11}$,
    we see that  $\hbeta_{\cad} = \widehat S_{e,12} / \widehat S_{e,11}$.
    Theorem \ref{thm:debiased-estimate} follows by apply Theorem \ref{thm:noise-est} and the $\delta$-method, using that $S_{e,12}$ is bounded above and $S_{e,11}$ is bounded below and above by $\cPmodel$-dependent constants.
\end{proof}

\section*{Acknowledgements}

This work was partially supported by NSF grants CCF-2006489, IIS-1741162 and the ONR grant N00014-18-1-
2729. M.C. was supported by the National Science Foundation Graduate Research Fellowship
under grant DGE-1656518.

\bibliographystyle{alpha}
\bibliography{debiasing}

\newpage
\appendix

\section{Convention on using constants $C,c,\sC(\epsilon),\sc(\epsilon)$}

Throughout the appendices, the constants $C,C',c,c' > 0$ and 
$\sC(\epsilon),\sc(\epsilon): \reals_{>0} \rightarrow \reals_{>0}$ (or these quantities 
with subscripts) are always assumed be positive, finite, and depend only on $\cPregr,\cPmodel$, 
and the regression methods used. Unless otherwise stated, these may change at each appearance.
Further, the exponent $\sc(\epsilon)$ is understood to satisfy $\sc(\epsilon)>c'(\eps^{\nu}\wedge 1)$
for some constant $\nu<\infty$ that also can change from line to line.
These conventions will usually not be repeated throughout the text. 

We use these conventions to make several statements more concise. 
For example, the statement ``we conclude $\| \hat \btheta_k - \btheta_k \|_{\bSigma} > c$'' should be interpreted to mean ``we conclude $\| \hat \btheta_k - \btheta_k \|_{\bSigma}$ is bounded below by a positive constant depending only on $\cPregr,\cPmodel$, and the regression methods used,'' although it will usually not be stated in this more unwieldy way.

\section{The $\alpha$-smoothed Lasso}

For technical reasons, it is difficult to prove Theorem \ref{thm:joint-characterization} directly for the Lasso estimator.
In particular, 
the non-differentiability of the $\ell_1$-norm leads to technical challenges in characterizing the behavior of the debiased Lasso $\hat \btheta_k^{\de}$.
To address these challenges, we use a smoothing technique, which was also used in \cite{celentano2020lasso}.

We introduce a regression estimator which uses a smoothed version of the $\ell_1$-norm. 
The $\alpha$-smoothed Lasso with smoothing parameter $\alpha_k > 0$ and regularization $\lambda_k$ uses the penalty $\Omega_k(\bpi) := \lambda_k \sM_{\alpha_k/\sqrt{n_k}}(\bpi)/\sqrt{n_k}$, where
\begin{equation}
    \sM_t(\bpi) 
        := 
        \inf_{\bv \in \reals^p}\Big\{\frac1{2t}\| \bpi - \bv \|_2^2 + \| \bv \|_1\Big\}.
\end{equation}
For $t = 0$, we define $\sM_0(\bpi) = \|\bpi\|_1$.
It is straightforward to check that for all $\alpha_k > 0$, $\nabla \sM_{\alpha_k/\sqrt{n_k}}(\bpi)/\sqrt{n_k}$ exists and is 
$1/\alpha_k$-Lipschitz. 
We add to assumption $\textsf{A2}$ the following assumptions for the case that the $\alpha$-smoothed Lasso is used:
\begin{description}

    \item{$\textsf{A2}$ (continued)} 
    If we use the $\alpha$-smoothed Lasso, we make the following assumption:
    \begin{itemize}

        \item 
        We assume the parameter $\bar \btheta$ (resp.\ $\bgamma$) is $(s,\sqrt{ n_k /p}(1-\Delta_{\min}),M)$-approximately sparse for some $s/p > \nu_{\min} $, $\Delta_{\min} \in (0,1)$, and $ n_k /p \in [\delta_{\min},\delta_{\max}]$ for $0 < \delta_{\min} \leq \delta_{\max} < \infty$,
        and $0 < \lambda_{\min} < \lambda_k < \lambda_{\max} < \infty$.
        We assume $\alpha_k < \alpha_{\max}$.

    \end{itemize}

\end{description}
We include in $\cPregr$ the constant $\alpha_{\max}$ which appears in our assumption about the $\alpha$-smoothed Lasso.
We will prove the joint-characterization under assumption $\textsf{A1}$ and $\textsf{A2}$ even when the $\alpha$-smoothed Lasso is used for one or both of the regressions.
In the remainder of the paper, we will develop results for the $\alpha$-smoothed Lasso alongside our development of results for OLS, ridge-regression, and the Lasso.

\section{The fixed-point parameters}

Our proofs will require certain bounds on and identities satisfied by the fixed point parameters $\bS_g$ and $\bS_h$.
They will sometimes benefit from new notation for references entries of these matrices.
The purpose of this section is to provide these results and new notation.

\subsection{Non-matricial form of fixed-point parameters}
\label{sec:non-matrix-fix-pt}

For any matrix $\bS$, possibly with subscripts,
we will denote the variances and correlations implied by the covariance structure $\bS$ by $\tau^2$ and $\rho$, respectively, with the corresponding subscripts.
For example,
we will denote
\begin{equation}
    \bS_g =:
        \begin{pmatrix}
            \tau_{g_1}^2 & \tau_{g_1}\tau_{g_2}\rho_g \\
            \tau_{g_1}\tau_{g_2}\rho_g & \tau_{g_2}^2
        \end{pmatrix},
\end{equation}
and will similarly define the parameters $\tau_{h_k}^2,\rho_h,\tau_{\hat e_k^{\de}}^2,\rho_{\hat e^{\de}}$.
We will also denote $\rho^\perp = \sqrt{1-\rho^2}$, with appropriate subscripts.
For example, $\rho_e^\perp = \sqrt{1 - \rho_e^2}$, and so on.

It will frequently be useful to reference individual entries of the fixed-point equations \eqref{eq:fixed-pt-eqns}.
Our new notation facilitates this. 
The diagonal equations in Eq.~\eqref{eq:fixed-pt-eqns} are 
\begin{equation}\label{eq:marg-fix-pt-eq}
    \begin{gathered}
        \tau_{g_k}^2 = \frac1{n_k}\big(\tau_{e_k}^2 + \sR_k(\tau_{g_k}^2,\zeta_k)\big),
        \\
        \zeta_k = 1 - \frac{\df_k(\tau_k^2,\zeta_k)}{n_k},
    \end{gathered}
\end{equation}
where  $\sR_k(\tau_{g_k}^2,\zeta_k) := \E\big[\| \hat \btheta_k^f - \btheta_k\|_{\bSigma}^2\big]$.
The first equation correpsonds to the $k^\text{th}$ diagonal entry of the first (matricial) fixed point equation \eqref{eq:fixed-pt-eqns}.
The right-hand side depends on $\tau_{g_k}^2$ via $\E\big[\| \hat \btheta_k^f - \btheta_k\|_{\bSigma}^2\big]$.
There is some abuse of notation in writing $\df_k$ as a function of $\tau_k^2,\zeta_k$ rather than $\bS_g,\{ \zeta_l \}$.
This is because $\df_k$ depends on $\bS_g,\{ \zeta_l \}$ only via $\tau_k^2,\zeta_k$.

We also have equation
\begin{equation}\label{eq:simul-fix-pt}
    \tau_{g_1}\tau_{g_2} \rho_g
        = 
        \frac{n_{12}}{n_1n_2}\big( \tau_{e_1}\tau_{e_2}\rho_e + \sR_{12}(\bS_g,\{\zeta_k\}) \big),
\end{equation}
where 
\begin{equation}
    \sR_{12}(\bS_g,\{\zeta_k\})
        :=
        \E[\< \hat \btheta_1^{f,\de} - \btheta_1^f, \hat \btheta_2^{f,\de} - \btheta_2^f\>_{\bSigma}].
\end{equation}
This equation corresponds to the off diagonal entry of the first (matricial) fixed point equation \eqref{eq:fixed-pt-eqns}. 
The equations \eqref{eq:marg-fix-pt-eq} and \eqref{eq:simul-fix-pt} are a rewriting of Eq.~\eqref{eq:fixed-pt-eqns}.

\subsection{Identities satisfied by the fixed-point parameters}

The fixed-point parameters satisfy several identities which will be useful in the proofs.
These follow by straightforward algebra from the fixed point equations \eqref{eq:fixed-pt-eqns} (or, equivalently, Eqs.~\eqref{eq:marg-fix-pt-eq} and \eqref{eq:simul-fix-pt}) and the structure of the fixed design models (FD-P) and (FD-N).
For future reference, we collect these identities here.

    The variance of the noise in the random-design model satisfies
    \begin{equation}\label{tau-e-ref}
    \begin{gathered}
        \tau_{e_k}^2 = \E[ \|\be_{k,\cI}\|_2^2] / |\cI|,\\
        \tau_{e_1}\tau_{e_2}\rho_e = \E[ \<\be_{1,\cI},\be_{2,\cI}\>^2] / |\cI|,
    \end{gathered}
    \end{equation}
    for any subset of indices $\cI \neq \emptyset$.

    The variance of the noise in the fixed-design model for noise estimation satisfies
    \begin{equation}\label{eq:tau-h-ref}
    \begin{gathered}
        \tau_{h_k}^2 
            = \E[\| \bh_{k,\cI}^f \|_2^2] / |\cI|
            = \sR_k(\tau_{g_k}^2,\zeta_k) 
            = \E[\|\hat \btheta_k^f - \btheta_k\|_{\bSigma}^2],
        \\
        \tau_{h_1}\tau_{h_2}\rho_h
            = \E[\< \bh_{1,\cI}^f,\bh_{2,\cI}^f\>] / |\cI|, 
            = \sR_{12}(\bS_g,\{\zeta_k\})
            = \E[\<\hat \btheta_1^f - \btheta_1,\hat \btheta_2^f - \btheta_2\>_{\bSigma}]
    \end{gathered}
    \end{equation}
    for any subset of indices $\cI \neq \emptyset$.

    The variance of the noise in the fixed-design model for parameter estimation satisfies
    \begin{equation}\label{eq:tau-g-ref}
    \begin{gathered}
        \tau_{g_k}^2
            = 
            \E[\|\bg_k\|_2^2]/p
            =
            \frac1{n_k}\big(\tau_{e_k}^2 + \sR_k(\tau_{g_k}^2,\zeta_k)\big)
            =
            \frac1{n_k}\big(\tau_{e_k}^2 + \tau_{h_k}^2\big)
            =
            \frac1{n_k}\tau_{\hat e_k^{\de}}^2,
        \\
        \tau_{g_1}\tau_{g_2}\rho_g
            = 
            \frac{\E[\<\bg_1,\bg_2\>]}{p}
            =
            \frac{n_{12}}{n_1n_2}\big(\tau_{e_1}\tau_{e_2}\rho_e + \sR_{12}(\bS_g,\{\zeta_k\})\big)
            =
            \frac{n_{12}}{n_1n_2}\big(\tau_{e_1}\tau_{e_2}\rho_e + \tau_{h_1}\tau_{h_2}\rho_h\big)
            =
            \frac{n_{12}}{n_1n_2} \tau_{\hat e_1^{\de}} \tau_{\hat e_2^{\de}} \rho_{\hat e^{\de}},
        \\
        \rho_g = \frac{n_{12}}{\sqrt{n_1n_2}}\rho_{\hat e^{\de}}.
    \end{gathered}
    \end{equation}

    The variance of the debiased estimates in the fixed-design model for noise estimation satisfies
    \begin{equation}\label{eq:tau-ed-ref}
    \begin{gathered}
        \tau_{\hat e_k^{\de}}^2
            =
            \E[\|  \hat \be_{k,\cI}^{f,\de}\|_2^2]/|\cI|
            = 
            \tau_{e_k}^2 + \sR_k(\tau_{g_k}^2,\zeta_k)
            =
            \tau_{e_k}^2 + \tau_{h_k}^2,
        \\
        \tau_{\hat e^{\de}_1}\tau_{\hat e^{\de}_2}\rho_{\hat e^{\de}}
            =
            \E[\< \hat \be_{1,\cI}^{f,\de} , \hat \be_{2,\cI}^{f,\de} \>]/|\cI| 
            =
            \tau_{e_1}\tau_{e_2}\rho_e + \sR_{12}(\bS_g,\{\zeta_k\})
            =
            \tau_{e_1}\tau_{e_2}\rho_e + \tau_{h_1}\tau_{h_2}\rho_h.
    \end{gathered}
    \end{equation}

    The residuals in the fixed design model are written as $\bu_k = \hat \be_k^f$.
    They satisfy $\bu_{k,\cI_k^c} = \hat \be_{k,\cI_k^c}^f = 0$ and also
    \begin{equation}\label{eq:res-ref}
    \begin{gathered}
        \frac1{n_k} \E[\|\bu_k^f\|_2^2] 
            = 
            \zeta_k^2( \tau_{e_k}^2 + \tau_{h_k}^2)
            =
            \zeta_k^2 \tau_{\hat e^{\de}_k}^2 
            =
            n_k \zeta_k^2 \tau_{g_k}^2,
        \\
        \E[\<\bu_1^f, \bu_2^f\>] 
            = 
            n_{12}\zeta_1\zeta_2(\tau_{e_1}\tau_{e_2}\rho_e + \tau_{h_1}\tau_{h_2}\rho_h)
            =
            n_{12}\zeta_1\zeta_2 \tau_{\hat e_1^{\de}}\tau_{\hat e_2^{\de}} \rho_{\hat e^{\de}}
            =
            n_1n_2\tau_{g_1}\tau_{g_2}\zeta_1\zeta_2\rho_g,
        \\
        \E[\|\bu_{1,\cI_2}^f\|_2^2] 
            = 
            n_{12}\zeta_1^2( \tau_{e_1}^2 + \tau_{h_1}^2)
            =
            n_{12}\zeta_1^2 \tau_{\hat e_1^{\de}}^2 
            =
            n_{12}n_1 \zeta_1^2 \tau_{g_1}^2.
    \end{gathered}
    \end{equation}

\subsection{Bounds on the fixed-point parameters}

We will require the following lower and upper bounds on the fixed point parameters.

\begin{lemma}[Bounds on fixed-point parameters]\label{lem:bound-on-fixed-pt}
    For $k = 1,2$,
    there exist $\cPmodel$, $\cPregr$, and regression method dependent constants $C,c > 0$ (which may change at each appearance) such that
    such that 
    $ c / n_k < \tau_{g_k}^2 < C/n_k $, 
    $c <  \zeta_k \leq 1$, 
    $c p / n_k < \tau_{h_k}^2 < Cp/n_k$,
    $c < \tau_{\hat e_k^\de}^2 < C$,
    and $|\rho_g|,|\rho_h|,|\rho_e|,|\rho_{\hat e^{\de}}| < 1 - c$.
    Further, $\df_k / p \leq 1$.
\end{lemma}

\noindent The proofs of Lemma \ref{lem:bound-on-fixed-pt} is provided in Section \ref{sec:existence-fix-pt}.

\begin{remark}
    In addition to providing important technical results required for our proofs,
    Lemma \ref{lem:bound-on-fixed-pt} is interesting in its own right because it shows how the fixed point parameters scale with $p$ and $n_k$.
    Most existing exact characterization results either consider an asymptotic limit in $n_k/p \rightarrow \delta$ or assume that $n_k/p$ is upper or lower bounded by a constant. 
    In such results, $n_k/p$ can be replaced by a constant in upper and lower bounds.
    Because we do not assume $n_k/p$ is bounded above in the case of ridge regression or OLS,
    such replacement cannot in general be made, and tracking the dependence of the fixed point parameters on $n_k$ and $p$ will be necessary.
    For example, it establishes that $\tau_{g_k}^2 = \Theta(1/n_k)$, 
    indicating that $\hat \btheta_k^{\de}$ provides---on average across coordinates---an observation of each coordinate at the parametric rate.
    It also establishes that $\tau_{h_k}^2 = \Theta(p/n_k)$, indicating (recall Table \ref{tab:fix-pt}) that the prediction error for each regression is, with high probability, $\| \hat \btheta_k - \btheta_k \|_{\bSigma}^2 = \Theta(p/n_k)$.
    In the proportional regime, this is what we expect.
\end{remark}

\section{Proof of joint characterization (Theorem \ref{thm:joint-characterization}(i) and (ii))}

In this section, we prove Theorem \ref{thm:joint-characterization}(i) and (ii),
deferring several technical details to later appendices.
We prove prove the extension to empirical $\hat \df_k$ (i.e., Theorem \ref{thm:joint-characterization}(iii)) in Section \ref{sec:emp-df}.

We prove Theorem \ref{thm:joint-characterization}(i) and (ii) by
bounding the differences to an intermediate quantity that depends on only one of the regression models at a time. 
The intermediate quantity is a function of the vectors\footnote{Later in the proof, we will analyze the behavior of the regression problem conditional on $\mathsf{Cond}_k$, justifying our choice of notation.}
\begin{equation}\label{eq:Cond-k}
    \mathsf{Cond}_k
        :=
        (\hat \btheta_k, \bX^\top \hat \be_k , \hat \be_k, \bX (\hat \btheta_k - \btheta_k) ,\be_1,\be_2) \in (\reals^p)^2 \times (\reals^N)^4.
\end{equation}
In fact, it will depend on this vectors only via two random vectors
$\hat \bg_k := \hat \bg_k(\mathsf{Cond}_k) \in \reals^p$ and $\hat \bh_k = \hat \bh_k(\mathsf{Cond}_k )\in \reals^N$
which are deterministic functions of $\mathsf{Cond}_k$.
The definition of these functions is complicated and non-intuitive,
and is carefully chosen to make our proofs work.
We postpone providing their definitions to Section \ref{sec:cond-char} (see Eq.~\eqref{eq:hat-g1-hat-h1}).
For now, 
the reader should think of these quantities as approximating the errors of the debiased estimates in the first regression model.
Indeed:
\begin{lemma}\label{lem:hat-g-hat-h-db-err}
    Assume \textsf{A1} and \textsf{A2}.

    There exist $\cPmodel$, $\cPregr$ and regression method-dependent $C',c' > 0$ and $\sC,\sc: \reals_{>0} \rightarrow \reals_{>0}$ such that for $\epsilon < c'$,
    with probability at least $1 - \sC(\epsilon) e^{-\sc(\epsilon) p}$ 
    \begin{equation}
        \big\|\hat \bg_k - \bSigma^{1/2}(\hat \btheta_k^{\de} - \btheta_k)\big\|_2 \leq \sqrt{\frac{p}{n_k}}\,\epsilon,
        \qquad 
        \frac1{\sqrt{n_k}}\big\|\hat \bh_{k,\cI_k} - (\hat \be_{k,\cI_k}^{\de} - \be_{k,\cI_k})\big\|_2 \leq \sqrt{\frac{p}{n_k}}\,\epsilon.
    \end{equation}
\end{lemma}
\noindent We prove Lemma \ref{lem:hat-g-hat-h-db-err} in Section \ref{sec:char-support}.

Let $\phi_e,\phi_\theta$ and $M_1,M_2$ be as in Theorem \ref{thm:joint-characterization}.
The intermediate quantities are
\begin{equation}\label{eq:def-phi|1}
\begin{gathered}
    \phi_{\theta|1}(\mathsf{Cond}_1)
        := 
        \E\big[\phi_\theta(\hat \btheta_1,\hat \btheta_1^{\de},\hat \btheta_2^f,\hat \btheta_2^{f,\de})\bigm|\bg_1^f = \hat \bg_1\big],
    \\
    \phi_{e|1}(\mathsf{Cond}_1)
        := 
        \E\Big[\phi_e\Big(\frac{\be_{1,\cI_2}}{\sqrt{n_2}},\frac{\be_{2,\cI_2}}{\sqrt{n_2}},\frac{\hat \be_{1,\cI_2}}{\sqrt{n_2}},\frac{\hat \be_{2,\cI_2}^f}{\sqrt{n_2}},\frac{\hat \be_{1,\cI_2}^{\de}}{\sqrt{n_2}},\frac{\hat \be_{2,\cI_2}^{f,\de}}{\sqrt{n_2}}\Big)\Bigm|\bh_1^f = \hat \bh_1,\, \be_1, \be_2\Big].
\end{gathered}
\end{equation}
In this display, the expectations are taken in the fixed-design model, and $\hat \btheta_1,\hat \btheta_1^{\de},\hat \be_{1,\cI_2},\hat \be_{1,\cI_2}^\de$, which are functions of $\mathsf{Cond}_1$, are considered fixed.
For clarity and with some abuse of notation, we have written each argument passed to $\phi_e$ explicitly rather than use that set-notation as in Theorem \ref{thm:joint-characterization}.
This is because some of the arguments come from the fixed-design model, and some (which are functions of $\mathsf{Cond}_1$), come from the random design model.
We bound
\begin{equation}\label{eq:joint-bound-1}
\begin{aligned}
    &\P\Big(
        \Big|
            \phi_e\Big(\Big\{\frac{\be_{k,\cI_2}}{\sqrt{n_2}}\Big\},\Big\{\frac{\hat \be_{k,\cI_2}}{\sqrt{n_2}}\Big\},\Big\{\frac{\hat \be_{k,\cI_2}^{\de}}{\sqrt{n_2}}\Big\}\Big)
            -
            \E\Big[
                \phi_e\Big(\Big\{\frac{\be_{k,\cI_2}}{\sqrt{n_2}}\Big\},\Big\{\frac{\hat \be_{k,\cI_2}^f}{\sqrt{n_2}}\Big\},\Big\{\frac{\hat \be_{k,\cI_2}^{\de,f}}{\sqrt{n_2}}\Big\}\Big)
            \Big]
        \Big|
        > (M_1 \sqrt{p/n_2} + M_2)\epsilon
    \Big)
    \\
        &\qquad\qquad\qquad \leq 
        \P\Big(
            \Big|
                \phi_e\Big(\Big\{\frac{\hat \be_{k,\cI_2}}{\sqrt{n_2}}\Big\},\Big\{\frac{\hat \be_{k,\cI_2}^{\de}}{\sqrt{n_2}}\Big\}\Big)
                -
                \phi_{e|1}(\mathsf{Cond}_1)
            \Big|
            > \frac{M_2}{2}\,\epsilon
        \Big)
    \\
    &\qquad\qquad\qquad\qquad+
        \P\Big(
            \Big|
                \phi_{e|1}(\mathsf{Cond}_1)
                -
                \E\Big[
                    \phi_e\Big(\Big\{\frac{\be_{k,\cI_2}}{\sqrt{n_2}}\Big\},\Big\{\frac{\hat \be_{k,\cI_2}^f}{\sqrt{n_2}}\Big\},\Big\{\frac{\hat \be_{k,\cI_2}^{\de,f}}{\sqrt{n_2}}\Big\}\Big)
                \Big]
            \Big|
            > \Big(M_1 \sqrt{\frac{p}{n_2}} + \frac{M_2}2\Big)\,\epsilon
        \Big),
\end{aligned}
\end{equation}
and 
\begin{equation}
\begin{aligned}
    &\P\Big(
        \Big|
            \phi_\theta\big(\{\hat \btheta_k\},\{\hat \btheta_k^{\de}\}\big)
            -
            \E\big[
                \phi_\theta\big(\{\hat \btheta_k^f\},\{\hat \btheta_k^{f,\de}\}\big)
            \big]
        \Big|
        > (M_1 \sqrt{p/n_1} + M_2)\epsilon
    \Big)
    \\
        & \leq 
        \P\Big(
            \Big|
                \phi_\theta\big(\{\hat \btheta_k\},\{\hat \btheta_k^{\de}\}\big)
                -
                \phi_{\theta|1}(\mathsf{Cond}_1)
            \Big|
            > \frac12\Big(M_1 \sqrt{\frac{p}{n_1}} + M_2\Big)\,\epsilon
        \Big)
    \\
    &\qquad\qquad\qquad\qquad\qquad+
        \P\Big(
            \Big|
                \phi_{\theta|1}(\mathsf{Cond}_1)
                -
                \E\big[
                    \phi_\theta\big(\{\hat \btheta_k^f\},\{\hat \btheta_k^{f,\de}\}\big)
                \big]
            \Big|
            > \frac12 \Big(M_1 \sqrt{\frac{p}{n_1}} + M_2\Big)\,\epsilon
        \Big).
\end{aligned}
\end{equation}
The next two lemmas bound the terms in this decomposition.
\begin{lemma}\label{lem:phi-|1-conc}
    Assume \textsf{A1} and \textsf{A2}.

    There exist $\cPmodel$, $\cPregr$ and regression method-dependent $c' > 0$ and $\sC,\sc: \reals_{>0} \rightarrow \reals_{>0}$ such that
    for $\epsilon < c'$
    \begin{equation}
    \begin{gathered}
        \P\Big(
            \Big|
                \phi_{\theta|1}(\mathsf{Cond}_1)
                -
                \E\big[
                    \phi_\theta\big(\{\hat \btheta_k^f\},\{\hat \btheta_k^{f,\de}\}\big)
                \big]
            \Big|
            > \Big(M_1 \sqrt{\frac{p}{n_1}} + M_2 \sqrt{\frac{p}{n_2}}\;\Big)\,\epsilon
        \Big)
        \leq \sC(\epsilon)e^{-\sc(\epsilon)p},
        \\
        \P\Big(
            \Big|
                \phi_{e|1}(\mathsf{Cond}_1)
                -
                \E\Big[
                    \phi_e\Big(\Big\{\frac{\be_{k,\cI_2}}{\sqrt{n_2}}\Big\},\Big\{\frac{\hat \be_{k,\cI_2}^f}{\sqrt{n_2}}\Big\},\Big\{\frac{\hat \be_{k,\cI_2}^{\de,f}}{\sqrt{n_2}}\Big\}\Big)
                \Big]
            \Big|
            > \big(M_1 +M_2\big)\sqrt{\frac{p}{n_2}}\,\epsilon
        \Big)
        \leq \sC(\epsilon)e^{-\sc(\epsilon)p}.
    \end{gathered}
    \end{equation}
\end{lemma}

\begin{lemma}[Conditional characterization]\label{lem:conditional-characterization}
    Assume \textsf{A1} and $\mathsf{A2}$ are satisfied.

    There exist $\cPmodel,\cPregr$ and regression method-dependent $C',c' > 0$ and $\sC,\sc:\reals_{>0} \rightarrow \reals_{>0}$ such that for $\epsilon < c'$
    \begin{equation}
    \begin{gathered}
        \P\Big(
                \Big|
                    \phi_\theta\big(\{\hat \btheta_k\},\{\hat \btheta_k^{\de}\}\big)
                    -
                    \phi_{\theta|1}(\mathsf{Cond}_1)
                \Big|
                \geq  \Big(M_1 \sqrt{\frac{p}{n_1}} + M_2 \;\Big)\,\epsilon
            \Big)
            \leq 
            \sC(\epsilon)e^{-\sc(\epsilon)p},
        \\
        \P\Big(
                \Big|
                    \phi_e\Big(\Big\{\frac{\be_{k,\cI_2}}{\sqrt{n_2}}\Big\},\Big\{\frac{\hat \be_{k,\cI_2}}{\sqrt{n_2}}\Big\},\Big\{\frac{\hat \be_{k,\cI_2}^{\de}}{\sqrt{n_2}}\Big\}\Big)
                    -
                    \phi_{e|1}(\mathsf{Cond}_1)
                \Big|
                \geq  M_2\epsilon
            \Big)
            \leq 
            \sC(\epsilon)e^{-\sc(\epsilon)p}.
    \end{gathered}
    \end{equation}    
\end{lemma}
\noindent Lemma \ref{lem:phi-|1-conc} involves only the distribution of objects from the first regression.
To bound it, 
we will use a characterization result analogous to Theorem \ref{thm:joint-characterization} which applies to one regression at a time.
We call this the \emph{marginal characterization}. 
It is stated in Section \ref{sec:marg-char} and proved in Section \ref{sec:marg-char-proof}.
To prove Lemma \ref{lem:conditional-characterization},
we study the distribution of the estimators from the second regression conditional on $\mathsf{Cond}_1$. 
We will use a characterization result analogous to Theorem \ref{thm:joint-characterization}
which applies to one of the regressions conditionally on the other.
We call this the \emph{conditional characterization}. 
It is stated in Section \ref{sec:cond-char} and proved in Section \ref{sec:cond-char-proof}.

By assumption \textsf{A2}, $p/n_2 \leq C$ for some $\cPmodel,\cPregr$, and regression-method dependent $C$, whence, after adjusting constants, Lemma \ref{lem:phi-|1-conc} allows us to bound the second terms in the decompositions above by $\sC(\epsilon)e^{-\sc(\epsilon)p}$.
Lemma \ref{lem:conditional-characterization} allows us to bound the first terms in the decompositions above by $\sC(\epsilon)e^{-\sc(\epsilon)p}$.
Combining these bounds, we conclude Theorem \ref{thm:joint-characterization}(i) and (ii).
 
\section{Marginal characterization}
\label{sec:marg-char}

In this section we state and prove the marginal characterization, which describes the behavior of one regression at a time. 
Our primary application of the marginal characterization is to prove Lemma \ref{lem:phi-|1-conc}.
We postpone this proof of Lemma \ref{lem:phi-|1-conc} to Section \ref{sec:phi|1-conc-proof}.

\subsection{Statement of marginal characterization}

The next lemma states the marginal characterization.
\begin{lemma}\label{lem:marginal-characterization}
    Assume \textsf{A1} and \textsf{A2}.

    \begin{enumerate}[(i)]
        
        \item 
        There exist $\cPmodel$, $\cPregr$ and regression method-dependent $c' > 0$ and $\sC,\sc: \reals_{>0} \rightarrow \reals_{>0}$ such that
        for $\epsilon < c'$ and 1-Lipschitz $\phi_\theta: (\reals^p)^2 \rightarrow \reals$, $\phi_e: (\reals^{n_k})^4 \rightarrow \reals$ with probability at least $1 - \sC(\epsilon) e^{-\sc(\epsilon) p}$ 
        \begin{equation}
        \begin{gathered}
            \big| 
                \phi_\theta\big( \hat \btheta_k , \hat \btheta_k^{\de} \big) 
                    - 
                    \E\big[\phi_\theta\big( \hat \btheta_k^f , \hat \btheta_k^{f,\de}  \big)\big] 
            \big| 
                < 
                \sqrt{\frac{p}{ n_k }}\,\epsilon,
            \\
            \Big| 
                \phi_e\Big(
                    \frac{ \be_{1,\cI_k}}{\sqrt{n_k}},
                    \frac{ \be_{2,\cI_k}}{\sqrt{n_k}},
                    \frac{\hat \be_{k,\cI_k}}{\sqrt{n_k}},
                    \frac{\hat \be_{k,\cI_k}^{\de}}{\sqrt{n_k}}
                \Big) 
                - 
                \E\Big[
                    \phi_e\Big(
                        \frac{ \be_{1,\cI_k}}{\sqrt{n_k}},
                        \frac{ \be_{2,\cI_k}}{\sqrt{n_k}},
                        \frac{\hat \be_{k,\cI_k}^f}{\sqrt{n_k}},
                        \frac{\hat \be_{k,\cI_k}^{f,\de}}{\sqrt{n_k}}
                    \Big)\Big] \Big| < \sqrt{\frac{p}{ n_k }} \, \epsilon,
        \end{gathered}
        \end{equation}
        where the fixed design model and estimators are defined with $\bS_g,\{\zeta_k\}$ satisfying the fixed-point equations \eqref{eq:fixed-pt-eqns} and $\bS_h$ is defined by Eq.~\eqref{eq:def-R-rhoh}.

        \item 
        The same holds if we replace $\df_k$ by $\hat \df_k$ in the definitions of $\hat \btheta_k^\de$, $\hat \be_{k,\cI_k}^\de$ in Eqs.~\eqref{eq:param-est} and \eqref{eq:noise-est},
        where $\hat \df_k$ is defined by Eq.~\eqref{eq:hat-df}.

    \end{enumerate}

\end{lemma}

\noindent We provide several remarks on Lemma \ref{lem:marginal-characterization}. They are not essential to our exposition and may be skipped.
\begin{remark}[Relation to existing literature]\label{rmk:rel-lit}
    Marginal characterization results are now standard in the literature \cite{bayati2011lasso,stojnic2013framework,thrampoulidis2015,Sur14516,dobriban2018,miolane2018distribution,gerbelot2020,celentano2020lasso}.
    Lemma \ref{lem:marginal-characterization} differs from previous results in a couple respects.
    Most results are typically proved assuming an upper bound on $n_k/p$.
    In the case of OLS and ridge regression,
    we assume no such upper bound on $n_k/p$,
    and the characterization provided by Lemma \ref{lem:marginal-characterization} provides better concentration as $n_k/p$ gets large.
    Establishing the correct dependence on $n_k/p$ for $n_k/p$ large requires a more refined argument than existing results in the literature.
    We expect the same dependence on $n_k/p$ to hold for the Lasso and $\alpha$-smoothed Lasso, but do not pursue this in the present paper. 
    (Recall that for the Lasso, assumption \textsf{A2} imposes an upper bound on $n_k/p$).
    
    Further, as far as we are aware,
    in the case of ridge regression, a characterization result along the lines of Lemma \ref{lem:marginal-characterization} has not yet appeared in the literature even for $n_k/p$ bounded.
    This is perhaps surprising because the ridge regression estimator is simpler to study than the Lasso or other estimators which have been studied.
    The marginal characterization result we prove fills this gap.
\end{remark}

\begin{remark}[Rates of concentration]\label{rmk:marg-rates-conc}
    To assess the rates of concentration provided by Lemma \ref{lem:marginal-characterization}, 
    one should compare them to those which occur for the analogous fixed-design quantities.
    Because proximal operators are 1-Lipschitz, 
    $\bSigma^{1/2}\hat \btheta_k^f$ and $\bSigma^{1/2} \hat \btheta_k^{f,\de}$ are 1-Lipschitz functions of $\bg_k^f$.
    Because the variance $\tau_{g_k}^2 \leq C/n_k$ (see Lemma \ref{lem:bound-on-fixed-pt}), 
    Gaussian concentration of Lipschitz functions gives
    \begin{equation}
        \P
        \Big(
            \big| 
                \phi_\theta\big( \hat \btheta_k^f , \hat \btheta_k^{f,\de} \big) 
                    - 
                    \E\big[\phi_\theta\big( \hat \btheta_k^f , \hat \btheta_k^{f,\de}  \big)\big] 
            \big| 
                < 
                \sqrt{\frac{p}{ n_k }}\,\epsilon
        \Big)
            \geq 
            1 - Ce^{-cp\epsilon^2}.
    \end{equation}
    Similarly, $\hat \be_k^{f,\de}$ is Gaussian with variance $(\tau_{e_k}^2 + \tau_{h_k}^2)\id_N$, and $\hat \be_k^f$ is a 1-Lipschitz function of $\hat \be_k^f$.
    Because $\tau_{e_k}^2 \leq C$ (by assumption $\mathsf{A1}$) and $\tau_{h_k}^2 \leq Cp/n_k \leq C$ (by assumption $\mathsf{A2}$ and Lemma \ref{lem:bound-on-fixed-pt}), 
    Gaussian concentration of Lipshitz functions gives 
    \begin{equation}
        \P
        \Big(
            \Big| 
                \phi_e\Big(
                    \frac{\be_{1,\cI_k}}{\sqrt{n_k}},
                    \frac{\be_{2,\cI_k}}{\sqrt{n_k}},
                    \frac{\hat \be_{k,\cI_k}^f}{\sqrt{n_k}},
                    \frac{\hat \be_{k,\cI_k}^{f,\de}}{\sqrt{n_k}}
                \Big) 
                - 
                \E\Big[
                    \phi_e\Big(
                        \frac{\be_{1,\cI_k}}{\sqrt{n_k}},
                        \frac{\be_{2,\cI_k}}{\sqrt{n_k}},
                        \frac{\hat \be_{k,\cI_k}^f}{\sqrt{n_k}},
                        \frac{\hat \be_{k,\cI_k}^{f,\de}}{\sqrt{n_k}}
                    \Big)\Big] \Big| < \sqrt{\frac{p}{ n_k }} \, \epsilon
        \Big)
        \geq 1 - Ce^{-cp\epsilon^2}.
    \end{equation}
    That is, in the fixed design model, those fluctuations which have probability $e^{-\Theta(p)}$ have size  $O(\sqrt{p/n_k})$.
    Lemma \ref{lem:marginal-characterization} establishes the same for the random-design model.
\end{remark}

\begin{remark}[High probability in $p$ rather than $n_k$]
    The reader may wonder why the probability bounds we provide are of the form $\sC(\epsilon) e^{-\sc(\epsilon) p}$ rather than $\sC(\epsilon) e^{-\sc(\epsilon) n_k}$.
    One regime in which the distinction is manifest is a fixed-$p$, $n_k \rightarrow \infty$ asymptotics.
    In this limit, $\sC(\epsilon) e^{-\sc(\epsilon) n_k} \rightarrow 0$ whereas $\sC(\epsilon) e^{-\sc(\epsilon) p}$ remains constant.
    Our proof techniques are unable to establish vanishing probability bounds in this limit.
    Thus, we are required to use probability bounds of the form $\sC(\epsilon) e^{-\sc(\epsilon) p}$ rather than $\sC(\epsilon) e^{-\sc(\epsilon) n_k}$. 
    We suspect, though are not certain, that this is fundamental in some cases.

    Previous works (see \cite{miolane2018distribution,celentano2020lasso})
    typically write probability bounds in the form $\sC(\epsilon) e^{-\sc(\epsilon) n_k}$, which decays exponentially fast in $n_k$.
    Because these works assume an upper and lower bound on $n_k/p$, the two forms are in fact equivalent after adjusting constants.
    The distinction becomes relevant in the current paper because---at least in the case of ridge-regression---we do not assume an upper bound on $n_k/p$.
\end{remark}

\subsection{Proof of marginal characterization (Lemma \ref{lem:marginal-characterization}(i))}
\label{sec:marg-char-proof}

Here we prove Lemma \ref{lem:marginal-characterization}(i).
We postpone the proof of Lemma \ref{lem:marginal-characterization}(ii) to Section \ref{sec:emp-df}.

We prove Lemma \ref{lem:marginal-characterization} using Gordon's convex Gaussian min-max theorem (stated below as Lemma \ref{lem:gordon-marginal}) using a now standard argument \cite{thrampoulidis2015,miolane2018distribution,celentano2020lasso}.
Because 
assumption \textsf{A2} does not require $n_k/p$ be bounded above in the case of ridge-regression,
we are more explicit in tracking dependence on this quantity than current literature.

Recall Gordon's convex min-max theorem.
\begin{lemma}[Marginal Gordon]\label{lem:gordon-marginal}
    Consider any continuous function $\psi: \reals^N \times \reals^p \rightarrow \reals$ which is concave in its first argument and convex in its second argument.
    Let $\bA \in \reals^{N\times p}$ have entries $A_{ij} \stackrel{\mathrm{iid}}\sim \normal(0,1)$,
    $\bxi_g \stackrel{\mathrm{iid}}\sim \normal(0,\id_p)$, and $\bxi_h \sim \normal(0,\id_N)$, all independent,
    and define
    \begin{equation}
        \bg_{\mathrm{mg}}(\bu)
            :=
            \|\bu\|_2 \bxi_g,
        \qquad 
        \bh_{\mathrm{mg}}(\bv)
            :=
            \|\bv\|_2 \bxi_h,
    \end{equation}
    (where the subscript stands for ``marginal Gordon'').
    \begin{enumerate}[(i)]

        \item If $E_u \in \reals^N$, $E_v \in \reals^p$ are compact sets,
        then for any $t\in \reals$
        \begin{equation}
            \P\Big( \min_{\bv \in E_v} \max_{\bu \in E_u} \bu^\top \bA \bv + \psi(\bu,\bv) \leq t \Big)
                \leq 
                2 \P\Big( \min_{\bv \in E_v} \max_{\bu \in E_u} - \<\bg_{\mathrm{mg}}(\bu), \bv\> + \<\bh_{\mathrm{mg}}(\bv), \bu\> + \psi(\bu,\bv)  \leq t \Big).
        \end{equation}

        \item If $E_u \in \reals^N$, $E_v \in \reals^p$ are compact, convex sets,
        then for any $t\in \reals$
        \begin{equation}
            \P\Big( \min_{\bv \in E_v} \max_{\bu \in E_u} \bu^\top \bA \bv + \psi(\bu,\bv) \geq t \Big)
                \leq 
                2 \P\Big( \min_{\bv \in E_v} \max_{\bu \in E_u} - \<\bg_{\mathrm{mg}}(\bu), \bv\> + \<\bh_{\mathrm{mg}}(\bv), \bu\>  + \psi(\bu,\bv)  \geq t \Big).
        \end{equation}

    \end{enumerate}
\end{lemma}

To apply Gordon's inequality to study the regression estimators marginally,
we rewrite the optimization \eqref{eq:param-est} as a min-max problem
\begin{equation}
    \min_{\bpi \in \reals^p}
    \max_{\substack{\bu_{\cI_k} \in \reals^{n_k}\\\bu_{\cI_k^c}=0}}
        \Big\{
            \frac1{n_k}\bu^\top\big(\be_k + \bX(\btheta_k - \bpi)\big) - \frac1{2n_k}\|\bu\|_2^2 + \Omega_k(\bpi)
        \Big\}.
\end{equation}
Define the matrix $\bA := -\bX \bSigma^{-1/2}$, which is distributed as in Lemma \ref{lem:gordon-marginal}.
If we define $\bv = \bSigma^{1/2}(\bpi - \btheta_k)$, 
the min-max problem can be written as
\begin{equation}\label{eq:regr-min-max}
    \min_{\bv \in \reals^p}
        \max_{\substack{\bu_{\cI_k} \in \reals^{n_k}\\\bu_{\cI_k^c}=0}}
        L_k(\bu,\bv)
        :=
    \min_{\bv \in \reals^p}
    \max_{\substack{\bu_{\cI_k} \in \reals^{n_k}\\\bu_{\cI_k^c}=0}}
        \Big\{
            \frac1{n_k}\bu^\top \bA \bv + \frac1{n_k}\bu^\top\be_k - \frac1{2n_k}\|\bu\|_2^2 + \bar \Omega_k(\bv)
        \Big\},
\end{equation}
where $\bar \Omega_k(\bv) := \Omega_k(\btheta_k + \bSigma^{-1/2}\bv)$.
This problem is convex-concave, and with probability 1 it has unique saddle point
\begin{equation}\label{eq:vk-uk}
    \bv_k 
        :=
        \bSigma^{1/2}(\hat \btheta_k - \btheta_k)
        =
        \argmin_{\bv \in \reals^p}
        \max_{\substack{\bu_{\cI_k} \in \reals^{n_k}\\\bu_{\cI_k^c}=0}}
        L_k(\bu,\bv),
    \qquad 
    \bu_k := \hat \be_k
        =
        \argmax_{\substack{\bu_{\cI_k} \in \reals^{n_k}\\\bu_{\cI_k^c}=0}}\;
        \min_{\bv \in \reals^p}
        L_k(\bu,\bv).
\end{equation}
By \cite[Corollary 37.3.2]{rockafellar-1970a}, the minimization and maximization can be exchanged.

We first characterize the behavior of $\bv_k$ and $\bu_k$.
They are characterized by comparison to the fixed-design model quantities 
\begin{equation}\label{eq:vkf}
    \bv_k^f := \bSigma^{1/2}(\hat \btheta_k^f - \btheta_k)\quad \text{and} \quad \bu_k^f := \hat \be_k^f.
\end{equation}
For furture reference, we point out that by Eq.~\eqref{eq:fixed-design-est}
\begin{equation}\label{eq:vf-prox}
    \bv_k^f = \argmin_{\bpi \in \reals^p}
        \Big\{
            \frac12\| \bg_k^f - \bpi \|_2^2 + \frac1{\zeta_k} \bar \Omega_k(\bpi)
        \Big\}.
\end{equation}
Consider $1$-Lipschitz functions $\phi_v:\reals^p \rightarrow \reals$ 
and $\phi_u:(\reals^{n_k})^3 \rightarrow \reals$.
We will show that with high probability 
\begin{equation}\label{eq:marg-E}
\begin{gathered}
    \bv_k \in 
    E_v(\epsilon)
        :=
        \Big\{
            \bv \in \reals^p : 
            \big|
                \phi_v\big( \bv \big)
                -
                \E\big[\phi_v( \bv_k^f   )\big] 
            \big| 
            < \sqrt{\frac{p}{ n_k }} \, \epsilon
        \Big\},
    \\
    \bu_k \in 
    E_u(\epsilon)
        :=
        \Big\{
            \bu \in \reals^{n_k}
            :
            \Big|
                \phi_u\Big( 
                    \frac{\bu}{ \sqrt{n_k} } ,
                    \frac{\be_{1,\cI_k}}{ \sqrt{n_k} } ,
                    \frac{\be_{2,\cI_k}}{ \sqrt{n_k} } 
                \Big) 
                - 
                \E\Big[
                \phi_u\Big(
                    \frac{ \bu_{k,\cI_k}^f}{\sqrt{n_k}} ,
                    \frac{\be_{1,\cI_k}}{ \sqrt{n_k} } ,
                    \frac{\be_{2,\cI_k}}{ \sqrt{n_k} } 
                \Big)
                \Bigm| \{ \be_l \}_l
                \Big]
            \Big| 
            < \sqrt{\frac{p}{ n_k }} \, \epsilon
        \Big\}.
\end{gathered}
\end{equation}
Note that the set $E_u(\epsilon)$ depends implicitly on $\be_1,\be_2$, though for compactness we supress this in the notation.
To establish the previous display, we show that the objectives in Eq.~\eqref{eq:vk-uk} are sub-optimal when these constraints are violated.
Precisely:
\begin{lemma}[Marginal control of primary objective]\label{lem:marg-sub-opt}
    There exists $\ell_k^* = \ell_k^*(\be_k,n_k,p,\btheta_k,\sigma_k,\bSigma,\Omega_k)$ and $\cPmodel$, $\cPregr$ and regression method-dependent $C',c' > 0$ and $\sC,\sc: \reals_{>0} \rightarrow \reals_{>0}$ such that
    for $\epsilon < c'$ with probability at least $1 - \sC(\epsilon) e^{-\sc(\epsilon) p}$ 
    \begin{equation}\label{eq:marg-sub-opt}
    \begin{gathered}
        \Big|
            \max_{\substack{\bu \in \reals^N \\ \bu_{\cI_k^c} = 0} }
                    \min_{\bv \in \reals^p}
                L_k(\bu,\bv)
            - 
            \ell_k^*
        \Big|
        =
        \Big|
            \min_{ \bv \in \reals^p }\
            \max_{\substack{\bu \in \reals^N \\ \bu_{\cI_k^c} = 0} }
                L_k(\bu,\bv)
            - 
            \ell_k^*
        \Big|
        \leq 
             C'\frac{p}{n_k}\frac{\epsilon^2}2
    \\
        \max_{\substack{\bu_{\cI_k} \in E_u^c(\epsilon)\\ \bu_{\cI_k^c}= 0}}\;
              \min_{\bv \in \reals^p}
                L_k(\bu,\bv)
            \leq \ell_k^* - C'\frac{p}{n_k}\epsilon^2,
    \qquad
            \min_{\bv \in E_v^c(\epsilon)}\;
            \max_{\substack{\bu \in \reals^N \\ \bu_{\cI_k^c} = 0} }
                    L_k(\bu,\bv)
                \geq \ell_k^* + C'\frac{p}{n_k}\epsilon^2.
    \end{gathered}
    \end{equation}
\end{lemma}
\noindent Lemma \ref{lem:marg-sub-opt} is proved using the marginal Gordon inequality (Lemma \ref{lem:gordon-marginal}), as described in the following sections.

First we prove the marginal characterization (Lemma \ref{lem:marginal-characterization}) using Lemma \ref{lem:marg-sub-opt}.

\begin{proof}[Proof of marginal characterization (Lemma \ref{lem:marginal-characterization}(i))]
    By the optimality of $\bu_k,\bv_k$ (Eq.~\eqref{eq:vk-uk}),
    when Eq.~\eqref{eq:marg-sub-opt} occurs we have $\bu_k \in E_u(\epsilon)$ and $\bv_k \in E_v(\epsilon)$.
    Moreover, by Gaussian concentration of Lipschitz functions,
    with probability at least $1 - Ce^{-pc\epsilon^2}$,
    \begin{equation}
        \Big|
            \E\Big[
                \phi_u\Big(
                    \frac{ \bu_{k,\cI_k}^f}{\sqrt{n_k}} ,
                    \frac{\be_{1,\cI_k}}{ \sqrt{n_k} } ,
                    \frac{\be_{2,\cI_k}}{ \sqrt{n_k} } 
                \Big)
                \Bigm| \{ \be_l \}_l
            \Big]
            -
            \E\Big[
                \phi_u\Big(
                    \frac{ \bu_{k,\cI_k}^f}{\sqrt{n_k}} ,
                    \frac{\be_{1,\cI_k}}{ \sqrt{n_k} } ,
                    \frac{\be_{2,\cI_k}}{ \sqrt{n_k} } 
                \Big)
            \Big]
        \Big|
        \leq \sqrt{\frac{p}{n_k}}\,\epsilon.
    \end{equation}
    Combined with Lemma \ref{lem:marg-sub-opt},
    we conclude that for $\epsilon \leq c'$ with probability at least $1 - \sC(\epsilon)e^{-\sc(\epsilon)p}$
    \begin{gather}\label{eq:lipsch-of-tilb-tilu}
            \big|
                \phi_v( \bv_k ) 
                - 
                \E[\phi_v( \bv_k^f ) ]
            \big| 
            \leq \sqrt{\frac{p}{ n_k }} \, \epsilon,
        \qquad
            \Big|
                \phi_u\Big( \frac{\bu_k}{ \sqrt{n_k} } ,
                    \frac{\be_{1,\cI_k}}{ \sqrt{n_k} } ,
                    \frac{\be_{2,\cI_k}}{ \sqrt{n_k} }  \Big) 
                - 
                \E\Big[\phi_u\Big( \frac{\bu_k^f}{ \sqrt{n_k}} ,
                    \frac{\be_{1,\cI_k}}{ \sqrt{n_k} } ,
                    \frac{\be_{2,\cI_k}}{ \sqrt{n_k} }  \Big) \Big]
            \Big| 
            \leq \sqrt{\frac{p}{ n_k }} \, \epsilon.
    \end{gather}

    We now establish the second line of Lemma \ref{lem:marginal-characterization}.
    Restricted to the coordinates $\cI_k$,
    the debiased estimate of the noise can be written as (see Eq.~\eqref{eq:noise-est}) $\hat \be_{k,\cI_k}^{\de} = \zeta_k^{-1} \bu_{k,\cI_k}$,
    which is $C$-Lipschitz in $\bu_{k,\cI_k}$ because $\zeta_k > c > 0$ by Lemma \ref{lem:bound-on-fixed-pt}.
    Then $\phi_e\Big(\frac{\be_{1,\cI_k}}{\sqrt{n_k}},\frac{\be_{2,\cI_k}}{\sqrt{n_k}},\frac{\hat \be_{k,\cI_k}}{\sqrt{n_k}},\frac{\hat \be_{k,\cI_k}^{\de}}{\sqrt{n_k}}\Big)$ is $C$-Lipschitz in $\frac{\be_{1,\cI_k}}{\sqrt{n_k}},\frac{\be_{2,\cI_k}}{\sqrt{n_k}},\frac{\bu_{k,\cI_k}}{\sqrt{n_k}}$,
    so that the second line of Lemma \ref{lem:marginal-characterization} follows.
    
    We establish the first line of Lemma \ref{lem:marginal-characterization} first for OLS, ridge regression, and the $\alpha$-smoothed Lasso with $\alpha_k > 0$,
    and then extend the result to the Lasso.
    For simplicity of notation, we remove the subscript $k$ on $\alpha_k$ in what follows.
    For OLS, ridge regression, and the $\alpha$-smoothed Lasso,
    the penalty is differentiable, so that $\bX_{\cI_k}^\top(\by_k - \bX_{\cI_k} \hat \btheta_k) / n_k =  \nabla \Omega(\hat \btheta_k) = \bSigma^{1/2} \nabla \bar \Omega(\bv_k)$.
    Then, the estimate and
    the debiased estimate of the parameter can be written as 
    \begin{equation}\label{eq:v-to-theta}
        \hat \btheta_k
            := 
            \btheta_k + \bSigma^{-1/2} \bv_k,
        \qquad 
        \hat \btheta_k^{\de}
            :=
            \btheta_k
            + \bSigma^{-1/2}
            \Big(
                \bv_k + \frac1{\zeta_k}\nabla \bar \Omega_k(\bv_k)
            \Big),
    \end{equation}
    and likewise
    \begin{equation}
        \hat \btheta_k^f
            := 
            \btheta_k + \bSigma^{-1/2} \bv_k^f,
        \qquad 
        \hat \btheta_k^{f,\de}
            :=
            \btheta_k
            + \bSigma^{-1/2}
            \Big(
                \bv_k^f + \frac1{\zeta_k}\nabla \bar \Omega_k(\bv_k^f)
            \Big),
    \end{equation}
    To simplify the following exposition, we adopt the convention that if we use OLS or ridge regression we set $\alpha = 1$,
    and if we use the $\alpha$-smoothed Lasso, then $\alpha$ is the smoothing constant.
    For OLS, ridge regression, and the $\alpha$-smoothed Lasso, $\nabla \bar \Omega_k(\bv)/\zeta_k$ is $C/\alpha$-Lipschitz.
    Indeed, $\nabla \bar \Omega_k(\bv) = \bSigma^{-1/2} \nabla \Omega_k(\btheta_k + \bSigma^{-1/2} \bv)$,
    whence, by assumption $\mathsf{A1}$, the Lipschitz constant of $\nabla \bar \Omega_k(\bv)$ is bounded by $c$ times the Lipschitz constant of $\nabla \Omega_k$.
    For ridge-regression, this Lipschitz constant is $\sqrt{p/n_k}\,\lambda \leq C$, and for the $\alpha$-smoothed Lasso it is $\lambda/\alpha \leq C/\alpha$. 
    Because $\zeta_k \geq c$,
    we conclude $\phi_\theta(\hat \btheta_k,\hat \btheta_k^{\de})$ is $C/\alpha$-Lipschitz in $\bv_k$.
    Using Eq.~\eqref{eq:lipsch-of-tilb-tilu},
    we conclude that 
    for $\epsilon \leq c'$ with probability at least $1 - \sC(\epsilon)e^{-\sc(\epsilon)p}$
    \begin{gather}\label{eq:non-Lasso-conc}
            \big|
                \phi_\theta( \hat \btheta_k, \hat \btheta_k^{\de} ) 
                - 
                \E[\phi_\theta( \hat \btheta_k^f, \hat \btheta_k^{f,\de} ) ]
            \big| 
            \leq \sqrt{\frac{p}{ n_k }} \, \frac{\epsilon}{\alpha}.
    \end{gather}

    We now remove the dependence on $\alpha$ in the previous display so that the upper bound does not blow up as $\alpha \rightarrow 0$ and so that we may extend the results to the Lasso (i.e., $\alpha = 0$).
    To do so, we use the approximation technique of \cite{celentano2020lasso}.
    The main idea is to argue that that $\alpha$-smoothed Lasso estimate and debiased estimate are close to the Lasso estimate and debiased estimate for $\alpha$ small (but with the regularization parameter $\lambda$ kept unchanged).
    Recall that for the Lasso we assume $c \leq n_k/p \leq C$, so that in what follows we may exchange $n_k$ and $p$ by adjusting constants.

    We consider both that $\alpha$-smoothed Lasso for some $\alpha > 0$ and the Lasso estimators computed on the same data in the second regression.
    We introduce notation which distinguishes quantities related to the $\alpha$-smoothed Lasso from those related to the Lasso by adding a subscript $\alpha$ to the former.
    For example, $\hat \btheta_k$ will denote the Lasso estimate, and $\hat \btheta_{k,\alpha}$ will denote the $\alpha$-smoothed Lasso estimate.

    Lemma B.8 in \cite{celentano2020lasso} gives that for $\alpha < c'$ with probability at least $1 - Ce^{-cp}$ we have $\| \hat \btheta_k - \hat \btheta_{k,\alpha} \|_2 < C\sqrt{\alpha}$.
    The proof of Theorem 10 in \cite{celentano2020lasso} (see, in particular, Eq.~(80)) gives that for $\alpha < c'$
    with probability at least $1 - \sC(\alpha)e^{-p\sc(\alpha)}$ we have $\| \hat \btheta_k^\de - \hat \btheta_{k,\alpha}^\de \|_2 < C\sqrt{\alpha}$.
    (Note that the bound stated in \cite{celentano2020lasso} applies to the debiased Lasso which uses $1 - \hat \df_k / n_k$ in place of $\zeta_k$. Replacing $1 - \hat \df_k / n_k$ by $\zeta_k$ removes an error term, so that the bound of \cite{celentano2020lasso} still holds).
    Thus, with probability at least $1 - \sC(\alpha)e^{-p\sc(\alpha)}$ we have $\big|\phi_\theta(\hat \btheta_k,\hat \btheta_k^{\de}) - \phi_\theta(\hat \btheta_{k,\alpha},\hat \btheta_{k,\alpha}^{\de}) \big| < C\sqrt{\alpha}$.
    Applying the same bound with $\alpha'$ in place of $\alpha$ and $0 \leq \alpha' \leq \alpha$,
    we have with probability at least $1 - \sC(\alpha)e^{-p\sc(\alpha)}$ that $\big|\phi_\theta(\hat \btheta_{k,\alpha'},\hat \btheta_{k,\alpha'}^{\de}) - \phi_\theta(\hat \btheta_{k,\alpha},\hat \btheta_{k,\alpha}^{\de}) \big| < C\sqrt{\alpha}$.
    We further have the following lemma, proved in Section \ref{sec:existence-fix-pt}.
    \begin{lemma}\label{lem:marg-alpha-approx}
        Assume \textsf{A1} and \textsf{A2}.
        Then for $0 \leq \alpha'\leq \alpha < c'$,
        \begin{equation}
            \big|
                \E[\phi_\theta(\hat \btheta_{k,\alpha'},\hat \btheta_{k,\alpha'}^{\de})]
                -
                \E[\phi_\theta(\hat \btheta_{k,\alpha}^f,\hat \btheta_{k,\alpha}^{f,\de})]
            \big|
            \leq C'\sqrt{\alpha}.
        \end{equation}
    \end{lemma}
    \noindent 
    Combining these results,
    we have
    for $\alpha,\epsilon < c'$ and $\alpha'\leq \alpha$ with probability at least 
    $1 - \sC(\alpha,\epsilon)e^{-p\sc(\alpha\wedge \epsilon)}$ 
    \begin{equation}
    \begin{aligned}
        &\big|
            \phi_\theta(\hat \btheta_{k,\alpha'},\hat \btheta_{k,\alpha'}^{\de})        
            -
            \E[\phi_\theta(\hat \btheta_{k,\alpha'}^f,\hat \btheta_{k,\alpha'}^{f,\de})]
        \big| \leq 
            \big|
                \phi_\theta(\hat \btheta_{k,\alpha'},\hat \btheta_{k,\alpha'}^{\de})        
                -
                \phi_\theta(\hat \btheta_{k,\alpha},\hat \btheta_{k,\alpha}^{\de})
            \big| 
            +
            \big|
                \phi_\theta(\hat \btheta_{k,\alpha},\hat \btheta_{k,\alpha}^{\de})
                -
                \E[\phi_\theta(\hat \btheta_{k,\alpha}^f,\hat \btheta_{k,\alpha}^{f,\de})]
            \big|
        \\
            &\qquad\qquad\qquad\qquad\qquad\qquad\qquad\qquad\qquad\qquad\qquad\qquad+
            \big|
                \E[\phi_\theta(\hat \btheta_{k,\alpha}^f,\hat \btheta_{k,\alpha}^{f,\de})]
                -
                \E[\phi_\theta(\hat \btheta_{k,\alpha'}^f,\hat \btheta_{k,\alpha'}^{f,\de})]
            \big|
        \\
            &\quad \leq 
                C'\sqrt{\alpha} + C'\,\epsilon/\alpha + C' \sqrt{\alpha}.
    \end{aligned}
    \end{equation}
    Taking $\alpha = \epsilon^{2/3}$ and adjusting constants, 
    we conclude the first bound in Lemma \ref{lem:marginal-characterization} for any smoothing parameter $\alpha' \leq \epsilon^{2/3}$.
    For $\alpha' > \epsilon^{2/3}$, the result follows by Eq.~\eqref{eq:non-Lasso-conc}.
\end{proof}

\begin{remark}
    Existing applications of Gordon's inequality to regression problems prove a version of Lemma \ref{lem:marg-sub-opt} for a value $\ell_k^*$ which depends on the model parameters $n_k,p,\btheta_k$ and penalty $\Omega_k$ but not on the realization of the noise $\be_k$ \cite{thrampoulidis2015,miolane2018distribution,celentano2020lasso}.
    By proving Lemma \ref{lem:marg-sub-opt} for an $\be_k$-dependent value $\ell_k^*$, 
    we can achieve control of $\bv_k,\bu_k$ which improves as $n_k/p \rightarrow \infty$. (See Remark \ref{rmk:rel-lit}).

    To localize $\bv_k$ (the discussion for $\bu_k$ is analogous), Lemma \ref{lem:marg-sub-opt} shows the minimization problem Eq.~\eqref{eq:vk-uk} is sub-optimal on $E_v(\epsilon)^c$.
    For our proof strategy to work,
    we must choose a set $E_v(\epsilon)^c$ separated enough from the minimizer so that the sub-optimality on this set dominates the fluctuations of the minimal value of the optimization.
    The smaller these fluctuations, the smaller we can make the separation.

    Unconditionally on $\be_k$, the fluctuations will be too large. 
    Indeed, consider the case of least-squares. Then, the value of the optimization problem is distributed $\chi_{n_k-p}^2/n_k$.
    For $n_k /p > 1 + c$, this fluctuates by an amount $C'\sqrt{p/n_k}\,\epsilon^2$ with probability at least $Ce^{-cp\epsilon^4}$.
    Lemma \ref{lem:marg-sub-opt} shows that conditionally on $\be_k$, 
    the minimal value fluctuates by the much smaller amount $C'(p/n_k)\epsilon^2$ with probability at most $1 - \sC(\epsilon)e^{-\sc(\epsilon)p}$.
    By allowing $\ell_k^*$ to depend on $\be_k$, 
    we can localize the minimal value of the optimization problem with resolution $(p/n_k)\epsilon^2$ rather than $\sqrt{p/n_k}\,\epsilon^2$, and hence the location of the minimizer with resolution $\sqrt{p/n_k}\,\epsilon$ rather than $(p/n_k)^{1/4}\epsilon$.

    See the proof of Lemma \ref{lem:marg-sub-opt} for further details.
\end{remark}

We prove Lemma \ref{lem:marg-sub-opt} in the next section.

\subsection{Proof of Lemma \ref{lem:marg-sub-opt}: marginal control of primary objective}
\label{app:marginal-characterization-preliminaries}

We prove Lemma \ref{lem:marg-sub-opt} using the marginal Gordon inequality (Lemma \ref{lem:gordon-marginal}).
Define the auxilliary objective
\begin{equation}\label{eq:lk}
\begin{gathered}
    \ell_k(\bu,\bv) 
        := 
        - \frac1{n_k}\<\bg_{\mathrm{mg}}(\bu), \bv\> + \frac1{n_k}\<\bh_{\mathrm{mg}}(\bv), \bu\> + \frac1{n_k}\bu^\top\be_k - \frac1{2n_k}\|\bu\|_2^2 + \bar \Omega_k(\bv),
\end{gathered}
\end{equation}
where $\bg_{\mathrm{mg}}(\bu)$ and $\bh_{\mathrm{mg}}(\bv)$ are defined as in Lemma \ref{lem:gordon-marginal}.
Define 
\begin{equation}\label{eq:sig*-e*-ome*}
    \tau_{e_k}^* = \frac{\| \be_{k,\cI_k} \|_2}{\sqrt{n_k}},
    \qquad 
    e_k^*
        := \zeta_k (\tau_{e_k}^{*2} + \tau_{h_k}^2)^{1/2},
    \qquad
    \omega_k 
        :=
        \E\big[\bar\Omega_k(\bv_k^f)\big].
\end{equation}
We will see that $\| \bu_k^f \|_2/\sqrt{n_k}$ concentrates on $e_k^*$,
$\bar\Omega_k(\bv_k^f)$ concentrates on $\omega_k$, and
the value of the auxilliary optimization $L_k(\bu_k,\bv_k)$ concentrates on
\begin{equation}\label{eq:lk*}
    \ell_k^* := 
        \ell_k^*(\be_k,n_k,p,\btheta_k,\sigma_k,\bSigma,\Omega_k)
        :=
        \frac12 e_k^{*2} + \omega_k.
\end{equation}
(Note the dependence of $\ell_k^*$ on the model parameters $n_k,p,\btheta_k,\sigma_k,\bSigma$ and penalty $\Omega_k$ occurs via the fixed-point parameters defined via Eq.~\eqref{eq:fixed-pt-eqns}).
We will consider the min-max problem under the compact constraints
\begin{equation}\label{eq:marg-compact}
    \| \bv \|_2 \leq R,
    \qquad 
    \cS_u
        := 
        \Big\{
            \bu \in \reals^{n_k} 
            \Bigm| 
            \frac1{\sqrt{n_k}}\big\|\bu_{\cI_k} -\zeta_k \be_{k,\cI_k} \big\|_2 \leq R',
            \; \bu_{\cI_k^c} = 0
        \Big\}.
\end{equation}
For values of $R,R'$ to be chosen below.
It will suffice that $R,R'$ satisfy $R = C_v\sqrt{p/n_k}$ and $R' = C_u\sqrt{p/n_k}$ for $C_v \geq C_{\mathrm{min}}$ and $C_u \geq C_{\min}'(C_v)$, where $C_{\mathrm{min}}$ and $C_{\mathrm{min}}'(\cdot)$ only depend on $\cPmodel,\cPregr$, and the regression method.
Note here that the set  $\cS_u$ depends implicitly on $\be_k$, but for brevity of notation we supress this dependence.

In Section \ref{app:marginal-characterization-est-err},
we establish for $\epsilon < c'$ the lower bounds
\begin{equation}\label{eq:mgam-bounds}
\begin{gathered}
    \P\Big(
            \min_{ \substack{\bv \in E_v^c(\epsilon) \\ \| \bv \|_2 \leq R } }\;
            \max_{\bu \in \cS_u }
                    \ell_k(\bu,\bv)
                \geq 
                \ell_k^* + C'\frac{p}{ n_k }\,\epsilon^2 
        \Big) \geq 1 - \sC(\epsilon)e^{-\sc(\epsilon)p},
    \\
    \P\Big(
            \min_{\| \bv \|_2 \leq R  }\;
            \max_{\bu \in \cS_u }
                    \ell_k(\bu,\bv)
            \geq
                \ell_k^*
                - 
                C'\frac{p}{ n_k }\,\frac{\epsilon^2}{2}
        \Big) \geq 1 - \sC(\epsilon)e^{-\sc(\epsilon)p},
\end{gathered}
\end{equation}
where the $C'$ in both bounds is the same.
In Section \ref{app:marginal-characterization-residuals},
we establish for $\epsilon < c'$ the upper bounds
\begin{equation}\label{eq:mgam-upper-bound}
\begin{gathered}
    \P\Big(
        \max_{
            \substack{
                \bu \in E_u^c(\epsilon)\\
                                \bu \in \cS_u }}
                    \min_{\|\bv\|_2 \leq R }
                    \ell_k(\bu,\bv)
            \leq 
            \ell_k^* - C'\frac{p}{n_k}\,\epsilon^2 
    \Big) \geq 1 - \sC(\epsilon)e^{-\sc(\epsilon)p},
    \\
    \P\Big(
            \max_{\bu \in \cS_u }
            \min_{\| \bv \|_2 \leq R  }\;
                    \ell_k(\bu,\bv)
            \leq
                \ell_k^*
                + 
                C'\frac{p}{ n_k }\,\frac{\epsilon^2}{2}
        \Big) \geq 1 - \sC(\epsilon)e^{-\sc(\epsilon)p},
\end{gathered}
\end{equation}
where the $C'$ in both bounds is the same.
The preceding probabilities are taken over the randomness in $\be_1$, $\be_2$, $\bxi_g$, $\bxi_h$.

Applying the marginal Gordon inequality (Lemma \ref{lem:gordon-marginal}) conditionally on $\be_1,\be_2$ and then averaging over $\be_1,\be_2$,
we conclude the the previous two displays hold with $L_k$ in place of $\ell_k$ and adjusted constants.
Because the minimization and maximization are over compact convex sets,
by \cite[Corollary 37.3.2]{rockafellar-1970a} the problem
\begin{equation}
    \max_{ \bu \in \cS_u }\;
    \min_{\|\bv\|_2 \leq R }
        L_k(\bu,\bv)
    =
    \min_{\|\bv\|_2 \leq R }\;
    \max_{ \bu \in \cS_u }         
        L_k(\bu,\bv)        
\end{equation}
has a saddle point $\hat \bu, \hat \bv$ and minimization and maximization can be exchanged.
Thus,
for $\epsilon < c'$ with probability at least $1 - \sC(\epsilon)e^{-\sc(\epsilon)p}$ (over $\be_1,\be_2,\bA$)
\begin{equation}\label{eq:primal-prec-bound}
\begin{gathered}                  
    \max_{
    \substack{
        \bu \in E_u^c(\epsilon)\\
                        \bu \in \cS_u }}\;
            \min_{\|\bv\|_2 \leq R}
        L_k(\bu,\bv)
        \leq \ell_k^* - C'\frac{p}{n_k}\epsilon^2,
    \quad
    \min_{\substack{\bv \in E_v^c(\epsilon)\\\|\bv\|_2 \leq R}}\;
    \max_{\bu \in \cS_u }
            L_k(\bu,\bv)
        \geq \ell_k^* + C'\frac{p}{n_k}\epsilon^2,
    \\
    \Big|
        \max_{
        \substack{\bu \in \cS_u }}
                \min_{\|\bv\|_2 \leq R }
            L_k(\bu,\bv)
        - 
        \ell_k^*
    \Big|
    =
    \Big|
        \min_{\|\bv\|_2 \leq R }
        \max_{\bu \in \cS_u }
            L_k(\bu,\bv)
        - 
        \ell_k^*
    \Big|
    \leq 
         C'\frac{p}{n_k}\frac{\epsilon^2}2,
\end{gathered}
\end{equation}
where the $C'$ in the previous three bounds can be taken to be equal.

To relax the restriction of the optimizations to compact sets, 
we show that with high-probability $\hat \bu,\hat \bv$ fall in the interior of these sets.
Take $\phi_u(\bu/\sqrt{n_k}) = \|\bu - \zeta_k \be_{k,\cI_k} \|_2/\sqrt{n_k}$, which is 1-Lipschitz.
For $R = C\sqrt{p/n_k}$ and $C$ chosen sufficiently large, we have $\E[\phi_u(\bu_k^f/\sqrt{n_k})] \leq R/2$ (see bounds on fixed point parameters in Lemma \ref{lem:bound-on-fixed-pt}).
Then the event in the previous display (in particular, the first and third lines) for $\epsilon \leq c'$ implies $\hat \bu$ is in the interior of $\cS_u$. 
Next take $\phi_v(\bv) = \| \bv \|_2$.
For $R = C\sqrt{p/n_k}$ and $C$ chosen sufficiently large,
we have $\E[\phi_v(\bv_k^f)] \leq R/2$ (see again bounds on fixed point parameters in Lemma \ref{lem:bound-on-fixed-pt}).
Then the event in the previous display (in particular, the second and third lines) for $\Delta \leq c'$ implies $\| \hat \bv \|_2 < R$.
When both $\hat \bu $ is in the interior of $\cS_u$ and $\| \hat \bv \|_2 < R$,
the saddle point is in the interior of the problem's domain, in which case it remains a saddle point when we expand the domains to $\{ \bu \in \reals^N : \bu_{\cI_k^c} = 0\}$ and $\reals^p$.
Thus, $\hat \bu = \bu_k$ and $\hat \bv = \bv_k$ with probability at least $1 - \sC(\epsilon)e^{-\sc(\epsilon)p}$.
In particular, the second line of Eq.~\eqref{eq:primal-prec-bound} continues to hold when we remove the compactness constraints.

It only remains to show that we can remove the compactness constraints in the first line of Eq.~\eqref{eq:primal-prec-bound}.
Consider the first bound. 
We have
\begin{equation}\label{eq:max-decomp}
    \max_{\bu \in E_u^c(\epsilon)} \min_{\|\bv\|_2 \leq R} L_k(\bu,\bv)
        \leq 
        \max \Big\{
            \max_{
            \substack{
                \bu \in E_u^c(\epsilon)\\
                                \bu \in \cS_u }}\;
                    \min_{\|\bv\|_2 \leq R}
                L_k(\bu,\bv)
            ,\;
            \max_{
                \bu \in \cS_u^c}\;
                    \min_{\|\bv\|_2 \leq R}
                L_k(\bu,\bv)
        \Big\}.
\end{equation}
Eq.~\eqref{eq:primal-prec-bound} implies the first quantity in the maximum is smaller than $\ell_k^* - C'(p/n_k)\epsilon^2$ with conditional probability at least $1 - \sC(\epsilon)e^{-\sc(\epsilon)p}$.
To bound the second quantity in the maximum, we apply Eq.~\eqref{eq:primal-prec-bound} with a different choice of $\phi_u$; namely, $\phi_u'(\bu/\sqrt{n_k}) = \|\bu - \zeta_k\be_{k,\cI_k}\|_2/\sqrt{n_k}$.
Taking $R' = C'\sqrt{p/n_k}$ in Eq.~\eqref{eq:marg-compact} sufficiently large and $\epsilon < c'$ sufficiently small, 
we have that $E_u'(\epsilon)$ is contained in the interior of $\cS_u$ (where $E_u'$ is defined with respect to function $\phi_u'$).
When Eq.~\eqref{eq:primal-prec-bound} occurs with $\phi_u',E_u'$ in place of $\phi_u,E_u$,
then the saddle point $\hat \bu,\hat \bv$ satisfies $\hat \bu \in E_u'(\epsilon)$.
Consider now any $\bu' \in S_u^c$.
Because the function $\bu \mapsto \min_{\|\bv\|_2 \leq R}L_k(\bu,\bv)$ is concave, it is non-decreasing on the line segment joining $\bu'$ to $\hat \bu$, and because $E_u'(\epsilon)$ is contained in the interior of $\cS_u$ and $\hat \bu \in E_u'(\epsilon)$,
this line segment intersects with $E_u'^c(\epsilon) \cap \cS_u$ at some point $\bu''$.
Then, 
\begin{equation}\label{eq:S_uc-bound}
    \min_{\|\bv\|_2 \leq R} L_k(\bu',\bv)
    \leq 
    \min_{\|\bv\|_2 \leq R} L_k(\bu'',\bv)
    \leq 
    \max_{
        \substack{
            \bu \in E_u'^c(\epsilon)\\
                            \bu \in \cS_u } }\;
            \min_{\|\bv\|_2 \leq R}
        L_k(\bu,\bv)
    \leq 
    \ell_k^*- C'\frac{p}{n_k}\epsilon^2.
\end{equation}
Because this applies to all $\bu' \in S_u^c$, we conclude that with probability at least $1 - \sC(\epsilon)e^{-\sc(\epsilon)p}$ that the second term in the maximum in Eq.~\eqref{eq:max-decomp} is less than $\ell_k^* - C'(p/n_k)\epsilon^2$.
Then, the left-hand side of Eq.~\eqref{eq:max-decomp} is also smaller than $\ell_k^* - C'(p/n_k)\epsilon^2$.
Removing the constraint $\| \bv \|_2 \leq R$ on the inner minimization only decreases the value of the min-max problem.
We conclude the upper bound shown in the second line of Eq.~\eqref{eq:marg-sub-opt}.
A similar argument allows us to remove the compactness constraints in the lower bound in the first line of Eq.~\eqref{eq:primal-prec-bound}.

Provided we can prove Eqs.~\eqref{eq:mgam-bounds} and \eqref{eq:mgam-upper-bound},
we have concluded Lemma \ref{lem:marg-sub-opt}. 
We establish these bounds in the next three sections.

\subsubsection{The good marginal characterization event}
\label{sec:good-marg-char-event}

To prove Eqs.~\eqref{eq:mgam-bounds} and \eqref{eq:mgam-upper-bound},
we show that the upper and lower bounds on the min-max problem are implied by the occurence of a certain high-probability event. 
To describe this event, it is useful to couple the auxilliary objective \eqref{eq:lk} to the fixed-design models \eqref{eq:fixed-des}.
The coupling is given by taking 
\begin{equation}
    \bg_k^f = \tau_{g_k} \bxi_g \quad  \text{and}  \quad \bh_k^f = \tau_{h_k} \bxi_h.
\end{equation}
We introduce notation $\bT(\ba_1,\ldots,\ba_r) \in \SS_+^r$ for the positive semi-definite matrix whose $(i,j)$ entry is $\< \ba_i , \ba_j \>$, where $\ba_i \in \reals^s$ for all $i$ and some fixed $s$.
With some abuse of notation,
we will use the same notation $\bT$ regardless of the number $r$ of arguments to the function and regardless of the dimensionality $s$ of the vectors $\ba_k$. 
No confusion should result.

The good marginal characterization event is
\begin{equation}
\begin{aligned}\label{eq:cEgam-marginal}
    \cG_k(\epsilon,\Delta)
        :=
        \Big\{\;\;
            &\Big\| 
                \bT\Big(\frac{\be_{k,\cI_k}}{\sqrt{n_k}},\frac{\bxi_{h,\cI_k}}{\sqrt{n_k}}\Big) 
                - 
                \E\Big[
                    \bT\Big(\frac{\be_{k,\cI_k}}{\sqrt{n_k}},\frac{\bxi_{h,\cI_k}}{\sqrt{n_k}}\Big)
                    \Bigm|
                    \be_k
                \Big] 
            \Big\|_{\sF} <\sqrt{\frac{p}{n_k}}\,\epsilon,\;
        \\
            &\Big\| \bT\Big(\sqrt{\frac{n_k}{p}}\,\bv_k^f,\frac1{\sqrt{p}}\bxi_g\Big) - \E\Big[\bT\Big(\sqrt{\frac{n_k}{p}}\,\bv_k^f,\frac1{\sqrt{p}}\bxi_g\Big)\Big] \Big\|_{\sF} < \epsilon,\;
        \\
            &\big|
                \bar \Omega_k(\bv_k^f)
                -
                \omega_k
            \big| < \frac{p}{ n_k }\,\epsilon,\;
            |\tau_{e_k}^* - \tau_{e_k}| = 
            \Big|\frac{\| \be_{k,\cI_k} \|_2}{\sqrt{n_k}} - \tau_{e_k} \Big| < \sqrt{\frac{p}{n_k}}\,\epsilon,
        \\
            &\bv_k^f \in E_v(\Delta/2),\;
            \bu_k^f \in E_u(\Delta/2)
    \;\;\Big\}.
\end{aligned}
\end{equation}
The expectations are taken in the model which couples the auxilliary objective with the fixed-design models.
In all cases, we have normalized the arguments to $\bT$ so that their $\ell_2$ norm is of order 1 (see Lemma \ref{lem:bound-on-fixed-pt}).
Because the matrices involved are finite-dimensional, the Frobenius norm can be replaced with any norm at the cost of constant factors.
For future reference, we have written the expectations appearing in the definition of $\cG_k$ in Section \ref{sec:identity-ref}.
In words,
the good marginal characterization event is the event that several important quantities are close to their expectations.
It has high-probability.
\begin{lemma}[Good marginal characterization event]\label{lem:marginal-concentration-event}
    Assume \textsf{A1} and \textsf{A2}.
    There exist $\cPmodel$, $\cPregr$ and regression method-dependent $c' > 0$ and $\sC,\sc: \reals_{>0}^2 \rightarrow \reals_{>0}$ such that
    for all $\epsilon,\Delta < c'$ the event $\cG_k(\epsilon,\Delta)$ occurs with probability at least $1 - \sC(\epsilon,\Delta) e^{-\sc(\epsilon\wedge \Delta) p}$.
\end{lemma}
\noindent We prove Lemma \ref{lem:marginal-concentration-event} in Section \ref{sec:marg-char-support}.

\subsubsection{Lower bounds on the auxilliary min-max problem}
\label{app:marginal-characterization-est-err}

In this section, we prove Eq.~\eqref{eq:mgam-bounds}.
In particular, we seek a lower bound on 
\begin{equation}
    \min_{ \|\bv\|_2 \leq R }\;
        \max_{\bu \in \cS_u }
                \ell_k(\bu,\bv),
\end{equation}
and also on the problem where we further restrict the minimization to $\bv \in E_v^c(\epsilon)$.
It is enough to show that our lower bound holds on the event $\cG_k(\epsilon,\Delta)$ for $\epsilon < c'$, $\Delta = \Delta(\epsilon) < c'$, and $c'$ taken sufficiently small and depending only on $\cPmodel,\cPregr$, and the regression method.
We will henceforth assume we are on this event for $\epsilon,\Delta$ for $c'$ sufficiently small, without repeatedly reminding the reader of this fact,
and all statements will be deterministic.
In particular,
as we derive more implications of $\cG_k(\epsilon,\Delta)$, 
we may need to take $c'$ smaller,
but will not track $c'$ or announce when an additional implication requires we shrink $c'$ further.
To make notation more compact,
we denote
\begin{equation}
     \sT_{\sN,k} 
        := 
        \E\Big[
            \bT\Big(\frac{\be_{k,\cI_k}}{\sqrt{n_k}},\frac{\bxi_{h,\cI_k}}{\sqrt{n_k}}\Big)
            \Bigm|
            \be_k
        \Big] ,
        \qquad
    \sT_{\sP,k}
        :=
        \E\Big[\bT\Big(\sqrt{\frac{n_k}{p}}\,\bv_k^f,\frac1{\sqrt{p}}\bxi_g\Big)\Big].
\end{equation}
We remind the reader that explicit expressions for these can be found in Section \ref{sec:identity-ref}.

The major steps in the analysis are as follows.
\begin{enumerate}

    \item We replace $\ell_k(\bu,\bv)$ by a function $\ell_k^{(1)}(\bu,\bv)$ which approximates it uniformly well across its domain.
    We control the change in the value of the min-max problem incurred by this replacement.
    The objective $\ell_k^{(1)}(\bu,\bv)$ is introduced because it is easier to analyze.

    \item 
    Ideally, we would evaluate the internal maximization exactly.
    However, doing so explicitly leads to complicated expressions.
    Instead, we construct a lower bound on the internal maximization.
    In particular, for any function $\bu(\bv)$ for which $\bu(\bv) \in \cS_u$, the quantity $ \ell_k^{(1)}(\bu(\bv),\bv)$ is a lower bound on the value of the maximization problem. 
    Our strategy is to pick a function $\bu(\bv)$ so that the function $\ell_k^{(2)}(\bv) := \ell_k^{(1)}(\bu(\bv),\bv)$ is tractable to analyze and provides a good enough lower bound for our purposes.

    \item
    We establish several properties of the lower bound $\ell_k^{(2)}(\bv)$.  
    In particular, we show $\ell_k^{(2)}(\bv_k^f)$ is close to $\ell_k^*$,
    the subdifferential
    $\partial \ell_k^{(2)}(\bv_k^f)$ contains a small element, and $\ell_k^{(2)}(\bv)$ is strongly convex for bounded $\bv$.

    \item 
    Using standard convex analysis techniques, 
    these properties imply a lower bound on the min-max problem over $\| \bv\|_2 \leq R$ and over $\bv \in E_v^{c}(\epsilon) \cap \{ \|\bv\|_2 \leq R\}$.

\end{enumerate}
We now carry out these steps in detail. 
\\

\noindent \textit{Step 1: we replace $\ell_k(\bu,\bv)$ by a function $\ell_k^{(1)}(\bu,\bv)$ which approximates it uniformly well across its domain.}

We use a Gram-Schmidt type procedure to replace $\bxi_h$.
In particular, 
we replace $\bxi_h$ by $\bxi_h^{\mathrm{gs}}$
so that 
\begin{equation}\label{eq:gs-def}
    \bT(\be_{k,\cI_k}/\sqrt{n_k},\bxi_{h,\cI_k}^{\mathrm{gs}}/\sqrt{n_k}) = \sT_{\sN,k}.
\end{equation}
By Eq.~\eqref{eq:cEgam-marginal}, 
we may do this in such a way that $\| \bxi_{h,\cI_k}^{\mathrm{gs}} - \bxi_{h,\cI_k} \|_2 / \sqrt{n_k} \leq C'\sqrt{p/n_k}\,\epsilon$.
The only term in $\ell_k(\bu,\bv)$ which depends on $\bxi_h$ is $\| \bv\|_2 \<\bxi_{h,\cI_k},\bu_{\cI_k}\>/n_k$.
For $\| \bv \|_2 \leq R, \bu \in \cS_u$,
we bound
\begin{equation}
\begin{aligned}
    \Big|
        \frac{\| \bv\|_2 \<\bxi_{h,\cI_k}^{\mathrm{gs}},\bu_{\cI_k}\>}{n_k}
        -
        \frac{\| \bv\|_2 \<\bxi_{h,\cI_k},\bu_{\cI_k}\>}{n_k}
    \Big|
        &\leq 
    \frac{\| \bv\|_2 \|\bxi_{h,\cI_k}^{\mathrm{gs}}-\bxi_{h,\cI_k}\|_2\|\bu_{\cI_k}\|_2}{n_k}
        \leq C \frac{p}{n_k} \epsilon,
\end{aligned}
\end{equation}
where we have used that $\| \bv \|_2 \leq R = C_v\sqrt{p/n_k}$ (see Eq.~\eqref{eq:marg-compact}) and 
$\| \bu_{\cI_k}\|_2/\sqrt{n_k} \leq \| \be_{k,\cI_k}\|_2/\sqrt{n_k} + C \sqrt{p/n_k}$.
We define $\ell_k^{(1)}(\bu,\bv)$ like $\ell_k(\bu,\bv)$, except with $\bxi_h^{\mathrm{gs}}$ in place of $\bxi_h$.
Thus,
\begin{equation}\label{eq:marg-gs-approx}
    \sup_{\| \bv \|_2 \leq R} \; \sup_{\bu \in \cS_u} |\ell_k^{(1)}(\bu,\bv) - \ell_k(\bu,\bv)| \leq \frac{p}{n_k}\epsilon.
\end{equation}
\\

\noindent \textit{Step 2: we construct a lower bound $\ell_k^{(2)}(\bv)$ on the internal maximization.}

In particular, 
we evaluate $\ell_k^{(1)}(\bu,\bv)$ 
at $\bu(\bv)$ defined by $\bu_{\cI_k}(\bv) := e_k^* \frac{ \be_{k,\cI_k} + \| \bv \|_2 \bxi_{h,\cI_k}^{\mathrm{gs}}}{(\tau_{e_k}^{*2}+\|\bv\|_2^2)^{1/2}}$ and $\bu_{\cI_k^c}(\bv) = 0$.
We have chosen $\bu(\bv)$ so that $\| \bu (\bv) \|_2 / \sqrt{n_k} = e_k^*$ and $\bu_{\cI_k}(\bv)^\top (\be_{k,\cI_k} + \| \bv \|_2 \bxi_{h,\cI_k}^{\mathrm{gs}})/n_k = e_k^* (\tau_{e_k}^{*2} + \|\bv\|_2^2)^{1/2}$ (see Eqs.~\eqref{eq:gs-def} and \eqref{eq:T-marg-ref}).
Evaluating $\ell_k^{(1)}(\bu,\bv)$ at this value of $\bu$ gives
\begin{equation}
\begin{aligned}
    \ell_k^{(2)}(\bv)
        &:=
            \frac1{ n_k }\bu_{\cI_k}(\bv)^\top (\be_{k,\cI_k} + \| \bv \|_2 \bxi_{h,\cI_k}^{\mathrm{gs}})
            -
            \frac1{ n_k }\| \bu(\bv) \|_2 \< \bxi_g,\bv \>
            - \frac1{2{ n_k }} \| \bu(\bv) \|_2^2 
            + \bar \Omega_k(\bv)
    \\
        &= 
            e_k^*\Big(\sqrt{\tau_{e_k}^{*2} +  \|\bv\|_2^2 }
                -\frac{\bxi_g^\top \bv}{ \sqrt{n_k} }\Big)
            - \frac12 e_k^{*2}
            + \bar \Omega_k(\bv).
\end{aligned}
\end{equation}
Note 
\begin{equation}
    \frac{\|\bu_{\cI_k}(\bv) - \zeta_k\be_{k,\cI_k}\|_2}{\sqrt{n_k}}
        \leq 
        \zeta_k\Big|
            \sqrt{\frac{\tau_{e_k}^{*2}+\tau_{h_k}^2}{\tau_{e_k}^{*2}+\|\bv\|_2^2}}
            - 1
        \Big|\frac{\|\be_{k,\cI_k}\|_2}{\sqrt{n_k}} + \frac{e_k^*\|\bv\|_2}{(\tau_{e_k}^{*2} + \|\bv\|_2^2)^{1/2}} \frac{\|\bxi_{h,\cI_k}^{\mathrm{gs}}\|_2}{\sqrt{n_k}}
        \leq C \sqrt{p/n_k},
\end{equation}
where we have used that $\tau_{h_k}^2\leq Cp/n_k$, $\|\bv\|_2^2 \leq C_v^2 p/n_k$, $\zeta_k \leq 1$, $\tau_{e_k}^{*2} \geq c$, $e_k^* \leq C$, and $\| \bxi_{h,\cI_k}^{\mathrm{gs}}\|_2/\sqrt{n_k} = 1$ (see Lemma \ref{lem:bound-on-fixed-pt}, Eq.~\eqref{eq:marg-compact}, Eq.~\eqref{eq:gs-def}, and Eq.~\eqref{eq:T-marg-ref}).
Thus, we may choose $R' = C_u\sqrt{p/n_k}$ in Eq.~\eqref{eq:marg-compact} large enough so that
$\bu(\bv) \in \cS_u$ and $C_u$ depends only on $\cPmodel,\cPregr$, the regression method, and $C_v$.
Then
\begin{equation}\label{eq:mgam-to-mgam1-v}
    \min_{\substack{\bv \in E_v^c(\Delta) \\ \| \bv \|_2 \leq R } }\;
    \max_{\bu \in \cS_u }
        \ell_k^{(1)}(\bu,\bv)
        \geq 
        \min_{\substack{\bv \in E_v^c(\Delta) \\ \| \bv \|_2 \leq R } }
            \ell_k^{(2)}(\bv).
\end{equation}
\\

\noindent \textit{Step 3: we establish several properties of the lower bound $\ell_k^{(2)}(\bv)$.}

First, we show $\ell_k^{(2)}(\bv_k^f)$ is close to $\ell_k^*$.
We compute
\begin{equation}
\begin{aligned}
    \ell_k^{(2)}&(\bv_k^f)
        =
        e_k^* \Big(\sqrt{\tau_{e_k}^{*2} + \| \bv_k^f \|_2^2 }
            -\frac{\bxi_g^\top \bv_k^f}{ \sqrt{n_k} }\Big)
        - \frac12 e_k^{*2}
        + \bar \Omega_k(\bv_k^f).
\end{aligned}
\end{equation}
We approximate several quantities in this expression.
In particular, we make replacements
\begin{equation}\label{eq:vk-repacements}
\begin{gathered}
        \| \bv_k^f \|_2^2 
            \longrightarrow 
            \tau_{h_k}^2,
    \qquad
        \frac{\bxi_g^\top \bv_k^f}{ \sqrt{n_k}}
            \longrightarrow 
            \frac{\df_k}{n_k} \sqrt{\tau_{e_k}^{*2}+\tau_{h_k}^2},
    \qquad
        \bar \Omega_k(\bv_k^f)
            \longrightarrow
            \omega_k,
\end{gathered}
\end{equation}
after which, by simple algebra, we get the quantity $\ell_k^*$ (see its definition in Eq.~\eqref{eq:lk*}).

By Eqs.~\eqref{eq:cEgam} and \eqref{eq:T-marg-ref},
the terms above differ from their replacements by at must $(p/n_k)\epsilon$.
Indeed, for the first replacement, we use that $(n_k/p)\big|\|\bv_k^f\|_2^2-\tau_{h_k}^2\big| <\epsilon$.
For the second replacement, 
we use that $\big|\bxi_g^\top \bv_k^f/\sqrt{n_k} - \tau_{g_k}\df_k/\sqrt{n_k}\big| < (p/n_k)\epsilon$, 
$\tau_{g_k} = (\tau_{e_k}^2 + \tau_{h_k}^2)^{1/2}/\sqrt{n_k}$ (by Eq.~\eqref{eq:marg-fix-pt-eq}), 
and $(\df_k/n_k)\big|(\tau_{e_k}^2 + \tau_{h_k}^2)^{1/2} - (\tau_{e_k}^{*2} + \tau_{h_k}^2)^{1/2}\big| \leq (p/n_k)^{3/2}\epsilon$ (recall $\df_k\leq p$).
The error incurred by the third replacement is directly controlled by Eq.~\eqref{eq:cEgam}.
On the event $\cG_k(\epsilon)$ for $\epsilon < c'$, we have $e_k^* \leq C$.
Thus,
\begin{equation}\label{eq:mgam2-vgam*-approx}
        \Big|
            \ell_k^{(2)}(\bv_k^f)
            - 
            \ell_k^*
        \Big|
        < C'\frac{p}{ n_k }\,\epsilon.
\end{equation}

Second, we show $\partial \ell_k^{(2)}(\bv_k^f)$ contains a small element.
The subdifferential at $\bv_k^f$ is 
\begin{equation}
\begin{aligned}
    \partial \ell_k^{(2)}(\bv_k^f)
        &=
            e_k^* \frac{\bv_k^f}{\sqrt{\tau_{e_k}^{*2} + \| \bv_k^f \|_2^2 }}
            - \frac{e_k^* \bxi_g}{ \sqrt{n_k}}
            + \partial \bar \Omega_k(\bv_k^f).
\end{aligned}
\end{equation}
If we replace $\tau_{e_k}^*$ by $\tau_{e_k}$, $e_k^*$ by $\sqrt{n_k}\,\tau_{g_k}\zeta_k$, and $\| \bv_k^f \|_2^2$ by $\tau_{h_k}^2$ (we bound the error incurred by these replacements below) and recall that $(\tau_{e_k}^2 + \tau_{h_k}^2)^{1/2} = \sqrt{ n_k }\,\tau_{g_k}$ (see Eq.~\eqref{eq:marg-fix-pt-eq}), 
we get
\begin{equation}
    \zeta_k(\bv_k^f - \tau_{g_k} \bxi_g) 
    +
    \partial \bar \Omega_k(\bv_k^f).
\end{equation}
By the definition of $\bv_k^f$ in Eqs.~\eqref{eq:fixed-design-est} and \eqref{eq:vkf}, 
$\bv_k^f = \argmin_{\bv\in \reals^p}\big\{\| \tau_{g_k} \bxi_g - \bv \|_2^2/2 + \bar \Omega_k(\bv)/\zeta_k\big\}$.
By the KKT conditions for this optimization, the set in the previous display contains $\bzero$.

We control the errors introduced by these replacements.
By Eqs.~\eqref{eq:sig*-e*-ome*}, \eqref{eq:cEgam}, \eqref{eq:T-marg-ref}, \eqref{eq:marg-fix-pt-eq},
and the bound $\tau_{h_k} \leq C \sqrt{p/n_k}$ (see Lemma \ref{lem:bound-on-fixed-pt}),
each term we replace is within $\sqrt{p/n_k}\,\epsilon$ of its replacement.
Moreover, because $e_k^*\| \bv_k^f \|_2 \leq C \sqrt{p/n_k}$, $\tau_{e_k}^{*2} \geq c$, and $\|\bxi_g\|_2/\sqrt{n_k} \leq C\sqrt{p/n_k}$,
the derivatives with respect to the terms we replace are bounded by $C\sqrt{p/n_k}$.
We conclude
\begin{equation}\label{eq:mgam-grad-small-v}
    \inf \big\{
        \| \bdelta \|_2
        :
        \bdelta \in \partial \ell_k^{(2)}(\bv_k^f) 
    \big\}
    < C'\frac{p}{ n_k }\,\epsilon
\end{equation}
\\

Third, we show that $\ell_k^{(2)}(\bv)$ is strongly convex for bounded $\bv $.
Because $\bar \Omega_k$ is convex and $\bxi_g^\top\bv$ is linear, 
it suffices to lower bound the Hessian of $e_k^* \sqrt{\tau_{e_k}^{*2} + \| \bv \|_2^2 }$.
The Hessian is
\begin{equation}
    \frac{e_k^* }{\sqrt{\tau_{e_k}^{*2} + \| \bv \|_2^2 }}
    \Big(
        \id_p
        -
        \frac{\bv\bv^\top}{\tau_{e_k}^{*2} + \| \bv \|_2^2}
    \Big)
    \succeq
    \frac{e_k^*\tau_{e_k}^2}{(\tau_{e_k}^{*2} + \| \bv \|_2^2)^{3/2}}\, \id_p.
\end{equation}
By Lemma \ref{lem:bound-on-fixed-pt}, Eq.~\eqref{eq:cEgam}, and because $\|\bv\|_2 \leq R$ for $R = C_v\sqrt{p/n_k} \leq C$,
the coefficient in the previous display is no smaller than $c>0$.
Thus,
$\ell_k^{(2)}(\bv)$ is $c$-strongly convex on $\| \bv \|_2 \leq R$.
\\

\noindent \textit{Step 4: we use these properties to establish lower bounds on the min-max problem.}

Because $\ell_k^{(2)}(\bv)$ is $c$-strongly convex on $\| \bv \|_2 \leq R$,
for any $\| \bv \|_2 \leq R$ and $\bdelta \in \partial \ell_k^{(2)}(\bv_k^f)$
\begin{equation}
\begin{aligned}
    \ell_k^{(2)}(\bv)
        &\geq 
        \ell_k^{(2)}(\bv_k^f) 
        + 
        \bdelta^\top(\bv - \bv_k^f) 
        +
        \frac{c}2 \| \bv - \bv_k^f \|_2^2
    \\
        &\geq 
        \ell_k^{(2)}(\bv_k^f) 
        - 
        \frac{16}{c}\|\bdelta \|_2^2
        + 
        \frac{c}{4} \| \bv - \bv_k^f \|_2^2.
\end{aligned}
\end{equation}
Combining the previous display with 
Eqs.~\eqref{eq:marg-gs-approx}, \eqref{eq:mgam-to-mgam1-v}, \eqref{eq:mgam2-vgam*-approx}, and \eqref{eq:mgam-grad-small-v},
we conclude
\begin{equation}\label{eq:mgam-min-max-upper-bound-v}
\begin{gathered}        
            \min_{\substack{\bv \in E_v^c(\Delta) \\ \| \bv \|_2 \leq R } }\;
            \max_{\bu \in \cS_u }
                        \ell_k(\bu,\bv)
                \geq 
                \ell_k^* - C'\frac{p}{ n_k }\,\epsilon 
                +
                \min_{\substack{\bv \in E_v^c(\Delta) \\ \| \bv \|_2 \leq R } }\;
                    \frac{c}{4} \| \bv - \bv_k^f\|_2^2,
    \\
            \min_{\substack{\| \bv \|_2 \leq R } }\;
            \max_{\bu \in \cS_u }
                        \ell_k(\bu,\bv)
                \geq 
                \ell_k^* - C'\frac{p}{ n_k }\,\epsilon.
\end{gathered}
\end{equation} 
Further, by Eq.~\eqref{eq:cEgam},
\begin{equation}
        \big|
            \phi_v\big( \bv_k^f \big)
            -
            \phi_v^* 
        \big|
        \leq \sqrt{\frac{p}{ n_k }}\,\frac{\Delta}2.
\end{equation}
By the definition of $E_v^c(\Delta)$ and because $\phi_v$ is $1$-Lipschitz, 
\begin{equation}
    \min_{\substack{\bv \in E_v^c(\Delta) \\ \| \bv \|_2 \leq R } }\;
                \frac{c}{4} \| \bv - \bv_k^f\|_2^2 \geq c \,\frac{p}{ n_k } \frac{\Delta^2}{16}.
\end{equation}
Combining Eq.~\eqref{eq:mgam-min-max-upper-bound-v} with the previous display and taking $\Delta = \sqrt{32C'\epsilon/c}$ (with the same values $c,C'$ appearing in Eq.~\eqref{eq:mgam-min-max-upper-bound-v} and the previous display),
we conclude the lower bound inside the first probability in Eq.~\eqref{eq:mgam-bounds} holds. 
The lower bound inside the second probability in Eq.~\eqref{eq:mgam-bounds} is given in Eq.~\eqref{eq:mgam-min-max-upper-bound-v}.

In summary, we have shown that the lower bounds inside the probabilities in Eq.~\eqref{eq:mgam-bounds} hold on $\cG_k(\epsilon,\sqrt{32C'\epsilon/c})$ for $\epsilon < c'$. 
By Lemma \ref{lem:marginal-concentration-event}, we conclude that the probability bounds in Eq.~\eqref{eq:mgam-bounds} hold as well.

\subsubsection{Upper bounds on the auxilliary max-min problem}
\label{app:marginal-characterization-residuals}

In this section, we prove Eq.~\eqref{eq:mgam-upper-bound}.
In particular, we seek a upper bound on 
\begin{equation}
    \max_{\bu \in \cS_u }   
    \min_{ \|\bv\|_2 \leq R }\;    
                \ell_k(\bu,\bv),
\end{equation}
and also on the problem where we further restrict the minimization to $\bu \in E_u^c(\epsilon)$.
As in the previous section, we will show that our upper bound holds on the event $\cG_k(\epsilon,\Delta)$ for $\epsilon < c'$, $\Delta = \Delta(\epsilon) < c'$, and $c'$ taken sufficiently small and depending only on $\cPmodel,\cPregr$, and the regression method.
We will henceforce assume we are on this event for $\epsilon,\Delta$ for $c'$ sufficiently small, without repeatedly reminding the reader of this fact,
and all statements will be deterministic.

The major steps in the analysis are as in the previous section, except that the first and second steps occur in the opposite order.
We now carry out these steps in detail. 
\\

\noindent \textit{Step 1: we construct an upper bound on the internal minimization.}

In particular, 
we evaluate $\ell_k(\bu,\bv)$ at $\bv = \bv_k^f$.
Define
\begin{equation}
    \ell_k^{(1)}(\bu):=
        \frac1{ n_k }\bu_{\cI_k}^\top \be_{k,\cI_k}
        -\frac1{ n_k }\|\bu\|_2\bxi_g^\top \bv_k^f
            +
            \frac1{ n_k }
            \|\bv_k^f\|_2 \bxi_{h,\cI_k}^\top \bu_{\cI_k}
            - \frac1{2{ n_k }} \| \bu \|_2^2 
            + \bar\Omega_\gamma(\bv).
\end{equation}
Then
\begin{equation}\label{eq:mgam-to-mgam1-u}
        \max_{\mathclap{
            \substack{
                \bu \in E_u^c(\Delta)\\
                                \bu \in \cS_u }}}
            \;\;\; \;\;                              
                    \min_{\|\bv\|_2 \leq R }
            \ell_k(\bu,\bv)
            \leq 
            \max_{\substack{\bu \in E_u^c(\Delta)\\
                        \bu \in \cS_u }}
                \ell_k^{(1)}(\bu).
\end{equation}
Indeed, by taking $R = C\sqrt{p/n_k}$ with $C$ sufficiently large, $\|\bv_k^f\|_2 \leq R$ by
Eqs.~\eqref{eq:cEgam}, \eqref{eq:T-marg-ref}, and Lemma \ref{lem:bound-on-fixed-pt}.
The previous display follows.
\\

\noindent \textit{Step 2: we replace $\ell_k^{(1)}(\bu)$ by a function $\ell_k^{(2)}(\bu)$ which approximates it uniformly well across its domain.} 

In particular, we make the replacements given in Eq.~\eqref{eq:vk-repacements},
after which we get the objective
\begin{equation}\label{eq:mgam2-def}
\begin{aligned}
    \ell_k^{(2)}(\bu)
        &:=
        \frac1{ n_k } \bu_{\cI_k}^\top (\be_{k,\cI_k} + \tau_{h_k} \bxi_{h,\cI_k})
        - \frac{\df_k}{n_k}\sqrt{\tau_{e_k}^{*2} + \tau_{h_k}^2}\, \frac{\| \bu \|_2}{\sqrt{n_k}}
        - 
        \frac12 \frac{\| \bu \|_2^2}{ n_k }
        + 
        \omega_k.       
\end{aligned}
\end{equation}
As argued after Eq.~\eqref{eq:vk-repacements},
the terms in Eq.~\eqref{eq:vk-repacements} differ from their replacements by at most $(p/n_k)\epsilon$.
Using the $\delta$-method,
we can conclude
\begin{equation}\label{eq:mgam1-to-mgam2-u}
        \sup_{\|\bu\|_2/ \sqrt{n_k} < 2 e_k^*} |\ell_k^{(1)}(\bu) - \ell_k^{(2)}(\bu)| < C' \frac{p}{n_k}\, \epsilon,
\end{equation}
provided the coefficients of the terms we replace are appropriately bounded with high-probability.
Indeed, the coefficient of $\bxi_g^\top \bv_k^f/ \sqrt{n_k}$ is $\| \bu \|_2 /  \sqrt{n_k}$, which is bounded by $ C$ when $\bu \in \cS_u$ (we have used Eq.~\eqref{eq:cEgam} and assumption \textsf{A2}).
Because $\tau_{h_k} \geq C \sqrt{p/n_k}$ (see Lemma \ref{lem:bound-on-fixed-pt}), 
by the $\delta$-method we have $\big| \| \bv_k \|_2 - \tau_{h_k} \big| < \sqrt{p/n_k}\epsilon$, and the coefficient of $\| \bv_k^f \|_2$ is $\bxi_{h,\cI_k}^\top \bu_{\cI_k} /  n_k \leq \zeta_k \bxi_{h,\cI_k}^\top \be_{k,\cI_k}/n_k + R\|\bxi_{h,\cI_k}\|_2/\sqrt{n_k} \leq C\sqrt{p/n_k}\,\epsilon$ by Eq.~\eqref{eq:cEgam}, \eqref{eq:T-marg-ref}, and Lemma \ref{lem:bound-on-fixed-pt}.
Thus, the replacement of $\|\bv_k^f\|_2$ by $\tau_{h_k}$ incurs an error bounded by $C(p/n_k)\epsilon$.
Finally,
the coefficient of $\bar \Omega_k(\bv)$ is 1.
\\

\noindent \textit{Step 3: we establish several properties of the upper bound $\ell_k^{(2)}(\bu)$.}

First, we approximate the value $\ell_k^{(2)}(\bu_k^f)$.
Recalling that $\bu_{k,\cI_k}^f = \zeta_k(\be_{k,\cI_k} + \tau_{h_k} \bxi_{h,\cI_k})$,
we compute
\begin{equation}
\begin{aligned}
    \ell_k^{(2)}&(\bu_k^f)
        =
        \frac{\|\bu_k^f\|_2^2}{n_k\zeta_k}
        - 
        \frac{\df_k}{n_k}\sqrt{\tau_{e_k}^{*2} + \tau_{h_k}^2}\, \frac{\| \bu_k^f \|_2}{\sqrt{n_k}}
        - 
        \frac12 \frac{\| \bu_k^f \|_2^2}{ n_k }
        + 
        \omega_k.       
\end{aligned}
\end{equation}
We replace the quantity $\|\bu_k^f\|_2/ \sqrt{n_k}$ with $e_k^*$.
If we make this replacement, we get $\ell_k^*$ (recall the definition of $\ell_k^*$, Eq.~\eqref{eq:lk*}).
Using $\tau_{h_k} \leq C\sqrt{p/n_k}$ (see Lemma \ref{lem:bound-on-fixed-pt}), Eq.~\eqref{eq:cEgam}, and Eq.~\eqref{eq:T-marg-ref},
we conclude that the term $\|\bu_k^f\|_2/ \sqrt{n_k}$ differs from $e_k^*$ by at most $(p/n_k)\epsilon$.
By the $\delta$-method,
\begin{equation}\label{eq:mgam2-ugam*-approx}
        \big|
            \ell_k^{(2)}(\bu_k^f)
            - 
            \ell_k^*
        \big|
        < C' \frac{p}{n_k}\,\epsilon,
\end{equation}
provided we can show the derivatives with respect to the terms we replace are appropriately bounded.
Indeed, they are because $\zeta_k \geq c$ and $\df_k/n_k \leq 1$ (see Lemma \ref{lem:bound-on-fixed-pt}).

Second, we show the gradient of $\ell_k^{(2)}(\bu)$ at $\bu_k^f$ is small.
We only compute the gradient for the coordinates in $\cI_k$.
Its norm is
\begin{equation}
\begin{aligned}
    \big\|\nabla & \ell_k^{(2)}(\bu_k^f)_{\cI_k}\big\|_2
        =
        \Big\|
            \frac{\be_{k,\cI_k} + \tau_{h_k} \bxi_{h,\cI_k}}{n_k}
            -\frac{\df_k}{n_k}\sqrt{\tau_{e_k}^{*2} + \tau_{h_k}^2}\, \frac{\bu_{k,\cI_k}^f}{\sqrt{n_k}\|\bu_k^f\|_2}
            - 
            \frac{\bu_{k,\cI_k}^f}{ n_k }
        \Big\|_2
    \\
        &=
        \frac1{ \sqrt{n_k} }
        \Big(
            \zeta_k^{-1}
            - 
            \frac{\df_k}{n_k}\sqrt{\tau_{e_k}^{*2} + \tau_{h_k}^2} \frac{1}{\| \bu_k^f \|_2 / \sqrt{n_k} }
            -
            1
        \Big)
        \frac{\| \bu_k^f \|_2}{ \sqrt{n_k} }
\end{aligned}
\end{equation}
We replace the quantity $\| \bu_k^f \|_2 /  \sqrt{n_k}$ with $e_k^* = \zeta_k(\tau_{e_k}^{*2} + \tau_{h_k}^2)^{1/2}$.
As argued above, it differs from its replacement by at most $(p/n_k)\epsilon$.
After this replacement, the right-hand side becomes 0.
By the $\delta$-method,
\begin{equation}
        \big\|\nabla  \ell_k^{(2)}(\bu_k^f)_{\cI_k}\big\|_2
        < C' \frac{p}{n_k^{3/2}}\,\epsilon,
\end{equation}
provided we can show show the derivatives with respect to the terms we replace are appropriately bounded.
Indeed, they are because $\tau_{e_k}^{*2} + \tau_{h_k}^2 \leq C$ and $e_k^* \geq c > 0$ (see Eq.~\eqref{eq:cEgam} and Lemma \ref{lem:bound-on-fixed-pt}) and $\df_k/n_k \leq 1$.

Third, observe that by Eq.~\eqref{eq:mgam2-def},
$\ell_k^{(2)}(\bu)$ is $1/ n_k $-strongly concave everywhere.
\\

\noindent \textit{Step 4: we use these properties to establish a upper bounds on the max-min problem.}

Because $\ell_k^{(2)}(\bu)$ is $1/ n_k $-strongly concave,
for any $\bu \in \reals^N$ with $\bu_{\cI_k^c} = 0$,
\begin{equation}\label{eq:mgam-grad-u-small}
\begin{aligned}
    \ell_k^{(2)}(\bu)
        &\leq 
        \ell_k^{(2)}(\bu_k^f) 
        + 
        \nabla \ell_k^{(2)}(\bu_k^f)_{\cI_k}^\top(\bu_{\cI_k} - \bu_{k,\cI_k}^f) 
        - 
        \frac1{2 n_k } \| \bu - \bu_k^f \|_2^2
    \\
        &\leq 
        \ell_k^{(2)}(\bu_k^f) 
        + 
         n_k \cdot \|\nabla \ell_k^{(2)}(\bu_k^f)_{\cI_k}\|_2^2
        - 
        \frac1{4 n_k } \| \bu - \bu_k^f \|_2^2.
\end{aligned}
\end{equation}
Combining the previous display with Eqs.~\eqref{eq:mgam-to-mgam1-u}, \eqref{eq:mgam1-to-mgam2-u}, \eqref{eq:mgam2-ugam*-approx}, and \eqref{eq:mgam-grad-u-small},
we conclude
\begin{equation}\label{eq:mgam-min-max-upper-bound}
\begin{gathered}
        \max_{
            \substack{
                \bu \in E_u^c(\Delta)\\
                                \bu \in \cS_u }}\;
                    \min_{\|\bv\|_2 \leq R }
                    \ell_k(\bu,\bv)
            \leq 
            \ell_k^* + C' \frac{p}{n_k}\, \epsilon 
            - 
            \;\;\;
            \min_{\mathclap{
                \substack{\bu \in E_u^c(\Delta)\\
                                \bu \in \cS_u }}}
                \;\;\;
                \frac{\| \bu - \bu_k^f \|_2^2}{4 n_k },
    \\
        \max_{ \bu \in \cS_u }\;
        \min_{ \|\bv\|_2 \leq R }
            \ell_k(\bu,\bv)
            \leq 
            \ell_k^* + C' \frac{p}{n_k}\, \epsilon.
\end{gathered}
\end{equation} 
Further,
by Lemma \ref{lem:marginal-concentration-event}, after adjusting constants and using the $\sqrt{p/n_k} \leq C$,
\begin{equation}
        \Big|
            \phi_u\Big( \frac{\bu_{k,\cI_k}^f}{ \sqrt{n_k}}  \Big)
            -
            \phi_u^* 
        \Big|
        \leq \sqrt{\frac{p}{n_k}}\, \frac{\Delta}{2}.
\end{equation}
By the definition of $E_u^c(\Delta)$ and because $\phi_u$ is $1$-Lipschitz
\begin{equation}
    \min_{\substack{\bu \in E_u^c(\Delta)\\\bu \in \cS_u }}
        \frac{\| \bu - \bu_k^f \|_2^2}{4 n_k } \geq
        \frac{p}{n_k}\, \frac{\Delta^2}{16}.
\end{equation}
Combining Eq.~\eqref{eq:mgam-min-max-upper-bound} with the previous display and taking $\Delta = \sqrt{32C'\epsilon}$ (with the same value $C'$ appearing in Eq.~\eqref{eq:mgam-min-max-upper-bound}),
we conclude the upper bound inside the probability in the first line of Eq.~\eqref{eq:mgam-upper-bound} holds. 
The upper bound inside the probability in the second line of Eq.~\eqref{eq:mgam-upper-bound} is given in Eq.~\eqref{eq:mgam-min-max-upper-bound}.

In summary, we have shown that the lower bounds inside the probabilities in Eq.~\eqref{eq:mgam-upper-bound} hold on $\cG_k(\epsilon,\sqrt{32C'\epsilon})$ for $\epsilon < c'$. 
By Lemma \ref{lem:marginal-concentration-event}, we conclude that the probability bounds in Eq.~\eqref{eq:mgam-upper-bound} hold as well.

\section{Conditional characterization}
\label{sec:cond-char}

In this section, we prove the conditional characterization (Lemma \ref{lem:conditional-characterization}).

\subsection{The conditional Gordon inequality for regression}

To prove Lemma \ref{lem:conditional-characterization}, we use a Gordon-like comparison inequality, 
which we call the \emph{conditional Gordon inequality for regression}, stated below in Lemma \ref{lem:conditional-gordon}.
Like the marginal Gordon inequality (Lemma \ref{lem:gordon-marginal}),
the conditional Gordon inequality involves comparing the Gaussian process $L_2$ in 
Eq.~\eqref{eq:regr-min-max} to a simpler stochastic process.
Unlike in the marginal Gordon inequality, the simpler process is not Gaussian.
It is defined using Gaussian vectors $\bxi_g \sim \normal(0,\id_p)$ and $\bxi_h \sim \normal(0,\id_N)$,
but also using non-Gaussian vectors $\hat \bxi_g \in \reals^p$ and $\hat \bxi_h \in \reals^N$ which are functions of $\mathsf{Cond}_1$.
In particular, the simpler process is defined on the same probability space as the random design models.

We first define $\hat \bxi_g,\hat \bxi_h$. 
We warn the reader that the definition is complicated and non-intuitive, and is carefully chosen to make the comparison inequality hold.
Some intuition for these definitions is given in Section \ref{app:conditional-gordon}, where the comparison inequality is proved.
Define $ \bu_1,\bt_1 \in \reals^N$ and $\bv_1, \bs_1 \in \reals^p$ 
by
\begin{equation}
\label{eq:uvst}
    \bu_1 = \hat \be_1,
    \quad 
    \bv_1 = \bSigma^{1/2}(\hat \btheta_1 - \btheta_1),
    \quad 
    \bs_1 = \bSigma^{-1/2}\bX^\top \hat \be_1 = \bSigma^{-1/2}\bX_{\cI_1}^\top (\by_{1,\cI_1} - \bX_{\cI_1} \hat \btheta_1),
    \quad 
    \bt_1 = \bX(\hat \btheta_1 - \btheta_1).
\end{equation}
Note that $\bu_1,\bv_1,\bs_1,\bt_1$ are functions of $\mathsf{Cond}_1$.
Given $(\bu_1,\bv_1,\bs_1,\bt_1)$,
we define vectors $\hat \bxi_g \in \reals^p$, $\hat \bxi_h \in \reals^N$ as the unique solution to the equations 
\begin{gather}
    -\frac{ \bu_1}{\| \bu_1\|_2} \< \hat \bxi_g,  \bv_1 \> + \|  \bv_1\|_2 \hat \bxi_h -  \bt_1 = 0,
    \qquad
    -\|  \bu_1 \|_2 \hat \bxi_g + \frac{ \bv_1}{\|  \bv_1\|_2} \< \hat \bxi_h,  \bu_1 \> +  \bs_1 = 0,\label{eq:hat-xi}
    \\
    \frac{\< \hat \bxi_h , \bu _1 \>}{n_1 } = \zeta_1 \tau_{h_1}.\label{eq:constraint+}
\end{gather}
These equations indeed have a unique solution:
\begin{lemma}\label{lem:unique-soln}
    For $ \bu_1,\bt_1 \in \reals^N$, $\bv_1, \bs_1 \in \reals^p$ defined by Eq.~\eqref{eq:uvst},
    where $\hat \btheta_1$ is defined by Eq.~\eqref{eq:param-est},
    the equations \eqref{eq:hat-xi} and \eqref{eq:constraint+} have a unique solution.
\end{lemma}
\noindent Lemma \ref{lem:unique-soln} is proved in Section \ref{app:conditional-gordon}. 
We are ready to state the conditional Gordon inequality for regression.

\begin{lemma}[Conditional Gordon inequality for regression]\label{lem:conditional-gordon}
    Define $\hat \bxi_g,\hat \bxi_h$ to be the unique solution to Eqs.~\eqref{eq:hat-xi} and \eqref{eq:constraint+},
    where $\bu_1,\bv_1,\bs_1,\bt_1$ are defined by Eq.~\eqref{eq:uvst},
    and let $\bxi_g \sim \normal(0,\id_p)$ and $\bxi_h \sim \normal(0,\id_N)$ independent of each other and everything else.
    Define
    \begin{equation}
        \bg_{\mathrm{cg}}(\bu)
            :=
            \frac{\<  \bu_1,\bu \>}{\| \bu_1\|_2}\hat \bxi_g +  \| \sfP_{ \bu_1}^\perp \bu\|_2  \sfP_{ \bv_1}^\perp  \bxi_g,
        \qquad 
        \bh_{\mathrm{cg}}(\bv)
            :=
            \frac{\<  \bv_1 , \bv \> }{\|  \bv_1 \|_2} \hat \bxi_h + \| \sfP_{ \bv_1}^\perp \bv\|_2  \sfP_{ \bu_1}^\perp \bxi_h,
    \end{equation}
    (where the subscript stands for ``conditional Gordon''),
    and define the conditional auxilliary objective
    \begin{equation}\label{eq:m-theta-def}
    \begin{aligned}
        \ell_{2|1}(\bu,\bv)
            &:=
            -\frac1{n_2}\<\bg_{\mathrm{cg}}(\bu),\bv\> 
                + 
                \frac1{n_2}
                \<\bh_{\mathrm{cg}}(\bv),\bu\>
                + \frac1{n_2}\bu^\top \be_2
                - \frac1{2{n_2}} \| \bu_{ \cI_2 } \|_2^2 
                + \bar \Omega_2(\bv).
    \end{aligned}
    \end{equation}
    
    \begin{enumerate}[(i)]
        
        \item 
        If $E_u = E_u(\mathsf{Cond}_1) \in \reals^N$, $E_v = E_v(\mathsf{Cond}_1) \in \reals^p$ are (possibly $\mathsf{Cond}_1$-dependent) compact sets,
        then for any $t\in \reals$
        \begin{equation}
        \begin{aligned}
            \PP\left( \min_{\bv \in E_v} \max_{\bu \in E_u} L_2(\bu,\bv) \leq t \right) 
            \leq 
                2\PP\left( 
                    \min_{\bv \in E_v} \max_{\bu \in E_u} 
                    \ell_{2|1}(\bu,\bv) \leq t\right).
        \end{aligned}
        \end{equation}  

        \item 
        If $E_u = E_u(\mathsf{Cond}_1) \in \reals^N$, $E_v = E_v(\mathsf{Cond}_1) \in \reals^p$ are (possibly $\mathsf{Cond}_1$-dependent) compact, convex sets,
        then for any $t\in \reals$
        \begin{equation}
        \begin{aligned}
            \PP\left( \min_{\bv \in E_v} \max_{\bu \in E_u} L_2(\bu,\bv) \geq t \right) 
            \leq 
                2\PP\left( 
                    \min_{\bv \in E_v} \max_{\bu \in E_u} 
                    \ell_{2|1}(\bu,\bv) \geq t\right).
        \end{aligned}
        \end{equation}  

    \end{enumerate}  
\end{lemma}
\noindent Lemma \ref{lem:conditional-gordon} is a consequence of a more general comparison inequality which may be of independent interest.
Both Lemma \ref{lem:conditional-gordon} and the more general comparison inequality are proved in Section \ref{app:conditional-gordon}.

\begin{remark}
\label{rmk:cond-obj-interpretation}
    The conditional Gordon inequality for regression takes the same form as the marginal Gordon inequality in Section \ref{sec:marg-char-proof},
    except that $\bg_{\mathrm{mg}}(\bu)$, $\bh_{\mathrm{mg}}(\bv)$ are replaced by $\bg_{\mathrm{cg}}(\bu)$, $\bh_{\mathrm{cg}}(\bv)$.
    The quantity $\bg_{\mathrm{mg}}(\bu)$ has a Gaussian distribution with variance $\| \bu\|_2^2 \id_p$, and the quantity $\bh_{\mathrm{mg}}(\bv)$ has a Gaussian distribution with variance $\| \bv \|_2^2 \id_N$.
    This is not true of the quantities $\bg_{\mathrm{cg}}(\bu)$ and $\bh_{\mathrm{cg}}(\bv)$,
    but we encourage the reader to adopt the intuition that it is, in a certain sense, approximately true. 
    Indeed, 
    we will show using the marginal characterization (Lemma \ref{lem:marginal-characterization}) that $\hat \bxi_g,\hat \bxi_h$ behave like standard Gaussian vectors (where ``behave like'' will be in the sense of concentration of Lipschitz functions, as in the marginal characterization).
    Motivated by this intuition,
    if we replace $\hat \bxi_g$ with a standard Gaussian vector and replace $\sfP_{ \bv_1}^\perp  \bxi_g$ with $\bxi_g$ (which, we will see, negligibly impacts the min-max problem),
    then $\bg_{\mathrm{cg}}(\bu)$ has a Gaussian distribution with variance $\| \bu \|_2^2 \id_p$ and correlation $\< \bu_1, \bu \> / (\| \bu_1 \|_2 \| \bu \|_2)$ with $\hat \bxi_g$.
    Similarly, $\bh_{\mathrm{mg}}(\bv)$ behaves like a Gaussian distribution with variance $\| \bv \|_2^2 \id_p$ and correlation $\< \bv_1, \bv \> / (\| \bv_1 \|_2 \| \bv \|_2)$ with $\hat \bxi_h$.
    Thus, we should think of the conditional Gordon inequality as having the same structure as the marginal Gordon inequality, except that the Gaussian vectors are correlated with the quantities $\hat \bxi_g$ and $\hat \bxi_h$ coming from the first regression, with the correlation is determined implicitly by $\bu$ and $\bv$, respectively.
\end{remark}

\subsection{Proof of the conditional characterization (Lemma \ref{lem:conditional-characterization})}

We now define $\hat \bg_1(\mathsf{Cond}_1)$, $\hat \bh_1(\mathsf{Cond}_1)$ which appear 
in the statement of Lemma \ref{lem:conditional-characterization}:
\begin{equation}\label{eq:hat-g1-hat-h1}
    \hat \bg_1 :=
    \hat \bg_1(\mathsf{Cond}_1)
        := 
        \tau_{g_1} \hat \bxi_g,
    \quad 
    \hat \bh_1 :=
    \hat \bh_1(\mathsf{Cond}_1)
        := 
        \tau_{h_1} \hat \bxi_h,
    \quad 
    \text{for $\hat \bxi_g,\hat \bxi_h$ satisfying Eqs.~\eqref{eq:hat-xi} and \eqref{eq:constraint+}.}
\end{equation}
As in the proof of the marginal characterization (Lemma \ref{lem:marginal-characterization}),
we first characterize the behavior of $\bv_2,\bu_2$ by  comparison to the fixed-design model quantities $\bv_2^f$ and $\bu_2^f$.
Consider functions $\phi_v(\bv;\mathsf{Cond}_1)$ which is 1-Lipschitz in its first argument $\bv \in\reals^p$ and $\phi_u(\bu;\mathsf{Cond}_1)$ which is 1-Lipschitz in its first argument $\bu \in \reals^{n_2}$.
For simplicity of notation, we will often write $\phi_v(\bv)$ and $\phi_u(\bu)$, but the dependence of $\phi_v$ and $\phi_u$ on $\mathsf{Cond}_1$ will be understood.
We show that with high probability 
\begin{equation}\label{eq:cond-E}
\begin{gathered}
    \bv_2 \in 
    E_{v|1}(\epsilon)
        :=
        \Big\{
            \bv \in \reals^p : 
            \Big|
                \phi_v\big( \bv \big)
                -
                \E\big[\phi_v( \bv_2^f ) \bigm| \bg_1^f = \hat \bg_1 \big] 
            \Big| 
            < \epsilon
        \Big\},
    \\
    \bu_2 \in 
    E_{u|1}(\epsilon)
        :=
        \Big\{
            \bu \in \reals^{n_k}
            :
            \Big|
                \phi_u\Big( \frac{\bu}{ \sqrt{n_2}} \Big) 
                - 
                \E\Big[
                \phi_u\Big(
                    \frac{ \bu_{2,\cI_2}^f}{\sqrt{n_2}}
                \Big)
                \Bigm| \bh_1^f = \hat \bh_1,\, \be_1,\be_2
                \Big]
            \Big| 
            < \epsilon
        \Big\}.
\end{gathered}
\end{equation}
Analogous to our proof of the marginal characterization (Lemma \ref{lem:marginal-characterization}),
we will show this by proving the objectives in Eq.~\eqref{eq:vk-uk} are sub-optimal when these constraints are violated.
Precisely:
\begin{lemma}[Conditional control of primary objective]\label{lem:cond-sub-opt}
    There exists a constant $\ell_{2|1}^* = \ell_{2|1}^*(\tau_{e_2},n_2,p,\btheta_2,\Omega_2)$ and $\cPmodel$, $\cPregr$ and regression method-dependent $C',c' > 0$ and $\sC,\sc: \reals_{>0} \rightarrow \reals_{>0}$ such that
    with probability at least $1 - \sC(\epsilon) e^{-\sc(\epsilon) p}$
    \begin{equation}\label{eq:cond-sub-opt}
    \begin{gathered}
        \Big|
            \max_{\substack{\bu \in \reals^N \\ \bu_{\cI_2^c} = 0} }
                    \min_{\bv \in \reals^p}
                L_2(\bu,\bv)
            - 
            \ell_{2|1}^*
        \Big|
        =
        \Big|
            \min_{ \bv \in \reals^p }\
            \max_{\substack{\bu \in \reals^N \\ \bu_{\cI_2^c} = 0} }
                L_2(\bu,\bv)
            - 
            \ell_{2|1}^*
        \Big|
        \leq 
             C'\frac{\epsilon^2}2,
    \\
        \max_{\substack{\bu_{\cI_2} \in E_{u|1}^c(\epsilon)\\ \bu_{\cI_2^c}= 0}}\;
              \min_{\bv \in \reals^p}
                L_2(\bu,\bv)
            \leq \ell_{2|1}^* - C'\epsilon^2,
    \qquad
            \min_{\bv \in E_{v|1}^c(\epsilon)}\;
            \max_{\substack{\bu \in \reals^N \\ \bu_{\cI_2^c} = 0} }
                    L_2(\bu,\bv)
                \geq \ell_{2|1}^* + C'\epsilon^2.
    \end{gathered}
    \end{equation}
\end{lemma}
\noindent Lemma \ref{lem:cond-sub-opt} is proved using the conditional Gordon inequality (Lemma \ref{lem:conditional-gordon}), as described in the following sections.

First we prove the conditional characterization (Lemma \ref{lem:conditional-characterization}) using Lemma \ref{lem:cond-sub-opt}.

\begin{proof}[Proof of conditional characterization (Lemma \ref{lem:conditional-characterization})]
    By the optimality of $\bu_2,\bv_2$ (Eq.~\eqref{eq:vk-uk}),
    when Eq.~\eqref{eq:cond-sub-opt} occurs we have $\bu_2 \in E_{u|1}(\epsilon)$ and $\bv_2 \in E_{v|1}(\epsilon)$.
    We conclude that for $\epsilon \leq c'$ with probability at least $1 - \sC(\epsilon)e^{-\sc(\epsilon)p}$ we have $\bu_2 \in E_{u|1}(\epsilon)$ and $\bv_2 \in E_{v|1}(\epsilon)$.
 
    We now establish the second line of Lemma \ref{lem:conditional-characterization}.
    Restricted to the coordinates $\cI_2$,
    the debiased estimate of the noise can be written as (see Eq.~\eqref{eq:noise-est}) $\hat \be_{2,\cI_2}^{\de} = \zeta_2^{-1} \bu_{2,\cI_2}$,
    which is $C$-Lipschitz in $\bu_{2,\cI_2}$ by Lemma \ref{lem:bound-on-fixed-pt}.
    Further, the quantities $\hat \be_{1,\cI_2},\hat \be_{1,\cI_2}^{\de}$ are functions of $\mathsf{Cond}_1$ (see Eqs.~\eqref{eq:noise-est} and \eqref{eq:Cond-k}).
    Thus, $\phi_e\Big(\Big\{\frac{\be_{k,\cI_2}}{\sqrt{n_2}}\Big\},\Big\{\frac{\hat \be_{k,\cI_k}}{\sqrt{n_k}}\Big\},\Big\{\frac{\hat \be_{k,\cI_k}^{\de}}{\sqrt{n_k}}\Big\}\Big)$ is a function of $\hat \be_{2,\cI_2}/\sqrt{n_2}$, $\hat \be_{2,\cI_2}^{\de}/\sqrt{n_2}$ and $\mathsf{Cond}_1$ and is $M_2$-Lipschitz in the first two.
    It is thus a function of $\bu_{2,\cI_2}/\sqrt{n_2}$ and $\mathsf{Cond}_1$ and is $CM_2$-Lipschitz in $\bu_{2,\cI_2}/\sqrt{n_2}$.
    Taking $\phi_u$ to be $1/M_2$ times this function and adjusting the $\cPmodel,\cPregr$, and regression-method dependent constants,
    the event in the second line of Lemma \ref{lem:conditional-characterization} is equivalent to $\bu_2 \not \in E_{u|1}(\epsilon)$. 
    The probability bound in the second line of Lemma \ref{lem:conditional-characterization} follows by the above discussion.

    We now establish the first line of Lemma \ref{lem:conditional-characterization} for OLS, ridge regression, and the $\alpha$-smoothed Lasso with $\alpha_k > 0$, and then extend the result to the Lasso.
    To simplify notation, we remove the subscript $2$ from $\alpha_2$.
    As in the proof of Lemma \ref{lem:marginal-characterization}, we adopt the convention that if we use OLS or ridge regression we set $\alpha = 1$, and if we use the $\alpha$-smoothed Lasso, then $\alpha$ is the smoothing constant.
    For OLS, ridge regression, and the $\alpha$-smoothed Lasso,
    we recall the representation of the estimate and debiased estimate from Eq.~\eqref{eq:v-to-theta} in the proof of Lemma \ref{lem:marginal-characterization}.
    As argued there,
    $\hat \btheta_2,\hat \btheta_2^{\de}$ is $C/\alpha$-Lipschitz in $\bv_2$.
    Then $\phi_\theta(\{\hat \btheta_k^f\},\{\hat \btheta_k^{f,\de}\})$ is a function of $\hat \btheta_2$, $\hat \btheta_2^{\de}$, and $\mathsf{Cond}_1$,
    and is $M_2/\alpha$-Lipschitz in the first two.
    It it is thus $CM_2/\alpha$-Lipschitz in $\bv_2$.
    Thus, for $\epsilon < c'$, with probability at least $1 - \sC(\epsilon)e^{-\sc(\epsilon)p}$, 
    \begin{equation}
        \big|\phi_\theta(\{\hat \btheta_k^f\},\{\hat \btheta_k^{f,\de}\}) - \phi_{\theta|1}(\mathsf{Cond}_1)\big|   
            \leq M_2\epsilon/\alpha.
    \end{equation}

    We now remove the dependence on $\alpha$ in the previous display so that the upper bound does not blow up as $\alpha \rightarrow 0$ and so that we may extend the results to the Lasso (i.e., $\alpha = 0$).
    To do so, we use the approximation technique of \cite{celentano2020lasso} and in the proof of Lemma \ref{lem:marginal-characterization}.
    As we did there, we consider using the Lasso and the $\alpha$-smoothed Lasso in the second regression.
    We introduce notation which distinguishes quantities related to the $\alpha$-smoothed Lasso from those related to the Lasso by adding a subscript $\alpha$ to the former.
    For example, $\hat \btheta_2$ will denote the Lasso estimate, and $\hat \btheta_{2,\alpha}$ will denote the $\alpha$-smoothed Lasso estimate.    
    Each choice of regression method for the second regression leads to a distinct fixed design model. 
    We define $\phi_{\theta|1}^{(\alpha)}$ to be the function defined by Eq.~\eqref{eq:def-phi|1} in the fixed design model corresponding to the use of the $\alpha$-smoothed Lasso in the second regression.

    As argued in the proof of Lemma \ref{lem:marginal-characterization},
    for $0 \leq \alpha' < \alpha < c$,
    we have with probability at least $1 - \sC(\alpha)e^{-\sc(\alpha)p}$ that $\| \hat \btheta_{2,\alpha'} - \hat \btheta_{2,\alpha}\|_2 < C'\sqrt{\alpha}$ and $\| \hat \btheta_{2,\alpha'}^{\de} - \hat \btheta_{2,\alpha}^{\de}\|_2 < C'\sqrt{\alpha}$.
    Moreover,
    we have the following lemma,
    proved in Section \ref{sec:existence-fix-pt}.
    \begin{lemma}\label{lem:cond-alpha-approx}
        Under \textsf{A1} and $\textsf{A2}$,
        we have for $0 \leq \alpha' \leq \alpha < c'$ with probability at least $1 - \sC(\alpha)e^{-\sc(\alpha)p}$ that
        \begin{equation}
            \big|
                \phi_{\theta|1}^{(\alpha')}(\hat \bg_1)
                -
                \phi_{\theta|1}^{(\alpha)}(\hat \bg_1)
            \big|
            < C\Big(M_1 \sqrt{\frac{p}{n_1}} + M_2 \sqrt{\frac{p}{n_2}}\;\Big)\sqrt{\alpha}.
        \end{equation}
    \end{lemma}
    \noindent 
    Combining these bounds, 
    we have for $\alpha,\epsilon < c'$ and $0\leq \alpha' \leq \alpha$ with probability at least $1 - \sC(\alpha,\epsilon)e^{-\sc(\alpha,\epsilon)p}$ that 
    \begin{equation}
    \begin{aligned}
        \big|
            \phi_\theta(\hat \btheta_1,\hat \btheta_1^\de,\hat \btheta_{2,\alpha'},\hat \btheta_{2,\alpha'}^\de) 
            &- 
            \phi_{\theta|1}^{(\alpha')}(\hat \bg_1)
        \big|   
            \leq
            \big|
                \phi_\theta(\hat \btheta_1,\hat \btheta_1^\de,\hat \btheta_{2,\alpha'},\hat \btheta_{2,\alpha'}^\de) 
                - 
                \phi_\theta(\hat \btheta_1,\hat \btheta_1^\de,\hat \btheta_{2,\alpha},\hat \btheta_{2,\alpha}^\de) 
            \big|
        \\
            &\qquad \qquad +
            \big| 
                \phi_\theta(\hat \btheta_1,\hat \btheta_1^\de,\hat \btheta_{2,\alpha},\hat \btheta_{2,\alpha}^\de) 
                - 
                \phi_{\theta|1}^{(\alpha)}(\hat \bg_1)
            \big|  
            +
            \big| 
                \phi_{\theta|1}^{(\alpha)}(\hat \bg_1)
                -
                \phi_{\theta|1}^{(\alpha')}(\hat \bg_1)
            \big|
        \\
            &\leq 
            C'M_2\sqrt{\alpha} + M_2\epsilon / \alpha + C\Big(M_1 \sqrt{\frac{p}{n_1}} + M_2 \sqrt{\frac{p}{n_2}}\;\Big)\sqrt{\alpha}.
    \end{aligned}  
    \end{equation}
    If we take $\alpha = \epsilon^{2/3}$
    and 
    adjust $\sc(\epsilon),\sC(\epsilon)$,
    we get, in the case of the $\alpha$-smoothed Lasso with smoothing parameter $\alpha' \leq \epsilon^{2/3}$, 
    the upper bound provided by the first line in Lemma \ref{lem:conditional-characterization}.
    For $\alpha' > \epsilon^{2/3}$, we may directly apply Eq.~\eqref{eq:non-Lasso-conc} to get the first line in Lemma \ref{lem:conditional-characterization}.
\end{proof}

\begin{remark}[Rates of concentration]
    The rate of concentration provided by Lemma \ref{lem:conditional-characterization} is slower than that provided by the marginal characterization (Lemma \ref{lem:marginal-characterization}),
    and is slower than the corresponding rate of concentration in the fixed design models (see Remark \ref{rmk:marg-rates-conc}).
    In the fixed design models, $\bg_2^f$ and $\bh_2^f$ have a Gaussian distribution conditional on $\bg_1^f = \hat \bg_1(\mathsf{Cond}_1),\bh_1^f = \hat \bh_1(\mathsf{Cond}_1)$,
    with variance $\tau_{g_2}^2 \rho_g^{\perp^2}$ and $\tau_{h_2}^2 \rho_h^{\perp2}$, respectively.
    By Lemma \ref{lem:bound-on-fixed-pt},
    these variances are of order $1/n_2$ and $p/n_2$.
    Thus, in the fixed design models, with conditional probability at least $1 - Ce^{-cp\epsilon^2}$
    the quantities $\phi_\theta(\{\hat \btheta_k^f\},\{\hat \btheta_k^{f,\de}\})$ and $\phi_e\Big(\Big\{\frac{\hat \be_{k,\cI_2}}{\sqrt{n_2}}\Big\},\Big\{\frac{\hat \be_{k,\cI_2}^{\de}}{\sqrt{n_2}}\Big\}\Big)$ fluctuate by no more than $C'M_2\sqrt{p/n_2}\,\epsilon$ from their conditional mean. 
    The bound provided by Lemma \ref{lem:conditional-characterization} is worse by a factor $\sqrt{n_2/p}$, which can be large under assumption $\mathsf{A2}$.
    We do not know whether this discrepancy is fundamental or an artefact of our proofs.
\end{remark}

We prove Lemma \ref{lem:cond-sub-opt} in the next section.

\subsection{Proof of Lemma \ref{lem:cond-sub-opt}: conditional control of the primary objective}
\label{sec:cond-char-proof}

It is useful to introduce a new fixed-design model which we call \emph{conditional fixed-design model}.
The conditional fixed-design model is defined on the same probability space as the conditional auxilliary objective \eqref{eq:m-theta-def}.
It is
\begin{equation}\label{eq:cond-fixed-des}
    \begin{gathered}
    \textbf{Conditional fixed-design model}\\
    \textbf{for parameter estimation}\\
    \bY^{cf}
        := 
        \bSigma^{1/2} \bTheta + \bG^{cf}
        \;\;
        \text{(CFD-P)}
    \end{gathered}
    \qquad\qquad 
    \begin{gathered}
    \textbf{Conditional Fixed-design model}\\
    \textbf{for noise estimation}\\
    \bR^{cf} = \bE + \bH^{cf}
    \;\;
    \text{(CFD-N)}
    \end{gathered}
\end{equation}
where
\begin{equation}\label{eq:cf-quantities}
    \bg_1^{cf} = \tau_{g_1}\hat \bxi_g = \hat \bg_1,
    \quad 
    \bg_2^{cf} = \tau_{g_2}(\rho_g \hat \bxi_g + \rho_g^\perp \bxi_g),
    \quad 
    \bh_1^{cf} = \tau_{h_1} \hat \bxi_h = \hat \bh_1,
    \quad 
    \bh_2^{cf} = \tau_{h_2}(\rho_h \hat \bxi_h + \rho_h^\perp \bxi_h).
\end{equation}
We define parameter and noise estimates in the conditional fixed-design model just as we did in the (unconditional) fixed-design model.
In particular,
\begin{equation}\label{eq:cond-fixed-design-est}
\begin{gathered}
    \hat \btheta_k^{cf}
        :=
        \eta_k(\by_k^{cf};\zeta_k)
        :=
        \argmin_{\bpi \in \reals^p}
        \Big\{
            \frac12 \| \by_k^{cf} - \bSigma^{1/2}\bpi\|_2^2 + \frac1{\zeta_k} \Omega_k(\bpi)
        \Big\},
    \\
    \bv_k^{cf}
        :=
        \argmin_{\bpi \in \reals^p}
        \Big\{
            \frac{\zeta_k}2 \| \bg_k^{cf} - \bpi\|_2^2 + \bar  \Omega_k(\bpi)
        \Big\},
    \\
    \hat \btheta_k^{{cf},\de}
        :=
        \hat \btheta_k^{cf}
        +
        \bSigma^{-1/2}(\by_k^{cf} - \bSigma^{1/2}\hat \btheta_k^{cf})
        =
        \bSigma^{-1/2}\by_k^{cf}.
\end{gathered}
\end{equation}
and
\begin{equation}
    \bu_{k,\cI}^{cf} = \hat \be^{cf}_{k,\cI} 
        =
        \begin{cases}
            \zeta_k\br^{cf}_{\cI_k} \quad &\text{if } \cI = \cI_k,\\
            \bzero \quad & \text{if } \cI = \cI_k^c,
        \end{cases}
    \qquad 
    \hat \be_k^{cf,\de}
        =
        \frac{1-\zeta_k}{\zeta_k} \hat \be_k^{cf}  
        +
        \Big(
            \br_k^{cf} - \frac{1-\zeta_k}{\zeta_k} \hat \be_k^{cf}
        \Big)
        =
        \br_k^{cf}.
\end{equation}

The value of the primary min-max problem concentrates on
\begin{equation}\label{eq:cthe*-wthe*-def}
\begin{gathered}
        \ell_{2|1}^*
        := 
        \frac12 n_2 \tau_{g_2}^2\zeta_2^2 + \omega_2,
\end{gathered}
\end{equation}
where $\omega_2$ is defined in Eq.~\eqref{eq:sig*-e*-ome*}.
We will consider the min-max problem under the compact constraints
\begin{equation}
    \| \bv \|_2 \leq C_v,
    \qquad 
    \bu \in \cS_{u|1}
        := 
        \Big\{
            \bu \in \reals^N
            \Bigm| 
            \frac1{\sqrt{n_2}}\big\|\bu_{\cI_2}\big\|_2 \leq C_u,
            \; \bu_{\cI_2^c} = 0
        \Big\}.
\end{equation}
For values of $C_v,C_u$ to be chosen below.
It will suffice that $C_v,C_u \geq C_{\mathrm{min}}$, where $C_{\mathrm{min}}$ only depends on $\cPmodel,\cPregr$, and the regression method.

In Section \ref{app:conditional-characterization-est-err},
we establish for $\epsilon < c'$ the lower bounds
\begin{equation}\label{eq:m0-lower-bound}
\begin{gathered}
    \P\Big(
            \min_{\substack{\bv \in E_{v|1}^c(\epsilon) \\ \| \bv \|_2 \leq C_v } }\;
            \max_{\bu \in \cS_{u|1}  }
                    \ell_{2|1}(\bu,\bv)
                \geq 
                \ell_{2|1}^* + C'\epsilon^2 
        \Big) \geq 1 - \sC(\epsilon)e^{-\sc(\epsilon)p},
    \\
    \P\Big(
        \min_{ \| \bv \|_2 \leq C_v  }\;
        \max_{\bu \in \cS_{u|1}  }
                \ell_{2|1}(\bu,\bv)
        \geq
        \ell_{2|1}^* 
        - C'\frac{\epsilon^2}2
    \Big) \geq 1 - \sC(\epsilon)e^{-\sc(\epsilon)p},
\end{gathered}
\end{equation}
where the $C'$ in both bounds is the same.
In Section \ref{app:conditional-characterization-residuals},
we establish for $\epsilon < c'$ the upper bounds
\begin{equation}\label{eq:mthe-bound}
\begin{gathered}
    \P\Big(
            \max_{
                \substack{
                    \bu_{ \cI_2 } \in E_{u|1}^c(\epsilon)\\
                                    \bu \in \cS_{u|1}  }}
                        \min_{\|\bv\|_2 \leq C_v }
                        \ell_{2|1}(\bu,\bv)
                \leq 
                \ell_{2|1}^* - C'\epsilon^2 
        \Big) \geq 1 - \sC(\epsilon)e^{-\sc(\epsilon)p},
    \\
    \P\Big(
        \min_{\|\bv\|_2 \leq C_v }
        \max_{\bu \in \cS_{u|1} }
                \ell_{2|1}(\bu,\bv)
                \leq 
                \ell_{2|1}^* + C'\frac{\epsilon^2}{2}
        \Big) \geq 1 - \sC(\epsilon)e^{-\sc(\epsilon)p},
\end{gathered}
\end{equation} 
where the $C'$ in both bounds is the same.
The preceding probabilities are taken over the randomness in $\be_1,\be_2,\hat \bxi_g,\hat \bxi_h,\bxi_g,\bxi_h$.

Applying the conditional Gordon inequality for regression (Lemma \ref{lem:conditional-gordon}),
we conclude that the previous two displays hold with $L_2$ in place of $\ell_{2|1}$ and adjusted constants.
Because the minimization and maximization are over compact convex sets,
by \cite[Corollary 37.3.2]{rockafellar-1970a} the problem
\begin{equation}
    \max_{ \bu \in \cS_{u|1} }\;
    \min_{\|\bv\|_2 \leq C_v }
        L_2(\bu,\bv)
    =
    \min_{\|\bv\|_2 \leq C_v }\;
    \max_{ \bu \in \cS_{u|1} }         
        L_2(\bu,\bv)        
\end{equation}
has a saddle-point $\hat \bu, \hat \bv$ and minimization and maximization can be exchanged.
Thus,
for $\epsilon < c'$ with probability at least $1 - \sC(\epsilon)e^{-\sc(\epsilon)p}$ 
\begin{equation}\label{eq:cond-obj-bound}
\begin{gathered}                  
    \max_{
    \substack{
        \bu_{\cI_2} \in E_{u|1}^c(\epsilon)\\
                        \bu \in \cS_{u|1} }}\;
            \min_{\|\bv\|_2 \leq C_v}
        L_2(\bu,\bv)
        \leq \ell_{2|1}^* - C'\epsilon^2,
    \quad
    \min_{\substack{\bv \in E_{v|1}^c(\epsilon)\\\|\bv\|_2 \leq C_v}}\;
    \max_{\bu \in \cS_{u|1} }
            L_2(\bu,\bv)
        \geq \ell_{2|1}^* + C'\epsilon^2,
    \\
    \Big|
        \max_{
        \substack{\bu \in \cS_{u|1} }}
                \min_{\|\bv\|_2 \leq C_v }
            L_2(\bu,\bv)
        - 
        \ell_{2|1}^*
    \Big|
    =
    \Big|
        \min_{\|\bv\|_2 \leq C_v }
        \max_{\bu \in \cS_{u|1} }
            L_2(\bu,\bv)
        - 
        \ell_{2|1}^*
    \Big|
    \leq 
         C'\frac{\epsilon^2}2,
\end{gathered}
\end{equation}
where the $C'$ in the previous three bounds can be taken to be equal.

To relax the restriction of the optimizations to compact sets, 
we use Lemma \ref{lem:marg-sub-opt}.
In particular, applying Lemma \ref{lem:marg-sub-opt} for the functions $\phi_v(\bv) = \| \bv \|_2$ and $\phi_u(\bu) = \| \bu - \zeta_2 \be_{2,\cI_2}\|_2 / \sqrt{n_2}$, 
we conclude that the saddle point $\bv_2,\bu_2$ for the unconstrained problem satisfies with probability at least $1 - \sC(\epsilon)e^{-\sc(\epsilon)p}$ that $\| \bv_2 \|_2 < C \sqrt{p/n_2}$ and $\| \bu_{2,\cI_2} - \zeta_2 \be_{2,\cI_2}\|_2 / \sqrt{n_2} < C \sqrt{p/n_2}$ for sufficiently large $C$.
Further, with probability at least $1 - \sC(\epsilon)e^{-\sc(\epsilon)p}$, we have $\| \zeta_2 \be_{2,\cI_2} \|_2 / \sqrt{n_2} < C$, whence $\| \bu_{2,\cI_2} \|_2 / \sqrt{n_2} < C$.
Thus, taking $C_v,C_u$ sufficiently large, we have that with probability at least $1 - \sC(\epsilon)e^{-\sc(\epsilon)p}$ we have $\| \bv_2 \|_2 < C_v$ and $ \| \bu_{2,\cI_2} \|_2 / \sqrt{n_2} < C_u$.
On this event, $\bv_2,\bu_2$ is also the saddle point for the min-max problem on the restricted domains $\| \bv_2 \|_2 < C_v$ and $\bu \in \cS_{u|1}$.
That is, $\hat \bv = \bv_2$ and $\hat \bu = \bu_2$. Then, the value of the min-max problems in the second line of Eq.~\eqref{eq:cond-obj-bound} is unaffected by removing the compactness constraints.

To remove the compactness constraints in the first line of Eq.~\eqref{eq:extend-bound},
it suffices to show that with probability at least $1 - \sC(\epsilon)e^{-\sc(\epsilon)p}$
\begin{equation}\label{eq:extend-bound}
   \sup_{\bu \in \cS_{u|1}^c }\;
        \min_{\bv \in \reals^p}
        L_2(\bu,\bv)
        \leq \ell_{2|1}^* - C'\epsilon^2,
    \qquad
    \inf_{\|\bv\|_2 > C_v}\;
    \max_{ \bu_{\cI_2^c} = 0 }
            L_2(\bu,\bv)
        \geq \ell_{2|1}^* + C'\epsilon^2.
\end{equation}
Indeed, then 
\begin{equation}
\begin{aligned}
    \max_{
        \bu_{\cI_2} \in E_{u|1}^c(\epsilon) }\;
            \min_{\bv \in \reals^p}
        L_2(\bu,\bv)
        &\leq 
        \max\Big\{
            \max_{
                \substack{\bu_{\cI_2} \in E_{u|1}^c(\epsilon)\\ \bu \in \cS_{u|1} }}\;
                    \min_{\bv \in \reals^p}
                L_2(\bu,\bv)
            ,\;
            \sup_{\bu \in \cS_{u|1}^c }\;
            \min_{\bv \in \reals^p}
            L_2(\bu,\bv)
            \leq \ell_{2|1}^* - C'\epsilon^2
        \Big\}
    \\
        &\leq 
        \max\Big\{
            \max_{
                \substack{\bu_{\cI_2} \in E_{u|1}^c(\epsilon)\\ \bu \in \cS_{u|1} }}\;
                    \min_{ \| \bv \|_2 \leq C_v}
                L_2(\bu,\bv)
            ,\;
            \sup_{\bu \in \cS_{u|1}^c }\;
            \min_{\bv \in \reals^p}
            L_2(\bu,\bv)
            \leq \ell_{2|1}^* - C'\epsilon^2
        \Big\},
\end{aligned}
\end{equation}
and we can bound the right-hand side using Eqs.~\eqref{eq:cond-obj-bound} and \eqref{eq:extend-bound}.
We can combine Eqs.~\eqref{eq:cond-obj-bound} and \eqref{eq:extend-bound} to lower bound the min-max problem in the second line of Eq.~\eqref{eq:cond-sub-opt} similarly.

We now show Eq.~\eqref{eq:extend-bound} using Lemma \ref{lem:marg-sub-opt}.
First, consider $\phi_u(\bu) = \| \bu \|_2 / \sqrt{n_2}$.
By Lemma \ref{lem:phi-|1-conc} and because $\E[\|\bu_{2,\cI_2}^f\|_2 ]/ \sqrt{n_2} \leq \zeta_2(\tau_{e_2}^2 + \tau_{h_2}^2)^{1/2} \leq C'$ (see Eq.~\eqref{eq:res-ref} and Lemma \ref{lem:bound-on-fixed-pt}),
we have with probability at least $1 - Ce^{-cp}$ 
that $\E\big[\phi_u(\bu_{2,\cI_2}^f/\sqrt{n_2}) \bigm| \bh_1^f = \hat \bh_1\big] \leq C' $.
Recalling the definition of $E_{u|1}(\epsilon)$ (see Eq.~\eqref{eq:cond-E}), we have for $\epsilon < c'$ sufficiently small and $C_u$ sufficiently large with probability at least $1 - \sC(\epsilon)e^{-\sc(\epsilon)p}$ that $E_{u|1}(\epsilon)$ is contained in the interior of $\cS_{u|1}$.
The first bound in Eq.~\eqref{eq:extend-bound} now follows by a similar argument used to remove the compactness constraint in Eq.~\eqref{eq:primal-prec-bound} in the proof of Lemma \ref{lem:marg-sub-opt}.
For completeness, we repeat the details here.
Consider the event that $\hat \bu \in E_{u|1}(\epsilon)$ and $E_{u|1}(\epsilon)$ is contained in the interior of $\cS_{u|1}$.
Consider any $\bu' \in \cS_{u|1}^c$.
Because the function $\bu \mapsto \min_{\|\bv\|_2 \leq C_v} L_2(\bu,\bv)$ is concave, it is non-decreasing on the line segment joining $\bu'$ to $\hat \bu$, and because $E_{u|1}(\epsilon)$ is contained in the interior of $\cS_{u|1}$ and $\hat \bu \in E_{u|1}(\epsilon)$,
this line segment intersects with $E_{u|1}^c(\epsilon) \cap \cS_{u|1}$ at some point $\bu''$.
Then, 
\begin{equation}
    \min_{\|\bv\|_2 \leq C_v} L_2(\bu',\bv)
    \leq 
    \min_{\|\bv\|_2 \leq C_v} L_2(\bu'',\bv)
    \leq 
    \max_{
        \substack{
            \bu \in E_{u|1}^c(\epsilon)\\
                            \bu \in \cS_{u|1} } }\;
            \min_{\|\bv\|_2 \leq C_v}
        L_2(\bu,\bv)
    \leq 
    \ell_{2|1}^*- C'\epsilon^2.
\end{equation}
Because this applies to all $\bu' \in S_u^c$, we conclude that with probability at least $1 - \sC(\epsilon)e^{-\sc(\epsilon)p}$ 
the first bound in Eq.~\eqref{eq:extend-bound}.
The second bound in Eq.~\eqref{eq:extend-bound} occurs by a similar argument, using $\phi_v(\bv) = \| \bv \|_2$.

Provided we can prove Eqs.~\eqref{eq:m0-lower-bound} and \eqref{eq:mthe-bound},
we have concluded Lemma \ref{lem:cond-sub-opt}.
We establish these bounds in the next three sections.

\subsubsection{The good conditional characterization event}
\label{app:conditional-characterization-preliminaries}

To prove Eqs.~\eqref{eq:m0-lower-bound} and \eqref{eq:mthe-bound},
we show that the upper and lower bounds on the min-max problem are implied by the occurrence of a certain high-probability event.
Recall the functions $\bT(\ba_1,\ldots,\ba_r) \in \SS_+^r$ introduced in Section \ref{sec:good-marg-char-event}.
The good conditional characterization event is
\begin{equation}
\begin{aligned}\label{eq:cEgam}
    &\cG_{2|1}(\epsilon,\Delta)
        :=
    \\
        \Big\{\;
            &\Big\| 
                \bT\Big(
                    \frac{\be_{2,\cI_2}}{\sqrt{n_2}},
                    \frac{\bu_{1,\cI_2}}{\sqrt{n_1}},
                    \frac{\hat \bxi_{h,\cI_2}}{\sqrt{n_2}},
                    \frac{\bxi_{h,\cI_2}}{\sqrt{n_2}}
                \Big) 
                - 
                \E\Big[
                    \bT\Big(
                        \frac{\be_{2,\cI_2}}{\sqrt{n_2}},
                        \frac{\bu_{1,\cI_2}^f}{\sqrt{n_1}},
                        \frac{\bh_{1,\cI_2}^f}{\sqrt{n_2}\, \tau_{h_1}} ,
                        \frac{\bh_{2,\cI_2}^f - (\tau_{h_2}\rho_h/\tau_{h_1})\bh_{1,\cI_2}^f}{\sqrt{n_2}\, \tau_{h_2}\rho_h^\perp }
                    \Big) 
                \Big] 
            \Big\|_{\sF} < \epsilon,\;
        \\
            &\Big\| 
                \bT\Big(
                    \sqrt{\frac{n_1}{p}}\,\bv_1,
                    \sqrt{\frac{n_2}{p}}\,\bv_2^{cf},
                    \frac{\hat \bxi_g}{\sqrt{p}},
                    \frac{\bxi_g}{\sqrt{p}}
                \Big) 
                - 
                \E\Big[
                    \bT\Big(
                        \sqrt{\frac{n_1}{p}}\,\bv_1^f,
                        \sqrt{\frac{n_2}{p}}\,\bv_2^f,
                        \frac{\bg_1^f}{\sqrt{p}\,\tau_{g_1}},
                        \frac{\bg_2^f - (\tau_{g_2}\rho_g/\tau_{g_1})\bg_1^f}{\sqrt{p}\,\tau_{g_2}\rho_g^\perp}
                    \Big) 
                \Big] 
            \Big\|_{\sF} < \epsilon,\;
        \\
            &\big|
                \bar \Omega_k(\bv_k^{cf})
                -
                \omega_k
            \big| < \epsilon,\;
            \Big| \frac{\|\bu_{1,\cI_2}\|_2}{\sqrt{n_1}} - \sqrt{n_{12}}\tau_{g_1}\zeta_1 \Big| < \epsilon,\;
            \Big| \frac{\|\bu_{1,\cI_2^c}\|_2}{\sqrt{n_1}} - \sqrt{n_1-n_{12}}\tau_{g_1}\zeta_1 \Big| < \epsilon,
        \\
            &\Big| \frac{\| \be_{2,\cI_2} \|_2}{\sqrt{n_2}} - \tau_{e_2} \Big| < \epsilon,\;
            \frac{\| \proj_{\bv_1} \bxi_g \|_2}{\sqrt{n_2}} < \epsilon,\;
            \frac{\| \proj_{\bu_2} \bxi_h \|_2}{\sqrt{n_2}} < \epsilon,\;
            \bv_2^{cf} \in E_{v|1}(\Delta/2),\;
            \bu_2^{cf} \in E_{u|1}(\Delta/2)
    \;\Big\}.
\end{aligned}
\end{equation}
Note the quantity $\| \bu_{1,\cI_2}^f \|_2 / \sqrt{n_1}$ is at most of order 1, though it can be substantially smaller if $\cI_1 \cap \cI_2$ is small.
In all other cases, we have normalized the arguments to $\bT$ so that their $\ell_2$ norm is of order 1 (see Lemma \ref{lem:bound-on-fixed-pt}).
For future reference, we have written the expectations appearing in the definition of $\cG_{2|1}$ in Section \ref{sec:identity-ref}.
In words,
the good conditional characterization event is the event that several important quantities are close to their expectations.
It has high probability.
\begin{lemma}[Gordon's good conditional characterization event]\label{lem:conditional-concentration-event}
    Assume \textsf{A1} and \textsf{A2}.
    There exist $\cPmodel$, $\cPregr$ and regression method-dependent $c' > 0$ and $\sC,\sc: \reals_{>0}^2 \rightarrow \reals_{>0}$ such that
    for $\epsilon,\Delta < c'$
    with probability at least $1 - \sC(\epsilon,\Delta)e^{-p\sc(\epsilon,\Delta)}$, 
    the event $\cG_{2|1}(\epsilon,\Delta)$ occurs.
\end{lemma}
\noindent We prove Lemma \ref{lem:conditional-concentration-event} in Section \ref{sec:char-support}.

\begin{remark}
    We point out an asymmetry in the definition of $\cG_{2|1}(\epsilon,\Delta)$.
    From the first regression, the random-design quantity $\bv_1$ is passed as an argument to $\bT$, whereas from the second regression, the conditional fixed-design quantity $\bv_2^{cf}$ is passed as an argument to $\bT$. 
    Note that in general, $\bv_1 \neq \bv_1^{cf}$.
    We point out this asymmetry to assure the reader that it is not a typo.
    Our choice is driven by the requirements of analyzing the conditional auxilliary objective. 
    Note that the conditional auxilliary objective is defined using $\bv_1$ (via the definition of $\bh_{cg}(\bv)$), not $\bv_1^{cf}$.
    On the other hand, we will approximate the minimizer of the conditional auxilliary objective by $\bv_2^{cf}$. 
    The quantity $\bv_2$ does not appear in the definition of $\cG_{2|1}(\epsilon,\Delta)$. 
    Indeed, it is not approximated well by quantities appearing in the conditional auxilliary objective.
    See Sections \ref{app:conditional-characterization-est-err} and \ref{app:conditional-characterization-residuals} for complete details.
\end{remark}

\subsubsection{Lower bounds on the auxilliary min-max problem}
\label{app:conditional-characterization-est-err}

In this section, we prove Eq.~\eqref{eq:m0-lower-bound}.
In particular, we seek a lower bound on 
\begin{equation}
    \min_{ \|\bv\|_2 \leq C_v }\;
        \max_{\bu \in \cS_u }
                \ell_{2|1}(\bu,\bv),
\end{equation}
and also on the problem where we further restrict the minimization to $\bv \in E_{v|1}^c(\epsilon)$.
It is enough to show that our lower bound holds on the event $\cG_{2|1}(\epsilon,\Delta)$ for $\epsilon < c'$, $\Delta = \Delta(\epsilon) < c'$, and $c'$ taken sufficiently small and depending only on $\cPmodel,\cPregr$, and the regression method.
We will henceforce assume we are on this event for $\epsilon,\Delta$ for $c'$ sufficiently small, without repeatedly reminding the reader of this fact,
and all statements will be deterministic.
In particular,
as we derive more implications of $\cG_{2|1}(\epsilon,\Delta)$, 
we may need to take $c'$ smaller,
but will not track $c'$ or announce when an additional implication requires we shrink $c'$ further.
To make notation more compact,
we denote
\begin{equation}
\begin{gathered}
     \sT_{\sN,2|1} 
        := 
        \E\Big[
            \bT\Big(
                \frac{\be_{2,\cI_2}}{\sqrt{n_2}},
                \frac{\bu_{1,\cI_2}^f}{\sqrt{n_1}},
                \sqrt{\frac{n_1}{p}}\bh_{1,\cI_2}^f,
                \sqrt{\frac{n_2}{p}}\bh_{2,\cI_2}^f,
                \frac{\bh_{1,\cI_2}^f}{\sqrt{n_2}\, \tau_{h_1}} ,
                \frac{\bh_{2,\cI_2}^f}{\sqrt{n_2}\, \tau_{h_2}}
            \Big) 
        \Big] ,
    \\
    \sT_{\sP,2|1}
        :=
        \E\Big[
            \bT\Big(
                \sqrt{\frac{n_1}{p}}\,\bv_1^f,
                \sqrt{\frac{n_2}{p}}\,\bv_2^f,
                \sqrt{\frac{n_1}{p}} \bg_1^f,
                \sqrt{\frac{n_2}{p}} \bg_2^f
            \Big) 
        \Big].
\end{gathered}
\end{equation}
We remind the reader that explicit expressions for these can be found in Section \ref{sec:identity-ref}.

The major steps in the analysis are similar to those in the proof of Eq.~\eqref{eq:mgam-bounds} from the marginal characterization.
We repeat those steps in the current context here:
\begin{enumerate}

    \item We replace $\ell_{2|1}(\bu,\bv)$ by a function $\ell_{2|1}^{(1)}(\bu,\bv)$ which approximates it uniformly well across its domain.
    We control the change in the value of the min-max problem incurred by this replacement.
    The objective $\ell_{2|1}^{(1)}(\bu,\bv)$ is introduced because it is easier to analyze.

    \item 
    Ideally, we would evaluate the internal maximization exactly.
    However, doing so explicitly leads to complicated expressions.
    Instead, we construct a lower bound on the internal maximization.
    In particular, for any function $\bu(\bv)$ for which $\bu(\bv) \in \cS_u$, the quantity $ \ell_{2|1}^{(1)}(\bu(\bv),\bv)$ is a lower bound on the value of the maximization problem. 
    Our strategy is to pick a function $\bu(\bv)$ so that the function $\ell_{2|1}^{(2)}(\bv) := \ell_{2|1}^{(1)}(\bu(\bv),\bv)$ is tractable to analyze and provides a good enough lower bound for our purposes.

    \item
    We establish several properties of the lower bound $\ell_{2|1}^{(2)}(\bv)$.  
    In particular, we show $\ell_{2|1}^{(2)}(\bv_2^{cf})$ is close to $\ell_{2|1}^*$,
    the subdifferential
    $\partial \ell_{2|1}^{(2)}(\bv_2^{cf})$ contains a small element, and $\ell_{2|1}^{(2)}(\bv)$ is strongly convex for bounded $\bv$.

    \item 
    Using standard convex analysis techniques, 
    these properties imply a lower bound on the min-max problem over $\| \bv\|_2 \leq C_v$ and over $\bv \in E_{v|1}^{c}(\epsilon) \cap \{ \|\bv\|_2 \leq C_v\}$.

\end{enumerate}
We now carry out these steps in detail. 
\\

\noindent \textit{Step 1: we replace $\ell_{2|1}(\bu,\bv)$ by a function $\ell_{2|1}^{(1)}(\bu,\bv)$ which approximates it uniformly well across its domain.}

We may use a Gram-Schmidt construction to replace $\be_{2,\cI_2}$, $\hat \bxi_{h,\cI_2}$, and $\bxi_{h,\cI_2}$ by $\be_{2,\cI_2}^{\mathrm{gs}}$, $\hat \bxi_{h,\cI_2}^{\mathrm{gs}}$, and $\bxi_{h,\cI_2}^{\mathrm{gs}}$ so that
\begin{equation}
    \bT\Big(
        \frac{\be_{2,\cI_2}^{\mathrm{gs}}}{\sqrt{n_2}},
        \frac{\hat \bxi_{h,\cI_2}^{\mathrm{gs}}}{\sqrt{n_2}},
        \frac{\bxi_{h,\cI_2}^{\mathrm{gs}}}{\sqrt{n_2}}
    \Big) 
    =
    \diag(\tau_{e_2}^2,1,1),
\end{equation}
and
\begin{equation}\label{eq:gs-diff}
    \max\Big\{
        \frac{\|\be_{2,\cI_2}^{\mathrm{gs}} - \be_{2,\cI_2}\|_2}{\sqrt{n_2}},
        \frac{\|\hat \bxi_{h,\cI_2}^{\mathrm{gs}} - \hat \bxi_{h,\cI_2}\|_2}{\sqrt{n_2}},
        \frac{\| \bxi_{h,\cI_2}^{\mathrm{gs}} -  \bxi_{h,\cI_2}\|_2}{\sqrt{n_2}}
    \Big\}
    < C' \epsilon.
\end{equation}
This follows by Eqs.~\eqref{eq:cEgam} and \eqref{eq:T-cond-ref} and Lemma \ref{lem:gs-approx}.
Define
\begin{equation}
    \bZ
        := 
        \frac1{\sqrt{n_2}}
        \begin{pmatrix}
            \vert & \vert & \vert \\[4pt]
            \be_{2, \cI_2 }^{\mathrm{gs}}/\tau_{e_2} & \hat \bxi_{h,\cI_2}^{\mathrm{gs}} &  \bxi_{h, \cI_2 }^{\mathrm{gs}} \\[4pt]
            \vert & \vert & \vert
        \end{pmatrix}.
\end{equation}
By the Gram-Schmidt construction, $\bZ^\top \bZ = \id_3$.

We next apply a Gram-Schmidt step to $\bu_1$. 
In particular, 
define
\begin{equation}\label{eq:gs-u1}
    \frac{\bu_{1,\cI_2}^{\mathrm{gs}}}{\sqrt{n_1}} 
        =
        \bZ \begin{pmatrix}
                \frac{n_{12}}{\sqrt{n_1n_2}} \zeta_1 \tau_{e_1} \rho_e \\[4pt]
                \frac{n_{12}}{\sqrt{n_1n_2}} \zeta_1 \tau_{h_1} \\[4pt]
                0
            \end{pmatrix}
        + 
        \alpha_\perp \frac{\proj_{\bZ}^\perp \bu_{1,\cI_2}}{\|\proj_{\bZ}^\perp \bu_{1,\cI_2}\|_2},
\end{equation}
where we define
\begin{equation}
    \alpha^2 := \frac{n_{12}}{n_1}\zeta_1^2(\tau_{e_1}^2 + \tau_{h_1}^2) = n_{12} \zeta_1^2 \tau_{g_1}^2,
    \qquad 
    \alpha_\|^2 := \frac{n_{12}^2}{n_1n_2} \zeta_1^2 ( \tau_{e_1}^2 \rho_e^2 + \tau_{h_1}^2 ),
    \qquad 
    \alpha_\perp^2 := \alpha^2 - \alpha_\|^2.
\end{equation}
If $\proj_{\bZ}^\perp \bu_{1,\cI_2}=0$, we may replace $\frac{\proj_{\bZ}^\perp \bu_{1,\cI_2}}{\|\proj_{\bZ}^\perp \bu_{1,\cI_2}\|_2}$ with any unit vector orthogonal to the column space of $\bZ$, and all computations to come will remain valid.
Note that $\| \bu_{1,\cI_2}^{\mathrm{gs}} \|_2 / \sqrt{n_1} = \alpha$, $\| \proj_{\bZ} \bu_{1,\cI_2}^{\mathrm{gs}} \|_2 / \sqrt{n_1} = \alpha_\|$, and $\| \proj_{\bZ}^\perp \bu_{1,\cI_2}^{\mathrm{gs}}\|_2/\sqrt{n_1} = \alpha_\perp$. 
Our goal is to show 
\begin{equation}\label{eq:gs-u-diff}
    \frac{\| \bu_{1,\cI_2}^{\mathrm{gs}} - \bu_{1,\cI_2}\|_2}{\sqrt{n_1}} < C'\epsilon.
\end{equation}
Because there is no $\cPmodel,\cPregr$, and regression method dependent constant $c > 0$ such that $\alpha > c$, we cannot apply Lemma \ref{lem:gs-approx} to bound this quantity. (Indeed, under assumption \textsf{A1} and $\textsf{A2}$ we can even have $n_12 = 0$, in which case $\alpha = 0$). 
Thus, to bound this difference, we require a more careful argument.

First observe\footnote{This part of the argument is essentially the same as in the proof of Lemma \ref{lem:gs-approx}.}
\begin{equation}
\begin{aligned}
    &\frac{\| \proj_{\bZ} \bu_{1,\cI_2}^{\mathrm{gs}} - \proj_{\bZ} \bu_{1,\cI_2} \|_2}{\sqrt{n_1}}
        = 
        \left\|
            \begin{pmatrix}
                \frac{\zeta_1 \tau_{e_1} \rho_e n_{12}}{\sqrt{n_1n_2}} \\[4pt]
                \frac{\zeta_1 n_{12}}{\sqrt{n_1n_2}} \\[4pt]
                0
            \end{pmatrix}
            -
            \frac{\bZ^\top \bu_{1,\cI_2}}{\sqrt{n_1}} 
        \right\|_2
    \\
        &\qquad\qquad\leq 
        \left\|
            \begin{pmatrix}
                \frac{\zeta_1 \tau_{e_1} \rho_e n_{12}}{\sqrt{n_1n_2}} \\[4pt]
                \frac{\zeta_1 n_{12}}{\sqrt{n_1n_2}} \\[4pt]
                0
            \end{pmatrix}
            -
            \begin{pmatrix}
                \frac{\< \be_{2,\cI_2} , \bu_{1,\cI_2} \>}{\sqrt{n_1n_2}\,\tau_{e_2}} \\[4pt]
                \frac{\< \hat \bxi_{h,\cI_2} , \bu_{1,\cI_2} \>}{\sqrt{n_1n_2}} \\[4pt]
                \frac{\< \bxi_{h,\cI_2} , \bu_{1,\cI_2} \>}{\sqrt{n_1n_2}}
            \end{pmatrix}
        \right\|_2
        +
        \left\|
            \begin{pmatrix}
                \frac{\< \be_{2,\cI_2} , \bu_{1,\cI_2} \>}{\sqrt{n_1n_2}\,\tau_{e_2}} \\[4pt]
                \frac{\< \hat \bxi_{h,\cI_2} , \bu_{1,\cI_2} \>}{\sqrt{n_1n_2}} \\[4pt]
                \frac{\< \bxi_{h,\cI_2} , \bu_{1,\cI_2} \>}{\sqrt{n_1n_2}}
            \end{pmatrix}
            -
            \begin{pmatrix}
                \frac{\< \be_{2,\cI_2}^{\mathrm{gs}} , \bu_{1,\cI_2} \>}{\sqrt{n_1n_2}\,\tau_{e_2}} \\[4pt]
                \frac{\< \hat \bxi_{h,\cI_2}^{\mathrm{gs}} , \bu_{1,\cI_2} \>}{\sqrt{n_1n_2}} \\[4pt]
                \frac{\< \bxi_{h,\cI_2}^{\mathrm{gs}} , \bu_{1,\cI_2} \>}{\sqrt{n_1n_2}}
            \end{pmatrix}
        \right\|_2
        < C'\epsilon,
\end{aligned}
\end{equation}
where we bound the first term on the right-hand side by Eqs.~\eqref{eq:cEgam} and \eqref{eq:T-cond-ref},
and we bound the second term by Eq.~\eqref{eq:gs-diff}.

Thus, to show $\| \bu_{1,\cI_2}^{\mathrm{gs}} - \bu_{1,\cI_2}\|_2 / \sqrt{n_1} < C'\epsilon$, it now suffices to show $\big| \alpha_\perp - \| \proj_{\bZ}^\perp \bu_{1,\cI_2}\|_2/\sqrt{n_1} \big| < C'\epsilon $.
The previous display implies that $\big| \alpha_\| - \| \proj_{\bZ} \bu_{1,\cI_2}\|_2/\sqrt{n_1} \big| < C'\epsilon $.
Because, by Eq.~\eqref{eq:cEgam}, we also have
\begin{equation}
    \left|
        \left\|
            \begin{pmatrix}
                \alpha_\| \\[4pt] \alpha_\perp
            \end{pmatrix}
        \right\|_2
        -
        \left\|
            \begin{pmatrix}
                \| \proj_{\bZ} \bu_{1,\cI_2} \|_2 / \sqrt{n_1} \\[4pt] 
                \| \proj_{\bZ}^\perp \bu_{1,\cI_2} \|_2 / \sqrt{n_1}
            \end{pmatrix}
        \right\|_2
    \right|
        =
        \left|
            \alpha - \frac{\|\bu_{1,\cI_2}\|_2}{\sqrt{n_1}}
        \right|
        \leq 
        C'\epsilon,
\end{equation}
we conclude
\begin{equation}
    \left|
        \left\|
            \begin{pmatrix}
                \alpha_\| \\[4pt] \alpha_\perp
            \end{pmatrix}
        \right\|_2
        -
        \left\|
            \begin{pmatrix}
                \alpha_\| \\[4pt] 
                \| \proj_{\bZ}^\perp \bu_{1,\cI_2} \|_2 / \sqrt{n_1}
            \end{pmatrix}
        \right\|_2
    \right| < C'\epsilon.
\end{equation}
Thus, 
\begin{equation}
    \left|
        \sqrt{\alpha_\|^2 + \Big(\alpha_\perp + \frac{\| \proj_{\bZ}^\perp \bu_{1,\cI_2} \|_2}{\sqrt{n_1}} - \alpha_\perp\Big)^2}
        -
        \sqrt{\alpha_\|^2 + \alpha_\perp^2}
    \right|
        \leq 
        C'\epsilon.
\end{equation}
It is straightforward to check that for $\alpha_\perp/\alpha_\| > c'$,
the function $f(x) \mapsto \sqrt{\alpha_\|^2 + (\alpha_\perp + x)^2}$ satisfies $|f(x) - f(0)| > c'' |x|$ for $x \in [-\alpha_\perp,\infty)$, where $c'' > 0$ depends only on $c'$.
By Lemma \ref{lem:bound-on-fixed-pt},
\begin{equation}
    \frac{\alpha_\|}{\alpha} \leq \sqrt{\frac{n_{12}}{n_2}} \sqrt{\frac{\tau_{e_1}^2 \rho_e^2 + \tau_{h_1}^2}{\tau_{e_1}^2 + \tau_{h_1}^2}} < 1 - c
    \quad \text{whence} \quad 
    \frac{\alpha_\perp}{\alpha_\|} \geq c'.
\end{equation}
Thus, we conclude that $c'' \big| \alpha_\perp - \| \proj_{\bZ}^\perp \bu_{1,\cI_2}\|_2/\sqrt{n_1} \big| < C'\epsilon$, as desired.
Eq.~\eqref{eq:gs-u-diff} follows.

Finally, by Eq.~\eqref{eq:cEgam},
we can replace $\bu_{1,\cI_2^c}$ by $\bu_{1,\cI_2^c}^{\mathrm{gs}}$ so that $\| \bu_{1,\cI_2^c}^{\mathrm{gs}} \|_2 / \sqrt{n_1} = \sqrt{n_1-n_{12}}\,\tau_{g_1}\zeta_1$ and $\| \bu_{1,\cI_2^c} - \bu_{1,\cI_2^c}^{\mathrm{gs}} \|_2 / \sqrt{n_1} < \epsilon $.
We define $\bu_1^{\mathrm{gs}}$ to be $\bu_{1,\cI_2}^{\mathrm{gs}}$ on $\cI_2$ and $\bu_{1,\cI_2^c}^{\mathrm{gs}}$ on $\cI_2^c$.
The above discussion implies that $\| \bu_1^{\mathrm{gs}} - \bu_1 \|_2 / \sqrt{n_1} < C'\epsilon$. 
Because $\| \bu_1^{\mathrm{gs}}\|_2  / \sqrt{n_1} = \sqrt{n_1} \zeta_1 \tau_{g_1} > c$,
we conclude that also $\| \bu_1^{\mathrm{gs}} / \| \bu_1^{\mathrm{gs}} \|_2 - \bu_1 / \| \bu_1 \|_2 \|_2 < C'\epsilon$. 
Our Gram-Schmidt constructions are now complete.

Define
\begin{equation}\label{eq:m-theta-1-def}
\begin{aligned}
    \ell_{2|1}^{(1)}(\bu,\bv)
        &:=
        -\frac1{n_2}\<\bg_{\mathrm{cg}}^{(1)}(\bu),\bv\> 
            + 
            \frac1{n_2}
            \<\bh_{\mathrm{cg},\cI_2}^{(1)}(\bv),\bu_{\cI_2}\>
            + \frac1{n_2}\bu_{\cI_2}^\top \be_{2,\cI_2}^{\mathrm{gs}}
            - \frac1{2{n_2}} \| \bu_{ \cI_2 } \|_2^2 
            + \bar \Omega_2(\bv),
\end{aligned}
\end{equation}
where
\begin{equation}\label{eq:gtheta-u-htheta-v}
\begin{gathered}
        \bg_{\mathrm{cg}}^{(1)}(\bu)
            :=
            \frac{\<  \bu_1^{\mathrm{gs}},\bu \>}{\| \bu_1^{\mathrm{gs}}\|_2}\hat \bxi_g +  \| \sfP_{ \bu_1^{\mathrm{gs}}}^\perp \bu\|_2 \bxi_g,
        \qquad 
        \bh_{\mathrm{cg},\cI_2}^{(1)}(\bv)
            :=
            \frac{\<  \bv_1 , \bv \> }{\|  \bv_1 \|_2} \hat \bxi_{h,\cI_2}^{\mathrm{gs}} + \| \sfP_{ \bv_1}^\perp \bv\|_2 \bxi_{h,\cI_2}^{\mathrm{gs}}.
\end{gathered}
\end{equation}
Observe
\begin{equation}
\begin{aligned}
    \frac1{n_2}(\bg_{\mathrm{cg}}^{(1)}(\bu) - \bg_{\mathrm{cg}}(\bu))
        =
        \frac{\hat \bxi_g}{\sqrt{n_2}}
        \Big\<
            \frac{\bu_1^{\mathrm{gs}}}{\|\bu_1^{\mathrm{gs}}\|_2}
            -
            \frac{\bu_1}{\|\bu_1\|_2}
            ,
            \frac{\bu}{\sqrt{n_2}}
        \Big\>
        + 
        \frac{\bxi_g}{\sqrt{n_2}}
        \Big(
            \frac{\| \sfP_{ \bu_1^{\mathrm{gs}}}^\perp \bu\|_2}{\sqrt{n_2}}
            -
            \frac{\| \sfP_{ \bu_1}^\perp \bu\|_2}{\sqrt{n_2}}
        \Big)
        +
        \frac{\proj_{\bv_1}\bxi_g}{\sqrt{n_2}}\frac{\| \sfP_{ \bu_1}^\perp \bu\|_2}{\sqrt{n_2}}.
\end{aligned}
\end{equation}
By Eqs.~\eqref{eq:cEgam} and \eqref{eq:T-cond-ref},
we have $\| \hat \bxi_g \|_2 / \sqrt{n_2} \leq C \sqrt{p/n_2} \leq C$, $\| \bxi_g \|_2 \leq C\sqrt{p/n_2} \leq C$, and $\|\proj_{\bv_1}\bxi_g\|/\sqrt{n_2} \leq C'\epsilon$.
By the construction of $\bu_1^{\mathrm{gs}}$ and the discussion above, 
we have
$\| \proj_{\bu_1}^\perp - \proj_{\bu_1^{\mathrm{gs}}}^\perp\|_{\mathrm{op}} \leq \sqrt{2}\big\|\bu_1 / \| \bu_1\|_2 - \bu_1^{\mathrm{gs}} / \| \bu_1^{\mathrm{gs}} \|_2 \big\|_2 \leq C\epsilon$,
so that for $\| \bu \|_2 / \sqrt{n_2} \leq C_u$ we have $\| \bg_{\mathrm{cg}}^{(1)}(\bu) - \bg_{\mathrm{cg}}(\bu) \|_2 / \sqrt{n_2} < C' \,\epsilon$.

Similarly,
\begin{equation}
\begin{aligned}
    \frac1{\sqrt{n_2}}(\bh_{\mathrm{cg}}^{(1)}(\bu) - \bh_{\mathrm{cg}}(\bu))
        =
        \frac{\<\bv_1,\bv\>}{\|\bv_1\|_2}
        \cdot
        \frac{\hat \bxi_h^{\mathrm{gs}}-\hat \bxi_h}{\sqrt{n_2}}
        + 
        \| \proj_{\bv_1}^\perp \bv \|_2
        \cdot 
        \frac{\bxi_h^{\mathrm{gs}}-\bxi_h}{\sqrt{n_2}}
        +
        \| \proj_{\bv_1}^\perp \bv \|_2
        \frac{\proj_{\bu_1}\bxi_h}{\sqrt{n_2}}.
\end{aligned}
\end{equation}
By Eqs.~\eqref{eq:cEgam} and \eqref{eq:T-cond-ref},
we have $\| \hat \bxi_h^{\mathrm{gs}} - \hat \bxi_h \|_2 / \sqrt{n_2} \leq C$, $\| \bxi_h^{\mathrm{gs}} -  \bxi_h \|_2 / \sqrt{n_2} \leq C$, and $\| \proj_{\bu_1} \bxi_h \|_2/\sqrt{n_2} \leq C'\epsilon$.
Thus, for $\| \bv \|_2 \leq C \sqrt{p/n_2}$ we have $\| \bh_{\mathrm{cg}}^{(1)}(\bu) - \bh_{\mathrm{cg}}(\bu) \|_2 / \sqrt{n_2} < C'\epsilon$.
Finally, by Eq.~\eqref{eq:gs-diff},
$\big|\bu_{\cI_2}^\top \be_{2,\cI_2}^{\mathrm{gs}} - \bu_{\cI_2}^\top \be_{2,\cI_2}\big|/n_2 \leq C'\epsilon$ for $\| \bu \|_2/\sqrt{n_2} < C_u$.
Combining the preceding results,
we conclude
\begin{equation}\label{eq:mO1-approx-mO}
    \sup_{
        \substack{\bu \in \cS_{u|1}  }}\;
        \sup_{\|\bv\|_2 \leq C_v }
    |\ell_{2|1}^{(1)}(\bu,\bv) - \ell_{2|1}(\bu,\bv)| < \epsilon.    
\end{equation}

Define
\begin{equation}
    \bb(\bv)
        := 
        \begin{pmatrix}
            \tau_{e_2} \\[3pt] 
            \< \bv_1 , \bv \> / \| \bv_1 \|_2 \\[3pt]
            \| \proj_{\bv_1}^\perp \bv \|_2 
        \end{pmatrix}.
\end{equation}
We may rewrite
\begin{equation}
    \ell_{2|1}^{(1)}(\bu,\bv)
        := 
        \frac1{\sqrt{n_2}}\bu_{ \cI_2 }^\top \bZ\bb(\bv)
        -\frac1{n_2} \<\bg_{\mathrm{cg}}^{(1)}(\bu), \bv\>
        - \frac1{2{n_2}} \| \bu \|_2^2 
        + \bar \Omega_2(\bv).
\end{equation}
\\

\noindent \textit{Step 2: we construct a lower bound $\ell_{2|1}^{(2)}(\bv)$ on the internal maximization.}

To do so, we define $\ell_{2|1}^{(2)}(\bv) = \ell_{2|1}^{(1)}(\bu(\bv),\bv)$ for some $\bu(\bv) \in \cS_u$.
To motivate our choice of $\bu(\bv)$, 
recall that in the fixed-design model (see Eq.~\eqref{eq:res-ref})
\begin{equation}
\begin{gathered}
    \frac{1}{n_2}\E[\| \bu_2^f \|_2^2] 
        = n_2 \zeta_2^2 \tau_{g_2}^2,
    \\
    \frac{\E[\< \bu_{1,\cI_2}^f,\bu_{2,\cI_2}^f\>]}{\E[\|\bu_{1,\cI_2}^f\|_2^2]^{1/2} \sqrt{n_2}}
        = 
        \frac{n_{12}\zeta_1\zeta_2\tau_{\hat e_1^{\de}}\tau_{\hat e_2^{\de}} \rho_{\hat e^{\de}}}{\sqrt{n_{12}n_2}\zeta_1\tau_{\hat e_1^{\de}}}
        =
        \sqrt{\frac{n_{12}}{n_2}} \zeta_2 \tau_{\hat e_2^{\de}} \rho_{\hat e^{\de}}
        =
        \sqrt{n_{12}} \zeta_2 \tau_{g_2}\rho_{\hat e^{\de}},
\end{gathered}
\end{equation} 
where we use that $\tau_{\hat e_2^{\de}} = \sqrt{n_2} \tau_{g_2}$ (see Eq.~\eqref{eq:tau-g-ref}).
We choose $\bu(\bv)$ to be maximally aligned with $\bZ\bb(\bv)$ subject to the empirical counterpart of these equations holding with equality and $\bu(\bv)$ having support restricted to $\cI_2$.
Namely,
we maximize $\bu_{\cI_2}(\bv)^\top \bZ \bb(\bv)$ subject to
$\| \bu(\bv)\|_2^2/n_2 = n_2 \zeta_2^2  \tau_{g_2}^2$,
$\< \bu_{1,\cI_2}^{\mathrm{gs}}/\| \bu_{1,\cI_2}^{\mathrm{gs}} \|_2, \bu_{\cI_2}(\bv)/\sqrt{n_2} \> = \sqrt{n_{12}}\,\zeta_2\tau_{g_2}\rho_{\hat e^{\de}}$, and $ \bu_{\cI_2^c}(\bv) = 0$.
This gives
\begin{equation}
    \frac{\bu_{\cI_2}(\bv)}{\sqrt{n_2}}
        :=
        \sqrt{n_{12}}\,\zeta_2 \tau_{g_2}\rho_{\hat e^{\de}} \frac{\bu_{1,\cI_2}^{\mathrm{gs}}}{\| \bu_{1,\cI_2}^{\mathrm{gs}} \|_2}
        +
        \sqrt{n_2 \zeta_2^2 \tau_{g_2}^2 - n_{12}\zeta_2^2 \tau_{g_2}^2\rho_{\hat e^{\de}}^2} \, \frac{\proj_{\bu_{1,\cI_2}^{\mathrm{gs}}}^\perp \bZ\bb(\bv)}{\|\proj_{\bu_{1,\cI_2}^{\mathrm{gs}}}^\perp \bZ\bb(\bv)\|_2},
\end{equation} 
where we adopt the convention that $\bzero / \| \bzero \|_2 = \bzero$.
The projection $\proj_{\bu_{1,\cI_2}^{\mathrm{gs}}}^\perp$ is a projection in $\reals^{n_2}$ not in $\reals^N$.

We define $\ell_{2|1}^{(2)}(\bv) = \ell_{2|1}^{(1)}(\bu(\bv),\bv)$.
By design, $\| \bu(\bv)\|_2/\sqrt{n_2} = \sqrt{n_2}\,\zeta_2 \tau_{g_2} < C$ by Lemma \ref{lem:bound-on-fixed-pt}.
We immediately get that $\ell_{2|1}^{(2)}(\bv)$ is a lower-bound on the inner maximization:
\begin{equation}\label{eq:mO1-to-mO2-v}
    \min_{\substack{\bv \in E_{v|1}^c(\Delta) \\ \| \bv \|_2 \leq R } }\;
    \max_{\bu \in \cS_{u|1}  }
        \ell_{2|1}^{(1)}(\bu,\bv)
        \geq 
        \min_{\substack{\bv \in E(\Delta) \\ \| \bv \|_2 \leq R } }
            \ell_{2|1}^{(2)}(\bv).
\end{equation}

We now write $\ell_{2|1}^{(2)}(\bv)$ explicitly.
One simplification arising from our choice of $\bu(\bv)$ (and a primary reason for our choice) is that $\bg_{\mathrm{cg}}^{(1)}(\bu(\bv))$ and $\|\bu(\bv)\|_2^2/n_2$ do not depend on $\bv$, so that the term $\< \bg_{\mathrm{cg}}^{(1)}(\bu(\bv)) , \bv\>$ has a straightforward linear dependence on $\bv$.
In particular, 
due the Gram-Schmidt construction, we have $\|\bu_{1,\cI_2}^{\mathrm{gs}}\|_2 / \|\bu_1^{\mathrm{gs}}\|_2 =  \sqrt{n_{12}/n_1}$, whence
\begin{equation}
    \frac1{n_2} \frac{\< \bu_1^{\mathrm{gs}} , \bu(\bv) \>}{\| \bu_1^{\mathrm{gs}} \|_2}
    = 
    \frac1{\sqrt{n_2}} \frac{\< \bu_{1,\cI_2}^{\mathrm{gs}} , \bu_{\cI_2}(\bv) \> }{ \| \bu_{1,\cI_2}^\mathrm{gs} \|_2 \sqrt{n_2} } \frac{\| \bu_{1,\cI_2}^{\mathrm{gs}}\|_2 }{ \| \bu_1^{\mathrm{gs}} \|_2 }
    = \frac1{\sqrt{n_2}} \sqrt{n_{12}} \zeta_2 \tau_{g_2} \rho_{\hat e^{\de}} \sqrt{\frac{n_{12}}{n_1}}
    = \zeta_2 \tau_{g_2} \rho_g,
\end{equation}
where we have used that $\rho_g = (n_{12}/\sqrt{n_1n_2})\rho_{\hat e^{\de}}$ (see Eq.~\eqref{eq:tau-g-ref}).
Recalling the definition of the conditional fixed design model \eqref{eq:cf-quantities}, we have
\begin{equation}
\begin{gathered}
    \frac1{n_2^2}\| \bu(\bv)\|_2^2 = \zeta_2^2 \tau_{g_2}^2,
    \quad 
    \frac1{n_2} \| \proj_{\bu_1^{\mathrm{gs}}}^\perp \bu(\bv) \|_2
    = 
    \sqrt{ \zeta_2^2 \tau_{g_2}^2 - \| \proj_{\bu_1^{\mathrm{gs}}} \bu(\bv) \|_2^2/n_2^2 }
    =
    \sqrt{ \zeta_2^2 \tau_{g_2}^2 - \zeta_2^2 \tau_{g_2}^2 \rho_g^2 }
    =\zeta_2 \tau_{g_2}\rho_g^\perp,
    \\
    \frac1{n_2}
    \bg_{\mathrm{cg}}^{(1)}(\bu(\bv))
        = 
        \tau_{g_2}\zeta_2(\rho_g \hat \bxi_g + \rho_g^\perp \bxi_g)
        =
        \zeta_2 \bg_2^{cf}.
\end{gathered}
\end{equation}
Thus,
\begin{equation}
\begin{aligned}
    \ell_{2|1}^{(2)}(\bv)
        =
        \zeta_2 \sqrt{n_{12}}\,\tau_{g_2} \rho_{\hat e^{\de}}\frac{\<\bZ^\top\bu_{1,\cI_2}^{\mathrm{gs}},\bb(\bv)\>}{\|\bu_{1,\cI_2}^{\mathrm{gs}}\|_2}
        +
        \zeta_2\sqrt{n_2 \tau_{g_2}^2  - n_{12} \tau_{g_2}^2\rho_{\hat e^{\de}}^2} \,
        &\Big(
            \| \bZ \bb(\bv)\|_2^2 - \frac{\<\bZ^\top \bu_{1,\cI_2}^{\mathrm{gs}},\bb(\bv)\>^2}{\|\bu_{1,\cI_2}^{\mathrm{gs}}\|_2^2}
        \Big)^{1/2}
    \\
        &-\zeta_2 \bg_2^{cf\top}\bv
        - \frac{n_2 \tau_{g_2}^2\zeta_2^2}{2}
        + \bar \Omega_2(\bv).
\end{aligned}
\end{equation}
By the Gram-Schmidt construction (see Eq.~\eqref{eq:gs-u1} and recall $\| \bu_{1,\cI_2}^{\mathrm{gs}}\|_2/\sqrt{n_1} = \sqrt{n_{12}}\,\zeta_1 \tau_{g_1}$),
we have
\begin{equation}
    \ba 
        :=
        \frac{\bZ^\top \bu_{1,\cI_2}^{\mathrm{gs}}}{\|\bu_{1,\cI_2}^{\mathrm{gs}}\|_2}
        =
        \begin{pmatrix}
            \sqrt{\frac{n_{12}}{n_2}}\, \frac{\tau_{e_1}\rho_e}{\sqrt{n_1}\,\tau_{g_1}}\\[4pt]
            \sqrt{\frac{ n_{12} }{n_2}} \frac{\tau_{h_1}}{\sqrt{n_1}\,\tau_{g_1}}\\[4pt]
           0            
        \end{pmatrix}.
\end{equation}
Plugging this into the previous display and using that $\bZ^\top \bZ = \id_3$, 
we conclude that 
\begin{equation}
\begin{aligned}
    \ell_{2|1}^{(2)}(\bv)
        &= 
        \zeta_2\sqrt{n_{12}}\,\tau_{g_2}\rho_{\hat e^{\de}}\ba^\top \bb(\bv)
        +
        \zeta_2\sqrt{n_2 \tau_{g_2}^2  - n_{12} \tau_{g_2}^2\rho_{\hat e^{\de}}^2}\,
        \|\bb(\bv) \|_{\id - \ba\ba^\top}
        - 
        \zeta_2 \bg_2^{cf\top} \bv
        - \frac{n_2 \tau_{g_2}^2\zeta_2^2}{2}
        + \bar \Omega_2(\bv).
\end{aligned}
\end{equation}
Note that by Eq.~\eqref{eq:tau-g-ref}, $\| \ba \|_2^2 = (n_{12}/(n_1n_2))(\tau_{e_1}^2 \rho_e^2 + \tau_{h_1}^2)/\tau_{g_1}^2 \leq (1/n_1))(\tau_{e_1}^2 + \tau_{h_1}^2)/\tau_{g_1}^2  = 1$, so that $\| \cdot \|_{\id - \ba\ba^\top} $ is indeed a norm.
\\

\noindent \textit{Step 3: we establish several properties of the lower bound $\ell_{2|1}^{(2)}(\bv)$.}

First, we show $\ell_{2|1}^{(2)}(\bv_2^{cf})$ is close to $\ell_{2|1}^*$.
To do so, we will replace several quantities with those on which they concentrate.
In particular, we make replacements
\begin{equation}
\begin{gathered}
    \bg_2^{cf\top} \bv_2^{cf}
            \longrightarrow 
            \tau_{g_2}^2 \df_2,
    \qquad 
    \bar \Omega_2(\bv_2^{cf})
            \longrightarrow
            \omega_2.
\end{gathered}
\end{equation}
By Eqs.~\eqref{eq:cEgam} and \eqref{eq:T-cond-ref},
the term 
$\bg_2^{cf\top} \bv_2^{cf}$ differs from its replacement by at most $C'(p/n_2)\epsilon < C'\epsilon$.
The term $\bar \Omega_2(\bv_2^{cf})$ differs from its replacement by at most $C'(p/n_2)\epsilon < C'\epsilon$.
The coefficient of $\bg_2^{cf\top}\bv_2^{cf}$ is $\zeta_2 \leq 1$.
The coefficient of $\bar \Omega_2(\bv_2^{cf})$ is 1.
Thus, the errors incurred by making these replacements in the expression for $\ell_{2|1}^{(2)}(\bv_2^{cf})$ is at most $C'\epsilon$.

The quantity $\bb(\bv)$ involves the quantities $\< \bv_1,\bv_2^{cf}\>/\| \bv_1 \|_2$ and $\| \proj_{\bv_1}^\perp \bv_2^{cf}\|_2$.
We replace these by
\begin{equation}
\begin{gathered}
    \frac{\< \bv_1 , \bv_2^{cf} \>}{\| \bv_1 \|_2}
            \longrightarrow
            \tau_{h_2} \rho_h,
    \qquad
        \big\| \proj_{\bv_1}^\perp \bv_2^{cf} \big\|_2 
            \longrightarrow 
            \tau_{h_2}\rho_h^\perp.
\end{gathered}
\end{equation}
Recalling the definitions of $\ba$ and $\bb(\bv)$,
these replacements induce the replacements
\begin{equation}
\begin{gathered}
    \ba^\top \bb(\bv_2^{cf}) 
        \longrightarrow \sqrt{\frac{n_{12}}{n_2}} \frac{\tau_{e_1}\tau_{e_2}\rho_e + \tau_{h_1}\tau_{h_2}\rho_h}{\sqrt{n_1}\tau_{g_1}}
        =
        \frac{\sqrt{n_{12}}\, (\tau_{\hat e_1^\de}/\sqrt{n_1}) (\tau_{\hat e_2^\de}/\sqrt{n_2}) \rho_{\hat e^\de}}{\tau_{g_1}}
        = 
        \sqrt{n_{12}}\, \tau_{g_2} \rho_{\hat e^\de},
    \\
    \| \bb(\bv_2^{cf} )\|_2^2 
        \longrightarrow 
        \tau_{e_2}^2 + \tau_{h_2}^2 
        = n_2 \tau_{g_2}^2,
\end{gathered}
\end{equation}
where in the first line we have applied Eqs.~\eqref{eq:tau-g-ref} and \eqref{eq:tau-ed-ref}, and in the second line we have applied Eq.~\eqref{eq:tau-h-ref}.
By Eqs.~\eqref{eq:cEgam} and Eq.~\eqref{eq:T-cond-ref},
the term $\<\bv_1,\bv_2^{cf}\> / \| \bv_1 \|_2 = \sqrt{p/n_2}\, \< \sqrt{n_1/p}\,\bv_1,\sqrt{n_2/p}\,\bv_2^{cf}\>/\| \sqrt{n_1/p}\,\bv_1\|_2$ differs from its replacement by at most $C'\sqrt{p/n_2}\,\epsilon < C'\epsilon$,
where we use that the denominator $\| \sqrt{n_1/p}\,\bv_1\|_2$ concentrates on $\sqrt{n_1/p}\,\tau_{h_1} > c >0$.
The term $\big\| \proj_{\bv_1}^\perp \bv_2^{cf} \big\|_2 = (\| \bv_2^{cf} \|_2^2 - \<\bv_1,\bv_2^{cf}\>^2 / \| \bv_1 \|_2^2)^{1/2}$ differs from it replacement by at most $C'\sqrt{p/n_k}\,\epsilon < C'\epsilon$, where we use that $\rho_h^\perp > c$ to bound the derivative of the square-root (See Lemma \ref{lem:bound-on-fixed-pt}).
Using Lemma \ref{lem:bound-on-fixed-pt}, we can check that the etnries of $\ba$ are bounded above by $C$,
whence $\ba^\top \bb(\bv_2^{cf}) $ differs from its replacement by at most $C'\epsilon$.
Using Lemma \ref{lem:bound-on-fixed-pt} again, we have $c < n_2 \tau_{g_2}^2 < C$, whence $\| \bb(\bv_2^{cf} )\|_2^2 $ differs from its replacement by at most $C'\epsilon$.
These replacements induce the replacements
\begin{equation}
\begin{gathered}
    \zeta_2\sqrt{n_{12}}\,\tau_{g_2}\rho_{\hat e^{\de}}\ba^\top \bb(\bv_2^{cf})
        \longrightarrow
        \zeta_2n_{12}\tau_{g_2}^2 \rho_{\hat e^{\de}}^2,
    \\
    \| \bb(\bv_2^{cf}) \|_{\id - \ba\ba^\top}
        =
        \sqrt{\| \bb(\bv_2^{cf})\|_2^2 - (\ba^\top \bb(\bv_2^{cf}))^2}
        \longrightarrow \sqrt{n_2 \tau_{g_2}^2  - n_{12} \tau_{g_2}^2\rho_{\hat e^{\de}}^2}.
\end{gathered}
\end{equation}
By Lemma \ref{lem:bound-on-fixed-pt}, $\zeta_2\sqrt{n_{12}}\,\tau_{g_2}\rho_{\hat e^{\de}} \leq C$, whence the first of these replacements incurs an error of at most $C'\epsilon$.
By Lemma \ref{lem:bound-on-fixed-pt},
$n_2\tau_{g_2}^2 - n_{12} \tau_{g_2}^2 \rho_{\hat e^\de}^2 \geq c n_2 \tau_{g_2}^2 \geq c$,
whence the second of these repalcements incurs an error of at most $C'\epsilon$.

Applying the definition of $\ell_{2|1}^*$ (Eq.~\eqref{eq:cthe*-wthe*-def}) and $\zeta_2$ (Eq.~\eqref{eq:simul-fix-pt}) and performing some algebra,
we see that the replacements we have made for $\bg_2^{cf\top} \bv_2^{cf}$, $\bar \Omega_2(\bv_2^{cf})$, $\zeta_2\sqrt{n_{12}}\,\tau_{g_2}\rho_{\hat e^{\de}}\ba^\top \bb(\bv_2^{cf})$, and $\| \bb(\bv_2^{cf}) \|_{\id - \ba\ba^\top}$ induce the replacement of $\ell_{2|1}^{(2)}(\bv_2^{cf})$ by the quantity $\ell_{2|1}^*$. 
By the above discussion,
\begin{equation}\label{eq:mO3-vO*-approx}
        \big|
            \ell_{2|1}^{(2)}(\bv_2^{cf})
            - 
            \ell_{2|1}^*
        \big|
        < C'\,\epsilon.
\end{equation}

Second, we show $\partial \ell_{2|1}^{(2)}(\bv_2^{cf})$ contains a small element.
We compute
\begin{equation}\label{eq:mO3-grad-v}
\begin{aligned}
    \partial \ell_{2|1}^{(2)}(\bv_2^{cf})
        &=
        \sD \bb(\bv_2^{cf})^\top 
        \Big(
            \zeta_2\sqrt{n_{12}}\,\tau_{g_2}\rho_{\hat e^{\de}} \ba 
            +
            \zeta_2\sqrt{n_2 \tau_{g_2}^2  - n_{12} \tau_{g_2}^2\rho_{\hat e^{\de}}^2}
            \frac{(\id - \ba\ba^\top)\bb(\bv_2^{cf})}{\| \bb(\bv_2^{cf}) \|_{\id-\ba\ba^\top}}
        \Big)
        - \bg_2^{cf}
        + \partial \bar \Omega_2(\bv_2^{cf})
    \\
        &= 
        \sD \bb(\bv_2^{cf})^\top \ba
        \Big(
            \zeta_2\sqrt{n_{12}}\,\tau_{g_2}\rho_{\hat e^{\de}}
            -
            \zeta_2\sqrt{n_2 \tau_{g_2}^2  - n_{12} \tau_{g_2}^2\rho_{\hat e^{\de}}^2}
            \frac{\<\ba,\bb(\bv_2^{cf})\>}{\| \bb(\bv_2^{cf}) \|_{ \id - \ba\ba^\top }}
        \Big)
    \\
        &\qquad\qquad \qquad\qquad \qquad\qquad \qquad\qquad 
        +
        \frac{ \zeta_2\sqrt{n_2 \tau_{g_2}^2  - n_{12} \tau_{g_2}^2\rho_{\hat e^{\de}}^2} }{\| \bb(\bv_2^{cf}) \|_{ \id - \ba \ba^\top }} \bv_2^{cf}
        - \zeta_2 \bg_2^{cf}
        + \partial \bar \Omega_2(\bv_2^{cf}), 
\end{aligned}
\end{equation}
where $\sD\bb(\bv)$ is the Jacobian of $\bb(\bv)$,
given by
\begin{equation}
    \sD\bb(\bv)^\top
        =
        \begin{pmatrix}
            \bzero & \frac{\bv_1}{\|\bv_1\|_2} & \frac{\proj_{\bv_1}^\perp \bv}{\| \proj_{\bv_1}^\perp \bv \|_2}
        \end{pmatrix},
\end{equation}
and in the second equality we have used $\sD\bb(\bv_2^{cf})^\top \bb(\bv_2^{cf}) = \bv_2^{cf}$.
We make replacements
\begin{equation}
    \<\ba, \bb(\bv_2^{cf})\> \longrightarrow \sqrt{n_{12}}\,\tau_{g_2}\rho_{\hat e^{\de}},
    \qquad 
    \| \bb(\bv_2^{cf}) \|_{\id - \ba \ba^\top}
        \longrightarrow 
        \sqrt{n_2 \tau_{g_2}^2  - n_{12} \tau_{g_2}^2\rho_{\hat e^{\de}}^2},
\end{equation}
after which the first term becomes 0 and the subdifferential expression becomes
\begin{equation}
    \partial \ell_{2|1}^{(2)}(\bv_2^{cf})
        \longrightarrow 
        \zeta_2(\bv_2^{cf} - \bg_2^{cf}) 
        +
        \partial \bar \Omega_2(\bv_2^{cf}).
\end{equation}
The previous display is the subdifferential of the objective in Eq.~\eqref{eq:cond-fixed-design-est} evaluated at its minimizer $\bv_2^{cf}$. Thus, it contains $\bzero$.
As we have already shown, 
$\<\ba, \bb(\bv_2^{cf})\>$ and $ \| \bb(\bv_2^{cf}) \|_{\id - \ba \ba^\top}$ differ from their replacements by at most $C'\,\epsilon$. 
Our goal is to show that the subdifferential is also perturbed by at most $C'\epsilon$.
Indeed,
$\sD \bb(\bv_2^{cf})^\top \ba \leq \| \ba \|_2 = (n_{12}/n_2)(\tau_{e_1}^2 \rho_e + \tau_{h_1}^2) / (n_1 \tau_{g_1}^2) \leq C$ by Lemma \ref{lem:bound-on-fixed-pt}.
Also by Lemma \ref{lem:bound-on-fixed-pt}, 
$\zeta_2\sqrt{n_{12}}\,\tau_{g_2}\rho_{\hat e^{\de}} \leq \zeta_2 \sqrt{n_2}\tau_{g_2} \leq C$ and $c < (n_2 \tau_{g_2}^2  - n_{12} \tau_{g_2}^2\rho_{\hat e^{\de}}^2)^{1/2} < C$.
By Eqs.~\eqref{eq:cEgam} and \eqref{eq:T-cond-ref} and Lemma \ref{lem:bound-on-fixed-pt},
$\| \bv_2^{cf} \|_2 \leq C \sqrt{p/n_2} \leq C$.
Thus, because $\<\ba, \bb(\bv_2^{cf})\>$ concentrates on something bounded above by $C$, $\| \bb(\bv_2^{cf}) \|_{\id - \ba \ba^\top}$ concentrates on something bound below by $c$, and all relevant coefficients are bounded above by $C$,
the sub-differential set is perturbed in $\ell_2$ by the replacements above by at most $C'\epsilon$.
We conclude
\begin{equation}\label{eq:grad-small-v}
        \inf \big\{
            \| \bdelta \|_2
            :
            \bdelta \in \partial \ell_{2|1}^{(2)}(\bv_2^{cf}) 
        \big\}
        < C'\,\epsilon.
\end{equation}

Third, we show that $\ell_{2|1}^{(2)}(\bv)$ is $c$-strongly convex in a neighborhood of $\bv_2^{cf}$.
Because $\bar \Omega_2$ is convex, $\ba^\top \bb(\bv)$ is linear (because $a_3 = 0$), and $\bg_2^{cf\top}\bv$ is linear, 
it suffices to show $ \zeta_2(n_2 \tau_{g_2}^2  - n_{12} \tau_{g_2}^2\rho_{\hat e^{\de}}^2)^{1/2} \| \bb(\bv) \|_{\id - \ba \ba^\top}$ is strongly convex for $\bv$ bounded.
Because, as we have already justified, $ \zeta_2(n_2 \tau_{g_2}^2  - n_{12} \tau_{g_2}^2\rho_{\hat e^{\de}}^2)^{1/2} > c$, it suffices to show that $\big\| \bb(\bv) \big\|_{\id - \ba\ba^\top}$ is $c$-strongly convex in $\bv$.
For this
it is convenient to consider an alternative representation of the function $\| \bb(\bv) \|_{\id-\ba\ba^\top}$.
Let $\bU \in \reals^{p\times p}$ be an orthonormal matrix whose first column is $\bv_1 / \| \bv_1 \|_2$,
and let $\bv' = \bU^\top \bv$.
Thus, $\bv'$ can be interpreted as $\bv$ after an orthogonal change of basis,
and $\bb(\bv) = (\tau_{e_2} , b'_1 , \| \bv'_{-1} \|_2)^\top$.
Define $\breve\ba = (\ba^\top,\bzero_{p-2}^\top)^\top$, $\breve\bv = (\tau_{e_2},\bv'^\top)^\top$, and $\bK =  \id_{p+1} - \breve\ba\breve\ba^\top$.
It is straightforward to show that $\| \bb(\bv) \|_{\id - \ba\ba^\top} = \big\| \breve \bv \big\|_{\bK}$.

The Hessian of $ \big\| \breve\bv \big\|_{\bK}$ is 
\begin{equation}\label{eq:Hessian-in-b}
    \bH :=
    \frac{  1 }{\| \breve\bv \|_{\bK}}
    \bK^{1/2}
    \Big(
        \id_{p+1} - \frac{\bK^{1/2}\breve\bv\breve\bv^\top \bK^{1/2}}{\|\breve\bv\|_{\bK}^2}
    \Big)
    \bK^{1/2}.
\end{equation}
We want to lower bound $\bdelta'^\top \bH \bdelta'$ for all $\bdelta ' = (0,\bdelta^\top)^\top$ with $\bdelta \in \reals^p$, $\| \bdelta'\|_2 = 1$.
Because $n_{12}/n_2 \leq 1$ and
$n_1 \tau_{g_1}^2 = \tau_{e_1}^2 + \tau_{h_k}^2$ (see Eq.~\eqref{eq:tau-g-ref}),
\begin{equation}
    \| \breve \ba\|_2^2 
        = \|\ba\|_2^2 
        =
        \frac{n_{12}}{n_1n_2\tau_{g_1}^2}(\tau_{e_1}^2 \rho_e^2 + \tau_{h_1}^2)
        \leq \frac{\tau_{e_1}^2 \rho_e^2 + \tau_{h_1}^2}{ \tau_{e_1}^2 + \tau_{h_1}^2 }
        = 1 - \frac{\tau_{e_1}^2(1-\rho_e^{2})}{\tau_{e_1}^2 + \tau_{h_1}^2},
\end{equation}
whence
\begin{equation}
        \bK 
        \succeq
        \frac{\tau_{e_1}^2(1-\rho_e^{2})}{\tau_{e_1}^2 + \tau_{h_1}^2}\id_{p+1}
        \succeq
        c' \id_{p+1},
\end{equation}
where in the last step we use $c < \tau_{e_1}^2 < C$, $\tau_{h_1}^2 < C$, and $\rho_e^2 < 1-c$ by Lemma \ref{lem:bound-on-fixed-pt}.
Next define $\cos_{\bK}(\breve\bv,\bdelta') = \breve\bv^\top \bK \bdelta' / (\| \breve\bv \|_{\bK}\|\bdelta'\|_{\bK})$.
Straightforward algebra gives for $\| \bv \|_2 \leq C_v$ that
\begin{equation}
    \bdelta'^\top \bH \bdelta'
        = 
        \frac{ \|\bdelta'\|_{\bK}^2 }{\| \breve\bv \|_{\bK}}
            \Big(1 - \cos_{\bK}^2(\breve\bv,\bdelta')\Big)
        \geq 
        \frac{4c'\| \bdelta'\|_{\bK}^2}{\|\breve \bv\|_{\bK}}
        \Big( 
            1 - \frac{\< \bdelta' , \breve \bv \>^2}{\| \bdelta' \|_2^2 \| \breve \bv\|_2^2 }
        \Big)
        \geq c\Big( 
            1 - \frac{\< \bdelta' , \breve \bv \>^2}{\| \bdelta' \|_2^2 \| \breve \bv\|_2^2 }
        \Big),
\end{equation}
where the first inequality uses Lemma \ref{lem:cos-bound} and that
the condition number of $\bK$ is bounded by $1/c'$,
and the second inequality uses that $\| \bdelta' \|_2 = 1$, $\sigma_{\min}(\bK) \geq c$, $\sigma_{\max}(\bK) \leq 1$, and $\|\breve \bv\|_2 = \tau_{e_2}^2 + \| \bv'\|_2^2 = \tau_{e_2}^2 + \| \bv \|_2^2 \leq C$.
Further, recalling the definition of $\bdelta',\breve \bv$,
\begin{equation}
    \frac{\< \bdelta' , \breve \bv \>^2}{\| \bdelta' \|_2^2 \| \breve \bv\|_2^2 }
        =
        \frac{\< \bdelta , \bv' \>^2 }{ \tau_{e_2}^2 + \| \bv ' \|_2^2} 
        \leq 
        \frac{ \|\bv'\|_2^2 }{ \tau_{e_2}^2 + \| \bv ' \|_2^2} 
        =
        \frac{ \|\bv\|_2^2 }{ \tau_{e_2}^2 + \| \bv \|_2^2}.
\end{equation}
Using that $\tau_{e_2}^2 > c$, 
we have for $\| \bv \|_2 \leq C_v$ that $\< \bdelta' , \breve \bv \>^2/(\| \bdelta' \|_2^2 \| \breve \bv\|_2^2 ) \geq 1 - c$. 
Combining with the above bounds, 
we conclude that 
\begin{equation}
    \bdelta'^\top \bH \bdelta'
        \geq 
        c.
\end{equation}
Thus, we conclude that $\ell_{2|1}^{(2)}(\bv)$ is $c$-strongly convex for $\| \bv \|_2 \leq C_v$.
\\

\noindent \textit{Step 4: we use these properties to establish lower bounds on the min-max problem.}

Because $\ell_{2|1}^{(2)}(\bv)$ is $c$-strongly convex on $\| \bv \|_2 \leq C_v$,
we have for any $\| \bv \|_2 \leq C_v$ and $\bdelta \in \partial \ell_{2|1}^{(2)}(\bv_2^{cf})$ that
\begin{equation}
\begin{aligned}
    \ell_{2|1}^{(2)}(\bv)
        &\geq 
        \ell_{2|1}^{(2)}(\bv_2^{cf}) 
        + 
        \bdelta^\top(\bv - \bv_2^{cf}) 
        +
        \frac{c}2 \| \bv - \bv_2^{cf} \|_2^2
    \\
        &\geq 
        \ell_{2|1}^{(2)}(\bv_2^{cf}) 
        - 
        \frac{16}{c}\|\bdelta\|_2^2
        + 
        \frac{c}{4} \| \bv - \bv_2^{cf} \|_2^2.
\end{aligned}
\end{equation}
Combining the previous display with Eqs.~\eqref{eq:mO1-approx-mO}, \eqref{eq:mO1-to-mO2-v}, \eqref{eq:mO3-vO*-approx}, and \eqref{eq:grad-small-v},
we conclude that
\begin{equation}\label{eq:mO-min-max-upper-bound-v}
\begin{gathered}   
        \min_{\substack{\bv \in E_{v|1}^c(\Delta) \\ \| \bv \|_2 \leq C_v } }\;
        \max_{\bu \in \cS_{u|1}  }
                    \ell_{2|1}(\bu,\bv)
            \geq 
            \ell_{2|1}^* - C'\epsilon 
            +
            \min_{\substack{\bv \in E_{v|1}^c(\Delta) \\ \| \bv \|_2 \leq C_v } }\;
                \frac{c}{4} \| \bv - \bv_2^{cf}\|_2^2,
        \\
        \min_{\| \bv \|_2 \leq C_v } \;
        \max_{\bu \in \cS_{u|1}  }
                    \ell_{2|1}(\bu,\bv)
            \geq 
            \ell_{2|1}^* - C'\epsilon.
\end{gathered}
\end{equation} 
Further, by Eq.~\eqref{eq:cEgam},
\begin{equation}
        \big|
            \phi_v\big( \bv \big)
            -
            \E[\phi_v\big( \bv_2^{cf} \big)]
        \big|
        \leq 
        \frac{\Delta}{2}.
\end{equation}
By the definition of $E_{v|1}^c(\Delta)$ and because $\phi_v$ is $1$-Lipschitz, 
\begin{equation}
    \min_{\substack{\bv \in E_{v|1}^c(\Delta) \\ \| \bv \|_2 \leq C_v } }\;
                \frac{c}{4} \| \bv - \bv_2^{cf}\|_2^2 \geq c \frac{\Delta^2}{16}.
\end{equation}
Combining Eq.~\eqref{eq:mO-min-max-upper-bound-v} with the previous display and taking $\epsilon = c\Delta^2 / (32 C')$ (with the same values $c,C'$ appearing in Eq.~\eqref{eq:mO-min-max-upper-bound-v}),
we conclude the lower bound inside the first probability in Eq.~\eqref{eq:m0-lower-bound} holds.
The lower bound inside the second probability in Eq.~\eqref{eq:m0-lower-bound} is given in Eq.~\eqref{eq:mO-min-max-upper-bound-v}.

In summary, we have shown that the lower bounds inside the probabilities in Eq.~\eqref{eq:mO-min-max-upper-bound-v} hold on $\cG_{2|1}(\epsilon,\sqrt{32C'\epsilon/c})$ for $\epsilon < c'$. 
By Lemma \ref{lem:conditional-concentration-event}, we conclude that the probability bounds in Eq.~\eqref{eq:mO-min-max-upper-bound-v} hold as well.

\subsubsection{Upper bounds on the auxilliary max-min problem}
\label{app:conditional-characterization-residuals}

In this section, we prove Eq.~\eqref{eq:mthe-bound}.
As in the previous section, we will show that our upper bound holds on the event $\cG_{2|1}(\epsilon,\Delta)$ for $\epsilon < c'$, $\Delta = \Delta(\epsilon) < c'$, and $c'$ taken sufficiently small and depending only on $\cPmodel,\cPregr$, and the regression method.
We will henceforce assume we are on this event for $\epsilon,\Delta$ for $c'$ sufficiently small, without repeatedly reminding the reader of this fact,
and all statements will be deterministic.

The major steps in the analysis are as in the previous section, except that the first and second steps occur in the opposite order.
We now carry out these steps in detail. 
\\

\noindent \textit{Step 1: we construct an upper bound on the internal minimization.}

In particular, we evaluate $\ell_{2|1}(\bu,\bv)$ at $\bv = \bv_2^{cf}$.
Define
\begin{equation}
    \ell_{2|1}^{(1)}(\bu):=
        -\frac1{n_2}\<\bg_{\mathrm{cg}}(\bu),\bv_2^{cf}\>
            +
            \frac1{n_2}\<\bh_{\mathrm{cg}, \cI_2 }(\bv_2^{cf}),\bu_{ \cI_2 }\>
            +
            \frac1{n_2}\bu^\top \be_2
            - 
            \frac1{2{n_2}} \| \bu_{ \cI_2 } \|_2^2 
            + \bar \Omega_2(\bv_2^{cf}).
\end{equation}
Indeed, by Eqs.~\eqref{eq:cEgam} and \eqref{eq:T-cond-ref} and Lemma \ref{lem:bound-on-fixed-pt},
we have $\| \bv_2^{cf}\|_2 \leq C \sqrt{p/n_k} \leq C$, whence, for $C_v$ chosen sufficiently large,
\begin{equation}\label{eq:mthe1-to-mthe2-max-min}
        \min_{\|\bv\|_2 \leq C_v }
            \ell_{2|1}^{(1)}(\bu,\bv)
            \leq 
            \ell_{2|1}^{(1)}(\bu).
\end{equation}
\\

\noindent \textit{Step 2: we replace $\ell_{2|1}^{(1)}(\bu)$ by a function $\ell_{2|1}^{(2)}(\bu)$ which approximates it uniformly well across its domain.}

First, we remove the projections of $\bxi_g$ and $\bxi_h$.
That is, we replace the terms $\frac1{n_2}\| \proj_{\bu_1}^\perp \bu\|_2 \< \proj_{\bv_1}^\perp \bxi_g,\bv_2^{cf}\>$ and $\frac1{n_2} \| \proj_{\bv_1}^\perp \bv \|_2 \< (\proj_{\bu_1}^\perp \bxi_h)_{\cI_2} , \bu_{\cI_2}\> $ in the objective $\ell_{2|1}^{(1)}(\bu)$ by $\frac1{n_2}\| \proj_{\bu_1}^\perp \bu\|_2 \<  \bxi_g,\bv_2^{cf}\>$ and $\frac1{n_2} \| \proj_{\bv_1}^\perp \bv \|_2 \< \bxi_{h,\cI_2} , \bu_{\cI_2}\> $.
These replacements induce the replacements
\begin{equation}
\begin{gathered}
    \frac1{n_2}\<\bg_{\mathrm{cg}}(\bu),\bv_2^{cf}\>
        \longrightarrow 
        \frac{\< \bu_1,\bu\>}{\|\bu_1\|_2\sqrt{n_2}} \frac{\< \hat \bxi_g, \bv_2^{cf} \>}{\sqrt{n_2}}
        + 
        \frac{\| \proj_{\bu_1}^\perp\bu\|_2}{\sqrt{n_2}} \frac{\< \bxi_g, \bv_2^{cf} \>}{\sqrt{n_2}},
    \\
    \frac1{n_2}\<\bh_{\mathrm{cg},\cI_2}(\bu),\bu_{\cI_2}\>
        \longrightarrow 
        \frac{\< \bv_1,\bv\>}{\|\bv_1\|_2} \frac{\< \hat \bxi_{h,\cI_2}, \bu_{\cI_2} \>}{n_2}
        + 
        \| \proj_{\bv_1}^\perp\bv\|_2 \frac{\< \bxi_{h,\cI_2}, \bu_{\cI_2} \>}{n_2}.
\end{gathered}
\end{equation}
By Eqs.~\eqref{eq:cEgam} and \eqref{eq:T-cond-ref} and Lemma \ref{lem:bound-on-fixed-pt},
we have $\| \proj_{\bv_1}^\perp \bxi_g \|_2 / \sqrt{n_2} \leq C'\epsilon$, $\| \proj_{\bu_1}^\perp \bxi_h \|_2 / \sqrt{n_2} \leq C'\epsilon$, and $\|\bv_2^{cf}\|_2 \leq C'\epsilon$.
Thus, the value of the objective is perturbed by at most $C'\epsilon$ uniformly over
$\| \bu \|_2 /\sqrt{n_2} \leq C_u$ and $\| \bv \|_2 \leq C_v$ by this replacement.

Next, we make replacements
\begin{equation}
\begin{gathered}
    \frac{\hat \bxi_g^\top\bv_2^{cf}}{\sqrt{n_2}} 
        \longrightarrow 
        \frac{\tau_{g_2} \df_2}{\sqrt{n_2}} \rho_g,
    \qquad
    \frac{\bxi_g^\top\bv_2^{cf}}{\sqrt{n_2}} 
        \longrightarrow 
        \frac{\tau_{g_2} \df_2}{\sqrt{n_2}} \rho_g^\perp,
    \qquad   
    \frac{\<\bv_1, \bv_2^{cf} \>}{\|\bv_1\|_2}
        \longrightarrow 
        \tau_{h_2} \rho_h,
    \qquad
    \| \proj_{\bv_1}^\perp \bv_2^{cf} \|_2 
        \longrightarrow 
        \tau_{h_2} \rho_h^\perp,
\end{gathered}
\end{equation}
which induce the replacements (recall the definition of $\bh_2^{cf}$ in Eq.~\eqref{eq:cf-quantities})
\begin{equation}
\begin{gathered}
    \frac{\< \bu_1,\bu\>}{\|\bu_1\|_2\sqrt{n_2}} \frac{\< \hat \bxi_g, \bv_2^{cf} \>}{\sqrt{n_2}}
        + 
        \frac{\| \proj_{\bu_1}^\perp\bu\|_2}{\sqrt{n_2}} \frac{\< \bxi_g, \bv_2^{cf} \>}{\sqrt{n_2}}
        \longrightarrow 
        \frac{\tau_{g_2}\df_2}{n_2}        
        \Bigg( 
            \rho_g \, \frac{\<\bu_1 , \bu \>}{\|\bu_1 \|_2 } 
            + 
            \rho_g^\perp \| \proj_{\bu_1}^\perp \bu \|_2 
        \Bigg) 
    \\
    \frac{\< \bv_1,\bv\>}{\|\bv_1\|_2} \frac{\< \hat \bxi_{h,\cI_2}, \bu_{\cI_2} \>}{n_2}
        + 
        \| \proj_{\bv_1}^\perp\bv\|_2 \frac{\< \bxi_{h,\cI_2}, \bu_{\cI_2} \>}{n_2}
        \longrightarrow 
        \frac{\< \bh_{2,\cI_2}^{cf} , \bu_{\cI_2} \>}{n_2}.
\end{gathered}
\end{equation}
By Eqs.~\eqref{eq:cEgam} and \eqref{eq:T-cond-ref},
$\hat \bxi_g^\top \bv_2^{cf} / \sqrt{n_2}$ and $\bxi_g^\top \bv_2^{cf}/\sqrt{n_2}$ differ from their replacements by at most $C'(p/n_2)\epsilon\leq C'\epsilon$;
and the term $\<\bv_1,\bv_2^{cf}\>/\|\bv_1\|_2 = \sqrt{p/n_2}\< \sqrt{n_1/p}\,\bv_1,\sqrt{n_2/p}\,\bv_2^{cf}\>/\|\sqrt{n_1/p}\,\bv_1\|_2$ differs from its replacement by at most $C'\sqrt{p/n_2}\,\epsilon \leq C'\epsilon$,
where we use that the denominator $\|\sqrt{n_1/p}\,\bv_1\|_2$ concentrates on $\sqrt{n_1/p}\,\tau_{h_1} > c$ (see Lemma \ref{lem:bound-on-fixed-pt}).
Then $\| \proj_{\bv_1}^\perp \bv_2^{cf} \|_2  = \sqrt{p/n_2}((n_2/p)\| \bv_2^{cf} \|_2^2 - (n_2/p)\<\bv_1,\bv_2^{cf}\>^2/\|\bv_1\|_2^2 )^{1/2}$ differs from its replacement by at most $C'\sqrt{p/n_2}\,\epsilon$,
where we use that the quantity inside the square-root concentrates on $(n_2/p)\tau_{h_2}^2 \rho_h^{\perp2} \geq c$ (see Lemma \ref{lem:bound-on-fixed-pt}).
Further, the coefficients of these terms are bounded by $C$.
Indeed,
the coefficients of $\hat \bxi_g^\top \bv_2^{cf}/\sqrt{n_2}$ and $\bxi_g^\top \bv_2^{cf}/\sqrt{n_2}$ are $\<\bu_1,\bu\>/(\|\bu_1\|_2\sqrt{n_2})$ and  $\| \proj_{\bu_1}^\perp \bu\|_2 / \sqrt{n_2}$, which are bounded by $C_u$ when $\|\bu\|_2 / \sqrt{n_2} \leq C_u$. 
The coefficients of $\< \bv_1 , \bv_2^{cf} \> / \| \bv_1 \|_2$ and $\| \proj_{\bv_1}^\perp \bv_2^{cf} \|_2$ are $\hat \bxi_{h,\cI_2}^\top \bu_{ \cI_2 } / n_2$ and $\bxi_{h, \cI_2 }^\top \bu_{ \cI_2 }/ n_2$,
which are bounded by $C$ when $\|\bu\|_2 / \sqrt{n_2}\leq C_u$ because $\| \hat \bxi_{h,\cI_2} \|_2 / \sqrt{n_2} \leq 2$ and $\| \bxi_h \|_2 / \sqrt{n_2} \leq 2$ by Eqs.~\eqref{eq:cEgam} and \eqref{eq:T-cond-ref}.
Thus, the value of the objective is perturbed by at most $C'\epsilon$ uniformly over
$\| \bu \|_2 /\sqrt{n_2} \leq C_u$ and $\| \bv \|_2 \leq C_v$ by these replacements.

Finally, we make the replacement $\bar \Omega_2(\bv_2^{cf})\longrightarrow \omega_2$.
This combined with all previous replacements gives the objective
\begin{equation}\label{eq:mO3-def}
\begin{aligned}
    \ell_{2|1}^{(2)}(\bu)
        &:= 
        - 
        \frac{\tau_{g_2}\df_2}{n_2} 
        \Bigg( 
            \rho_g \, \frac{\<\bu_1 , \bu \>}{\|\bu_1 \|_2 } 
            + 
            \rho_g^\perp \| \proj_{\bu_1}^\perp \bu \|_2 
        \Bigg) 
        +
        \frac1{n_2} 
        \bu_{\cI_2}^\top 
        \big( 
            \be_{2,\cI_2} + \bh_{2,\cI_2}^{cf}
        \big)
        - 
        \frac12 \frac{\| \bu_{ \cI_2 } \|_2^2}{n_2} 
        + 
        \omega_2. 
\end{aligned}
\end{equation}
By Eq.~\eqref{eq:cEgam}, $\bar \Omega_2(\bv_2^{cf})$ differs from its replacement by at most $\epsilon$.
Together with the above discussion, we conclude
\begin{equation}\label{eq:mO3-approx-mO2}
        \sup_{ \substack{ \bu_{\cI_2} \in E_{u|1}^c(\Delta) \\ \bu \in \cS_{u|1}  }} |\ell_{2|1}^{(2)}(\bu) - \ell_{2|1}^{(1)}(\bu)| < C' \epsilon.
\end{equation}
\\

\noindent \textit{Step 3: we establish several properties of the upper bound $\ell_{2|1}^{(2)}(\bu)$.}

First, we show $\ell_{2|1}^{(2)}(\bu_2^{cf})$ is close to $\ell_{2|1}^*$.
Recalling that $\bu_{2,\cI_2}^{cf} = \zeta_2(\be_{2,\cI_2} + \bh_{2,\cI_2}^{cf})$,
we compute
\begin{equation}\label{eq:mO3-at-uO*}
\begin{aligned}
    \ell_{2|1}^{(2)}&(\bu_2^{cf})
        =
        -\frac{\tau_{g_2}\df_2}{n_2}
        \Bigg(
            \rho_g\,\frac{\< \bu_1, \bu_2^{cf} \>}{\|\bu_1 \|_2}  
            + 
            \rho_g^\perp \|\proj_{\bu_1}^\perp \bu_2^{cf}\|_2 
        \Bigg)
        +
        \frac{ \zeta_2 }{n_2}
        \big\| \be_{2,\cI_2} + \bh_{2,\cI_2}^{cf} \big\|_2^2
        - 
        \frac{\zeta_2^2}{2n_2} \big\| \be_{2,\cI_2} + \bh_{2,\cI_2}^{cf} \big\|_2^2
        + \omega_2.
\end{aligned}
\end{equation}
We make replacements
\begin{equation}
    \frac{\big\| \be_{2,\cI_2} + \bh_{2,\cI_2}^{cf} \big\|_2^2}{n_2}
        \longrightarrow 
        n_2\tau_{g_2}^2,
    \qquad 
    \frac{\< \bu_1 , \bu_2^{cf} \>}{\|\bu_1 \|_2 \sqrt{n_2}}
        \longrightarrow 
        \zeta_2\sqrt{n_2}\,\tau_{g_2}\rho_g,
    \qquad 
    \frac{\| \proj_{\bu_1}^\perp \bu_2^{cf} \|_2}{\sqrt{n_2}}
        \longrightarrow 
        \zeta_2 \sqrt{n_2}\,\tau_{g_2}\rho_g^\perp.
\end{equation} 
If we make these replacements, after some algebra the objective becomes $-\zeta_2\tau_{g_2}^2 \df_2 + n_2\zeta_2\tau_{g_2}^2 - \zeta_2^2 n_2 \tau_{g_2}^2/2 + \omega_2 = n_2 \tau_{g_2}^2 \zeta_2^2 + \omega_2 = \ell_{2|1}^*$, where we have used that $n_2 - \df_2 = n_2 \zeta_2$ by the fixed point equation \eqref{eq:marg-fix-pt-eq}.

We bound the errors incurred by these replacements.
By Eq.~\eqref{eq:cf-quantities}, we may write the term $\| \be_{2,\cI_2} + \bh_{2,\cI_2}^{cf} \|_2^2/n_2$ as $\| \be_{2,\cI_2} + \tau_{h_2}\rho_h \hat \bxi_{h,\cI_2} + \tau_{h_2} \rho_{h,\cI_2}^\perp \bxi_h \|_2^2/n_2$.
Expanding the square and using that $\tau_{h_2} \leq C \sqrt{p/n_2} \leq C$ by Lemma \ref{lem:bound-on-fixed-pt},
we have by Eqs.~\eqref{eq:cEgam} and \eqref{eq:T-cond-ref} that $\| \be_{2,\cI_2} + \tau_{h_2}\rho_h \hat \bxi_{h,\cI_2} + \tau_{h_2} \rho_{h,\cI_2}^\perp \bxi_h \|_2^2/n_2$ differs from $\tau_{e_2}^2 + \tau_{h_2}^2$ by at most $C'\epsilon$.
Recall that by Eq.~\eqref{eq:tau-g-ref}, $\tau_{e_2}^2 + \tau_{h_2}^2 = n_2 \tau_{g_2}^2$, whence $\| \be_{2,\cI_2} + \bh_{2,\cI_2}^{cf} \|_2^2/n_2$ differs from its replacement by at most $C'\epsilon$.
Because the coefficients of this term are $\zeta_2$ and $\zeta_2^2/2$, both of which are bound by $C$, 
the objective is perturbed by at most $C'\epsilon$ by this replacement.

We now turn to the replacements of $\< \bu_1 , \bu_2^{cf} \>/(\|\bu_1 \|_2 \sqrt{n_2})$ and $\| \proj_{\bu_1}^\perp \bu_2^{cf} \|_2/\sqrt{n_2}$.
Using the representation $\bu_{2,\cI_2}^{cf} = \zeta_2(\be_{2,\cI_2} + \tau_{h_2}\rho_h \hat \bxi_{h,\cI_2} + \tau_{h_2} \rho_h^\perp \bxi_{h,\cI_2})$,
we write $\< \bu_1,\bu_2^{cf}\>/\sqrt{n_1n_2} = \zeta_2\< \bu_{1,\cI_2}, \be_{2,\cI_2}\>/\sqrt{n_1n_2} + \zeta_2\tau_{h_2}\rho_h\< \bu_{1,\cI_2} , \hat \bxi_{h,\cI_2}\>/\sqrt{n_1n_2} + \zeta_2\tau_{h_2}\rho_h^\perp\< \bu_{1,\cI_2} , \bxi_{h,\cI_2}\>/\sqrt{n_1n_2}$.
By Eqs.~\eqref{eq:cEgam} and \eqref{eq:T-cond-ref} and using that $\tau_{h_2} \leq C\sqrt{p/n_2}$ (see Lemma \ref{lem:bound-on-fixed-pt}),
we conclude that $\< \bu_1,\bu_2^{cf}\>/\sqrt{n_1n_2}$ differs from $\zeta_1\zeta_2(n_{12}/\sqrt{n_1n_2})(\tau_{e_1}\tau_{e_2}\rho_e + \tau_{h_1}\tau_{h_2}\rho_h) $ by at most $C'\epsilon$.
Also by Eq.~\eqref{eq:cEgam}, $\| \bu_1 \|_2 / \sqrt{n_1}$ differs from $\zeta_1 \sqrt{n_1}\,\tau_{g_1}$ by at most $C'\epsilon$.
By Lemma \ref{lem:bound-on-fixed-pt}, $\zeta_1 \sqrt{n_1}\,\tau_{g_1} > c$ and $\zeta_1\zeta_2(n_{12}/\sqrt{n_1n_2})(\tau_{e_1}\tau_{e_2}\rho_e + \tau_{h_1}\tau_{h_2}\rho_h) \leq C$,
so that we may combine the previous bounds to conclude that $\< \bu_1, \bu_2^{cf} \> / (\| \bu_1 \|_2 \sqrt{n_2})$ differs $\zeta_1\zeta_2(n_{12}/\sqrt{n_1n_2})(\tau_{e_1}\tau_{e_2}\rho_e + \tau_{h_1}\tau_{h_2}\rho_h)/(\zeta_1 \sqrt{n_1}\,\tau_{g_1})$ by at most $C'\epsilon$.
Using Eq.~\eqref{eq:tau-g-ref},
we see that $\zeta_1\zeta_2(n_{12}/\sqrt{n_1n_2})(\tau_{e_1}\tau_{e_2}\rho_e + \tau_{h_1}\tau_{h_2}\rho_h)/(\zeta_1 \sqrt{n_1}\,\tau_{g_1}) = \zeta_2 \sqrt{n_2}\,\tau_{g_2}\rho_g$,
whence $\< \bu_1, \bu_2^{cf} \> / (\| \bu_1 \|_2 \sqrt{n_2})$ differs from its replacement by at most $C'\epsilon$.
Its coefficient is $\tau_{g_2} \df_2 / \sqrt{n_2} = \sqrt{n_2}\,\tau_{g_2} (1-\zeta_2) \leq C$, where the equality uses the fixed point equation \eqref{eq:marg-fix-pt-eq} and the inequality uses Lemma \ref{lem:bound-on-fixed-pt}. 
Thus, the objective is perturbed by at most $C'\epsilon$ by this replacement.

We write
$\| \proj_{\bu_1}^\perp \bu_2^{cf} \|_2/\sqrt{n_2} = (\| \bu_2^{cf}\|_2^2 / n_2 - \< \bu_1, \bu_2^{cf} \>^2 / (\| \bu_1 \|_2^2n_2) )^{1/2}$.
Because $\bu_{2,\cI_2}^{cf} = \zeta_2(\be_{2,\cI_2} + \bh_{2,\cI_2}^{cf})$ and $\bu_{2,\cI_2^c} = 0$,
from the above discussion we have that $\| \bu_2^{cf}\|_2^2 / n_2$ differs from $\zeta_2^2 n_2 \tau_{g_2}^2$ by at most $C'\epsilon$.
From the above bounds, we conclude that $\| \proj_{\bu_1}^\perp \bu_2^{cf} \|_2/\sqrt{n_2}$ differs from  its replacement by at most $C'\epsilon$,
where we use that the term inside the square-root is $\zeta_2^2 n_2 \tau_{g_2}^2 \rho_g^{\perp2} > c$ by Lemma \ref{lem:bound-on-fixed-pt}.

Combining all replacements, we conclude that 
\begin{equation}\label{eq:mO3-uO*-approx}
        \Big|
            \ell_{2|1}^{(2)}(\bu_2^{cf})
            - 
            \ell_{2|1}^*
        \Big|
        < C'\epsilon.
\end{equation}

Second, we show the gradient of $\ell_{2|1}^{(2)}(\bu)$ at $\bu_2^{cf}$ is small.
Because the coordinates $\bu_{ \cI_2 ^c}$ are constrained to be 0, we only compute the gradient for coordinates in $ \cI_2 $.
The gradient is
\begin{equation}\label{eq:mO3-grad}
\begin{aligned}
    \nabla & \ell_{2|1}^{(2)}(\bu_2^{cf})_{ \cI_2 }
        =
        -
        \frac{\tau_{g_2}\df_2}{n_2}
        \Bigg(\rho_g\frac{\bu_1}{\| \bu_1 \|_2} 
        + 
        \rho_g^\perp 
        \frac{\proj_{\bu_1}^\perp \bu_2^{cf}}{\| \proj_{\bu_1}^\perp \bu_2^{cf} \|_2} \Bigg)_{ \cI_2 }
        +
        \frac1{n_2}
        \big(
            \be_{2,\cI_2} + \bh_{2,\cI_2}^{cf}
        \big)
        - 
        \frac{\bu_{2, \cI_2 }^{cf}}{n_2}
    \\
        &=
        \underbrace{
        -
        \frac{\tau_{g_2}\df_2}{n_2}
        \Big(
            \rho_g - \rho_g^\perp \frac{\< \bu_1 , \bu_2^{cf}\> /\|\bu_1 \|_2}{ \| \proj_{\bu_1}^\perp \bu_2^{cf} \|_2 }
        \Big)\frac{ \bu_{1,\cI_2} }{ \| \bu_1 \|_2 } 
        }_{=:\bA_1}
        + 
        \underbrace{
        \Big(
            - 
            \frac{\tau_{g_2}\df_2}{\sqrt{n_2}}\rho_g^\perp \frac{1}{\| \proj_{\bu_1}^\perp \bu_2^{cf}\|_2 / \sqrt{n_2} }
            +
            \frac1{\zeta_2}
            -
            1
        \Big)
            \frac{\bu_{2,\cI_2}^{cf}}{ n_2 }
        }_{=:\bA_2},
\end{aligned}
\end{equation}
where we use that $\bu_{2,\cI_2}^{cf} = \zeta_2(\be_{2,\cI_2} + \bh_{2,\cI_2}^{cf})$.
If $ \cI_1 \cap  \cI_2  = \emptyset$, 
then $\bu_{1,\cI_2} = 0$ and 0 appears in the denominator in the previous display.
In this case we interpret $\bu_1 / \| \bu_1 \|_2 = 0$, $\proj_{\bu_{1,\cI_2}}^\perp = \id_{n_2}$.
Under this interpretation, all calculations to come will be correct.
Note that because we have already argued that $\| \proj_{\bu_1}^\perp \bu_2^{cf} \|_2/\sqrt{n_2}$ 
differs from $\zeta_2 \sqrt{n_2}\,\tau_{g_2} \rho_g^\perp > c$ by at most $C'\epsilon$,
we have that $\| \proj_{\bu_1}^\perp \bu_2^{cf}\|_2\neq0 $ for $\epsilon < c'$, so that no other terms in the previous display will cause an issue.

We bound
\begin{equation}
\begin{aligned}
    \| \bA_1 \|_2
    \leq 
        \frac1{\sqrt{n_2}}
        \sqrt{n_2}\,\tau_{g_2}
        \frac{\df_2}{n_2}
        \Bigg|
            \rho_g - \rho_g^\perp \frac{\< \bu_1 , \bu_2^{cf} \> / \|\bu_1 \|_2 }{\| \proj_{\bu_1}^\perp \bu_2^{cf}\|_2}
        \Bigg|
    \leq 
        \frac{C}{\sqrt{n_2}}
        \Bigg| 
            \rho_g - \rho_g^\perp \frac{ \< \bu_1 , \bu_2^{cf} \> / \| \bu_1 \|_2 }{\| \proj_{\bu_1}^\perp \bu_2^{cf}\|_2} 
        \Bigg|, 
\end{aligned}
\end{equation}
where the inequality uses Lemma \ref{lem:bound-on-fixed-pt} and that $\df_2 / n_2 = 1 - \zeta_2$ from the fixed point equation \eqref{eq:marg-fix-pt-eq}.
We make replacements $\< \bu_1 , \bu_2^{cf} \>/(\|\bu_1\|_2\sqrt{n_2}) \rightarrow  \zeta_2 \sqrt{n_2}\tau_{g_2}\rho_g$ 
and $\| \proj_{\bu_1}^\perp \bu_2^{cf}\|_2/\sqrt{n_2} \rightarrow \zeta_2 \sqrt{n_2} \tau_{g_2} \rho_g^\perp$,
after which the right-hand side becomes 0.
As we have already argued,
these terms differ by at most $C'\epsilon$ from their replacements.
Moreover, $\zeta_2 \sqrt{n_2}\tau_{g_2}\rho_g < C$ and $\zeta_2 \sqrt{n_2} \tau_{g_2} \rho_g^\perp > c$ by Lemma \ref{lem:bound-on-fixed-pt}.
We conclude
\begin{equation}\label{eq:grad-first-term-small}
        \big\|
            \bA_1
        \big\|_2
        < \frac{C'\epsilon}{\sqrt{n_2}}.
\end{equation}

If we replace $\| \proj_{\bu_1}^\perp \bu_2^{cf}\|_2/\sqrt{n_2}$ 
with $ \sqrt{n_2} \tau_{g_2} \zeta_2 \rho_g^\perp$ in the expression for $\bA_2$, 
we get the quantity $\big( -\df_2/(\zeta_2 n_2) + 1/\zeta_2 - 1 \big)\bu_{2,\cI_2}^{cf}/\sqrt{n_2} = 0$.
As we have already shown, $\| \bu_{2,\cI_2}^{cf} \|_2 / \sqrt{n_2} \leq C$.
Observe $\tau_{g_2} \df_2 / \sqrt{n_2}  = \sqrt{n_2}\,\tau_{g_2} (1-\zeta_2) \leq C$, where the equality uses the fixed point equation \eqref{eq:marg-fix-pt-eq} and the inequality uses Lemma \ref{lem:bound-on-fixed-pt}. 
Further, $\sqrt{n_2} \tau_{g_2} \zeta_2 \rho_g^\perp > c$.
Because $\| \proj_{\bu_1}^\perp \bu_2^{cf}\|_2/\sqrt{n_2}$ 
differs from $ \sqrt{n_2} \tau_{g_2} \zeta_2 \rho_g^\perp$ by at most $C'\epsilon$,
we conclude
\begin{equation}
        \big\|
            \bA_2
        \big\|_2 
        < \frac{C'\epsilon}{\sqrt{n_2}}.
\end{equation}
Combining the bounds on $\bA_1$ and $\bA_2$, we conclude
\begin{equation}\label{eq:grad-small}
    \big\|
        \nabla \ell_{2|1}^{(2)}(\bu_2^{cf})_{ \cI_2 }
    \big\|_2
        < \frac{C'\epsilon}{\sqrt{n_2}}.
\end{equation}

Third, we observe that by Eq.~\eqref{eq:mO3-def},
$\ell_{2|1}^{(2)}(\bu)$ is $1/n_2$-strongly concave everywhere.
\\

\noindent \textit{Step 4: we use these properties to establish upper bounds on the max-min problem.}

Because
$\ell_{2|1}^{(2)}(\bu)$ is $1/n_2$-strongly concave,
for any $\bu \in \reals^n$ with $\bu_{ \cI_2 ^c}=0$,
\begin{equation}\label{eq:mthe-taylor-bound}
\begin{aligned}
    \ell_{2|1}^{(2)}(\bu)
        &\leq 
        \ell_{2|1}^{(2)}(\bu_2^{cf}) 
        + 
        \nabla \ell_{2|1}^{(2)}(\bu_2^{cf})_{ \cI_2 }^\top(\bu_{ \cI_2 } - \bu_{2, \cI_2 }^{cf}) 
        - 
        \frac1{2n_2} \| \bu_{ \cI_2 } - \bu_{2, \cI_2 }^{cf} \|_2^2
    \\
        &\leq 
        \ell_{2|1}^{(2)}(\bu_2^{cf}) 
        + 
        n_2\cdot \|\nabla \ell_{2|1}^{(2)}(\bu_2^{cf})_{ \cI_2 }\|_2^2
        - 
        \frac1{4n_2} \| \bu_{ \cI_2 } - \bu_{2, \cI_2 }^{cf} \|_2^2.
\end{aligned}
\end{equation}
Combining the previous display with Eqs.~\eqref{eq:mO1-approx-mO}, \eqref{eq:mthe1-to-mthe2-max-min}, \eqref{eq:mO3-approx-mO2}, \eqref{eq:mO3-uO*-approx}, and \eqref{eq:grad-small},
we conclude that
\begin{equation}\label{eq:mO-min-max-upper-bound}
\begin{gathered}
        \max_{\mathclap{
            \substack{
                \bu_{ \cI_2 } \in E_{u|1}^c(\Delta)\\
                                \bu \in \cS_{u|1}  }}}
            \;\;\; \;\;                              
                    \min_{\|\bv\|_2 \leq C_v }
                    \ell_{2|1}(\bu,\bv)
            \leq 
            \ell_{2|1}^* + C'\epsilon 
            - 
            \;\;\;\;
            \min_{\mathclap{
                \substack{\bu_{ \cI_2 } \in E_{u|1}^c(\Delta)\\
                                \bu \in \cS_{u|1}  }}}
                \;\;\;\;
                \frac{\| \bu - \bu_2^{cf} \|_2^2}{4n_2},
    \\
        \max_{\bu \in \cS_{u|1} }                           
        \min_{\|\bv\|_2 \leq C_v }    
            \ell_{2|1}(\bu,\bv)
            \leq 
            \ell_{2|1}^* + C'\epsilon.
\end{gathered}
\end{equation} 
By Eq.~\eqref{eq:cEgam},
\begin{equation}
        \Big|
            \phi_u\Big( \frac{\bu_{2, \cI_2 }}{\sqrt{n_2}}  \Big)
            -
            \E\Big[\phi_u\Big( \frac{\bu_{2, \cI_2 }^{cf}}{\sqrt{n_2}}  \Big)\Big]
        \Big|
        \leq\frac{\Delta}{2}.
\end{equation}
By the definition of $E_{u|1}^c(\Delta)$ and because $\phi_u$ is $1$-Lipschitz, on this event 
\begin{equation}
    \min_{\substack{\bu_{ \cI_2 } \in E_{u|1}^c(\Delta)\\\bu \in \cS_{u|1}  }}\frac{\| \bu - \bu_2^{cf} \|_2^2}{4n_2} \geq \frac{\Delta^2}{16}.
\end{equation}
Combining Eq.~\eqref{eq:mO-min-max-upper-bound} with the previous display and taking $\epsilon = \Delta^2 / (32 C')$ (with the same value $C'$ appearing in Eq.~\eqref{eq:mO-min-max-upper-bound}),
we conclude the upper bound inside the probability in the first line of Eq.~\eqref{eq:mthe-bound} holds. 
The upper bound insdie the probability in the second line of Eq.~\eqref{eq:mO-min-max-upper-bound} is given in Eq.~\eqref{eq:mO-min-max-upper-bound}.

In summary, we have shown that the upper bounds inside the probabilities in Eq.~\eqref{eq:mO-min-max-upper-bound} hold on $\cG_{2|1}(\epsilon,\sqrt{32C'\epsilon})$ for $\epsilon < c'$. 
By Lemma \ref{lem:conditional-concentration-event}, we conclude that the probability bounds in Eq.~\eqref{eq:mO-min-max-upper-bound} hold as well.

\section{Concentration of empirical second moments}
\label{sec:marg-char-support}

We will frequently need to use the concentration of second moment statistics of the form $\< \ba^{(1)},\ba^{(2)}\>$ for random vectors $\ba^{(1)},\ba^{(2)}$ (possibly the same). 
For example, this will be required in the proofs of Lemmas \ref{lem:marginal-concentration-event} and \ref{lem:conditional-concentration-event}, but also elsewhere.
Although $\< \ba^{(1)},\ba^{(2)}\>$ is not globally Lipschitz, 
its concentration will follow from Lipschitz concentration results along the lines of the marginal concentration (Lemma \ref{lem:marginal-characterization}) and conditional characterization (Lemma \ref{lem:conditional-characterization}) or standard Gaussian concentration of Lipschitz functions.
Some care is needed to track the correct dependence on $p$ and $n_k$ when extending these results to $\< \ba^{(1)},\ba^{(2)}\>$.
The following lemma provides the required extension.

\begin{lemma}\label{lem:T-conc}
    Consider two random vectors $\ba^{(1)},\ba^{(2)} \in \reals^n$, constants $C,c> 0$ and $\sC(\epsilon),\sc(\epsilon): \reals_{>0} \rightarrow \reals_{>0}$, $p \geq 1$, $M,K > 0$, and $\bS \in \SS_+^2$, which satisfy the following concentration guarantee:

    \begin{itemize}

        \item 
        There exist $M$-Lipschitz functions $\bphi^{(1)},\bphi^{(2)}: \reals^n \rightarrow \reals$
        such that for Gaussian vectors $\bxi^{(1)},\bxi^{(2)} \sim \normal(0,\bS \otimes \id_n) $ and any $1$-Lipschitz functions $\phi: \reals^n\rightarrow \reals$, $\phi': (\reals^n)^2\rightarrow \reals$,
        \begin{equation}\label{eq:generic-concentration}
        \begin{gathered}
            \P\Big(
                \Big|
                    \phi(\ba^{(k)})
                    -
                    \E\big[
                        \phi\big(
                            \bphi^{(k)}(\bxi^{(k)})
                        \big)
                    \big]
                \Big| 
                > 
                C \epsilon
            \Big)
            \leq \sC(\epsilon)e^{-\sc(\epsilon)p}
            \;\; \text{for $k = 1,2$},
            \\
            \P\Big(
                \Big|
                    \phi'(\ba^{(1)},\ba^{(2)})
                    -
                    \E\big[
                        \phi'\big(
                            \bphi^{(1)}(\bxi^{(1)}),
                            \bphi^{(2)}(\bxi^{(2)})
                        \big)
                    \big]
                \Big| 
                > 
                K \epsilon
            \Big)
            \leq \sC(\epsilon)e^{-\sc(\epsilon)p},
        \end{gathered}
        \end{equation}  
        and
        \begin{equation}\label{eq:E-bound}
            \E[\|\bphi^{(i)}(\bxi^{(i)})\|_2^2] \leq C,\quad i = 1,2.
        \end{equation}

        \item The parameters above satisfy $M \leq CK/\sqrt{p}$ and the singular values of $\bS_\xi$ are bounded by $C$.

    \end{itemize}
    Then there exist $C',c'> 0$ and $\sC'(\epsilon),\sc'(\epsilon): \reals_{>0} \rightarrow \reals_{>0}$ depending only on $C,c> 0$ and $\sC(\epsilon),\sc(\epsilon): \reals_{>0} \rightarrow \reals_{>0}$
    such that for any $\epsilon < c'$ we have
    \begin{equation}
        \P\Big(
            \Big|
                \<\ba^{(1)},\ba^{(2)}\>
                -
                \E\big[
                    \big\<
                        \bphi^{(1)}(\bxi^{(1)}),
                        \bphi^{(2)}(\bxi^{(2)})
                    \big\>
                \big]
            \Big| 
            > 
            K \epsilon
        \Big)
        \leq \sC'(\epsilon)e^{-\sc'(\epsilon)p}.
    \end{equation}  
\end{lemma}

\noindent In all applications of Lemma \ref{lem:T-conc}, we will take $C,c,\sC(\epsilon),\sc(\epsilon)$ to be $\cPregr,\cPmodel$, and regression-method dependent constants, 
and $K,M$ to be constants which possibly depend also on $p/n_k$.

\begin{proof}[Proof of Lemma \ref{lem:T-conc}]
    The function $\<\ba^{(1)},\ba^{(2)}\>$ is $2C$-Lipschitz on $\{\| \ba^{(1)}\|_2,\| \ba^{(2)}\|_2 \leq 2C \}$. 
    Define $\overline \bt(\ba^{(1)},\ba^{(2)})$ to be a $2C$-Lipschitz extension of $\<\ba^{(1)},\ba^{(2)}\>$ which agrees with $\<\ba^{(1)},\ba^{(2)}\>$ on $\{\| \ba^{(1)}\|_2,\| \ba^{(2)}\|_2 \leq 2C\}$.
    By the first line of Eq.~\eqref{eq:generic-concentration},
    we have with probability at least $1 - 2Ce^{-cp}$ that 
    $\max\{\| \ba^{(1)} \|_2,\| \ba^{(1)} \|_2\} \leq 2C$.
    Thus, by the second line of Eq.~\eqref{eq:generic-concentration}, with probability at least $1 - \sC(\epsilon)e^{-\sc(\epsilon)p} - 2Ce^{-cp}$
    \begin{equation}
        \| \<\ba^{(1)},\ba^{(2)}\> - \E[\overline \bt(\bphi^{(1)},\bphi^{(2)})] \|_{\sF}
            =
            \| \<\ba^{(1)},\ba^{(2)}\> - \E[\overline \bt(\bphi^{(1)},\bphi^{(2)})] \|_{\sF}
            \leq 
            2CK \epsilon.
    \end{equation}
    Moreover, becuase $\overline \bt(0,0) = 0$ and is $2C$-Lipschitz, 
    we have $\big|\overline \bt(\bphi^{(1)},\bphi^{(2)})\big| \leq 2C (\|\bphi^{(1)}\|_2 + \|\bphi^{(2)}\|_2) $.
    By Cauchy-Schwartz, we have $|\< \bphi^{(1)},\bphi^{(2)}\>| \leq \| \bphi^{(1)} \|_2 \| \bphi^{(2)}\|_2$.
    Thus
    \begin{equation}
        \big| \overline \bt(\bphi^{(1)},\bphi^{(2)}) - \<\bphi^{(1)},\bphi^{(2)}\> \big|
            \leq 
            \big(2C(\| \bphi^{(1)} \|_2 + \|\bphi^{(2)} \|_2) + \| \bphi^{(1)} \|_2 \|\bphi^{(2)} \|_2\big)
            \indic{\max\{\| \bphi^{(1)} \|_2, \|\bphi^{(2)} \|_2\} > 2C}.
    \end{equation}
    By sub-Gaussian concentration, using Eq.~\eqref{eq:E-bound} and that the singular values of $\bS$ are bounded above by $C$,
    we have
    $\P(\| \bphi^{(i)} \|_2 > 2C + t) \leq e^{-(C+t)^2/(2CM^2)}$,
    whence
    \begin{equation}
    \begin{aligned}
        \big| \E\big[\overline \bt(\bphi^{(1)},\bphi^{(2)})] - \E\big[\<\bphi^{(1)},\bphi^{(2)}\> \big] \big|
            &\leq \int_0^\infty (4C(C+t) + (C+t)^2) e^{-(C+t)^2/(CM^2)}\de t
            \leq C' M
        \\
            &\leq \int_0^\infty C'(1+(t/M)^2) e^{- \big(2(t/M) + (t/M)^2\big)/C }\de t
            \leq C' M .
    \end{aligned}
    \end{equation}
    where $C'$ depends only on $C$ and changes at each appearance, and in the second inequality we have used $M \leq C$ to conclude that $(4C(C+t) + (C+t)^2) e^{-C/M^2} \leq C'(1+(t/M)^2)$ and $-2t/M^2 \leq - 2 t/(CM)$.
    We conclude that 
    \begin{equation}
        \P\Big(
            \Big|
                \<\ba^{(1)},\ba^{(2)}\>
                -
                \E\big[
                    \big\<
                        \bphi^{(1)}(\bxi^{(1)}),
                        \bphi^{(2)}(\bxi^{(2)})
                    \big\>
                \big]
            \Big| 
            > 
            C'M + K\epsilon
        \Big)
        \leq 
        \sC(\epsilon)e^{-\sc(\epsilon)p} + 2Ce^{-cp},
    \end{equation}
    For $K\epsilon \leq C'M$, we have $e^{1-C^2p\epsilon^2/(C'(M/K)\sqrt{p})^2} > 1$,
    whence
    \begin{equation}
        \P\Big(
            \Big|
                \<\ba^{(1)},\ba^{(2)}\>
                -
                \E\big[
                    \big\<
                        \bphi^{(1)}(\bxi^{(1)}),
                        \bphi^{(2)}(\bxi^{(2)})
                    \big\>
                \big]
            \Big| 
            > 
            2K\epsilon
        \Big)
        \leq 
        \sC(\epsilon)e^{-\sc(\epsilon)p} + 2Ce^{-cp} + e^{1-(K/M)^2\epsilon^2/C'^2},
    \end{equation}
    because the bound becomes trivial for $K\epsilon \leq C'M$ and is implied by the previous display for $K\epsilon > C'M$.
    Because $(K/M)^2 \geq p/C^2$, 
    the result follows.
\end{proof}

\section{Empirical degrees-of-freedom: proof of Theorem \ref{thm:joint-characterization}(iii) and Lemma \ref{lem:marginal-characterization}(ii)}
\label{sec:emp-df}

We extend Theorem \ref{thm:joint-characterization}(i) and (ii) to the case that $\hat \btheta_k^\de$ and $\hat \be_k^\de$ are computed using $\hat \df_k$ in place of $\df_k$ (recall $\hat \df_k$ defined in Eq.~\eqref{eq:hat-df}),
and make the similar extension for Lemma \ref{lem:marginal-characterization}(i).
To distinguish between $\hat \btheta_k^\de,\hat \be_k^\de$ which are computed using $\hat \df_k$ from those computed using $\df_k$,
in this section we denote by $\hat \btheta_k^{\de,\mathrm{emp}},\hat \be_{k}^{\de,\mathrm{emp}}$ 
the fully empirical debiased estimates which use $\hat \df_k$.
As in the rest of the paper, the notation $\hat \btheta_k^\de,\hat \be_k^\de$ will denote the quantities computed using $\df_k$, as in Eqs.~\eqref{eq:param-est} and \eqref{eq:noise-est}.

Our main tool is the following.
\begin{lemma}
\label{lem:emp-df-conc}
    Assume \textsf{A1} and \textsf{A2}. 

    \begin{enumerate}[(i)]

        \item 
        There exist $\cPmodel$, $\cPregr$ and regression method-dependent $c' > 0$ and $\sC,\sc: \reals_{>0} \rightarrow \reals_{>0}$ such that
        \whp
        \begin{equation}
        \begin{gathered}
        \Big| 
            \frac{\hat \df_k}{p} - \frac{\df_k}{p}        
        \Big| < \epsilon.
        \end{gathered}
        \end{equation}

        \item 
        There exist $\cPmodel$, $\cPregr$ and regression method-dependent $c' > 0$ and $\sC,\sc: \reals_{>0} \rightarrow \reals_{>0}$ such that
        \whp
        \begin{equation}
        \begin{gathered}
            \| \hat \btheta_k^{\de,\mathrm{emp}} - \hat \btheta_k^\de \|_2 < \frac{p}{n_k}\,\epsilon,
           \qquad\qquad 
           \frac{\| \hat \be_{k,\cI_l}^{\de,\mathrm{emp}} - \hat \be_{k,\cI_l}^{\de,\mathrm{emp}} \|_2 }{\sqrt{n_l}}
           \leq \frac{p}{n_k}\,\epsilon.
        \end{gathered}
        \end{equation}

    \end{enumerate}
\end{lemma}

\begin{proof}[Proof of Lemma \ref{lem:emp-df-conc}(i)]
    In the case of least-squares, $\hat \df_k = p = \df_k$, and there is nothing to prove.

    In the case of the Lasso, this follows from Theorem 8 of \cite{celentano2020lasso},
    where we recall that for the lasso we assume in \textsf{A2} that $c < n_k / p < C$, so that $n$ and $p$ may be interchanged at the cost of changing constants.

    For ridge-regression, the bound follows from Theorem 3.7 of \cite{knowles2017anisotropic}, as we now describe.
    By Eq.~\eqref{eq:hat-df},
    we have that $\hat \df_k/p = 1 - \sqrt{\frac{p}{n_k}}\,\lambda_k \Tr\big(\frac1p\big(\frac1{n_k}\bX_{\cI_k}^\top \bX_{\cI_k} + \sqrt{\frac{p}{n_k}}\,\lambda_k\id_p\big)^{-1}\big)$.
    The trace on the right-hand side is the resolvent of $\bX_{\cI_k}^\top \bX_{\cI_k}/n_k$ evaluated at $-\lambda_k \sqrt{p/n_k}$ and normlized by $1/p$.
    Theorem 3.7 of \cite{knowles2017anisotropic} implies that this resolvent has fluctations $O_p(1/\sqrt{n_k})$ from a deterministic quantity defined by Eqs.~(3.5) and (2.1) of that paper, where $O_p$ hides only $\cPregr$ and $\cPmodel$-dependent constants.
    We compare these equations to the second line of the fixed-point equations Eqs.~\eqref{eq:marg-fix-pt-eq}, which in the case of ridge regression do not depend on $\tau$. 
    Some algebra shows that the deterministic quantity given by \cite{knowles2017anisotropic} corresponds to $\zeta_k/(\sqrt{p/n_k}\,\lambda_k)$.
    Thus, we conclude that $\hat \df_k/p$ has fluctuations $O_p(\sqrt{p}/n_k) = O_p(1/\sqrt{p})$ around $\df_k$.

    We now show that the concentration occurs with exponentially (in $p$) high probability.
    By Weyl's inequality, the singular values of $\bX_{\cI_k}/\sqrt{n_k}$ are $1$-Lipschitz in $\bX_{\cI_k}/\sqrt{n_k}$ with respect to Frobenius norm. 
    Because $x \mapsto (x^2 + \lambda_k \sqrt{p/n_k})^{-1}$ is $C (n_k/p)^{3/4} $-Lipschitz,
    we see that $\hat \df_k/p$ is $C \sqrt{p/n_k} \lambda_k \,(n_k/p)^{3/4} \leq C(n_k/p)^{1/4} \leq C\sqrt{n_k/p} $-Lipschitz in the $\bX_{\cI_k}/\sqrt{n_k}$ with respect to Frobenius norm. 
    By Gaussian concentration of Lipschitz functions,
    we conclude that with probability at least $1 - Ce^{-cp\epsilon^2}$ we have $\big| \hat \df_k / p - \E[\hat \df_k/p]\big| < \epsilon$.
    Combined with the fluctuation bound in the previous paragraph, we conclude that $\big|\E[\hat \df_k/p] - \df_k/p\big| \leq C/\sqrt{p}$.
    By adjusting constants so that the bound becomes trivial for $\epsilon \leq C/\sqrt{p}$,
    we conclude Lemma \ref{lem:emp-df-conc} in the case of ridge regression.
\end{proof}

\begin{proof}[Proof of Lemma \ref{lem:emp-df-conc}(ii)]
    Note
    \begin{equation}
        \| \hat \btheta_k^{\de,\mathrm{emp}} - \hat \btheta_k^\de \|_2 
            \leq 
            \frac{\| \bX_{\cI_k} \bSigma^{-1} \|_{\mathrm{op}}}{\sqrt{n_k}}
            \frac{\| \hat \be_{k,\cI_k} \|_2}{\sqrt{n_k}}
            \Big|
                \frac1{1 - \hat \df_k/n_k} - \frac1{1 - \df_k/n_k}
            \Big|.
    \end{equation}
    By Eq.~\eqref{eq:noise-est} and because $\hat \be_{k,\cI_k^c} = 0$, 
    $\hat \be_k^{\de,\mathrm{emp}}$ and $\hat \be_k^{\de}$ differ only on the indices in $\cI_k$.
    Moreover, 
    one can write $\hat \be_{k,\cI_k}^\de = \hat \be_{k,\cI_k}/(1 - \df_k/n_k)$ 
    and $\hat \be_{k,\cI_k}^{\de,\mathrm{emp}} = \hat \be_{k,\cI_k}/(1 - \hat \df_k/n_k)$.
    Thus,
    \begin{equation}
        \frac{\| \hat \be_{k,\cI_2}^{\de,\mathrm{emp}} - \hat \be_{k,\cI_2}^{\de,\mathrm{emp}} \|_2 }{\sqrt{n_2}}
            =
            \frac{\| \hat \be_{k,\cI_k \cap \cI_2}^{\de,\mathrm{emp}} - \hat \be_{k,\cI_k \cap \cI_2}^{\de,\mathrm{emp}} \|_2 }{\sqrt{n_2}}
            \leq 
            \frac{\| \hat \be_{k, \cI_2}\|_2}{\sqrt{n_2}} 
            \Big|
                \frac1{1 - \hat \df_k/n_k} - \frac1{1 - \df_k/n_k}
            \Big|.
    \end{equation}
    where we have used in the inequality that $\| \hat \be_{k,\cI_k \cap \cI_2}\|_2 \leq \| \hat \be_{k,\cI_2}\|_2$.

    We have that \whp, 
    $\| \bX_{\cI_k} \bSigma^{-1} \|_{\mathrm{op}}/\sqrt{n_k} < C$ and $\| \hat \be_{k,\cI_k}\|_2 / \sqrt{n_k} < C$,
    where in the first bound we use
    \cite[Corollary 5.35]{vershynin2010} and that the singular values of $\bSigma$ are bouned below and above by $c$ and $C$, 
    and in the second bound we use the marginal characterization (Lemma \ref{lem:marginal-characterization}) and that $\E[\|\hat \be_{k,\cI_k}^f\|_2]/\sqrt{n_k} \leq \E[\|\hat \be_{k,\cI_k}^f\|_2^2]^{1/2}/\sqrt{n_k} \leq C$. 
    By Lemma \ref{lem:emp-df-conc} and using that $1 - \df_k/n_k = \zeta_k > c$ by Eq.~\eqref{eq:marg-fix-pt-eq} and Lemma \ref{lem:bound-on-fixed-pt},
    we conclude that \whp, $\big|(1 - \hat\df_k/n_k)^{-1} - (1 - \df_k/n_k)^{-1}\big| < (p/n_k)\epsilon$.
    We conclude Lemma \ref{lem:emp-df-conc}(ii) by combining these bounds with the previous two displays.        
\end{proof}

Using Lemma \ref{lem:emp-df-conc}, Theorem \ref{thm:joint-characterization}(iii) follows from Theorem \ref{thm:joint-characterization}(i) and (ii) and Lemma \ref{lem:marginal-characterization}(ii) follows from Lemma \ref{lem:marginal-characterization}(i).

\section{The good marginal characterization event: proof of Lemma \ref{lem:marginal-concentration-event}}
\label{sec:marg-event-proof}

\begin{proof}[Proof of Lemma \ref{lem:marginal-concentration-event} (Good marginal characterization event)]
    First, consider the concentration of $\bT\Big(\frac{\be_{k,\cI_k}}{\sqrt{n_k}},\frac{\bxi_{h,\cI_k}}{\sqrt{n_k}}\Big) $.
    For the first diagonal entry, there is nothing to show because it is a function of $\be_k$.
    The conditional expectation of the second diagonal entry is $1$ for all realization of $\be_k$, because $\bxi_h$ is independent of $\be_k$. 
    Then, 
    Because $\bxi_{h,\cI_k}$ is a Gaussian vector with variance $\id_{n_k}$,
    by Gaussian concentration of Lipschitz functions Eq.~\eqref{eq:generic-concentration} is satisfied with $\ba^{(1)} = \ba^{(2)} = \bxi_{h,\cI_k}/\sqrt{n_k}$, $\bxi^{(1)},\bxi^{(2)} = \bxi_{h,\cI_k}$, $\bphi^{(1)}(\ba) = \bphi^{(2)}(\ba) = \ba / \sqrt{n_k}$, and $K = \sqrt{p/n_k}$.
    We see that $M = 1/\sqrt{n_k} \leq C / \sqrt{p}$ and $\bS_\xi$ has bounded singular values.
    Moreover, $\E[\|\bxi_{h,\cI_k}\|_2^2]/n_k = 1$, so Eq.~\eqref{eq:E-bound} is satisfied.
    Thus, Lemma \ref{lem:T-conc} implies the desired concentration of the second diagonal entry.
    
    Now consider the off-diagonal entries.
    Because $\be_{k,\cI_k}$ is a Gaussian vector with variance $\tau_{e_k}^2\id_{n_k}$, and $\tau_{e_k}^2 \leq C$, 
    by Gaussian concentration Eq.~\eqref{eq:generic-concentration} is satisfied with $\ba^{(1)} = \ba^{(2)} = \be_{k,\cI_k}/\sqrt{n_k}$, $\bxi^{(1)},\bxi^{(2)} = \be_{k,\cI_k}$, $\bphi^{(1)}(\ba) = \bphi^{(2)}(\ba) = \ba / \sqrt{n_k}$, and $K = \sqrt{p/n_k}$.
    We see that $M = 1/\sqrt{n_k} \leq C / \sqrt{p}$ and $\bS_\xi$ has bounded singular values.
    Moreover, $\E[\|\be_{k,\cI_k}\|_2^2]/n_k = \tau_{e_k}^2 \leq C$, so Eq.~\eqref{eq:E-bound} is satisfied.
    Thus, we may apply Lemma \ref{lem:T-conc}.
    We conclude that \whp, $\big| \| \hat \be_{k,\cI_k} \|_2^2 / n_k - \tau_{e_k}^2\big| < \sqrt{p/n_k}\,\epsilon$.
    Thus, 
    with probability at least $1 - Ce^{-cp}$, $\| \be_{k,\cI_k}\|_2^2 / n_k \leq \tau_{e_k}^2 + \sqrt{p/n_k} \leq C$, where the second inequality uses \textsf{A1}, \textsf{A2}.
    For such $\be_{k,\cI_k}$, the distribution of $\< \be_{k,\cI_k} , \xi_{h,\cI_k} \> / n_k$ conditional on $\be_{k,\cI_k}$ is Gaussian with variance $C^2 / n_k$.
    Then, the desired concentration of the off diagonal entries occurs by Gaussian concentration and the fact that $\| \be_{k,\cI_k}\|_2^2 / n_k \leq C$ with high-probability.

    Second, consider the concentration of $\| \be_{k,\cI_k} \|_2 / \sqrt{n_k}$.
    The previous paragraphs shows that \whp, $\big| \| \hat \be_{k,\cI_k} \|_2^2 / n_k - \tau_{e_k}^2\big| < \sqrt{p/n_k}\,\epsilon$. Because by assumption \textsf{A1} $\tau_{e_k}^2 > c$,
    taking square roots gives the desired concentration of $\| \be_{k,\cI_k} \|_2 / \sqrt{n_k}$.

    Third, consider the concentration of $\bT\Big(\sqrt{\frac{n_k}{p}}\,\bv_k^f,\frac1{\sqrt{p}}\bxi_g\Big)$.
    By Eqs.~\eqref{eq:vkf}, \eqref{eq:vf-prox}, and because proximal operators are $1$-Lispchitz \cite{parikh2014}, 
    the vectors $\sqrt{n_k/p}\, \bv_k^f, \bxi_g^f / \sqrt{p}$ are $(\sqrt{n_k/p}\,\tau_{g_k} \vee 1/\sqrt{p})$-Lipschitz functions of $\bxi_g$,
    and by Lemma \ref{lem:bound-on-fixed-pt}, we have $\sqrt{n_k/p}\,\tau_{g_k} \leq C/\sqrt{p}$.
    Thus, we may take $M = C/\sqrt{p}$ in Lemma \ref{lem:T-conc}.
    By Gaussian concentration, we get Eq.~\eqref{eq:generic-concentration} for $K = 1$.
    Finally, by Eq.~\eqref{eq:T-marg-ref}, Eq.~\eqref{eq:E-bound} is satisfied,
    whence we may apply Lemma \ref{lem:T-conc} and get the desired concentration.

    Next, we establish the concentration of $\bar \Omega_k(\bv_k^f)$.
    In the case of least-squares, $\bar \Omega_k(\bv_k^f) = 0$ always, 
    and there is nothing to show.
    In the case of ridge-regression,
    we apply Lemma \ref{lem:T-conc}.
    Recall that $\bar \Omega_k(\bv_k^f) = \sqrt{p/n_k}\,(\lambda/2)\|\hat \btheta_k^f\|_2^2$.
    By Eqs.~\eqref{eq:vkf} and \eqref{eq:vf-prox}, because proximal operators are 1-Lipschitz \cite{parikh2014} and $\bSigma$ has singular values bounded below by $c$ and $\tau_{g_k} \leq C / \sqrt{n_k}$ by Lemma \ref{lem:bound-on-fixed-pt},
    we have that $\hat \btheta_k^f$ is $M= 1 / \sqrt{n_k}$-Lipschitz in $\bg_k^f / \tau_{g_k}$, which is standard Gaussian.
    Moreover, by the marginal characterization (Lemma \ref{lem:marginal-characterization}),
    Eq.~\eqref{eq:generic-concentration} holds for $K = \sqrt{p/n_k}$.
    By \textsf{A2} and Lemma \ref{lem:bound-on-fixed-pt}, we have $\E[\|\btheta_k + \bSigma^{-1/2}\bv_k^f\|_2^2] \leq 2\|\btheta_k\|_2^2 + 2 C\E[\|\bv_k^f\|_2^2] \leq C(2 + p/n_k) \leq C$.
    Thus, using that $\omega_k = \E[\bar \Omega_k(\bv_k^f)]$ (see Eq.~\eqref{eq:sig*-e*-ome*}),
    we my apply Lemma \ref{lem:T-conc} to get the required concentration of $\bar \Omega_k(\bv_k^f)$, where we use additionally that $\lambda < C$ by \textsf{A2}.
    In the case of the $\alpha$-smoothed Lasso, 
    we have that $\bar \Omega_k(\bv_k^f) = \lambda \| \btheta_k + \bSigma^{-1/2} \bv_k^f \|_1/\sqrt{n_k}$ is 
    $C \lambda \tau_{g_k} \sqrt{p/n_1} \leq C \sqrt{p}/n_k$-Lipschitz in $\bxi_g$.
    The required concentration of $\bar \Omega_k(\bv_k^f)$ follows by Gaussian concentration of Lipschitz functions.

    Because $\bv_k^f$ is $\tau_{g_k} \leq C/\sqrt{n_k}$-Lipschitz in $\bxi_g$, 
    Gaussian concentration of Lipschitz functions implies that $\bv_k^f \in E_v(\Delta/2)$ with the desired high-probability.
    Because $\bu_k^f$ is $\zeta_k \tau_{h_k} \leq  C\sqrt{p/n_k}$-Lipschitz function of in $\bxi_h$ conditional on $\be_1,\be_2$,
    Gaussian concentration of Lipschitz functions implies that $\bu_k^f \in E_u(\Delta/2)$ with the desired high-probability.

    This completes the proof of Lemma \ref{lem:marginal-concentration-event}.
\end{proof}

\section{Applications of the marginal characterization}

In this section, we prove several consequences of the marginal characterization (Lemma \ref{lem:marginal-characterization}).
In Section \ref{sec:phi|1-conc-proof}, we prove that $\phi_{\theta|1}(\mathsf{Cond}_1)$ and $\phi_{e|1}(\mathsf{Cond}_1)$ concentrate (Lemma \ref{lem:phi-|1-conc}).
In Section \ref{sec:cond-event-proof},
we prove that the good conditional characterization event occurs with high probability (Lemma \ref{lem:conditional-concentration-event}).
The proof of Lemmas \ref{lem:phi-|1-conc} and \ref{lem:conditional-concentration-event} will require some corollaries of the marginal characterization which are stated and proved in Section \ref{sec:char-support}.

\subsection{Corollaries of the marginal characterization}
\label{sec:char-support}

In this section, we prove Lemma \ref{lem:hat-g-hat-h-db-err} and state and prove two corollaries of the marginal characterization which will be useful in proving Lemmas \ref{lem:phi-|1-conc} and \ref{lem:conditional-concentration-event} and Theorem \ref{thm:noise-est}.
These proofs rely on the concentration of certain random-design quantities, which are easily established using Lemma \ref{lem:marginal-characterization}.
For convenient reference, we collect these concentration statements in the next lemma.
\begin{lemma}\label{lem:marg-T-conc-helper}
    Assume \textsf{A1} and \textsf{A2}.
    Then \whp,
    \begin{equation}\label{eq:param-T}
    \begin{gathered}
        \Big\|
            \bT\Big(
                \sqrt{\frac{n_k}{p}}\, \bv_k ,\,
                \sqrt{\frac{n_k}{p}}\, \bSigma^{1/2}(\hat \btheta_k^\de - \btheta_k)
            \Big)
            - 
            \E\Big[
                \bT\Big(
                    \sqrt{\frac{n_k}{p}}\, \bv_k^f ,\,
                    \sqrt{\frac{n_k}{p}}\,\bg_k^f 
                \Big)  
            \Big]
        \Big\|_{\sF}
        < C'\epsilon,
    \end{gathered}
    \end{equation}
    and 
    \begin{equation}\label{eq:noise-T}
    \begin{gathered}
        \Big\|
            \bT\Big(
                \frac{\be_{1,\cI_k}}{\sqrt{n_k}},\,
                \frac{\be_{2,\cI_k}}{\sqrt{n_k}},\,
                \frac{\bu_{k,\cI_k}}{\sqrt{n_k}}
            \Big)
            -
            \E\Big[
                \bT\Big(
                    \frac{\be_{1,\cI_k}}{\sqrt{n_k}},\,
                    \frac{\be_{2,\cI_k}}{\sqrt{n_k}},\,
                    \frac{\bu_{k,\cI_k}^f}{\sqrt{n_k}}
                \Big)
            \Big]
        \Big\|_{\sF}
        < C' \sqrt{\frac{p}{n_k}}\, \epsilon,
    \end{gathered}
    \end{equation}
    where
    \begin{equation}
    \begin{gathered}
        \E\Big[
                \bT\Big(
                    \sqrt{\frac{n_k}{p}}\, \bv_k^f ,\,
                    \sqrt{\frac{n_k}{p}}\,\bg_k^f 
                \Big)  
            \Big]
            = 
            \begin{pmatrix}
                (n_k/p)\tau_{h_k}^2 & (n_k/p) \tau_{g_k}^2 \df_k \\[4pt]
                (n_k/p) \tau_{g_k}^2 \df_k & n_k \tau_{g_k}^2
            \end{pmatrix},
        \\
        \E\Big[
                \bT\Big(
                    \frac{\be_{1,\cI_k}}{\sqrt{n_k}},\,
                    \frac{\be_{2,\cI_k}}{\sqrt{n_k}},\,
                    \frac{\bu_{k,\cI_k}^f}{\sqrt{n_k}}
                \Big)
            \Big]
            = 
            \begin{pmatrix}
                \tau_{e_1}^2 & \tau_{e_1}\tau_{e_2}\rho_e & \zeta_k \tau_{e_1} \tau_{e_k} \rho_e^{\indic{k=2}} \\[4pt]
                \tau_{e_1}\tau_{e_2}\rho_e & \tau_{e_2}^2 & \zeta_k \tau_{e_2} \tau_{e_k} \rho_e^{\indic{k=1}} \\[4pt]
                \zeta_k \tau_{e_1} \tau_{e_k} \rho_e^{\indic{k=2}} & \zeta_k \tau_{e_2} \tau_{e_k} \rho_e^{\indic{k=1}} & n_k \zeta_k^2 \tau_{g_k}^2
            \end{pmatrix},
    \end{gathered}
    \end{equation}
    where we interpret $0^0 = 1$ when $\rho_e = 0$,
    and the entries of both matrices are bounded by $C$.

    The same holds if we replace $\df_k$ by $\hat \df_k$ in the definition of $\hat \btheta_k^\de$ in Eq.~\eqref{eq:param-est},
    where $\hat \df_k$ is defined by Eq.~\eqref{eq:hat-df}.
\end{lemma}

\begin{proof}[Proof of Lemma \ref{lem:marg-T-conc-helper}]
    Both Eq.~\eqref{eq:param-T} and \eqref{eq:noise-T} follows from Lemma \ref{lem:marginal-characterization} using Lemma \ref{lem:T-conc}.
    Indeed,
    by Eqs.~\eqref{eq:vkf}, \eqref{eq:vf-prox}, and because proximal operators are $1$-Lispchitz \cite{parikh2014}, 
    the vectors $\sqrt{n_k/p}\, \bv_k^f, \sqrt{n_k/p}\,\bg_k^f$ are $\sqrt{n_k/p}\,\tau_{g_k}$-Lipschitz functions of $\bg_k^f / \tau_{g_k}$,
    and by Lemma \ref{lem:bound-on-fixed-pt}, we have $\sqrt{n_k/p}\,\tau_{g_k} \leq C/\sqrt{p}$.
    Thus, we may take $M = C/\sqrt{p}$ in Lemma \ref{lem:T-conc}.
    Further, by Eq.~\eqref{eq:vk-uk},
    we have $\bv_k = \bSigma^{1/2}(\hat \btheta_k - \btheta_k)$, 
    whence the marginal characterization (Lemma \ref{lem:marginal-characterization}) gives us Eq.~\eqref{eq:generic-concentration} for the pair of vectors $\sqrt{n_k/p}\,\bv_k,\sqrt{n_k/p}\,\bSigma^{1/2}(\hat \btheta_k^{\de} - \btheta_k)$ with $K = C$.
    Finally, by Lemma \ref{lem:bound-on-fixed-pt}, $(n_k/p)\tau_{h_k}^2 \leq C$ and $n_k \tau_{g_k}^2 \leq C$, whence Eq.~\eqref{eq:E-bound} is satisfied.
    Thus, Lemma \ref{lem:T-conc} implies Eq.~\eqref{eq:param-T}.

    By Eqs.~\eqref{eq:vkf} and \eqref{eq:fixed-design-err-est} and because $\zeta_k \leq 1$, 
    the vectors $\be_{1,\cI_k}/\sqrt{n_k},\be_{2,\cI_k}/\sqrt{n_k}$, and $\bu_{k,\cI_k}^f/\sqrt{n_k}$ are $M=1/\sqrt{n_k}$-Lipschitz functions of Gaussian vectors with variance $\tau_{e_1}^2$, $\tau_{e_2}^2$, and $\tau_{e_k}^2 + \tau_{h_k}^2$, respectively.
    By Lemma \ref{lem:bound-on-fixed-pt} these variance are all bounded by $C$.
    Further, by Eq.~\eqref{eq:vk-uk},
    we have $\bu_{k,\cI_k} = \hat \be_{k,\cI_k}$, 
    whence the marginal characterization (Lemma \ref{lem:marginal-characterization}) gives us Eq.~\eqref{eq:generic-concentration} for the vectors $\be_{1,\cI_k}/\sqrt{n_k},\be_{2,\cI_k}/\sqrt{n_k},\bu_{k,\cI_k}^f/\sqrt{n_k}$ with $K = \sqrt{p/n_k}$.
    Finally, by Lemma \ref{lem:bound-on-fixed-pt}, 
    Eq.~\eqref{eq:E-bound} is satisfied,
    whence we may apply Lemma \ref{lem:T-conc} to conclude Eq.~\eqref{eq:noise-T}.
\end{proof}

\noindent We now prove Lemma \ref{lem:hat-g-hat-h-db-err}. 

\begin{proof}[Proof of Lemma \ref{lem:hat-g-hat-h-db-err}]

    Recall that $\hat \bg_k,\hat \bh_k$ are defined in Eq.~\eqref{eq:hat-g1-hat-h1} as $\hat \bg_k = \tau_{g_k}\hat \bxi_g$ and $\hat \bh_k = \tau_{h_k} \hat \bxi_h$, where $\hat \bxi_g, \hat \bxi_h$ solve the system of equations \eqref{eq:hat-xi} and \eqref{eq:constraint+}.
    (For notational simplicity, in most of the paper---including in Eqs.~equations \eqref{eq:hat-xi} and \eqref{eq:constraint+}---we considered conditioning on the first regression, $k = 1$. This was arbitrary, and in the current proof we consider general $k$.)
    Writing the solution to these equations explicitly gives
    \begin{equation}\label{eq:def-ggam-hgam}
    \begin{gathered}
        \hat \bxi_g 
            =
            \frac{n_k \zeta_k \tau_{h_k}}{\| \bu_k \|_2 \| \bv_k \|_2}  
            \bSigma^{1/2}\big(\hat \btheta_k - \btheta_k\big) 
            + 
            \frac{\bSigma^{-1/2}\bX_{ \cI_k }^\top(\by_{ k,\cI_k }-\bX_{ \cI_k }\hat \btheta_k)}{\| \bu_k \|_2},
        \\[4pt]
         \hat \bxi_h 
            := 
            \begin{pmatrix}
                - \frac{\bX_{ \cI_k }(\hat \btheta_k - \btheta_k)}{\|\bv_k\|_2} + \frac{\< \hat \bxi_g , \bv_k \> }{ \| \bu_k \|_2 \| \bv_k \|_2} (\by_{ k,\cI_k } - \bX_{ \cI_k }\hat \btheta_k)
                \\[3pt]
                -\frac{\bX_{ \cI_k ^c}(\hat \btheta_k - \btheta_k)}{\| \bv_k \|_2}
            \end{pmatrix}.
    \end{gathered}
    \end{equation}
    The first line follows from the second equation in Eq.~\eqref{eq:hat-xi} by plugging in $n_k \zeta_k \tau_{h_k}$ for $\< \hat \bxi_h,\bu_k\>$, recalling the definition of $\bs_1$ in Eq.~\eqref{eq:uvst}, and rearranging.
    The second line follows by rearranging the first equation in Eq.~\eqref{eq:hat-xi} and using the definition of $\bt_1$ in Eq.~\eqref{eq:uvst}.

    First consider $\big\| \hat \bg_k - \bSigma^{1/2}(\hat \btheta_k^{\de} - \btheta_k)\big\|_2$.
    Recalling the definition of $\hat \btheta_k^{\de}$ in Eq.~\eqref{eq:param-est} and that $n_k - \df_k = n_k\zeta_k $ by Eq.~\eqref{eq:marg-fix-pt-eq},
    the first line of Eq.~\eqref{eq:def-ggam-hgam} and the triangle inequality imply
    \begin{equation}
        \big\| \hat \bg_k - \bSigma^{1/2}(\hat \btheta_k^{\de} - \btheta_k)\big\|_2
            \leq 
            \Big|
                \frac{n_k \zeta_k \tau_{h_k}\tau_{g_k}}{\|\bu_k\|_2\|\bv_k\|_2}
                -
                1
            \Big|
            \cdot 
            \| \hat \btheta_k - \btheta_k \|_{\bSigma}
            +
            \Big|
                \frac{n_k\tau_{g_k}}{\|\bu_k\|_2} - \frac{1}{\zeta_k}
            \Big|
            \frac{\| \bSigma^{-1/2}\bX_{ \cI_k }^\top(\by_{ k,\cI_k }-\bX_{ \cI_k }\hat \btheta_k) \|_2}{n_k}.
    \end{equation}
    Because $\bv_k = \bSigma^{1/2}(\hat \btheta_k-\btheta_k)$ (see Eq.~\eqref{eq:vkf}),
    by Lemma \ref{lem:marg-T-conc-helper}
    with probability at least $1 - Ce^{-cp}$ we have $\| \hat \btheta_k - \btheta_k \|_{\bSigma} \leq C \sqrt{p/n_k}$.
    Further, by Lemma \ref{lem:marg-T-conc-helper}, 
    with probability at least $1 - \sC(\epsilon)e^{-\sc(\epsilon)p}$
    \begin{equation}
        \Big|
            \frac{n_k \zeta_k \tau_{h_k}\tau_{g_k}}{\|\bu_k\|_2\|\bv_k\|_2}
            -
            1
        \Big|
        = 
        \Big|
            \frac{ (\sqrt{n_k}\,\tau_{g_k}\zeta_k) (\sqrt{n_k/p}\,\tau_{h_k}) }{(\|\bu_k\|_2/\sqrt{n_k})(\sqrt{n_k/p}\,\|\bv_k\|_2)}
            -
            1
        \Big|
        <
        C'\epsilon,
    \end{equation}
    where we use that $\sqrt{n_k/p}\,\|\bv_k\|_2$ and $\|\bu_k\|_2 / \sqrt{n_k} = \| \bu_{k,\cI_k} \|_2 / \sqrt{n_k}$ concentrate on $\sqrt{n_k/p}\,\tau_{h_k}$ and $\sqrt{n_k}\,\tau_{g_k}\zeta_k$, which are bounded above and below by $C$ and $c$.
    Recalling that $\bu_{k,\cI_k} = \by_{k,\cI_k} - \bX_{\cI_k} \hat \btheta_k$ and applying \cite[Corollary 5.35]{vershynin2010} and Lemma \ref{lem:marg-T-conc-helper}, 
    we have with probability at least $1 - Ce^{-cp}$
    \begin{equation}
        \frac{\| \bSigma^{-1/2}\bX_{ \cI_k }^\top(\by_{ k,\cI_k }-\bX_{ \cI_k }\hat \btheta_k) \|_2}{n_k}
            \leq 
            \frac{\| \bSigma^{-1/2} \bX_{\cI_k}^\top\|_{\mathrm{op}}}{\sqrt{n_k}} 
            \frac{\| \by_{k,\cI_k}-\bX_{\cI_k}\hat \btheta_k\|_2}{\sqrt{n_k}}
            \leq 
            C',
    \end{equation}
    and with probability at least $1 - \sC(\epsilon)e^{-\sc(\epsilon)p}$,
    \begin{equation}
        \Big|
            \frac{n_k\tau_{g_k}}{\|\bu_k\|_2} - \frac{1}{\zeta_k}
        \Big|
        = 
        \Big|
            \frac{\sqrt{n_k}\,\tau_{g_k}}{\|\bu_k\|_2/\sqrt{n_k}} - \frac{1}{\zeta_k}
        \Big|
        < C'\sqrt{\frac{p}{n_k}}\,\epsilon,
    \end{equation}
    where we use that $\| \bu_k \|_2 / \sqrt{n_k}$ concentrates on $\sqrt{n_k}\,\tau_{g_k}$ which is bounded above and bewlo by $C$ and $c$.
    Combining the above bounds gives the high-probability bound on $\big\| \hat \bg_k - \bSigma^{1/2}(\hat \btheta_k^{\de} - \btheta_k)\big\|_2$ in Lemma \ref{lem:hat-g-hat-h-db-err}.

    Next, consider $\big\| \hat \bh_{k,\cI_k} - (\hat \be_{k,\cI_k}^\de - \be_{k,\cI_k}) \big\|_2 / \sqrt{n_k}$.
    Recalling the definition of $\hat \be_k^{\de}$ in Eq.~\eqref{eq:noise-est} and that $\df_k/(n_k - \df_k) = (1-\zeta_k)/\zeta_k $ by Eq.~\eqref{eq:marg-fix-pt-eq}, 
    the second line of Eq.~\eqref{eq:def-ggam-hgam} and the triangle inequality imply
    \begin{equation}
        \frac{\big\| \hat \bh_{k,\cI_k} - (\hat \be_{k,\cI_k}^\de - \be_{k,\cI_k}) \big\|_2}{\sqrt{n_k}}   
            \leq 
            \Big|
                \frac{\tau_{h_k}}{\| \bv_k \|_2} - 1
            \Big|
            \cdot 
            \frac{\| \bX_{\cI_k}(\hat \btheta_k - \btheta_k) \|_2}{\sqrt{n_k}}
            + 
            \Big|
                \frac{\tau_{h_k}\< \hat \bxi_g, \bv_k \>}{\| \bu_k \|_2 \| \bv_k \|_2}
                -
                \frac{1-\zeta_k}{\zeta_k}
            \Big|
            \cdot 
            \frac{\| \by_{k,\cI_k} - \bX_{\cI_k} \hat \btheta_k \|_2}{\sqrt{n_k}}.
    \end{equation}
    By \cite[Corollary 5.35]{vershynin2010} and Lemma \ref{lem:marg-T-conc-helper},
    with probability at least $1 - Ce^{-cp}$ we have $\| \bX_{\cI_k}(\hat \btheta_k - \btheta_k) \|_2/\sqrt{n_k} \leq \| \bX_{\cI_k} \bSigma^{-1/2} \|_2 \| \hat \btheta_k - \btheta_k \|_{\bSigma} \leq C' \sqrt{p/n_k}$
    and $\| \by_{k,\cI_k} - \bX_{\cI_k} \hat \btheta_k \|_2/\sqrt{n_k} = \| \bu_{k,\cI_k}\|_2 / \sqrt{n_k} \leq C'$.
    By Lemma \ref{lem:marg-T-conc-helper}, 
    with probability at least $1 - \sC(\epsilon)e^{-\sc(\epsilon)p}$ we have
    \begin{equation}
        \Big|
            \frac{\tau_{h_k}}{\| \bv_k \|_2} - 1
        \Big|
        =
        \Big|
            \frac{\sqrt{n_k/p}\,\tau_{h_k}}{\sqrt{n_k/p}\,\| \bv_k \|_2} - 1
        \Big|
        < C' \epsilon,
    \end{equation}
    where we have used that $\sqrt{n_k/p}\,\|\bv_k\|_2$ concentrates on $\sqrt{n_k/p}\,\tau_{h_k}$, 
    which is bounded above and below by $C$ and $c$.
    By Lemma \ref{lem:marg-T-conc-helper} and the bound on $\big\| \hat \bg_k - \bSigma^{1/2}(\hat \btheta_k^{\de} - \btheta_k)\big\|_2$ in Lemma \ref{lem:hat-g-hat-h-db-err} (which we have already proved), 
    with probability at least $1 - \sC(\epsilon)e^{-\sc(\epsilon)p}$ 
    \begin{equation}
    \begin{aligned}
        &\Big|
                \frac{\< \hat \bxi_g , \bv_k\> }{\sqrt{n_k}}
                -
                \sqrt{n_k}\,\tau_{g_k}(1-\zeta_k)
            \Big|
            =
            \frac{p}{n_k} \frac{1}{\sqrt{n_k}\,\tau_{g_k}}
            \Big|
                \Big\< \sqrt{\frac{n_k}{p}}\, \hat \bg_k , \sqrt{\frac{n_k}{p}}\,\bv_k\Big\>
                -
                \frac{n_k\,\tau_{g_k}^2\df_k}{p}
            \Big|
        \\
            &\leq 
            C' \| \hat \bg_k - \bSigma^{1/2}(\hat \btheta_k^\de - \btheta_k)\|_2 \| \bv_k\|_2 
            +
            C'\frac{p}{n_k}
            \Big|
                \Big\< \sqrt{\frac{n_k}{p}}\, \bSigma^{1/2}(\hat \btheta_k^\de - \btheta_k) , \sqrt{\frac{n_k}{p}}\,\bv_k\Big\>
                -
                \frac{n_k\,\tau_{g_k}^2\df_k}{p}
            \Big|
            \leq C' \frac{p}{n_k}\,\epsilon,
    \end{aligned}
    \end{equation}
    where in the first equality we have used that $n_k(1 - \zeta_k) =  \df_k$ by Eq.~\eqref{eq:marg-fix-pt-eq},
    and in the first inequality we have used that $\sqrt{n_k}\tau_{g_k} \geq c$ by Lemma \ref{lem:bound-on-fixed-pt} to replace it by a constant.
    Combining the above bounds, 
    we have with probability at least $1 - \sC(\epsilon)e^{-\sc(\epsilon)p}$ 
    \begin{equation}    
        \Big|
            \frac{\tau_{h_k}\< \hat \bxi_g, \bv_k \>}{\| \bu_k \|_2 \| \bv_k \|_2}
            -
            \frac{1-\zeta_k}{\zeta_k}
        \Big|
        =
        \Big|
            \frac{(\< \hat \bxi_g, \bv_k \>/\sqrt{n_k})(\sqrt{n_k/p}\,\tau_{h_k})}{(\| \bu_k \|_2/\sqrt{n_k}) (\sqrt{n_k/p}\,\| \bv_k \|_2)}
            -
            \frac{1-\zeta_k}{\zeta_k}
        \Big|
        < C' \sqrt{\frac{p}{n_k}}\,\epsilon,
    \end{equation}  
    where we use that $\sqrt{n_k/p}\,\|\bv_k\|_2$ concentrates on $\sqrt{n_k/p}\,\tau_{h_k}$,
    $\|\bu_k\|_2 / \sqrt{n_k}$ concentrates on $\sqrt{n_k}\,\tau_{g_k}\zeta_k$, which are bounded above and below by $C$ and $c$, 
    and that $\< \hat \bxi_g,\bv_k\> / \sqrt{n_k}$ concentrates on $n_k \tau_{g_k}^2 \df_k / p$, which is bounded above by $C$ by Lemma \ref{lem:bound-on-fixed-pt}.
    Combining the above bounds gives the high probability bound on $\| \hat \bh_{k,\cI_k} - (\hat \be_{k,\cI_k}^\de - \be_{k,\cI_k})\|_2 / \sqrt{n_k}$ in Lemma \ref{lem:hat-g-hat-h-db-err}.
\end{proof}

\noindent To prove Lemmas \ref{lem:phi-|1-conc} and \ref{lem:conditional-concentration-event} and Theorem \ref{thm:noise-est}, 
we must extend the marginal characterization.
In particular, we extend the characterization to the vectors $\hat \bg_1$ and $\hat \bh_1$, 
and extend the characterization from the index set $\cI_1$ to the index set $\cI_2$.
The next corollaries 
perform these extensions.
\begin{corollary}\label{cor:ext-to-hat-gh}
    Assume \textsf{A1} and \textsf{A2}. 

    There exist $\cPmodel$, $\cPregr$ and regression method-dependent $c' > 0$ and $\sC,\sc: \reals_{>0} \rightarrow \reals_{>0}$ such that
    for $\epsilon < c'$ and 1-Lipschitz $\phi_\theta:(\reals^p)^2 \rightarrow \reals$, $\phi_e: (\reals^{n_k})^4 \rightarrow \reals$, with probability at least $1 - \sC(\epsilon) e^{-\sc(\epsilon) p}$ 
    \begin{equation}
    \begin{gathered}
    \big| \phi_\theta(\hat \btheta_k, \hat \bg_k) - \E[\phi_\theta(\hat \btheta_k^f,\bg_k^f)]\big| < \sqrt{\frac{p}{n_k}}\,\epsilon,
    \\
    \Big| 
            \phi_e\Big(
                \frac{ \be_{1,\cI_k}}{\sqrt{n_k}},
                \frac{ \be_{2,\cI_k}}{\sqrt{n_k}},
                \frac{\hat \be_{k,\cI_k}}{\sqrt{n_k}},
                \frac{\hat \bh_{k,\cI_k}}{\sqrt{n_k}}
            \Big) 
            - 
            \E\Big[
                \phi_e\Big(
                    \frac{ \be_{1,\cI_k}}{\sqrt{n_k}},
                    \frac{ \be_{2,\cI_k}}{\sqrt{n_k}},
                    \frac{\hat \be_{k,\cI_k}^f}{\sqrt{n_k}},
                    \frac{\bh_k^f}{\sqrt{n_k}}
                \Big)\Big] \Big| < \sqrt{\frac{p}{ n_k }} \, \epsilon.
    \end{gathered}
    \end{equation}
\end{corollary}

\begin{proof}[Proof of Corollary \ref{cor:ext-to-hat-gh}]
    By Lemma \ref{lem:hat-g-hat-h-db-err},
    \whp
    \begin{equation}
        \Big|
            \phi_e\Big(
                \frac{ \be_{1,\cI_k}}{\sqrt{n_k}},
                \frac{ \be_{2,\cI_k}}{\sqrt{n_k}},
                \frac{\hat \be_{k,\cI_k}}{\sqrt{n_k}},
                \frac{\hat \bh_{k,\cI_k} }{\sqrt{n_k}}
            \Big) 
            -
            \phi_e\Big(
                \frac{ \be_{1,\cI_k}}{\sqrt{n_k}},
                \frac{ \be_{2,\cI_k}}{\sqrt{n_k}},
                \frac{\hat \be_{k,\cI_k}}{\sqrt{n_k}},
                \frac{\hat \be_{k,\cI_k}^{\de} - \be_{k,\cI_k}}{\sqrt{n_k}}
            \Big) 
        \Big|
        < \sqrt{\frac{p}{n_k}}\,\epsilon.
    \end{equation}
    Using the marginal characterization (Lemma \ref{lem:marginal-characterization}) and recalling that $\hat \be_{k,\cI_k}^{f,\de} - \be_{k,\cI_k} = \bh_{k,\cI_k}^f$ (see Eqs.~\eqref{eq:fixed-des} and 
    \eqref{eq:fixed-design-err-est}),
    we have \whp
    \begin{equation}
        \Big|
            \phi_e\Big(
                \frac{ \be_{1,\cI_k}}{\sqrt{n_k}},
                \frac{ \be_{2,\cI_k}}{\sqrt{n_k}},
                \frac{\hat \be_{k,\cI_k}}{\sqrt{n_k}},
                \frac{\hat \be_{k,\cI_k}^{\de} - \be_{k,\cI_k}}{\sqrt{n_k}}
            \Big) 
            - 
            \E\Big[
                \phi_e\Big(
                    \frac{ \be_{1,\cI_k}}{\sqrt{n_k}},
                    \frac{ \be_{2,\cI_k}}{\sqrt{n_k}},
                    \frac{\hat \be_{k,\cI_k}^f}{\sqrt{n_k}},
                    \frac{\bh_k^f}{\sqrt{n_k}}
                \Big)
            \Big] 
        \Big|
        < \sqrt{\frac{p}{n_k}}\,\epsilon.
    \end{equation}
    The corollary follows by combining the above two bounds using the triangle inequality.
    The first line of the corollary follows similarly, using the bound on $\| \hat \bg_k - \bSigma^{1/2}(\hat \btheta_k^\de - \btheta_k)\|_2$ from Lemma \ref{lem:hat-g-hat-h-db-err} and recalling that in the fixed design model $\bSigma^{1/2}(\hat \btheta_k^{f,\de} - \hat \btheta_k) = \bg_k^f$.
\end{proof}

\begin{corollary}\label{cor:ext-to-I2}
    Assume \textsf{A1} and \textsf{A2}. 

    \begin{enumerate}[(i)]

        \item 
        Let $\phi_e: (\reals^{n_2})^4 \rightarrow \reals$ by $M_e$-Lipschitz in its first three arguments and $M_h$ Lipschitz in its final argument.
        There exist $\cPmodel$, $\cPregr$ and regression method-dependent $c' > 0$ and $\sC,\sc: \reals_{>0} \rightarrow \reals_{>0}$ such that
        \whp
        \begin{equation}\label{eq:I2-marg}
        \begin{gathered}
        \Big| 
                \phi_e\Big(
                    \frac{ \be_{1,\cI_2}}{\sqrt{n_2}},
                    \frac{ \be_{2,\cI_2}}{\sqrt{n_2}},
                    \frac{\hat \be_{1,\cI_2}}{\sqrt{n_2}},
                    \frac{\hat \bh_{1,\cI_2}}{\sqrt{n_2}}
                \Big) 
                - 
                \E\Big[
                    \phi_e\Big(
                        \frac{ \be_{1,\cI_2}}{\sqrt{n_2}},
                        \frac{ \be_{2,\cI_2}}{\sqrt{n_2}},
                        \frac{\hat \be_{1,\cI_2}^f}{\sqrt{n_2}},
                        \frac{\bh_{1,\cI_2}^f}{\sqrt{n_2}}
                    \Big)\Big] \Big| < \Big(
                        M_e \sqrt{\frac{p}{n_2}} + M_h \sqrt{\frac{p(p+n_{12})}{n_1n_2}}\;
                    \Big) \, \epsilon.
        \end{gathered}
        \end{equation}

        \item 
        Further, we have \whp
        \begin{equation}
            \frac{\| \hat \bh_{1,\cI_2} - (\hat \be_{1,\cI_2}^\de - \be_{1,\cI_2})\|_2}{\sqrt{n_2}} 
                \leq \sqrt{\frac{p}{n_1}}\,\epsilon.
        \end{equation}
        Thus, if we replace $\hat \bh_{1,\cI_2}$ by $\hat \be_{1,\cI_2}^{\de}-\be_{1,\cI_2}$ in Eq.~\eqref{eq:I2-marg},
        the same bound holds except that we must replace the term $ M_h \sqrt{p(p+n_{12})/(n_1n_2)} $ on the right-hand side by $M_h \sqrt{p/n_1}$.

        \item 
        Part (ii) still holds if we replace $\df_k$ by $\hat \df_k$ in the definition of $\hat \be_1^\de$ in Eq.~\eqref{eq:noise-est},
        where $\hat \df_k$ is defined by Eq.~\eqref{eq:hat-df}.

    \end{enumerate}

\end{corollary}

\begin{proof}[Proof of Corollary \ref{cor:ext-to-I2}]
    Observe that the data $\be_{1,\cI_1},\be_{2,\cI_1}\bX_{\cI_1}$ is rotationally invariant in the sense that for any orthogonal matrix $\bQ \in \reals^{n_1\times n_1}$ (i.e.,~$\bQ^\top\bQ = \id_{n_k}$), 
    we have
    $(\bQ \be_{1,\cI_1},\bQ \be_{2,\cI_1},\bQ \bX_{\cI_k}) \stackrel{\mathrm{d}}= (\be_{1,\cI_1},\be_{2,\cI_1},\bX_{\cI_1})$.
    By the rotational invariance of the least-squares loss, 
    under this replacement $(\be_{1,\cI_1},\be_{2,\cI_1},\hat \be_{1,\cI_1},\hat \bh_{1,\cI_2})$ 
    is replaced with $(\bQ\be_{1,\cI_1},\bQ\be_{2,\cI_1},\bQ\hat \be_{1,\cI_1},\bQ\hat \bh_{1,\cI_2})$ and $\bv_k$ is unaffected.
    We conclude that $(\be_{1,\cI_1},\be_{2,\cI_1},\hat \be_{1,\cI_1},\hat \bh_{1,\cI_2})$ has a rotationally invariant distribution.
    Further, from Eq.~\eqref{eq:def-ggam-hgam} and recalling that $\bv_k = \bSigma^{1/2}(\hat \btheta_k - \btheta_k)$ (see Eq.~\eqref{eq:vk-uk}) and $\bX_{\cI_1^c}$ is independent of $\bX_{\cI_1}$, 
    we have $\hat \bxi_{h,\cI_1^c} = - \bX_{\cI_1^c}(\hat \btheta_1-\btheta_1)/\| \hat \btheta_1 - \btheta_1 \|_{\bSigma} \sim \normal(0,\id_{N-n_1})$ independent of $\be_{1,\cI_1},\be_{2,\cI_1}\bX_{\cI_1}$.
    Therefore, $\hat \bh_{1,\cI_1^c} = \tau_{h_1}\hat \bxi_{h,\cI_1^c} \sim \normal(0,\tau_{h_1}^2 \id_{N-n_1})$ independent of $\be_{1,\cI_1},\be_{2,\cI_1}\bX_{\cI_1}$.
    Thus, we may represent
    \begin{equation}
    \begin{aligned}
        &\Big(
                \frac{\be_{1,\cI_2}}{\sqrt{n_2}},
                \frac{\be_{2,\cI_2}}{\sqrt{n_2}},
                \frac{\hat \be_{1,\cI_2}}{\sqrt{n_2}},
                \frac{\hat \bh_{1,\cI_2}}{\sqrt{n_2}}
            \Big)
        \\
            &\quad\stackrel{\mathrm{d}}=
            \frac1{\sqrt{n_2}}
            \left(
                \begin{pmatrix}
                    (\bQ \be_{1,\cI_1})_{\cI_1 \cap \cI_2} \\[3pt]
                    \tau_{e_1} \bxi^{(1)}
                \end{pmatrix},
                \begin{pmatrix}
                    (\bQ \be_{2,\cI_1})_{\cI_1 \cap \cI_2} \\[3pt]
                    \tau_{e_2} \bxi^{(2)}
                \end{pmatrix},
                \begin{pmatrix}
                    (\bQ \hat \be_{1,\cI_1})_{\cI_1 \cap \cI_2} \\[3pt]
                    \bzero
                \end{pmatrix},
                \begin{pmatrix}
                    (\bQ \hat \bh_{1,\cI_1})_{\cI_1 \cap \cI_2} \\[3pt]
                    \tau_{h_1} \bxi^{(3)}
                \end{pmatrix}
            \right)
            =: \bD,
    \end{aligned}
    \end{equation}
    where $\bxi^{(1)},\bxi^{(2)},\bxi^{(3)}\stackrel{\mathrm{iid}}\sim\normal(0,\id_{n_2 - n_{12}})$ and $\bQ\in \reals^{n_1 \times n_1}$ is drawn uniformly from the special orthogonal group $SO(n_1)$ and independently of everything else.
    We introduce the notation $\bD \in \reals^{n_2 \times 4}$ for notational compactness,
    and denote $\phi_e$ applied to its columns by $\phi_e(\bD)$.
    By the distributional equivalence above, 
    it is enough to establish the concentration result for $\phi_e(\bD)$ in place of 
    $\phi_e\Big(
        \frac{\be_{1,\cI_2}}{\sqrt{n_2}},
        \frac{\be_{2,\cI_2}}{\sqrt{n_2}},
        \frac{\hat \be_{1,\cI_2}}{\sqrt{n_2}},
        \frac{\hat \bh_{1,\cI_2}}{\sqrt{n_2}}
    \Big)$.
    
    Conditioning on $\bQ ,\be_{1,\cI_1},\be_{2,\cI_1},\hat \be_{1,\cI_1},\hat \bh_{1,\cI_1}$, the only randomness that remains is in $\bxi^{(1)},\bxi^{(2)},\bxi^{(3)}$.
    By Lemma \ref{lem:bound-on-fixed-pt} and assumption \textsf{A1}, we have $\tau_{e_1},\tau_{e_2} \leq C$ and $\tau_{h_1} \leq C\sqrt{p/n_1}$.
    Thus, by Gaussian concentration,
    with probability at least $1 - \sC(\epsilon)e^{-\sc(\epsilon)p}$
    \begin{equation}\label{eq:bound1}
        \big| \phi_e(\bD) - \E[\phi_e(\bD) \mid \bQ ,\be_{1,\cI_1},\be_{2,\cI_1},\hat \be_{1,\cI_1},\hat \bh_{1,\cI_1}]\big| < \Big(M_e\sqrt{\frac{p}{n_2}} + M_h \frac{p}{\sqrt{n_1n_2}}\;\Big)\epsilon.
    \end{equation}
    Now note that $\E[\phi_e(\bD) \mid \bQ ,\be_{1,\cI_1},\be_{2,\cI_1},\hat \be_{1,\cI_1},\hat \bh_{1,\cI_1}]$ viewed as a function of $\bQ$ with $\be_1,\be_2,\hat \be_1,\hat \bh_1,\bxi^{(1)},\bxi^{(2)},\bxi^{(3)}$ fixed is $M_e(\| \be_{1,\cI_1}\|_2 + \| \be_{2,\cI_1} \|_2 + \| \hat \be_{1,\cI_1}\|_2)/\sqrt{n_2} + M_h \| \hat \bh_{1,\cI_1} \|_2 / \sqrt{n_2}$-Lipschitz in $\bQ$ in Frobenius norm.
    Indeed, $\| \bQ \be_{1,\cI_1} - \bQ \be_{1,\cI_1}' \|_2/\sqrt{n_2} \leq \| \bQ - \bQ' \|_{\sF} \| \be_{1,\cI_1} \|_2/\sqrt{n_2}$, and likewise for the other arguments.
    By concentration on the special orthogonal group (see, e.g.,~\cite[Theorem 5.2.7]{vershynin2019}),
    with probability at least $1 - \sC(\epsilon)e^{-\sc(\epsilon)p}$ 
    \begin{equation}
    \begin{aligned}
        &\big| 
            \E[\phi_e(\bD)\mid \bQ ,\be_{1,\cI_1},\be_{2,\cI_1},\hat \be_{1,\cI_1},\hat \bh_{1,\cI_1}]
            -
            \E[\phi_e(\bD)\mid \be_{1,\cI_1},\be_{2,\cI_1},\hat \be_{1,\cI_1},\hat \bh_{1,\cI_1}]
        \big|
        \\
            &\qquad\qquad\qquad\qquad\qquad\qquad\qquad
            \leq \sqrt{\frac{p}{n_1}} \frac{M_e(\| \be_{1,\cI_1}\|_2 + \| \be_{2,\cI_1} \|_2 + \| \hat \be_{1,\cI_1}\|_2) + M_h \| \hat \bh_{1,\cI_1} \|_2}{ \sqrt{n_2} }\,\epsilon.
    \end{aligned}
    \end{equation}
    Using Lemma \ref{lem:marg-T-conc-helper} and recalling that $\hat \be_{1,\cI_1} = \bu_{1,\cI_1}$ (see Eq.~\eqref{eq:vk-uk}),
    with probability at least $1 - Ce^{-cp}$
    we have $\big(M_e(\| \be_{1,\cI_1}\|_2 + \| \be_{2,\cI_1} \|_2 + \| \hat \be_{1,\cI_1}\|_2) + M_h \| \hat \bh_{1,\cI_1} \|_2\big)/\sqrt{n_2} \leq C(M_e \sqrt{n_1/n_2} + M_h \sqrt{p/n_2})$.
    Thus,
    for $\epsilon < c'$ with probability at $1 - \sC(\epsilon)e^{-\sc(\epsilon)p}$
    \begin{equation}\label{eq:bound2}
        \big| 
            \E[\phi_e(\bD)|\bQ ,\be_{1,\cI_1},\be_{2,\cI_1},\hat \be_{1,\cI_1},\hat \bh_{1,\cI_1}]
            -
            \E[\phi_e(\bD)|\be_{1,\cI_1},\be_{2,\cI_1},\hat \be_{1,\cI_1},\hat \bh_{1,\cI_1}]
        \big| < \Big( M_e \sqrt{\frac{p}{n_2}} + M_h \frac{p}{\sqrt{n_1n_2}}\;\Big)\,\epsilon.
    \end{equation}
    Now denote $\tilde \phi_e(\be_{1,\cI_1}/\sqrt{n_1},\be_{2,\cI_1}/\sqrt{n_1},\hat \be_{1,\cI_1}/\sqrt{n_1},\hat \bh_{1,\cI_1}/\sqrt{n_1}) := \E[\phi_e(\bD)|\be_{1,\cI_1},\be_{2,\cI_1},\hat \be_{1,\cI_1},\hat \bh_{1,\cI_1}]$.
    Using the Lipschitz properties of $\phi_e$,
    we have
    \begin{equation}
    \begin{aligned}
        &\Big|
            \tilde \phi_e
            \Big(
                \frac{\be_{1,\cI_1}}{\sqrt{n_1}},
                \frac{\be_{2,\cI_1}}{\sqrt{n_1}},
                \frac{\hat \be_{1,\cI_1}}{\sqrt{n_1}},
                \frac{\hat \bh_{1,\cI_1}}{\sqrt{n_1}}
            \Big)
            -
            \tilde \phi_e
            \Big(
                \frac{\be_{1,\cI_1}'}{\sqrt{n_1}},
                \frac{\be_{2,\cI_1}}{\sqrt{n_1}},
                \frac{\hat \be_{1,\cI_1}}{\sqrt{n_1}},
                \frac{\hat \bh_{1,\cI_1}}{\sqrt{n_1}}
            \Big)
        \Big|
        \\
        &\quad \leq 
            M_e \sqrt{\frac{n_1}{n_2}}
                \frac{
                    \E_{\bQ}\big[\|(\bQ (\be_{1,\cI_1}-\be_{1,\cI_1}'))_{\cI_1 \cap \cI_2}\|_2\big]
                    }{\sqrt{n_1}}   
        \leq 
        M_e \sqrt{\frac{n_1}{n_2}}
                \frac{
                    \E_{\bQ}\big[\|(\bQ (\be_{1,\cI_1}-\be_{1,\cI_1}'))_{\cI_1 \cap \cI_2}\|_2^2\big]^{1/2}
                    }{\sqrt{n_1}}
        \\
        &\quad= M_e \sqrt{\frac{n_{12}}{n_2}} \frac{\|\be_{1,\cI_1}-\be_{1,\cI_1}'\|_2}{\sqrt{n_1}},
    \end{aligned}
    \end{equation}
    where we have used that $\E_{\bQ}\big[\|(\bQ (\be_{1,\cI_1}-\be_{1,\cI_1}'))_{\cI_1 \cap \cI_2}\|_2^2\big] = n_{12} \| \be_{1,\cI_1} - \be_{1,\cI_1}'\|_2^2 / n_1$ by symmetry and the rotationally invariant distribution of $\bQ$.
    Thus, $\tilde \phi_e$ is $M_e\sqrt{n_{12}/n_2}$-Lipschitz in $\be_{1,\cI_1}/\sqrt{n_1}$.
    Similarly, it is $M_e\sqrt{n_{12}/n_2}$-Lipschitz in $\be_{2,\cI_1}/\sqrt{n_1}$ and $\hat \be_{1,\cI_1}/\sqrt{n_1}$
    and $M_h\sqrt{n_{12}/n_2}$-Lipschitz in $\hat \bh_{1,\cI_1}/\sqrt{n_1}$.
    In particular, it is $(M_e+M_h)\sqrt{n_{12}/n_2}$-Lipschitz in all its arguments.
    Recalling the definition of $\tilde \phi_e$,
    Corollary \ref{cor:ext-to-hat-gh} implies that \whp
    \begin{equation}\label{eq:bound3}
        \big|
            \E[\phi_e(\bD)|\be_{1,\cI_1},\be_{2,\cI_1},\hat \be_{1,\cI_1},\hat \bh_{1,\cI_1}]
            -
            \E[\phi_e(\bD)]
        \big| < ( M_e + M_h) \sqrt{\frac{n_{12}p}{n_1n_2}}\,\epsilon.
    \end{equation}
    Combining Eqs.~\eqref{eq:bound1}, \eqref{eq:bound2}, and \eqref{eq:bound3} and using $n_{12} \leq n_1$ gives Eq.~\eqref{eq:I2-marg}, 
    and the proof of Corollary \ref{cor:ext-to-I2}(i) is complete.

    Now we consider replacing $\hat \bh_{1,\cI_2}$ by $\hat \be_{1,\cI_2}^\de - \be_{1,\cI_2}$.
    By Eq.~\eqref{eq:I2-marg}, with probability at least $1 - Ce^{-pc}$ we have $\| \hat \bh_{1,\cI_2}\|_2/\sqrt{n_2} \leq \E[\| \bh_{1,\cI_2}^f \|_2] / \sqrt{n_2} + C\sqrt{p/n_1} \leq \E[\| \bh_{1,\cI_2}^f \|_2^2]^{1/2} / \sqrt{n_2} + \sqrt{p(p+n_{12})/(n_1n_2)} \leq C\sqrt{p/n_1}$, 
    where in the second inequality we have used Jensen and that $p/n_2 \leq C$ and $n_{12} /n_2 \leq 1$, and in the third inequality have used that $\E[\| \bh_{1,\cI_2}^f \|_2^2]^{1/2} = \tau_{h_1} \leq C \sqrt{p/n_1}$ by Eq.~\eqref{eq:tau-h-ref} and Lemma \ref{lem:bound-on-fixed-pt}.
    By Eq.~\eqref{eq:noise-est} 
    we have $\hat \be_{1,\cI_2\setminus \cI_1}^\de - \be_{1,\cI_2\setminus \cI_1} = - \bX_{\cI_2 \setminus \cI_1}(\hat \btheta_1 - \btheta_1)$ and by 
    Eq.~\eqref{eq:def-ggam-hgam}
    we have $\hat \bh_{1,\cI_2 \setminus \cI_1} = - (\tau_{h_1}/\|\bv_1\|_2)\bX_{\cI_2 \setminus \cI_1}(\hat \btheta_1 - \btheta_1)$, whence \whp\
    \begin{equation}
        \frac{\|\hat \bh_{1,\cI_2 \setminus \cI_1} - (\hat \be_{1,\cI_2\setminus \cI_1}^\de - \be_{1,\cI_2\setminus \cI_1})\|_2}{\sqrt{n_2}}
        = 
        \Big|
            1
            - 
            \frac{\sqrt{n_1/p}\,\|\bv_1\|_2}{\sqrt{n_1/p}\,\tau_{h_1}}
        \Big|
        \cdot 
        \frac{\| \hat \bh_{1,\cI_2 \setminus \cI_1} \|_2}{\sqrt{n_2}}
        \leq 
        C' \sqrt{\frac{p}{n_1}}\,\epsilon,
    \end{equation}
    where the inequality uses Lemma \ref{lem:marg-T-conc-helper} to bound $\big|1 - (\sqrt{n_1/p}\,\|\bv_1\|_2)/(\sqrt{n_1/p}\,\tau_{h_1})\big| \leq C'\epsilon$.
    Further, by Eq.~\eqref{eq:noise-est} 
    we have $\hat \be_{1,\cI_2 \cap \cI_1}^\de - \be_{1,\cI_2 \cap \cI_1} = \hat \be_{1,\cI_2 \cap \cI_1}/\zeta_1 - \be_{1,\cI_2 \cap \cI_1} $, 
    whence by Eq.~\eqref{eq:I2-marg} we have \whp\
    \begin{equation}
        \frac{\|\hat \bh_{1,\cI_2 \cap \cI_1} - (\hat \be_{1,\cI_2\cap \cI_1}^\de - \be_{1,\cI_2\cap \cI_1})\|_2}{\sqrt{n_2}}
        = 
        \frac{\|\hat \bh_{1,\cI_2 \cap \cI_1} - (\hat \be_{1,\cI_2\cap \cI_1}/\zeta_1 - \be_{1,\cI_2\cap \cI_1})\|_2}{\sqrt{n_2}}
        \leq 
        C' \sqrt{\frac{p(p+n_{12})}{n_1n_2}}\,\epsilon,
    \end{equation}
    where we have used that in the fixed design model $\bh_{1,\cI_2 \cap \cI_1}^f = \hat \be_{1,\cI_2\cap \cI_1}^f/\zeta_1 - \be_{1,\cI_2\cap \cI_1}$ (see Eq.~\eqref{eq:fixed-design-err-est}).
    Combining the previous two displays and using $\sqrt{p(p + n_{12})/(n_1n_2)} \leq C \sqrt{p/n_1}$ because $n_{12} \leq n_2$ and $p \leq Cn_2$ by \textsf{A2}, Corollary \ref{cor:ext-to-I2}(ii) follows.

    Corollary \ref{cor:ext-to-I2}(iii) holds by applying Lemma \ref{lem:emp-df-conc}.
\end{proof}

\subsection{Concentration of $\phi_{\theta|1}$ and $\phi_{e|1}$: proof of Lemma \ref{lem:phi-|1-conc}}
\label{sec:phi|1-conc-proof}

\begin{proof}[Proof of Lemma \ref{lem:phi-|1-conc}]
    Note $\phi_{\theta|1}$ is a function of $\mathsf{Cond}_1$ only via $\hat \bg_1$. 
    With some abuse of notation, 
    we will write $\phi_{\theta|1}(\hat \bg_1)$ for this function.
    Similary, $\phi_{e|1}$ is a function of $\mathsf{Cond}_1$ only via $\hat \bh_1,\be_1,\be_2$. 
    In fact, because the coordinates of the noise are independent in the fixed-design model,
    $\phi_{e|1}$ is a function of $\mathsf{Cond}_1$ only via $\hat \bh_{1,\cI_2}, \be_{1,\cI_2},\be_{2,\cI_2}$. 
    With some abuse of notation, 
    we will write $\phi_{e|1}(\be_{1,\cI_2}/\sqrt{n_2},\be_{2,\cI_2}/\sqrt{n_2},\hat \bh_{1,\cI_2}/\sqrt{n_2})$ for this function.
    
    We claim that $\phi_{\theta|1}$ is a $M_1 + M_2(\tau_{g_2}\rho_g/\tau_{g_1}) \leq M_1 + C M_2 \sqrt{n_1/n_2}$-Lipschitz function of $\hat \bg_1$. 
    Indeed, by Eq.~\eqref{eq:vf-prox} and because proximal operators are 1-Lipschitz \cite{parikh2014},
    we see that $\bv_k^f$ is a $1$-Lipschitz function of $\bg_k^f$.
    Becuase $\hat \btheta_k^f = \btheta_k + \bSigma^{-1/2} \bv_k^f$ (see Eq.~\eqref{eq:vkf}) and the singular values of $\bSigma$ are bounded below,
    we have that $\hat \btheta_k^f$ is $C$-Lipschitz functions of $\bg_k^f$.
    Because $\hat \btheta_k^{f,\de} = \btheta_k + \bg_k^f$, we have $\hat \btheta_2^{f,\de}$ is a 1-Lipschitz function of $\bg_k^f$ as well.
    In the fixed design model, we may represent $\bg_2^f = (\tau_{g_2}\rho_g/\tau_{g_1})\bg_1^f + \tau_{g_2}\rho_g^\perp \bxi_g$, where $\bxi_g \sim \normal(0,\id_p)$ independent of $\bg_1^f$.
    Because the expectation conditional on $\bg_1^f = \hat \bg_1$ in Eq.~\eqref{eq:def-phi|1} can be represented as an expectation over $\bxi_g$,
    the above observations imply that
    $\phi_{\theta|1}(\hat \bg_1)$ is $M_1 + M_2(\tau_{g_2}\rho_g/\tau_{g_1}) \leq M_1 + C M_2 \sqrt{n_1/n_2}$-Lipschitz in $\hat \bg_1$, where we have used Lemma \ref{lem:bound-on-fixed-pt} to bound $\tau_{g_2}\rho_g/\tau_{g_1} \leq C\sqrt{n_1/n_2}$.
    By the definition of $\phi_{\theta|1}$ (see Eq.~\eqref{eq:def-phi|1}),
    we have $\E[\phi_{\theta|1}(\bg_1^f)] = \E\big[\phi_\theta\big( \hat \btheta_k^f , \hat \btheta_k^{f,\de}  \big)\big]$.
    Thus, by Corollary \ref{cor:ext-to-hat-gh},
    \whp\
    \begin{equation}
    \begin{aligned}
        \Big|
            \phi_{\theta|1}\big(\hat \bg_1\big)
            -
            \E\big[\phi_\theta\big( \hat \btheta_k^f , \hat \btheta_k^{f,\de}  \big)\big]
        \Big|
        \leq 
        \Big(M_1 + M_2\sqrt{\frac{n_1}{n_2}}\Big)\sqrt{\frac{p}{n_1}}\,\epsilon 
        = 
        \Big(M_1\sqrt{\frac{p}{n_1}}+ M_2 \sqrt{\frac{p}{n_2}}\Big)\,\epsilon,
    \end{aligned}
    \end{equation}
    so that the first bound in Lemma \ref{lem:phi-|1-conc} is established.

    Similarly, we claim that $\phi_{e|1}$ is a $C\sqrt{n_1/n_2}$-Lipschitz function of $\hat \bh_{1,\cI_2}/\sqrt{n_2}, \be_{1,\cI_2}/\sqrt{n_2}, \be_{2,\cI_2}/\sqrt{n_2}$. 
    Indeed,
    by Eq.~\eqref{eq:fixed-design-err-est},
    $\hat \be_k^f = (\be_k + \bh_k^f)/\zeta_k$ and $\be_k^{f,\de} = \be_k + \bh_k^f$ are $C$-Lipschitz functions of $\bh_k^f$ because $1 / \zeta_k \leq C$ by Lemma \ref{lem:bound-on-fixed-pt}.
    In the fixed design model, we may represent $\bh_2^f = (\tau_{h_2}\rho_h/\tau_{h_1})\bh_1^f + \tau_{h_2}\rho_h^\perp \bxi_h$, where $\bxi_h \sim \normal(0,\id_N)$ independent of $\bh_1^f$.
    Because the expectation conditional on $\bh_1^f = \hat \bh_1,\be_1,\be_2$ in Eq.~\eqref{eq:def-phi|1} can be represented as an expectation over $\bxi_h$,
    we see that $\phi_{e|1}(\be_{1,\cI_2}/\sqrt{n_2},\be_{2,\cI_2}/\sqrt{n_2},\hat \bh_{1,\cI_2}/\sqrt{n_2})$ 
    is 
    $M_1$-Lipschitz in $\be_{1,\cI_2}/\sqrt{n_2}$,
    $M_2$-Lipschitz in $\be_{2,\cI_2}/\sqrt{n_2}$,
    and 
    $M_1 + M_2(\tau_{h_2}\rho_h/\tau_{h_1}) \leq M_1 +  C M_2 \sqrt{n_1/n_2}$-Lipschitz in $\hat \bh_{1,\cI_2}/\sqrt{n_2}$, where we have used Lemma \ref{lem:bound-on-fixed-pt} to bound $\tau_{h_2}/\tau_{h_1} \leq C \sqrt{n_1/n_2}$.
    By the definition of $\phi_{e|1}$ (see Eq.~\eqref{eq:def-phi|1}),
    we have $\E[\phi_{e|1}(\bh_{1,\cI_2}^f/\sqrt{n_2})] =                 \E\Big[
                    \phi_e\Big(\Big\{\frac{\hat \be_{k,\cI_2}^f}{\sqrt{n_2}}\Big\},\Big\{\frac{\hat \be_{k,\cI_2}^{\de,f}}{\sqrt{n_2}}\Big\}\Big)
                \Big]$.
    By Corollary \ref{cor:ext-to-I2},
    \whp\
    \begin{equation}
    \begin{aligned}
        \Big|
                \phi_{e|1}\Big(
                    \frac{\hat \bh_{1,\cI_2}}{\sqrt{n_2}}
                \Big)
                - 
                \E\Big[
                    \phi_e\Big(\Big\{\frac{\hat \be_{k,\cI_2}^f}{\sqrt{n_2}}\Big\},\Big\{\frac{\hat \be_{k,\cI_2}^{\de,f}}{\sqrt{n_2}}\Big\}\Big)
                \Big]
            \Big|
            &< 
            \Bigg[
                (M_1 + M_2) \sqrt{\frac{p}{n_2}} + \Big(M_1 + M_2\sqrt{\frac{n_1}{n_2}}\Big)\sqrt{\frac{p(p+n_{12})}{n_1n_2}}\,
            \Bigg]\epsilon 
        \\
        &\leq C\big(M_1+M_2\big) \sqrt{\frac{p}{n_2}}\,\epsilon,
    \end{aligned}
    \end{equation}
    where we have used that $M_1 \sqrt{p(p+n_{12})/(n_1n_2)} \leq CM_1 \sqrt{p/n_2}$ because $n_{12} \leq n_2$ and $p \leq Cn_1$,
    and $M_2 \sqrt{n_1/n_2}\,\sqrt{p(p+n_{12})/(n_1n_2)} = M_2 \sqrt{p(p+n_{12})/n_2^2} \leq CM_2 \sqrt{p/n_2}$ because $n_{12} \leq n_2$ and $p \leq Cn_2$.
    This complete the proof.
\end{proof}

\subsection{The good conditional characterization event: proof of Lemma \ref{lem:conditional-concentration-event}}
\label{sec:cond-event-proof}

\begin{proof}[Proof of Lemma \ref{lem:conditional-concentration-event}]
    For many of the objects controlled on the event $\cG_{2|1}(\epsilon,\Delta)$, our proof in fact shows a bettr rate (i.e., have a prefactor on $\epsilon$ which is smaller in $p/n_k$). 
    We state the Lemma only with the rate that we need so as not to clutter notation.

    First, we consider the concentration of 
    $\bT\Big(
        \frac{\be_{2,\cI_2}}{\sqrt{n_2}},
        \frac{\bu_{1,\cI_2}}{\sqrt{n_1}},
        \frac{\hat \bxi_{h,\cI_2}}{\sqrt{n_2}},
        \frac{\bxi_{h,\cI_2}}{\sqrt{n_2}}
    \Big)$.
    The required concentration follows from applying Lemma \ref{lem:T-conc} to its arguments pairwise, as we now show.
    Recall the definitions of $\be_2,\bu_1^f,\bh_1^f,\bh_2^f$ in Eqs.~\eqref{eq:fixed-des} and \eqref{eq:vkf} and of $\tau_{h_1},\tau_{h_2},\rho_h$ in Section \ref{sec:non-matrix-fix-pt}.
    Using the bound $\zeta_1^2(\tau_{e_1}^2 + \tau_{h_1}^2) \leq C$ by Lemma \ref{lem:bound-on-fixed-pt},
    we see that $\bu_1^f$ is a $1$-Lipschitz function of a Gaussian vector with variance bounded by $C$, and by the definition of $\tau_{h_1},\tau_{h_2},\rho_h$, we see that $\bh_{1,\cI_2}^f/\tau_{h_1}$ and $\bh_{2,\cI_2}^f - (\tau_{h_2}\rho_h/\tau_{h_1})/(\tau_{h_2}\rho_h^\perp)$ are standard Gaussian.
    Thus, the arguments to $\bT$ inside the expectation in the first line of Eq.~\eqref{eq:cEgam} are $M = 1 / (\sqrt{n_1 \vee n_2}) \leq C/\sqrt{p}$-Lipschitz functions of a Gaussian vector with variance bounded by $C$.
    Recall that $\bu_{1,\cI_2} = \hat \be_{1,\cI_2}$ (see Eq.~\eqref{eq:vkf}) and $\hat \bxi_{h,\cI_2} = \hat \bh_{1,\cI_2} / \tau_{h_1}$ (see Eq.~\eqref{eq:hat-g1-hat-h1}), and $\bxi_h$ is standard Gaussian independent of everything else.
    Using that $\tau_{h_1} \geq c \sqrt{p/n_1}$ by Lemma \ref{lem:bound-on-fixed-pt},
    we conclude that any function which is $1$-Lipschitz in $\be_{2,\cI_2}/\sqrt{n_2},\bu_{1,\cI_2}/\sqrt{n_1},\hat \bxi_{h,\cI_2}/\sqrt{n_2},\bxi_{h,\cI_2}/\sqrt{n_2}$ is in fact $M_e = (1 \vee \sqrt{n_2/n_1})$-Lipschitz in $\be_{1,\cI_2}/\sqrt{n_2},\be_{2,\cI_2}/\sqrt{n_2},\hat \be_{1,\cI_2}/\sqrt{n_2},\bxi_{h,\cI_2}/\sqrt{n_2}$ and $M_h = C \sqrt{n_1/p}$-Lipschitz in $\hat \bh_{1,\cI_2}/\sqrt{n_2}$.
    Then, Corollary \ref{cor:ext-to-I2} and Gaussian concentration of Lipschitz functions gives us Eq.~\eqref{eq:generic-concentration} with $K = C(1 \vee \sqrt{n_1/n_2})\sqrt{p/n_2} + C \sqrt{n_1/p}\,\sqrt{p(p+n_{12})/(n_1n_2)} \leq C$, where we use that $p/n_1< C$, $p/n_2 < C$ and $n_{12} \leq n_2$.
    By Eq.~\eqref{eq:T-cond-ref}, Eq.~\eqref{eq:E-bound} is satisfied,
    whence we may apply Lemma \ref{lem:T-conc} with $M = C/\sqrt{p}$ and $K = C$ to get the desired concentration.

    Second, we consider the concentration of 
    $
    \bT\Big(
        \sqrt{\frac{n_1}{p}}\,\bv_1,
        \sqrt{\frac{n_2}{p}}\,\bv_2^{cf},
        \frac{\hat \bxi_g}{\sqrt{p}},
        \frac{\bxi_g}{\sqrt{p}}
    \Big).
    $
    As above, the required concentration follows from applying Lemma \ref{lem:T-conc} to its arguments pairwise, as we now show.
    Recall the definitions of $\bv_1^f,\bv_2^f,\bg_1^f,\bv_2^f$ in Eqs.~\eqref{eq:vkf}, \eqref{eq:vf-prox}
    and of $\tau_{g_1},\tau_{g_2},\rho_g$ in Section \ref{sec:non-matrix-fix-pt}.
    Using the bound $\tau_{g_k} \leq C/\sqrt{n_k}$ by Lemma \ref{lem:bound-on-fixed-pt},
    we see that $\sqrt{n_1/p}\,\bv_1^f$, $\sqrt{n_2/p}\,\bv_2^f$ is a $1$-Lipschitz function of a Gaussian vector with variance bounded by $C$, and by the definition of $\tau_{g_1},\tau_{g_2},\rho_g$, we see that $\bg_{1,\cI_2}^f/\tau_{g_1}$ and $\bg_{2,\cI_2}^f - (\tau_{g_2}\rho_g/\tau_{g_1})/(\tau_{g_2}\rho_g^\perp)$ are standard Gaussian.
    Thus, the arguments to $\bT$ inside the expectation in the second line of Eq.~\eqref{eq:cEgam} are $M = C/\sqrt{p}$-Lipschitz functions of a Gaussian vector with variance bounded by $C$.
    By Eqs.~\eqref{eq:hat-g1-hat-h1}, \eqref{eq:cf-quantities}, \eqref{eq:cond-fixed-design-est}, the bounds $\tau_{g_2}  \leq C / \sqrt{n_2}$ and $\tau_{g_1} \geq c / \sqrt{n_1}$,
    and because proximal operators are 1-Lipschitz \cite{parikh2014},
    the vectors $\sqrt{n_2/p}\,\bv_2^{cf},\hat \bxi_g/\sqrt{p},\bxi_g/\sqrt{p}$ are $C\sqrt{n_1/p}$-Lipschitz functions of $\hat \bg_1,\bxi_g/\sqrt{n_1}$;
    and by Eq.~\eqref{eq:v-to-theta} and because the singular values of $\bSigma$ are bounded below by $c$, 
    we have $\sqrt{n_1/p}\,\bv_1$ is a $C\sqrt{n_1/p}$-Lipschitz function of $\hat \btheta_1$.
    Then, Corollary \ref{cor:ext-to-hat-gh} and Gaussian concentration of Lipschitz functions give us Eq.~\eqref{eq:generic-concentration} with $K = C$.
    By Eq.~\eqref{eq:T-cond-ref}, Eq.~\eqref{eq:E-bound} is satisfied,
    whence we may apply Lemma \ref{lem:T-conc} with $M = C/\sqrt{p}$ and $K = C$ to get the desired concentration.
        
    Third, we consider the concentration of $\bar \Omega_2(\bv_2^{cf})$.
    In the case of least-squares, $\bar \Omega_2(\bv_2^{cf}) = 0$ always, 
    and there is nothing to show.
    In the case of ridge-regression,
    we apply Lemma \ref{lem:T-conc}.
    Recall that $\bar \Omega_2(\bv_2^f) = \sqrt{p/n_2}\,(\lambda/2)\|\hat \btheta_2^f\|_2^2$.
    By Eqs.~\eqref{eq:vkf} and \eqref{eq:vf-prox}, because proximal operators are 1-Lipschitz \cite{parikh2014} and $\bSigma$ has singular values bounded below by $c$ and $\tau_{g_2} \leq C / \sqrt{n_2}$ by Lemma \ref{lem:bound-on-fixed-pt},
    we have that $\hat \btheta_2^f$ is $M= 1 / \sqrt{n_2}$-Lipschitz in $\bg_2^f / \tau_{g_2}$, which is standard Gaussian.
    Moreover, by the previous paragraph and because the singular values of $\bSigma$ are bounded below by $c$, we have that $\hat \btheta_2^{cf} = \btheta_2 + \bSigma^{-1/2}\bv_2^{cf}$ 
    is a $C\sqrt{n_1/n_2}$-Lipschitz function of $\hat \bg_1,\bxi_g/\sqrt{n_1}$, 
    whence by Corollary \ref{cor:ext-to-hat-gh} and Gaussian concentration of Lipschitz functions,
    Eq.~\eqref{eq:generic-concentration} holds for $K = \sqrt{p/n_2}$.
    By \textsf{A2} and Lemma \ref{lem:bound-on-fixed-pt}, we have $\E[\|\btheta_2 + \bSigma^{-1/2}\bv_2^f\|_2^2] \leq 2\|\btheta_2\|_2^2 + 2 C\E[\|\bv_2^f\|_2^2] \leq C(2 + p/n_2) \leq C$.
    Thus, using that $\omega_2 = \E[\bar \Omega_2(\bv_2^f)]$ (see Eq.~\eqref{eq:sig*-e*-ome*}),
    we my apply Lemma \ref{lem:T-conc} to get the required concentration of $\bar \Omega_2(\bv_2^{cf})$, where we use additionally that $\lambda < C$ by \textsf{A2}.
    In the case of the $\alpha$-smoothed Lasso, 
    we have that $\bar \Omega_2(\bv_2^{cf}) = \lambda \| \btheta_2 + \bSigma^{-1/2} \bv_2^{cf} \|_1/\sqrt{n_2}$ is 
    $C \lambda \sqrt{pn_1}/n_2 $-Lipschitz in $\hat \bg_1,\bxi_g/\sqrt{n_1}$.
    The required concentration of $\bar \Omega_2(\bv_2^{cf})$ follows by Corollary \ref{cor:ext-to-hat-gh} and Gaussian concentration of Lipschitz functions.

    To get the required concentration of $\bu_1$, 
    recall that $\bu_1 = \hat \be_1$ (see Eq.~\eqref{eq:vk-uk}).
    Then we apply Lemma \ref{lem:marginal-characterization} to conclude that \whp\
    we have $\big| \| \bu_{1,\cI_2} \|_2 / \sqrt{n_1} - \E[\| \bu_{1,\cI_2}^f \|_2 ]/\sqrt{n_1} \big| < \sqrt{p/n_1}\,\epsilon$ and $\big| \| \bu_{1,\cI_2^c} \|_2 / \sqrt{n_1} - \E[\| \bu_{1,\cI_2^c}^f \|_2 ]/\sqrt{n_1} \big| < \sqrt{p/n_1}\,\epsilon$.
    Then, because $\bu_{1,\cI_1}$ is a Gaussian with variance bounded by $C$ (see Eq.~\eqref{eq:res-ref} and Lemma \ref{lem:bound-on-fixed-pt}) and $\bu_{1,\cI_1^c} = 0$,
    we can use Gaussian concentration of Lipschitz functions to conclude that $\big| \E[\| \bu_{1,\cI_2}^f \|_2]/\sqrt{n_1} - \E[\| \bu_{1,\cI_2}^f \|_2^2]^{1/2}/\sqrt{n_1} \big| < C/\sqrt{n_1}$ and $\big| \E[\| \bu_{1,\cI_2^c}^f \|_2 ]/\sqrt{n_1} - \E[\| \bu_{1,\cI_2^c}^f \|_2^2 ]^{1/2}/\sqrt{n_1}\big| < C $.
    Thus, by adjusting $\sC(\epsilon)$ and $\sc(\epsilon)$ so that the bound becomes trivial with $\sqrt{p/n_1}\epsilon < C / \sqrt{n_1}$,
    we conclude the required concentration of $\| \bu_{1,\cI_2} \|_2 / \sqrt{n_1}$ and $\| \bu_{1,\cI_2^c} \|_2 / \sqrt{n_1}$.

    The required concentration of $\|\be_{2,\cI_2}\|_2/\sqrt{n_2}$ follows by Lemma \ref{lem:marg-T-conc-helper}. 
    The required concentration of $\| \proj_{\bv_1}\bxi_g\|_2 / \sqrt{n_2}$ and $\| \proj_{\bu_2}\bxi_h\|_2/\sqrt{n_2}$ holds because $\| \proj_{\bv_1}\bxi_g\|_2$ and $\| \proj_{\bu_2}\bxi_h\|_2$ are nothing but the norm of univariate Gaussians.

    Finally, we show $\bv_2^{cf} \in E_{v|1}(\Delta/2)$ and $\bu_2^{cf} \in E_{u|1}(\Delta/2)$
    By the definition of the conditional fixed design model (see Eq.~\eqref{eq:cf-quantities}),
    we have $\E\big[ \phi_v(\bv_2^f) \bigm| \bg_1^f = \hat \bg_1 \big] = \E\big[ \phi_v(\bv_2^{cf}) \bigm| \hat \bg_1 \big]$.
    Moreover, the above discussion implies that conditional on $\hat \bg_1$, $\bv_2^{cf}$ is a $C/\sqrt{n_2}$-Lipschitz function of $\bxi_g$. Thus, the required high-probability bound on $\bv_2^{cf} \in E_{v|1}(\Delta/2)$ holds by Gaussian concentration of Lipschitz functions and because $p/n_2 \leq C$ by \textsf{A2}.
    Similarly, 
    we have $\E\big[ \phi_u(\bu_{2,\cI_2}^f/\sqrt{n_2}) \bigm| \bh_1^f = \hat \bh_1,\, \be_1,\be_2 \big] = \E\big[ \phi_u(\bu_{2,\cI_2}^{cf}/\sqrt{n_2}) \bigm| \hat \bh_1,\, \be_1,\be_2 \big]$.
    Further, conditional on $\hat \bh_1,\be_1,\be_2$, we have $\bu_{2,\cI_2}^{cf}/\sqrt{n_2}$ is a $\tau_{h_2}/\sqrt{n_2} \leq C \sqrt{p}\,/n_2$-Lipschitz function of $\bxi_h$, where the inequality follows from Lemma \ref{lem:bound-on-fixed-pt}.
    Thus, we also conclude the required high-probability bound on $\bu_2^{cf} \in E_{u|1}(\Delta/2)$ by Gaussian concentration of Lipschitz functions.

    The proof of Lemma \ref{lem:conditional-concentration-event} is complete. 
\end{proof}

\section{Proof of Theorem \ref{thm:noise-est}: estimating noise covariance}
\label{sec:proof-noise-est}

We treat the on-diagonal and off-diagonal entries of $\bS_e$ separately.
\\

\noindent \textbf{On-diagonal estimation: noise variance.}
In Theorem \ref{thm:noise-est}, $k$ indexes the set $\cI_k$ of samples used to perform the estimation,
but the theorem provides both an estimate of the noise for the first and second linear model. 
We use $l$ to index the linear model whose noise variance we are estimating.

Note that $\| \hat \btheta_l^\de - \hat \btheta_l \|_{\bSigma}^2 = \| \bSigma^{1/2}(\hat \btheta_l^\de - \btheta_l) - \bv_l \|_2^2$,
were recall $\bv_l = \bSigma^{1/2}(\hat \btheta_l - \btheta_l)$ by Eq.~\eqref{eq:vk-uk}.
Applying Lemma \ref{lem:marg-T-conc-helper},\footnote{\label{ftn:df-replace}Recall that Lemmas \ref{lem:marginal-characterization} and \ref{lem:marg-T-conc-helper}, Corollary \ref{cor:ext-to-I2}, and Theorem \ref{thm:joint-characterization} apply also if $\hat \btheta_k^\de$ is computed using $\hat \df_k$ in place of $\df_k$.}
\whp, $\big| \| \hat \btheta_l^\de - \hat \btheta_l \|_{\bSigma}^2 - (p\tau_{g_l}^2 - 2\tau_{g_l}^2\,\df_l + \tau_{h_l}^2)\big| < C (p/n_l)\epsilon$.
We study the concentration of $\| \hat \be_{l,\cI_k}^\de \|_2^2 / n_k$ using Lemma \ref{lem:T-conc} and either Lemma \ref{lem:marginal-characterization} (in the case the $l = k$) or Corollary \ref{cor:ext-to-I2} (in the case that $l \neq k$).\footnote{See footnote \ref{ftn:df-replace}.}
By Eq.~\eqref{eq:fixed-design-err-est} and because $\tau_{\hat e_l^\de}^2 < C$ by Lemma \ref{lem:bound-on-fixed-pt},
we have that $\hat \be_{l,\cI_k}^{f,\de}$ is Gaussian with variance bounded by $C$.
Thus, in Lemma \ref{lem:T-conc},
we may take $M = 1 / \sqrt{n_k}$. 
Further, by Lemma \ref{lem:marginal-characterization} (in the case that $\ell = k$) or Corollary \ref{cor:ext-to-I2} (in the case the $l \neq k$),
we have Eq.~\eqref{eq:generic-concentration} for $\hat \be_{l,\cI_k}^\de/\sqrt{n_k}$
with $K = \sqrt{p/(n_1 \wedge n_2)}$.\footnote{In fact, for $l = k$, we get the better bound $K = \sqrt{p/n_l}$.}
Because $\tau_{\hat e_l^\de}^2 < C$,
Eq.~\eqref{eq:E-bound} is satisfied.
Thus, we may apply Lemma \ref{lem:T-conc} to conclude 
\whp, 
$\big|\hat \tau_{\hat e_l^\de}^2 - \tau_{\hat e_l^{\de}}^2 \big| = \big|\| \hat \be_{l,\cI_k}^{\de} \|_2^2/n_k - \tau_{\hat e_l^{\de}}^2 \big| \leq   C \sqrt{p/(n_1 \wedge n_2)}\,\epsilon$.
Recall $\hat \tau_{g_l}^2 = \hat \tau_{\hat e_l^{\de}}^2 / n_l = \| \hat \be_{l,\cI_k}^{\de} \|_2^2 / (n_kn_l)$ (see Eq.~\eqref{eq:Sg-hat}),
so that $(p-2\,\hat \df_l)\hat \tau_{g_l}^2 = (p/n_l - 2\,\hat \df_l /n_l) \hat \tau_{\hat e_l^{\de}}^2 $.
Because $p/n_l \leq C$ by \textsf{A2} and $\df_l \leq n_l = 1 - \zeta_l \leq 1$,
the concentration of $\hat \tau_{\hat e_l^\de}^2$ and Lemma \ref{lem:emp-df-conc} 
imply that \whp\ we have $\big|(p-2\,\hat \df_l)\hat \tau_{g_l}^2 - (p-2\,\df_l)\tau_{g_l}^2 \big| \leq C \sqrt{p/(n_1 \wedge n_2)}\,\epsilon$.
Combining these results, 
and using that $\tau_{\hat e_l^\de}^2 = \tau_{e_l}^2 + \tau_{h_l}^2$ (see Eq.~\eqref{eq:tau-ed-ref}),
we have \whp\
\begin{equation}
    \Big|\hat \tau_{\hat e_l^{\de}}^2 + (p-2\,\hat \df_l) \hat \tau_{g_l}^2 - \| \hat \btheta_l^{\de} - \hat \btheta_l \|_{\bSigma}^2 - \tau_{e_l}^2\Big| \leq \sqrt{\frac{p}{n_1 \wedge n_2}}\,\epsilon.
\end{equation}
\\

\noindent \textbf{Off-diagonal estimation: noise covariance.}
We establish the concentration of $\< \hat \btheta_1^\de - \hat \btheta_1 , \hat \btheta_2^\de - \hat \btheta_2 \>_{\bSigma} = \< \bSigma^{1/2}(\hat \btheta_1^\de - \hat \btheta_1) , \bSigma^{1/2}(\hat \btheta_2^\de - \hat \btheta_2) \>$ using the joint characterization (Theorem \ref{thm:joint-characterization}) and Lemma \ref{lem:T-conc}.
Recall by Eq.~\eqref{eq:fixed-des} that $\bSigma^{1/2} \hat \btheta^{f,\de} = \btheta_k + \bg_k^f$.
Further, because proximal operators are 1-Lipschitz \cite{parikh2014} and $\tau_{g_k} < C /\sqrt{n_k}$ by Lemma \ref{lem:bound-on-fixed-pt},
we have that $\sqrt{n_1/p}\,\bSigma^{1/2}(\hat \btheta_1^{f,\de} - \hat \btheta_1^f)$ and $\sqrt{n_2/p}\,\bSigma^{1/2}(\hat \btheta_2^{f,\de} - \hat \btheta_2^f)$ are $C/\sqrt{p}$-Lipschitz in $\bg_1^f / \tau_{g_1}$ and $\bg_2^f / \tau_{g_2}^f$, which are Gaussian vectors with variance bounded by $C$.
Thus we will apply Lemma \ref{lem:T-conc} with $M = C/\sqrt{p}$.
Lemma \ref{lem:marginal-characterization} gives us the first line in Eq.~\eqref{eq:generic-concentration} for vectors $\ba^{(k)} = \sqrt{n_k/p}\,\bSigma^{1/2}(\hat \btheta_k^\de - \hat \btheta_k)$.
Further, taking $M_1 = \sqrt{n_1/p}$ and $M_2 = \sqrt{n_2/p}$ in Theorem \ref{thm:joint-characterization} and using that the singular values of $\bSigma$ are bounded above by $C$ gives us the second line of Eq.~\eqref{eq:generic-concentration} for the these vectors with $K = \min\big\{ \sqrt{n_1/p} \sqrt{p/n_1} + \sqrt{n_2/p} , \sqrt{n_1/p} + \sqrt{n_2/p}\sqrt{p/n_2} \big\} \leq C \sqrt{(n_1 \wedge n_2)/p} $.\footnote{See footnote \ref{ftn:df-replace}}
Further, as we computed above, $(n_k/p)\E[\| \hat \btheta_k^{f,\de} - \hat \btheta_k \|_{\bSigma}^2] = n_k\tau_{g_k}^2(1-2\,\df_k/p) + (n_k/p)\tau_{h_k}^2 < C$ by Lemma \ref{lem:bound-on-fixed-pt},
so that Eq.~\eqref{eq:E-bound} is satisfied.
Thus, we may apply Lemma \ref{lem:T-conc}.
First, using that $\bSigma^{1/2} (\hat \btheta^{f,\de} - \btheta_k) = \bg_k^f$, 
we compute the expectation
\begin{equation}
\begin{aligned}
    &\E[ \< \bSigma^{1/2}(\hat \btheta_1^{f,\de} - \hat \btheta_1^f), \bSigma^{1/2}(\hat \btheta_1^{f,\de} - \hat \btheta_1^f) \> ] 
    \\
        &= \E[ \< \bg_1^f , \bg_2^f \>  ]
          - \E[ \< \bg_2^f , \bSigma^{1/2}(\hat \btheta_1^f - \btheta_1) \> ]
          - \E[ \< \bg_1^f , \bSigma^{1/2}(\hat \btheta_2^f - \btheta_2) \> ]
          + \E[ \< \hat \btheta_1^f - \btheta_1 , \hat \btheta_2^f - \btheta_2 \>_{\bSigma}  ]
    \\
        &= \tau_{g_1}\tau_{g_2}\rho_g(p - \df_1 - \df_2) + \tau_{h_1}\tau_{h_2}\rho_h = S_{g,12}(p - \df_1 - \df_2) + S_{h,12},
\end{aligned}
\end{equation}
where in the final equality we have used \eqref{eq:tau-h-ref} and Gaussian integration by parts with the definition of $\df_k$ (see Eq.~\eqref{eq:R-df}).
Thus, Lemma \ref{lem:T-conc} implies that \whp, 
$\big| \< \hat \btheta_1^\de - \hat \btheta_1 , \hat \btheta_2^\de - \hat \btheta_2 \>_{\bSigma} - S_{g,12}(p - \df_1 - \df_2) - S_{h,12} \big| < \sqrt{(n_1 \wedge n_2)/p} \sqrt{p^2/(n_1n_2)}\,\epsilon = \sqrt{p/(n_1 \vee n_2)}\,\epsilon$.

We now study the concentration of $\widehat S_{\hat e^d,12}^{(k)}$.
We decompose
\begin{equation}
    \widehat S_{\hat e^{\de},12}^{(k)}
        = 
        \frac{\< \be_{1,\cI_k},\be_{2,\cI_k}\>}{n_k}
        +
        \frac{\< \hat \be_{1,\cI_k}^\de - \be_{1,\cI_k},\be_{2,\cI_k}\>}{n_k}
        + 
        \frac{\< \be_{1,\cI_k}, \hat \be_{2,\cI_k}^{\de} -  \be_{2,\cI_k}\>}{n_k}
        +
        \frac{\< \hat \be_{1,\cI_k}^\de - \be_{1,\cI_k}, \hat \be_{2,\cI_k}^{\de} -  \be_{2,\cI_k}\>}{n_k}.
\end{equation}
Because $\be_1,\be_2$ are jointly Gaussian,
for $\epsilon < c'$, with probability at least $1 - \sC(\epsilon)e^{-\sc(\epsilon)p}$, $\big| \< \be_{1,\cI_k},\be_{2,\cI_k}\>/n_k - S_{e,12}\big| \leq C\sqrt{p/n_k}\,\epsilon \leq C \sqrt{p/(n_1 \wedge n_2)}\,\epsilon$.
To establish the concentration of $\< \hat \be_{1,\cI_k}^\de - \be_{1,\cI_k},\be_{2,\cI_k}\>/n_k$,
we use Lemma \ref{lem:marginal-characterization} (in the case $k = 1$) or Corollary \ref{cor:ext-to-I2} (in the case $k = 2$) together with Lemma \ref{lem:T-conc}.
Recall by Eq.~\eqref{eq:fixed-design-err-est} that $\hat \be_1^{f,\de} - \be_1 = \bh_1^f$.
Because $\tau_{h_1} < C \sqrt{p/n_1}$ by Lemma \ref{lem:bound-on-fixed-pt} and $\tau_{e_1}^2 \leq C$ by \textsf{A1},
we have that $\sqrt{n_1/p}(\hat \be_{1,\cI_k}^{f,\de} - \be_1)/\sqrt{n_k}$ and $\be_{2,\cI_k}/\sqrt{n_k}$ 
are $C / \sqrt{n_k}$-Lipschitz functions of Gaussian vectors with variance bounded by $C$.
Thus we will apply Lemma \ref{lem:T-conc} with $M = C/\sqrt{n_k}$.
In the case $k=1$,
Lemma \ref{lem:marginal-characterization} gives us Eq.~\eqref{eq:generic-concentration} for vectors $\sqrt{n_1/p}\,(\hat \be_{1,\cI_k}^\de - \be_{1,\cI_k})/\sqrt{n_k}$ and $\be_{1,\cI_k}/\sqrt{n_k}$ with $K = C$.
In the case $k = 2$,
Corollary \ref{cor:ext-to-I2} with $M_e = 1$ and $M_h = \sqrt{n_1/p}$ gives us Eq.~\eqref{eq:generic-concentration} for these vectors with $K = \sqrt{p/n_2} + \sqrt{n_1/p}\sqrt{p/n_1} \leq  C$.
Further, we have $\E[\| \be_{2,\cI_k} \|_2^2]/n_k = \tau_{e_2}^2 \leq C$ and $(n_1/p)\E[\| \bh_{1,\cI_k}^f \|_2^2]/n_k = (n_1/p)\tau_{h_1}^2 \leq C$ by Eqs.~\eqref{tau-e-ref} and \eqref{eq:tau-h-ref} and Lemma \ref{lem:bound-on-fixed-pt},
whence Eq.~\eqref{eq:E-bound} is satisfied.
Thus, we may apply Lemma \ref{lem:T-conc}.
Because $\E[\< \bh_{1,\cI_k}^f , \be_{2,\cI_k} \>]/n_k = 0$, we conclude that \whp, 
$|\< \hat \be_{1,\cI_k}^\de - \be_{1,\cI_k},\be_{2,\cI_k}\>/n_k| \leq \sqrt{p/n_1}\,\epsilon \leq \sqrt{p/(n_1 \wedge n_2)}\,\epsilon$.
An equivalent argument, switching the roles of the indices $1$ and $2$, gives that \whp,
$|\<\be_{1,\cI_k},\hat \be_{2,\cI_k}^\de - \be_{2,\cI_k}\>/n_k| \leq \sqrt{p/n_2}\,\epsilon \leq \sqrt{p/(n_1 \wedge n_2)}\,\epsilon$.

Finally, we deal with the last term
using Theorem \ref{thm:joint-characterization} and Lemma \ref{lem:T-conc}.
As justified in the previous paragraph,
we have that $\sqrt{n_1/p}\,(\hat \be_{1,\cI_k}^{f,\de} - \be_1)/\sqrt{n_k}$ and $\sqrt{n_2/p}\,(\hat \be_{2,\cI_k}^{f,\de} - \be_2)/\sqrt{n_k}$
are $C / \sqrt{n_k}$-Lipschitz functions of Gaussian vectors with variance bounded by $C$.
Thus, we will apply Lemma \ref{lem:T-conc} with $M = C / \sqrt{n_k}$.
In the case $k = 1$, 
Lemma \ref{lem:marginal-characterization} (in the case $k = 1$) and Corollary \ref{cor:ext-to-I2} (in the case $k = 2$) give the first line of Eq.~\eqref{eq:generic-concentration} for the vector $\sqrt{n_1/p}\,(\hat \be_{1,\cI_k}^\de - \be_1)/\sqrt{n_k}$.
Likewise, Eq.~\eqref{eq:generic-concentration} is satisfied for the vector $\sqrt{n_2/p}\,(\hat \be_{2,\cI_k}^\de - \be_2)/\sqrt{n_k}$.
Theorem \ref{thm:joint-characterization} with $M_1 = \sqrt{n_1/p}$ and $M_2 = \sqrt{n_2/p}$
gives the second line of Eq.~\eqref{eq:generic-concentration} with $K = \sqrt{n_1/p}(\sqrt{p/n_2} + p/n_1) + \sqrt{n_2/p} \leq C \sqrt{(n_1 \vee n_2)/p}$ in the case $k = 2$ and $K = \sqrt{n_1/p} + \sqrt{n_2/p}(\sqrt{p/n_1} + p/n_2) \leq C\sqrt{(n_1 \vee n_2)/p}$ in the case $k = 1$.\footnote{Note that we have used here the rate given by Theorem \ref{thm:joint-characterization} because we want to accommodate the case that $\hat \be_k^\de$ is compute with $\hat \df_k$ in place of $\df_k$.}
Finally, as we checked above $(n_1/p)\E[\|\bh_{1,\cI_k}^f\|_2^2]/n_k \leq C$ and $(n_2/p)\E[\|\bh_{2,\cI_k}^f\|_2^2]/n_k \leq C$.
Thus, we may apply Lemma \ref{lem:T-conc}.
Because $\E[\< \bh_{1,\cI_k}^f, \bh_{2,\cI_k}^f\>]/n_k = S_{h,12} $ (see Eq.~\eqref{eq:tau-ed-ref}),
we conclude that \whp,
$\big|\< \hat \be_{1,\cI_k}^{\de} - \be_{1,\cI_k}, \hat \be_{2,\cI_k}^{\de} -  \be_{2,\cI_k}\>/n_k - S_{h,12}\big| \leq \sqrt{(n_1 \vee n_2)/p}\sqrt{p^2/(n_1n_2)}\,\epsilon = \sqrt{p/(n_1 \wedge n_2)}\,\epsilon$.

Combining the concentration bounds on each term in the decomposition, 
we have \whp, $\big| \widehat S_{\hat e^{\de},12}^{(k)} - S_{e,12} - S_{h,12}\big| = \big| \widehat S_{\hat e^{\de},12}^{(k)} - S_{\hat e^{\de},12}\big| < \sqrt{p/(n_1 \wedge n_2)}\,\epsilon$,
where we have used that $ S_{\hat e^{\de},12} = S_{e,12} + S_{h,12}$ by Eq.~\eqref{eq:tau-ed-ref}.
Recalling that $\widehat S_{g,12}^{(k)} = (n_{12}/(n_1n_2)) \widehat S_{\hat e^{\de},12}^{(k)} $,
we conclude that \whp, $(p - \df_1 - \df_2)\widehat S_{g,12}^{(k)} = (n_{12}/(n_1n_2))(p-\df_1-\df_2)\widehat S_{\hat e^{\de},12}^{(k)} $.
Using Lemma \ref{lem:bound-on-fixed-pt},
we have that $(n_{12}/(n_1n_2))(p-\df_1-\df_2) \leq C$.
Thus,
recalling that $(n_{12}/(n_1n_2))S_{\hat e^{\de},12} = S_{g,12}$ by Eq.~\eqref{eq:fixed-pt-eqns},
the concentration bound for $\widehat S_{\hat e^{\de},12}^{(k)}$ implies that
\whp, $\big| (p - \df_1-\df_2)\widehat S_{g,12}^{(k)} - (p - \df_1-\df_2) S_{g,12}\big| < \sqrt{p/(n_1 \wedge n_2)}\,\epsilon$.
Combining the above results and using that $S_{\hat e^\de,12} = S_{e,12} + S_{h,12}$, we conclude that 
\begin{equation}
    \Big| \widehat S_{\hat e^{\de},12}^{(k)} + (p - \df_1 - \df_2) \widehat S_{g,12}^{(k)} - \<\hat \btheta_1^{\de} - \hat \btheta_1,\hat \btheta_2^{\de} - \hat \btheta_2\>_{\bSigma} - S_{e,12} \Big|
    \leq 
    \sqrt{\frac{p}{n_1 \wedge n_2}}\,\epsilon.
\end{equation}

The proof of Theorem \ref{thm:noise-est} is complete. \hfill $\square$

\section{Proof of Lemma \ref{lem:conditional-gordon}: the conditional Gordon inequality}\label{app:conditional-gordon}

The conditional Gordon inequality for regression (Lemma \ref{lem:conditional-gordon}) is derived by applying the marginal Gordon inequality (Lemma \ref{lem:gordon-marginal}) conditionally on $\mathsf{Cond}_1$, and then marginalizing over $\mathsf{Cond}_1$.
The key observation will be that conditioning on
$\mathsf{Cond}_1$
is equivalent to conditioning on a set of linear constraints on $\bA$ (recall we define $\bA = - \bX \bSigma^{-1/2}$).
First, we develop a comparison inequality which holds when conditioning on a certain type of linear constraints on $\bA$.

For any vectors $\bu_1',\bt_1',\be_1',\be_2' \in \reals^N$, $\bv_1',\bs_1' \in \reals^p$ 
consider the event (over the randomness in $\bA$)
\begin{equation}\label{eq:A-lin-constraint}
    \cE_{2|1} := \cE_{2|1}(\bu_1' , \bv_1', \bs_1', \bt_1')
    :=
    \Big\{
        \bA^\top \bu_1' +  \bs_1' = 0,
        \;
        \bA  \bv_1' -  \bt_1' = 0
    \Big\}.
\end{equation}
Consider the system of equations, repeated from Eq.~\eqref{eq:hat-xi} except now written for arbitrary $\bu_1',\bv_1',\bs_1',\bt_1'$ (not necessarily those defined in Eq.~\eqref{eq:uvst}).
\begin{equation}\label{eq:hat-xi-rep}
\begin{gathered}
    -\frac{ \bu_1'}{\| \bu_1'\|_2} \< \hat \bxi_g,  \bv_1' \> + \|  \bv_1'\|_2 \hat \bxi_h -  \bt_1' = 0,
    \qquad
    -\|  \bu_1' \|_2 \hat \bxi_g + \frac{ \bv_1'}{\|  \bv_1'\|_2} \< \hat \bxi_h,  \bu_1' \> +  \bs_1' = 0.
\end{gathered}
\end{equation}
These equations have a solution $\hat \bxi_g,\hat \bxi_h$
if and only if
the equations~\eqref{eq:A-lin-constraint} have a solution $\bA$:
\begin{lemma}\label{lem:KKT-equiv}
    The following are equivalent properties of the quadruplet $(\bu_1',\bv_1',\bs_1',\bt_1')$.
    \begin{enumerate}

        \item 
        $\<\bv_1',\bs_1'\> + \< \bu_1',\bt_1'\> = 0$.

        \item 
        There exists $\bA$ such that the constraints \eqref{eq:A-lin-constraint} are satisfied.

        \item 
        There exist $\hat \bxi_g,\hat \bxi_h$ such that the constraints \eqref{eq:hat-xi} are satisfied.

    \end{enumerate}
    For any $\alpha \in \reals$, under the additional constraint
    \begin{equation}
        \frac{\< \hat \bxi_h ,\bu_1 \>}{n_1} = \alpha,
    \end{equation}
    the solution $\hat \bxi_g,\hat \bxi_h$ to Eq.~\eqref{eq:hat-xi} is unique.
\end{lemma}
\noindent Lemma \ref{lem:KKT-equiv} is proved at the end of this section. 
We are ready to state the conditional Gordon inequality, of which Lemma \ref{lem:conditional-gordon} is corollary.
\begin{lemma}[Conditional Gordon]\label{lem:gen-conditional-gordon}
    Fix $ \bu_1',\bt_1' \in \reals^N$ and $\bv_1', \bs_1' \in \reals^p$ 
    such that $\<\bv_1',\bs_1'\> + \< \bu_1',\bt_1'\> = 0$.
    Fix any $\hat \bxi_g,\hat \bxi_h$ satisfying Eq.~\eqref{eq:hat-xi} (which exist by Lemma \ref{lem:KKT-equiv}).
    
    Let $\bA \in \reals^{N\times p}$ have entries $A_{ij} \stackrel{\mathrm{iid}}\sim \normal(0,1)$,
    $\bxi_g \stackrel{\mathrm{iid}}\sim \normal(0,\id_p)$, and $\bxi_h \sim \normal(0,\id_N)$, all independent,
    and define
    \begin{equation}
        \bg_{\mathrm{cg}}(\bu)
            :=
            \frac{\<  \bu_1',\bu \>}{\| \bu_1'\|_2}\hat \bxi_g +  \| \sfP_{ \bu_1'}^\perp \bu\|_2  \sfP_{ \bv_1'}^\perp  \bxi_g,
        \qquad 
        \bh_{\mathrm{cg}}(\bv)
            :=
            \frac{\<  \bv_1' , \bv \> }{\|  \bv_1' \|_2} \hat \bxi_h + \| \sfP_{ \bv_1'}^\perp \bv\|_2  \sfP_{ \bu_1'}^\perp \bxi_h,
    \end{equation}
    (where the subscript stands for ``conditional Gordon'').
    Define $\cE_{2|1} = \cE_{2|1}(\bu_1',\bv_1',\bs_1',\bt_1')$ as in Eq.~\eqref{eq:A-lin-constraint}.
    Then:
    \begin{enumerate}[(i)]
        
        \item 
        If $E_u \in \reals^N$, $E_v \in \reals^p$ are compact sets,
        then for any $t\in \reals$
        \begin{equation}
        \begin{aligned}
            \PP\left( \min_{\bv \in E_v} \max_{\bu \in E_u} \bu^\top \bA \bv + \psi(\bu,\bv) \leq t \bigm| \cE_{2|1} \right) 
            \leq 
                2\PP\left( 
                    \min_{\bv \in E_v} \max_{\bu \in E_u} 
                    -\<\bg_{\mathrm{cg}}(\bu),\bv\> 
                    + 
                    \<\bh_{\mathrm{cg}}(\bv),\bu\>
                    + 
                    \psi(\bu,\bv) \leq t\right).
        \end{aligned}
        \end{equation}  

        \item 
        If $E_u \in \reals^N$, $E_v \in \reals^p$ are compact, convex sets,
        then for any $t\in \reals$
        \begin{equation}
        \begin{aligned}
            \PP\left( \min_{\bv \in E_v} \max_{\bu \in E_u} \bu^\top \bA \bv + \psi(\bu,\bv) \geq t \bigm| \cE_{2|1} \right) 
            \leq 
                2\PP\left( 
                    \min_{\bv \in E_v} \max_{\bu \in E_u} 
                    -\<\bg_{\mathrm{cg}}(\bu),\bv\> 
                    + 
                    \<\bh_{\mathrm{cg}}(\bv),\bu\>
                    + 
                    \psi(\bu,\bv) \geq t\right).
        \end{aligned}
        \end{equation}  

    \end{enumerate}  
\end{lemma}
\noindent We prove Lemma \ref{lem:gen-conditional-gordon} at the end of this section.
First, we prove Lemma \ref{lem:conditional-gordon} as a straightforward consequence.

\begin{proof}[Proof of Lemma \ref{lem:conditional-gordon}]
    
Recall from Eq.~\eqref{eq:uvst} that
$\bt_1 = \bA \bv_1$
and $\bs_1 = -\bA^\top \bu_1$.
The KKT conditions for the min-max problem \eqref{eq:regr-min-max} state that $(\bv_1,\bu_1)$ is a saddle point for this min-max problem if and only if
\begin{equation}
    \frac1{n_1} \bt_{1,\cI_1} + \frac1{n_1} \be_{1,\cI_1} - \frac1{n_1} \bu_{1,\cI_1} = 0,
    \quad 
    -\frac1{n_1}\bs_1 \in \partial \bar \Omega_1(\bv_1),
    \quad 
    \bu_{1,\cI_1^c} = 0.
\end{equation}
Now consider any vectors $\bu_1',\bt_1',\be_1',\be_2' \in \reals^N$, $\bv_1',\bs_1' \in \reals^p$ satisfying $\bu_{1,\cI_1^c}' = \bzero$, 
$\bt_{1,\cI_1}' = \bu_{1,\cI_1}' - \be_{1,\cI_1}'$,
and $\bs_1' / n_1 \in \partial \bar \Omega_1(\bv_1)$.
(The prime indicates that these will be treated as dummy variables, not as those variables defined via Eq.~\eqref{eq:uvst}).
The KKT conditions above imply that the event (over the randomness is $\bA,\be_1,\be_2$)
\begin{equation}
    \cE_{2|1}(\bu_1' , \bv_1', \bs_1', \bt_1') \cap
    \Big\{
        \be_1 = \be_1',
        \; 
        \be_2 = \be_2'
    \Big\},
\end{equation}
is equivalent to the event
\begin{equation}
    \Big\{
        \hat \be_1 = \bu_1',
        \;
        \bSigma^{1/2}(\hat \btheta_1 - \btheta_1) = \bv_1',
        \;
        \bX(\hat \btheta_1 - \btheta_1) = -\bt_1',
        \;
        \bSigma^{-1/2}\bX^\top \hat \be_1 = \bs_1',
        \;
        \be_1 = \be_1',
        \; 
        \be_2 = \be_2'
    \Big\}.
\end{equation}
Thus,
the event that $\mathsf{Cond}_1 = (\hat \btheta_1, \hat \be_1, \bX (\hat \btheta_1-\btheta_1), \bX^\top \hat \be_1,\be_1,\be_2)$ obtains a certain value is equivalent to the event $\cE_{2|1}(\bu_1' , \bv_1', \bs_1', \bt_1') \cap
    \Big\{
        \be_1 = \be_1',
        \; 
        \be_2 = \be_2'
    \Big\}$ for $\bu_1' = \bu_1 , \bv_1' = \bv_1 , \bs_1' = \bs_1, \bt_1' = \bt_1$, and vice versa.
We can thus apply Lemma \ref{lem:gen-conditional-gordon} conditionally on $\mathsf{Cond}_1$.

Take $\bg_{\mathrm{cg}}(\bu)$ and $\bh_{\mathrm{cg}}(\bv)$ as defined in Lemma \ref{lem:conditional-gordon}.
Observe that $E_u = E_u(\mathsf{Cond}_1)$ is a deterministic set conditional on $\mathsf{Cond}_1$.
Thus, applying Lemma \ref{lem:gen-conditional-gordon}(i) conditionally,
we have for any $t \in \reals$ that with probability 1,
\begin{equation}
    \PP\left( \min_{\bv \in E_v} \max_{\bu \in E_u} L_2(\bu,\bv) \leq t \Bigm| \cE_{2|1}(\mathsf{Cond}_1)\right) 
    \leq 
        2\PP\left( 
            \min_{\bv \in E_v} \max_{\bu \in E_u} 
            \ell_{2|1}(\bu,\bv) \leq t \Bigm| 
            \bu_1,\bv_1,\hat \bxi_g,\hat \bxi_h
            \right).
\end{equation}
Taking the expectation of both sides of this inequality gives Lemma \ref{lem:conditional-gordon}(i).
Lemma \ref{lem:conditional-gordon}(ii) follows by the same argument using instead Lemma \ref{lem:gen-conditional-gordon}(ii).
\end{proof}

We now prove Lemma \ref{lem:gen-conditional-gordon}.
\begin{proof}[Proof of Lemma \ref{lem:gen-conditional-gordon}]
    Conditional on $\cE_{2|1}(\bu_1',\bv_1',\bs_1',\bt_1')$,
    \begin{align*}
        \bu^\top \bA \bv 
            &= 
            \bu^\top \sfP_{ \bu_1'}^\perp \bA \sfP_{ \bv_1'}^\perp \bv
            +
            \bu^\top \bA \sfP_{ \bv_1'}\bv
            +
            \bu^\top \sfP_{ \bu_1'} \bA \bv 
            -
            \bu^\top \sfP_{ \bu_1'} \bA \sfP_{ \bv_1'} \bv 
        \\
            &=
            \bu^\top \sfP_{ \bu_1'}^\perp \bA \sfP_{ \bv_1'}^\perp \bv
            -
            \frac{\< \bu ,  \bt_1' \> \<  \bv_1' , \bv \> }{ \|  \bv_1' \|_2^2 }
            - 
            \frac{ \< \bu ,  \bu_1' \> \<  \bs_1' , \bv \>  }{\|  \bu_1'\|_2^2 }
            +
            \frac{ \< \bu ,  \bu_1' \> \<  \bs_1' ,  \bv_1'\> \<  \bv_1' , \bv \>  }{\|  \bv_1' \|_2^2 \|  \bu_1'\|_2^2 }
        \\
            &=: 
            \bu^\top \sfP_{ \bu_1'}^\perp \bA \sfP_{ \bv_1'}^\perp \bv + \phi(\bu,\bv)
            \stackrel{\mathrm{d}}=
            \bu^\top \sfP_{ \bu_1'}^\perp \tilde \bA \sfP_{\bv_1'}^\perp \bv         
            +
            \phi(\bu,\bv),
    \end{align*}
    where $\tilde \bA$ is independent of and identically distributed to $\bA$. 
    We emphasize that the distributional equivalence holds conditional on the event \eqref{eq:A-lin-constraint}.
    Thus,
    \begin{equation}
    \begin{aligned}
        &\PP\left( \min_{\bv \in E_v} \max_{\bu \in E_u} \bu^\top \bA \bv + \psi(\bu,\bv) \leq t \bigm| \cE_{2|1} \right) 
        =
        \PP\left( 
            \min_{\bv \in E_v} \max_{\bu \in E_u} 
            \bu^\top \sfP_{\bu_1'}^\perp \tilde \bA \sfP_{\bv_1'}^\perp \bv
            +
            \phi(\bu,\bv)
            + 
            \psi(\bu,\bv) \leq t\right).
    \end{aligned}
    \end{equation}
    Let $\bxi_g \sim \normal(0,\id_p)$, $\bxi_h \sim \normal(0,\id_N)$.
    By the marginal Gordon inequality (Lemma \ref{lem:gordon-marginal}),
    \begin{equation}\label{eq:cond-gordon-unsimplified} 
    \begin{aligned} 
        &\PP\left( 
            \min_{\bv \in E_v} 
            \max_{\bu \in E_u} 
                \bu^\top \sfP_{\bu_1'}^\perp \tilde \bA \sfP_{\bv_1'}^\perp \bv 
                + 
                \phi(\bu,\bv) 
                + 
                \psi(\bu,\bv) \leq t\right) 
        \\ 
        &\qquad\qquad\leq 
            2\PP\left( 
                \min_{\bv \in E_v} \max_{\bu \in E_u} 
                -\| \sfP_{\bu_1'}^\perp \bu\|_2 \bxi_g^\top \sfP_{\bv_1'}^\perp \bv  
                + 
                \| \sfP_{\bv_1'}^\perp \bv\|_2 \bxi_h^\top \sfP_{\bu_1'}^\perp \bu 
                + 
                \phi(\bu,\bv) 
                + 
                \psi(\bu,\bv) \leq t\right). 
    \end{aligned}
    \end{equation}
    We now simplify the above expression.
    Using Eq.~\eqref{eq:hat-xi},
    \begin{equation}
    \begin{aligned}
        \phi(\bu,\bv)
            &=
            \frac{\big\< \bu , - \bu_1' \< \hat \bxi_g,\bv_1'\> /\| \bu_1' \|_2 + \| \bv_1' \|_2 \hat \bxi_h \big\> \< \bv_1' , \bv \> }{ \| \bv_1' \|_2^2 }
            + 
            \frac{ \< \bu , \bu_1' \> \big\< - \| \bu_1'\|_2 \hat \bxi_g + \bv_1' \< \hat \bxi_h,\bu_1'\> /\|\bv_1'\|_2 , \bv \big\>  }{\| \bu_1'\|_2^2 }
        \\
            &\qquad\qquad- \frac{ \< \bu , \bu_1' \> \< -\| \bu_1'\|_2 \hat \bxi_g + \bv_1' \< \hat \bxi_h,\bu_1'\> /\|\bv_1'\|_2  ,\bv_1' \> \< \bv_1', \bv \> }{\| \bv_1'\|_2^2 \| \bu_1' \|_2^2}
        \\
            &=
            - \frac{\< \bu_1' , \bu \> }{\| \bu_1' \|_2} \< \hat \bxi_g , \bv \> + \frac{\< \bv_1' , \bv \> }{\| \bv_1' \|_2} \< \hat \bxi_h , \bu \> .
    \end{aligned}
    \end{equation}
    Substituting into Eq.~\eqref{eq:cond-gordon-unsimplified},
    we conclude
    \begin{equation}
    \begin{aligned}
        &\PP\left( \min_{\bv \in E_v} \max_{\bu \in E_u}  \bu^\top \bA \bv + \psi(\bu,\bv) \leq t \bigm| \cE_{2|1} \right) 
        \\
        &\quad\leq 
            2\PP\left( 
                \min_{\bv \in E_v} \max_{\bu \in E_u} 
                -\big(\| \sfP_{\bu_1'}^\perp \bu\|_2  \sfP_{\bv_1'}^\perp \bxi_g + \frac{\< \bu_1' , \bu \> }{\| \bu_1' \|_2} \hat \bxi_g \big)^\top \bv 
                + 
                \big(\| \sfP_{\bv_1'}^\perp \bv\|_2  \sfP_{\bu_1'}^\perp \bxi_h + \frac{\< \bv_1' , \bv \> }{\| \bv_1' \|_2} \hat \bxi_h\big)^\top \bu
                + 
                \psi(\bu,\bv) \leq t\right).
    \end{aligned}
    \end{equation}
    The argument when $E_v,E_u$ are also convex and $\psi$ is convex-concave is equivalent but applies the convex version of the marginal Gordon inequality (Lemma \ref{lem:gordon-marginal}(ii)).
    The proof is complete. 
\end{proof}

\begin{proof}[Proof of Lemmas \ref{lem:KKT-equiv} and \ref{lem:unique-soln}]
    We first prove Lemma \ref{lem:KKT-equiv}.
    We write the system of equations \eqref{eq:hat-xi} in matrix-form
    \begin{equation}
        \begin{pmatrix}
             \bs_1 \\  \bt_1 
        \end{pmatrix}
        =
        \begin{pmatrix}
            \|  \bu_1' \|_2 \id_p  &  -  \bv_1'  \bu_1'^\top / \|  \bv_1' \|_2 \\
            -  \bu_1'  \bv_1'^\top / \|  \bu_1' \|_2 & \|  \bv_1' \|_2 \id_N 
        \end{pmatrix}
        \begin{pmatrix}
            \bg \\ \bh
        \end{pmatrix}
        =: 
        \bM \begin{pmatrix}
            \bg \\ \bh
        \end{pmatrix}. 
    \end{equation}
    The matrix $\bM$ in the preceding display is rank $p+N-1$,
    so that its column space has codimension 1.
    Checking that $\begin{pmatrix}  \bv_1'^\top &  \bu_1'^\top \end{pmatrix} \bM = 0$,
    we conclude Lemma \ref{lem:KKT-equiv}.

    Lemma \ref{lem:unique-soln} follows from Lemma \ref{lem:KKT-equiv} taking $\alpha = \zeta_1 \tau_{h_1}$ and observing that by the KKT conditions, Eq.~\eqref{eq:A-lin-constraint} is satisfied for $\bu_1' = \bu_1$, $\bs_1' = \bs_1$, $\bv_1' = \bv_1$, $\bt_1' = \bt_1$, and $\bA = - \bX\bSigma^{-1/2}$.
\end{proof}

\begin{remark}
    The conditional Gordon inequality makes an interesting---though to us still mysterious---connection between the primal and auxilliary KKT conditions.
    In particular,
    Eqs.~\eqref{eq:A-lin-constraint} are the KKT conditions for the regression optimization \eqref{eq:param-est}, and Eqs.~\eqref{eq:hat-xi-rep} are the KKT conditions for the marginal auxilliary objective \eqref{eq:lk}.
    Indeed, the gradient of $- \<\bg_{\mathrm{mg}}(\bu), \bv\> + \<\bh_{\mathrm{mg}}(\bv), \bu\>$ with respect to $\bu$ is $- (\bu / \| \bu \|_2)\< \bxi_g,\bv\> + \| \bv \|_2 \bxi_h$,
    and its gradient with respect to $\bv$ is $- \| \bu \|_2 \bxi_g + (\bv/\|\bv\|_2)\< \bxi_h,\bu\>$.
    Thus, for a quadruplet $(\bu_1',\bv_1',\bs_1',\bt_1')$ 
    for which $\bu_{1,\cI_1^c}' = \bzero$, 
    $\bt_{1,\cI_1}' = \bu_{1,\cI_1}' - \be_{1,\cI_1}'$,
    and $\bs_1' / n_1 \in \partial \bar \Omega_1(\bv_1)$,
    Eq.~\eqref{eq:A-lin-constraint} identifies those realizations of the random matrix $\bA$ for which the primal KKT conditions are satisfied at $(\bu_1',\bv_1',\bs_1',\bt_1')$,
    and Eq.~\eqref{eq:hat-xi-rep} identifies those realizations of the random noise in Gordon's objective for which the auxilliary KKT conditions are satisfied at $(\bu_1',\bv_1',\bs_1',\bt_1')$.
\end{remark}

\section{Technical lemmas}

\subsection{Properties of fixed-point solutions: proofs of Lemmas \ref{lem:fixed-pt-soln}, \ref{lem:bound-on-fixed-pt}, \ref{lem:marg-alpha-approx}, and \ref{lem:cond-alpha-approx}}
\label{sec:existence-fix-pt}

\begin{proof}[Proof of Lemma \ref{lem:fixed-pt-soln}: existence of fixed-point parameters]
    We first show that, for fixed $k$, the equations \eqref{eq:marg-fix-pt-eq}
    have a unique solution $\tau_{g_k},\zeta_k$ (note that these equations do not depend on $\rho_g$ or $\tau_{g_l},\zeta_l$ for $l \neq k$). 
    We consider least squares, ridge regression, and the Lasso separately.

    In the case of least squares, we may solve these two equations explicitly.
    Indeed, $\sR_k(\tau^2,\zeta) = p \tau^2$ and $\df_k(\tau^2,\zeta) = p$,
    whence by Eq.~\eqref{eq:marg-fix-pt-eq}
    \begin{equation}\label{eq:ols-fixed-pt}
        \text{for least squares,}
        \;\;
        \tau_{g_k}^2 = \frac{\tau_{e_k}^2/n_k}{1 - p / n_k}
        \;\;
        \text{and}
        \;\;
        \zeta_k = 1 - p/n_k
        \;\;
        \text{uniquely}.
    \end{equation}
    Because $\tau_{e_k}^2 > 0$ and $p < n_k$,
    we have $\tau_{g_k}^2 > 0$

    In the case of ridge regression with $\lambda  > 0$,
    we have the explicit form
    \begin{equation}\label{eq:ridge-R-df}
    \begin{gathered}
        \sR_k(\tau^2,\zeta)
            =
            \frac{p\lambda^2}{n_k \zeta^2} \btheta_k \bSigma^{1/2}\Big(\bSigma + \sqrt{\frac{p}{n_k}} \, \frac{\lambda}{\zeta} \id_p\Big)^{-2}\bSigma^{1/2} \btheta_k
            + 
            \tau^2 \Tr\Big(\bSigma^2\Big(\bSigma + \sqrt{\frac{p}{n_k}}\,\frac{\lambda}{\zeta}\id_p\Big)^{-2}\Big),
        \\
        \df_k(\tau^2,\zeta)
            =
            \Tr\Big(\bSigma\Big(\bSigma + \sqrt{\frac{p}{n_k}}\,\frac{\lambda}{\zeta}\id_p\Big)^{-1}\Big).
    \end{gathered}
    \end{equation}
    The degrees-of-freedom $\df_k(\tau^2,\zeta)$ does not depend on $\tau^2$, is continuous and strictly increasing in $\zeta$,
    and is 0 when $\zeta = 0$.
    Thus, the second fixed-point equation (i.e., $\zeta_k =  1 - \df_k(\tau_{g_k}^2,\zeta_k)/n_k$) has a unique solution $\zeta_k > 0$.
    Because $\bSigma( \bSigma + \sqrt{p/n_k}\,(\lambda/\zeta_k) \id_p)^{-1} \preceq \id_p$ and is diagonalizable,
    $\Tr\big(\bSigma^2( \bSigma + \sqrt{p/n_k}\,(\lambda/\zeta_k) \id_p)^{-2} \big) \leq \df_k(\tau_{g_2}^2,\zeta_k) < n_k$.
    Using the previous display and Eq.~\eqref{eq:marg-fix-pt-eq},
    we may solve 
    \begin{equation}\label{eq:ridge-tauO*}
        \tau_{g_k}^2 
            =
            \frac{\tau_{e_k}^2 + p\lambda^2/(n\zeta_k^2) \btheta_k^\top \bSigma^{1/2} (\bSigma + \sqrt{p/n_k}\, (\lambda / \zeta_k) \id_p)^{-2}\bSigma^{1/2} \btheta_k }{n_k - \Tr\big(\bSigma^2( \bSigma + \sqrt{p/n_k}\,(\lambda/\zeta_k) \id_p)^{-2} \big) } > 0.
    \end{equation}

    In the case of the Lasso and the $\alpha$-smoothed Lasso, 
    we apply results from the paper \cite{celentano2020lasso}, which considers the same setting with a different covariate normalization.
    In Section \ref{app:COV-lasso},
    we provide the change of variables to translate our normalization into theirs. 
    Under this change of variables,
    the fixed point equations \eqref{eq:marg-fix-pt-eq} are exactly equations (8a) and (8b) of \cite{celentano2020lasso}.
    Theorem 1 and Lemma A.2 of \cite{celentano2020lasso} guarantee these equations have a unique solution when $\bSigma$ 
    is invertible and $\tau_{e_k}^2 > 0$,
    and the solution $\tau_{g_k}^2 > 0$.

    Finally, we show that given a solution $\tau_{g_1},\zeta_1,\tau_{g_2},\zeta_2$ to Eqs.~\eqref{eq:marg-fix-pt-eq},
    there is a solution $\rho_g$ to the equation Eq.~\eqref{eq:simul-fix-pt}.
    By Cauchy-Schwartz, $\sR_{12}(\bS_g,\{\zeta_k\}) \leq \sqrt{\sR_1(\tau_{g_1}^2,\zeta_1)\sR_2(\tau_{g_2}^2,\zeta_2)}$. Because also $n_{12}/\sqrt{n_1 n_2} \leq 1$, 
    the right-hand side of the fixed point equation~\eqref{eq:simul-fix-pt} is between 
    \begin{equation}\label{eq:Q-bound}
        \pm 
        \frac{1}{\sqrt{n_1}}\sqrt{\tau_{e_1}^2\rho_e + \sR_1(\tau_{g_1}^2,\zeta_1)}  
        \frac{1}{ \sqrt{n_2} } \sqrt{\tau_{e_2}^2\rho_e + \sR_2(\tau_{g_2}^2,\zeta_2)}
        =
        \pm\tau_{g_1}\tau_{g_2}\sqrt{\big(1 - \tau_{e_1}^2\rho_e^{\perp2}/(n_1\tau_{g_1}^2)\big)\big(1 - \tau_{e_2}^2\rho_e^{\perp2}/(n_2\tau_{g_2}^2)\big)},
    \end{equation} 
    where we have used the fixed point equations \eqref{eq:marg-fix-pt-eq} to get the final equality.
    Because $\tau_{e_1},\tau_{e_2},\rho_e^\perp > 0$ (see Lemma \ref{lem:bound-on-fixed-pt}), 
    we conclude the right-hand side of the equation \eqref{eq:simul-fix-pt} is always in the open interval $(-\tau_{g_1}\tau_{g_2},\tau_{g_1}\tau_{g_2})$.
    Moroever, it is continuous in $\rho_g$.
    Indeed, 
    in the definition of $\sR_{12}$,
    we may represent $\bg_2^f = (\tau_{g_2}/\tau_{g_1})\rho_g \bg_1^f + \tau_{g_2}\sqrt{1-\rho_g^2}\bxi_g$ where $\bxi_g \stackrel{\mathrm{iid}}\sim \normal(0,\id_p)$ independent of everything else.
    Then continuity in $\rho_g$ follows by dominated convergence using that
    $\hat \btheta_2^{f,\de}$ and $\hat \btheta_2$ are $\tau_{g_2}$-Lipshitz in $\bg_2^f$ (because proximal operators are $1$-Lipschitz \cite{parikh2014}).
    At $\rho_g = 1$, the left-hand side of the fixed point equation \eqref{eq:marg-fix-pt-eq} is $\tau_{g_1}\tau_{g_2}$, and at $\rho_g = -1$, the left-hand side is $-\tau_{g_1}\tau_{g_2}$.
    The existence of a solution $\rho_g \in (-1,1)$ follows by the intermediate value theorem.

   We now show uniqueness of the solution $\rho_g$. 
   We denote by $\partial_{\rho_g} \sR_{12}(\bS_g,\{\zeta_k\}) $ the partial derivative taken with $\tau_{g_1},\tau_{g_2},\zeta_1,\zeta_2$ held fixed;
   by $\nabla_{\bg_k^f}$ the Jacobian with respect to $\bg_k^f$; 
   and for matrices $\bA,\bB \in \reals^{p\times p} $ by $\< \bA ,\bB \> = \Tr(\bA^\top \bB)$. 
   Then using the Dirichlet form for the Ornstein-Uhlenbeck process (see, for example, Sections 2.6 and 2.7 of \cite{bakry2013analysis}),
   we compute
    \begin{equation}\label{eq:cor-deriv}
    \begin{aligned}
        \big|\partial_{\rho_g} &\sR_{12}(\bS_g,\{\zeta_k\})\big|
            =
            \tau_1\tau_2 \big|
            \E\big[ \big\< \nabla_{\bg_1^f} \bSigma^{1/2} \eta_1(\bSigma^{1/2} \btheta_1 + \bg_1^f;\zeta_1) , \nabla_{\bg_2^f} \bSigma^{1/2} \eta_2(\bSigma^{1/2} \btheta_2 + \bg_2^f;\zeta_2) \big\>  \big]\big|
        \\
            &\leq 
            \tau_{g_1}\tau_{g_2}
            \E\big[\Tr\big((\nabla_{\bg_1^f} \bSigma^{1/2} \eta_1(\bSigma^{1/2} \btheta_1 + \bg_1^f;\zeta_1))^2\big)\big]^{1/2}
            \E\big[\Tr\big((\nabla_{\bg_2^f} \bSigma^{1/2} \eta_2(\bSigma^{1/2} \btheta_2 + \bg_2^f;\zeta_2))^2\big)\big]^{1/2}
        \\
        &\leq 
            \tau_{g_1}\tau_{g_2}
            \sqrt{\df_1(\tau_{g_1}^2,\zeta_1)\df_2(\tau_{g_2}^2,\zeta_2)}
            < \tau_{g_1} \tau_{g_2} \sqrt{n_1n_2},
    \end{aligned}
    \end{equation}
    where in the second-to-last inequality we have used that
    $\nabla_{\bg_1^f} \bSigma^{1/2} \eta_1(\bSigma^{1/2} \btheta_1 + \bg_1^f;\zeta_1)$ is symmetric (because $\bSigma^{1/2} \eta_1(\bSigma^{1/2} \btheta_1 + \bg_1^f;\zeta_1)$ is the solution to an optimization problem), and because $\bSigma \eta_k$ is $1$-Lipschitz, its eigenvalues are bounded by 1.
    Thus, the magnitude of the partial derivative with respect to $\rho_g$ of the right-hand side of Eq.~\eqref{eq:simul-fix-pt} is strictly less than $n_{12}/(n_1n_2)\tau_1\tau_2 \sqrt{n_1n_2} \leq \tau_1\tau_2$, where we have used that $n_{12} \leq n_1 \wedge n_2 \leq  \sqrt{n_1n_2}$.
    Because the partial derivative of the left-hand side of Eq.~\eqref{eq:simul-fix-pt} with respect to $\rho_g$ is $\tau_1\tau_2$,
    we conclude the solution $\rho_g$ is unique.
\end{proof}

To provide bounds on the fixed-point parameters in the proof of Lemma \ref{lem:bound-on-fixed-pt},
we will need the following lemma.

\begin{lemma}\label{lem:prox-err-bound}
    Taking any $\tau_{g_k},\zeta_k$ (not necessarily solutions to the fixed-point equations \eqref{eq:fixed-pt-eqns}),
    there exists constant $C$ depending only on $\cPmodel,\cPregr$ such that when the $\alpha$-smoothed lasso is used in the fixed design model
    \begin{equation}
        \| \hat \btheta_k^f - \btheta_k \|_{\bSigma} \leq  
            \| \bg_k^f\|_2
            +
            \frac{C}{\zeta_k} \sqrt{\frac{p}{n_k}}.
    \end{equation}
    Moreover,
    \begin{equation}
        \E[\| \hat \btheta_k^f - \btheta_k \|_{\bSigma}^2]
            \geq 
            \frac{c}{\zeta_k^2} \frac{p}{n_k} \Phi\Big(- \frac{C}{\sqrt{n_k}\,\tau_k \zeta_k}\Big),
    \end{equation}
    where $\Phi(z) = \P(Z \leq z)$ for $Z \sim \normal(0,1)$ is the Gaussian cdf.
\end{lemma}

\begin{proof}[Proof of Lemma \ref{lem:prox-err-bound}]
    The KKT conditions for the optimization \eqref{eq:fixed-design-est} are
    \begin{equation}
        \bSigma^{1/2}(\hat \btheta_k^f - \btheta_k )
            \in
            \bg_k^f - \frac1{\zeta_k} \bSigma^{-1/2} \partial \Omega_k(\hat \btheta_k),
    \end{equation}
    where for the Lasso and the $\alpha$-smoothed Lasso, we have $\| \bd \|_2 \leq C\sqrt{p/n_k}$ for any $\bd \in \partial \Omega_k(\hat \btheta_k)$ and any $\hat \btheta_k$.
    The first result follows.

    To derive the second result,
    we write
    \begin{equation}\label{eq:lasso-coord-bound}
        \bSigma(\hat \btheta_k^f - \btheta_k )
            \in
            \bSigma^{1/2}\bg_k^f - \frac1{\zeta_k} \partial \Omega_k(\hat \btheta_k^f).
    \end{equation}
    Note $[\partial \Omega_k(\hat \btheta_k^f)]_j \in [-\lambda/\sqrt{n_k},\lambda /\sqrt{n_k}]$,
    whence $\big| [\bSigma(\hat \btheta_k^f - \btheta_k )]_j \big| \geq \big| [\bSigma^{1/2}\bg_k^f]_j \big| - \lambda / (\sqrt{n_k}\,\zeta_k)$.
    Thus, 
    $\E\big[[\bSigma(\hat \btheta_k^f - \btheta_k )]_j^2\big] \geq \P\big( \big| [\bSigma^{1/2}\bg_k^f]_j \big| > 2 \lambda / (\sqrt{n_k}\zeta_k) \big) \lambda^2/(n_k\zeta_k^2) \geq c\Phi\big( - C/(\sqrt{n_k}\,\tau_{g_k} \zeta_k) \big) / (n_k \zeta_k^2)$.
    Because $\| \hat \btheta_k^f - \btheta_k \|_{\bSigma}^2 \geq c \| \bSigma(\hat \btheta_k^f - \btheta_k ) \|_2^2 = c \sum_{j=1}^p [\bSigma(\hat \btheta_k^f - \btheta_k )]_j^2$,
    the result follows.
\end{proof}

\begin{proof}[Proof of Lemma \ref{lem:bound-on-fixed-pt}: bounds on fixed-point parameters]
    By the first equation of \eqref{eq:marg-fix-pt-eq},    
    we have $\tau_{g_k}^2 \geq \tau_{e_k}^2 / n_k \geq c/n_k$.
    We next determine an upper bound on $\tau_{g_k}$ and lower bound on $\zeta_k$.
    We do so for least-squares, ridge regression, and the $\alpha$-smoothed Lasso separately (where $\alpha= 0$ is the Lasso).

    For least squares,
    by Eq.~\eqref{eq:ols-fixed-pt} we have $\zeta_k = 1 - p/n_k > c > 0$.
    Also by Eq.~\eqref{eq:ols-fixed-pt},
    we have $\tau_{g_k}^2 \leq (\tau_{e_k}^2 /n_k)/\zeta_k \leq C/n_k$.
    
    For ridge regression, there are two cases.
    In the first case, $n_k/p \geq c > 1$. 
    Then we use $\df_k(\tau^2,\zeta) \leq p$ (see Eq.~\eqref{eq:ridge-R-df}) 
    to get $\zeta_k = 1 - \df_k(\tau_{g_k}^2,\zeta_k)/n_k \geq 1 - 1/c > 0$.
    Note $n_k - \Tr(\bSigma^2(\bSigma + \sqrt{p/n_k}(\lambda / \zeta_k)\id_p)^{-2}) \geq n_k - \Tr(\bSigma(\bSigma + \sqrt{p/n_k}(\lambda / \zeta_k)\id_p)^-1) = n_k \zeta_k > n_k c$.
    Thus, by Eq.~\eqref{eq:ridge-tauO*},
    \begin{equation*}
        n_k\tau_{g_k}^2 
            \leq 
            \frac{\tau_{e_k}^2 + (p/n_k)(\lambda / \zeta_k)^2 \| \btheta_k \|_2^2 \| \bSigma(\bSigma + \sqrt{p/n_k}(\lambda / \zeta_k)\id_p)^{-2}\|_{\mathrm{op}}}{\zeta_k} \leq C.
    \end{equation*}
    In the second case, $\lambda \geq c > 0$.
    Then we use the second line of Eq.~\eqref{eq:ridge-R-df} to bound 
    $\df_k(\tau^2,\zeta_k) \leq Cp\zeta_k$.
    The second line of Eq.~\eqref{eq:marg-fix-pt-eq} implies $\zeta_k \geq 1 - C(p/n_k)\zeta_k \geq 1 - C\zeta_k$,
    whence $\zeta_k \geq 1/(1+C)$.
    Using, as above that $n_k - \Tr(\bSigma^2(\bSigma + \sqrt{p/n_k}(\lambda / \zeta_k)\id_p)^{-2}) \geq n_k c$, we conclude that $n_k \tau_{g_k}^2 \leq C$.

    For the Lasso and the $\alpha$-smoothed Lasso, 
    an upper bound on $\tau_{g_k}$ and lower bound on $\zeta_k$ holds by Theorem 2 and Lemma A.3 of \cite{celentano2020lasso}. See Section \ref{app:COV-lasso} for the change of normalization and notation connecting the current paper to \cite{celentano2020lasso}.

    We now derive bounds on $\tau_{h_k}$.
    For least squares, $\tau_{h_k}^2 = \sR_k(\tau_{g_k}^2,\zeta_k) = p \tau_{g_k}^2$, whence $cp/n_k \leq \tau_{h_k}^2 \leq Cp/n_k$ by the bounds on $\tau_{g_k}$.
    For ridge regression, 
    we upper bound
    \begin{align*}
        \tau_{h_k}^2 
            &= 
            \frac{p\lambda^2}{n_k \zeta_k^2} \btheta_k \bSigma^{1/2}\Big(\bSigma + \sqrt{\frac{p}{n_k}} \, \frac{\lambda}{\zeta_k} \id_p\Big)^{-2}\bSigma^{1/2} \btheta_k
            + 
            \tau_{g_k}^2 \Tr\Big(\bSigma^2\Big(\bSigma + \sqrt{\frac{p}{n_k}}\,\frac{\lambda}{\zeta_k}\id_p\Big)^{-2}\Big)
        \\
            &\leq 
            C\frac{p}{n_k} \frac{\| \btheta_k\|_2^2}{\sigma_{\min}(\bSigma + \sqrt{p/n_k}(\lambda / \zeta_k) \id_p)} + \tau_{g_k}^2 p \leq C p /n_k.
    \end{align*}
    For the lower bound, we use that $\Tr(\bSigma^2(\bSigma + \sqrt{p/n_k} (\lambda / \zeta_k)\id_p)^{-2}) \geq p\sigma_{\min}(\bSigma)^2(\sigma_{\min}(\bSigma)^2 + \sqrt{p/n_k}(\lambda /\zeta_k)) \geq cp$,
    whence, using that $\tau_{g_k}^2 \geq c/n_k$, we have $\tau_{h_k}^2 \geq cp/n_k$.
    Eq.~\eqref{eq:ridge-R-df} gives 
    $\tau_{h_k}^2 \leq (p/n_k)(\lambda/\zeta_k)^2 \| \btheta_k \|_2^2 / \sigma_{\min}(\bSigma + \sqrt{p/n_k}(\lambda / \zeta_k)\id_p) + \tau_{g_k}^2 p \leq Cp/n_k$.
    For the $\alpha$-smoothed Lasso, 
    we use that by Lemma \ref{lem:prox-err-bound},
    we have $\tau_{h_k}^2 = \E[\| \hat \btheta_k^f - \btheta_k \|_{\bSigma}^2 ] \leq 2\E[\| \bg_k^f \|_2^2] + (C/\zeta_k)^2 p/n_k \leq Cp/n_k$ by the upper bound on $\tau_{g_k}$.
    For the lower bound, we apply the second display in Lemma \ref{lem:prox-err-bound}, and used that $\zeta_k \leq  1$ and $\sqrt{n_k}\,\tau_{g_k} \zeta_k \geq c$, which gives $\tau_{h_k}^2 = \E[\| \hat \btheta_k^f - \btheta_k\|_{\bSigma}^2] \geq cp/n_k$.

    Recalling that $\tau_{\hat e_k^\de}^2 = \tau_{e_k}^2 + \tau_{h_k}^2$ (see Eq.~\eqref{eq:tau-ed-ref}),
    the above bounds imply $c < \tau_{\hat e_k^\de}^2 < C$.

    Next we upper bound $|\rho_e|$. Using the change of variables given in Section \ref{sec:gen-thm}, 
    we have $\rho_e^2 = \kappa^2\beta^2/(\kappa^2 \beta^2 + \sigma^2) = 1 - \sigma^2/(\kappa^2 \beta^2 + \sigma^2)  \leq 1 - c$ using that $C > \sigma^2 > c$ and $\kappa^2< C$, $\beta^2 < C$.
    We now upper bound $|\rho_g|$ and $|\rho_{\hat e^\de}|$.
    By Eq.~\eqref{eq:simul-fix-pt} and recalling the bounds on its right-hand side given by Eq.~\eqref{eq:Q-bound},
    we have $|\rho_g| \leq 1 - c$.
    Further, by Eq.~\eqref{eq:tau-ed-ref} and using $|\rho_h| \leq 1 $, we have $|\rho_{\hat e^{\de}}| \leq (\tau_{e_1}\tau_{e_2} + \tau_{h_1}\tau_{h_2})/(\tau_{\hat e_1^\de}\tau_{\hat e_2^\de}) - \tau_{e_1}\tau_{e_2}(1-\rho_e)/(\tau_{\hat e_1^\de}\tau_{\hat e_2^\de})$. 
    By Cauchy-Schwartz we have $(\tau_{e_1}\tau_{e_2} + \tau_{h_1}\tau_{h_2})/(\tau_{\hat e_1^\de}\tau_{\hat e_2^\de}) \leq 1$.
    Further, because $\tau_{h_k}^2 \leq Cp/n_k \leq C$, we have by Eq.~\eqref{eq:tau-ed-ref} that $\tau_{\hat e_k^{\de}}^2 \leq C$, whence $\tau_{e_1}\tau_{e_2}(1-\rho_e)/(\tau_{\hat e_1^\de}\tau_{\hat e_2^\de}) \geq c$, where we have used Lemma \ref{lem:bound-on-fixed-pt}. 
    We conclude that $|\rho_{\hat e^\de}| \leq 1 - c$.

    We now upper bound $|\rho_h|$.
    We may write
    $\bg_2^f = (\tau_{g_2} \rho_g / \tau_{g_1}) \bg_1^f + \tau_{g_2}\rho_g^\perp \bxi_g$ for $\bxi_g \sim \normal(0,\id_p)$ independent of everything else.
    Define $\tilde \bg_2^f = (\tau_{g_2} \rho_g / \tau_{g_1}) \bg_1^f + \tau_{g_2}\rho_g^\perp \tilde \bxi_g$, where $\tilde \bxi_g \sim \normal(0,\id_p)$ independent of everything else.
    Define $\tilde {\hat \btheta}_2^f$ as we defined $\hat \btheta_2^f$ except with $\tilde \bg_2^f$.
    We first claim that 
    \begin{equation}\label{eq:corr-diff}
        \E[\| \hat \btheta_2^f - \tilde {\hat \btheta}_2^f\|_{\bSigma}^2]
            \geq 
            c p / n_k.
    \end{equation}
    For ridge-regression (possibly with $\lambda = 0$ for least-squares), 
    observe that $\|\hat \btheta_2^f - \tilde {\hat \btheta}_2^f\|_{\bSigma} = \|\bSigma(\bSigma + \sqrt{p/n_k}\,(\lambda / \zeta_k)\id_p)^{-1}(\bg_k^f - \tilde \bg_k^f)\|_2 \geq c \| \bg_k^f - \tilde \bg_k^f\|_2$,
    where we have used that the singular values of $\bSigma$ are bounded below and $\sqrt{p/n_k}(\lambda / \zeta_k) \leq C$.   
    Then $\E[\| \hat \btheta_2^f - \tilde {\hat \btheta}_2^f\|_{\bSigma}^2] \geq c \E[\|\bg_2^f - \tilde \bg_2^f\|_2^2] = \E[\|\tau_{g_2}\rho_g^\perp(\bxi_g - \tilde \bxi_g)\|_2^2] = c \tau_{g_2}^2(1-\rho_g^2) \E[\| \bxi_g - \tilde \bxi_g\|_2^2] \geq cp/n_k$, where we have used that $|\rho_g| \leq 1-c$.
    
    For the $\alpha$-smoothed Lasso, 
    observe that for any $j \in [p]$,
    \begin{equation}
        \P\big([\bSigma^{1/2}\bg_2^f]_j > 2 \lambda / (\sqrt{n_2}\,\zeta_2) \; \text{and} \; [\bSigma^{1/2}\tilde \bg_2^f]_j < - 2 \lambda / (\sqrt{n_2}\,\zeta_2)\big) > c.
    \end{equation} 
    Indeed, this event is implied by the simultaneous occurrence of $[\bSigma^{1/2}\bg_2^f]_j \in [2 \lambda / (\sqrt{n_2}\,\zeta_2),3 \lambda / (\sqrt{n_2}\,\zeta_2)]$ and $\tau_{g_2}\rho_g^\perp [\bSigma^{1/2}\tilde \bxi_g]_j < - 5 \lambda / (\sqrt{n_2}\,\zeta_2)$.
    Because $\bSigma$ has bounded singular values and $\rho_g^\perp > c$, 
    the variance of both $[\bSigma^{1/2}\bg_2^f]_j$ and $\tau_{g_2}\rho_g^\perp [\bSigma^{1/2}\tilde \bxi_g]_j$ are bounded below and above by $c/n_k$ and $C/n_k$.
    Because further $\lambda$ and $\zeta_2$ are bounded above and below by a constant,
    we conclude that each of these events has probability lower bounded by $c > 0$. 
    Because $[\bSigma^{1/2}\bg_2^f]_j $ and $\tau_{g_2}\rho_g^\perp [\bSigma^{1/2}\tilde \bxi_g]_j$ are independent, we conclude the previous display. 
    Using Eq.~\eqref{eq:lasso-coord-bound} and the fact that $[\partial \Omega_k(\hat \btheta_k^f)]_j \in [-\lambda/\sqrt{n_k},\lambda /\sqrt{n_k}]$ for the $\alpha$-smoothed Lasso,
    we conclude that on the event in the previous display, $\big| [\bSigma(\hat \btheta_2^f - \btheta_2)]_j - [\bSigma(\tilde{\hat \btheta}_2^f - \btheta_2)]_j\big| > 2\lambda/(\sqrt{n_2}\zeta_2)$. 
    Summing over $j$ and using the the singular values of $\bSigma$ are bounded below,
    we conclude that
    \begin{equation}
        \E[\| \hat \btheta_2^f - \tilde {\hat \btheta}_2^f\|_{\bSigma}^2] \geq c\E[\| \bSigma(\hat \btheta_2 - \tilde{\hat \btheta}_2^f) \|_2^2] \geq cp\frac{4\lambda^2}{n_2\zeta_2^2} \geq cp/n_2.
    \end{equation}
    Thus, we have established Eq.~\eqref{eq:corr-diff} in all cases (i.e., least-squares, ridge regression, and the $\alpha$-smoothed Lasso).

    Now we use the triangle inequality to get $\| \hat \btheta_2^f - \tilde {\hat \btheta}_2^f\|_{\bSigma}^2 \leq 2\| (\hat \btheta_2^f - \btheta_2) - (\tau_{h_2}\rho_h/\tau_{h_1}) (\hat \btheta_1^f - \btheta_1) \|_{\bSigma}^2 + 2\| (\tilde{\hat \btheta}_2^f - \btheta_2) - (\tau_{h_2}\rho_h/\tau_{h_1}) (\hat \btheta_1^f - \btheta_1) \|_{\bSigma}^2$, and that $\E[\| (\hat \btheta_2^f - \btheta_2) - (\tau_{h_2}\rho_h/\tau_{h_1}) (\hat \btheta_1^f - \btheta_1) \|_{\bSigma}^2] = \E[\| (\tilde{\hat \btheta}_2^f - \btheta_2) - (\tau_{h_2}\rho_h/\tau_{h_1}) (\hat \btheta_1^f - \btheta_1) \|_{\bSigma}^2]$, to get
    \begin{equation}
    \begin{aligned}
        \frac{cp}{n_2} \leq \E[\| \hat \btheta_2^f - \tilde {\hat \btheta}_2^f\|_{\bSigma}^2]
            &\leq 
            4\E\big[\big\| 
                    (\hat \btheta_2^f - \btheta_2) - (\tau_{h_2}\rho_h/\tau_{h_1}) (\hat \btheta_1^f - \btheta_1)  
                \big\|_{\bSigma}^2\big]
        = \tau_{h_2}^2(1 - \rho_h^2) \leq \frac{Cp}{n_2}(1 - \rho_h^2).
    \end{aligned}
    \end{equation}
    We conclude $|\rho_h| \leq 1-c$, as claimed.

    Finally, by Eq.~\eqref{eq:fixed-design-est} and because proximal operators are 1-Lipschitz, 
    $\bSigma^{1/2} \eta_k(\by_k^f;\zeta_k)$ is $1$-Lipschitz in $\by_k^f$.
    Thus, $\div \bSigma^{1/2} \eta_k(\by_k^f;\zeta_k) \leq p$ almost surely, so by Eq.~\eqref{eq:R-df}, $\df_k \leq p$.
\end{proof}

\begin{proof}[Proof of Lemmas \ref{lem:marg-alpha-approx} and \ref{lem:cond-alpha-approx}]
    We prove Lemmas \ref{lem:marg-alpha-approx} and \ref{lem:cond-alpha-approx} simultaneously, because both proofs involve coupling the fixed design models corresponding to two different regression estimators.

    We compare two settings: one in which the second regression uses the Lasso with regularization parameter $\lambda$, and one in which the second regression uses the $\alpha$-smoothed Lasso with regularization parameter $\lambda$. 
    Across the the two settings, all other model parameters and the estimators in the first regression are the same.
    Each setting gives rise to a set of solutions to the fixed-point equations \eqref{eq:fixed-pt-eqns} and corresponding fixed design model \eqref{eq:fixed-des}.
    We will couple the fixed design models from the two settings so that the estimates from the second regression are close.

    We introduce some notation.
    The function $\sR$ appearing in the fixed point equations \eqref{eq:fixed-pt-eqns} will be denoted $\sR^{(0)}$ and $\sR^{(\alpha)}$ in the model in which the Lasso and $\alpha$-smoothed Lasso is used in the second regression, respectively.
    Likewise we define $\df_k^{(0)}$ and $\df_k^{(\alpha)}$.
    The solutions to the fixed point equations, in the notation of Section \ref{sec:non-matrix-fix-pt}, 
    will be denoted $\tau_{g_1,0},\tau_{g_2,0},\rho_{g,0},\zeta_{1,0},\zeta_{2,0}$ for the Lasso and  $\tau_{g_1,\alpha},\tau_{g_2,\alpha},\rho_{g,\alpha},\zeta_{1,\alpha},\zeta_{2,\alpha}$ for the $\alpha$-smoothed Lasso.
    We will denote by $\tau_{g_1},\tau_{g_2},\rho_{g},\zeta_{1},\zeta_{2}$ (i.e., with the second subscript omitted) generic arguments to the functions $\sR^{(0)},\df_k^{(0)},\sR^{(\alpha)},\df_k^{(\alpha)}$ which do not necessarily satisfy the fixed point equations \eqref{eq:fixed-pt-eqns} (or equivalently, \eqref{eq:marg-fix-pt-eq} and \eqref{eq:simul-fix-pt}).
    Note that because $\tau_{g_1,0},\zeta_{1,0}$ and $\tau_{g_1,\alpha},\zeta_{1,\alpha}$ are the solutions to the same equations \eqref{eq:marg-fix-pt-eq} for $k = 1$, we have $\tau_{g_1,0} = \tau_{g_1,\alpha}$ and $\zeta_{1,0} = \zeta_{1,\alpha}$.
    Denote the noise and estimates from the fixed design model in which the Lasso is used in the second regression by $\bg_{1,0}^f,\bg_{2,0}^f,\hat \btheta_{1,0}^f,\hat \btheta_{2,0}^f$,
    and denote by $\bg_{1,\alpha}^f,\bg_{2,\alpha}^f,\hat \btheta_{1,\alpha}^f,\hat \btheta_{2,\alpha}^f$ the corresponding objects for the fixed design model in which the $\alpha$-smoothed Lasso is used.

    We now construct a coupling between the two fixed design models.
    First, define $\bg_{2,*}^f = (\tau_{g_2,\alpha}/\tau_{g_2,0})\bg_{2,0}^f$ and
    \begin{equation}
        \hat \btheta_{2,*}^f 
            := 
            \argmin_{\bpi \in \reals^p}
            \Big\{
                \frac12\| \bg_{2,*}^f - \bSigma^{1/2} (\bpi - \btheta_2)\|_2^2 + \frac1{\zeta_{2,\alpha}} \frac{\lambda}{\sqrt{n_2}} \sM_{\alpha/\sqrt{n_2}}(\bpi)
            \Big\}.
    \end{equation}
    Note that $(\bg_{2,*}^f, \hat \btheta_{2,*}^f) \stackrel{\mathrm{d}}= (\bg_{2,\alpha}^f,\hat \btheta_{2,\alpha}^f)$.
    The quantity $\E[\| \hat \btheta_{2,*}^f - \hat \btheta_{2,0}^f \|_{\bSigma}^2]$ corresponds to the quantity $\| \bv_\alpha^* - \bv_0^* \|_{L_2}^2/n$ in the proof Lemma A.5 of \cite{celentano2020lasso}.
    There it is shown that this quantity is bounded by $C \alpha$ for $\alpha < c'$.
    Moreover, Lemma A.5 of \cite{celentano2020lasso} states that $|\sqrt{n_2}\,\tau_{g_2,\alpha} - \sqrt{n_2}\,\tau_{g_2,0}| < C \sqrt{\alpha}$ for $\alpha < c'$,
    whence, using Lemma \ref{lem:bound-on-fixed-pt}, we conclude that $|\tau_{g_2,\alpha}/\tau_{g_2,0} - 1| < C \sqrt{\alpha}$.
    We conclude that $\E[\| \bg_{2,*}^f - \bg_{2,0}^f \|_2^2] < C \alpha p / n_2 < C \alpha$.
    In particular,
    for any $1$-Lipschitz function $\phi_\theta:(\reals^p)^2 \rightarrow \reals$,
    we have
    \begin{equation}
        \big| \E[\phi_\theta(\hat \btheta_{2,\alpha}^f,\hat \btheta_{2,\alpha}^{f,\de})] - \E[\phi_\theta(\hat \btheta_{2,0}^f,\hat \btheta_{2,0}^{f,\de})] \big| < C\sqrt{\alpha}.
    \end{equation}
    Lemma \ref{lem:marg-alpha-approx} follows by applying this bound to $\E[\phi_\theta(\hat \btheta_{2,\alpha'}^f,\hat \btheta_{2,\alpha'}^{f,\de})]$ for $\alpha' \leq \alpha$ and using the triangle inequality and $C \sqrt{\alpha'} + C \sqrt{\alpha} \leq C \sqrt{\alpha}$.

    Next observe that 
    \begin{equation}
    \begin{aligned}
        &\big|
            \sR_{12}^{(\alpha)}( \tau_{g_1,0},\tau_{g_2,\alpha},\rho_{g,0},\zeta_{1,0},\zeta_{2,\alpha} ) 
            -
            \sR_{12}^{(0)}( \tau_{g_1,0},\tau_{g_2,0},\rho_{g,0},\zeta_{1,0},\zeta_{2,0} ) 
        \big|
    \\
            &\qquad\qquad\qquad= 
            \Big|
                \E[\< \hat \btheta_{1,0}^f - \btheta_1 , \hat \btheta_{2,*}^f - \btheta_2 \>_{\bSigma} ]
                - 
                \E[\< \hat \btheta_{1,0}^f - \btheta_1 , \hat \btheta_{2,0}^f - \btheta_2 \>_{\bSigma} ]
            \Big|
                \leq C\tau_{h_1}\sqrt{\alpha} \leq C \sqrt{\frac{\alpha p}{n_1}}.
    \end{aligned}
    \end{equation}
    Recall that $\tau_{g_1,0} = \tau_{g_1,\alpha}$ and $\zeta_{2,0} = \zeta_{2,\alpha}$.
    Using the fixed point equation \eqref{eq:simul-fix-pt},
    we conclude that 
    \begin{equation}
        \Big|
            \tau_{g_1,\alpha}\tau_{g_2,0}\rho_{g,0}
            -
            \frac{n_{12}}{n_1n_2}\big( \tau_{e_1}\tau_{e_2} \rho_e + \sR_{12}^{(\alpha)}(\tau_{g_1,\alpha},\tau_{g_2,\alpha},\rho_{g,0},\zeta_{1,\alpha},\zeta_{2,\alpha})\big) 
        \Big|
            \leq C\frac{n_{12}}{n_1n_2} \sqrt{\frac{ p}{n_1}}\,\sqrt{\alpha}
            \leq C \sqrt{\frac{\alpha}{n_1n_2}},
    \end{equation}
    where in the last inequality we use that $p/n_1 \leq C$ and $n_{12} \leq \sqrt{n_1n_2}$.
    Because $\big|\sqrt{n_2}\,\tau_{g_2,\alpha} - \sqrt{n_2}\,\tau_{g_2,0}\big|<C\sqrt{\alpha}$ and $\tau_{g_1,\alpha} \leq C / \sqrt{n_1}$,
    we conclude that 
    \begin{equation}\label{eq:approx-fix-pt}
        \Big|
            \tau_{g_1,\alpha}\tau_{g_2,\alpha}\rho_{g,0}
            -
            \frac{n_{12}}{n_1n_2}\big( \tau_{e_1}\tau_{e_2} \rho_e + \sR_{12}^{(\alpha)}(\tau_{g_1,\alpha},\tau_{g_2,\alpha},\rho_{g,0},\zeta_{1,\alpha},\zeta_{2,\alpha})\big)    
        \Big|
            \leq C \sqrt{\frac{\alpha}{n_1n_2}}.
    \end{equation}
    Finally, as argued in the proof of Lemma \ref{lem:fixed-pt-soln} (see Eq.~\eqref{eq:cor-deriv}),
    we have 
    \begin{equation}
    \begin{aligned}
        \frac{n_{12}}{n_1n_2}
            \big|
                \partial_{\rho_g}\sR_{12}^{(\alpha)}(\tau_{g_1,\alpha},\tau_{g_2,\alpha},\rho_g,\zeta_{1,\alpha},\zeta_{2,\alpha})\big)
            \big|
            &\leq \tau_{g_1,\alpha}\tau_{g_2,\alpha} \frac{n_{12}}{n_1n_2} \sqrt{\df_1^{(\alpha)}(\tau_{g_1,\alpha}^2,\zeta_{1,\alpha})\df_2^{(\alpha)}(\tau_{g_2,\alpha}^2,\zeta_{2,\alpha})}
        \\
            &\leq \tau_{g_1,\alpha}\tau_{g_2,\alpha} (1-c),
    \end{aligned}
    \end{equation}
    where we use that 
    $\df_k^{(\alpha)}(\tau_{g_k,\alpha}^2,\zeta_{k,\alpha})/n_k \leq 1-c$ by Lemma \ref{lem:bound-on-fixed-pt}.
    We see that the derivative of the term in the absolute value in Eq.~\eqref{eq:approx-fix-pt} has partial derivative with respect to $\rho_{g,0}$ no smaller than $c\tau_{g_1,\alpha}\tau_{g_2,\alpha}$.
    Because $\tau_{g_1,\alpha} \geq c/\sqrt{n_1}$ and $\tau_{g_2,\alpha} \geq c/\sqrt{n_2}$ by Lemma \ref{lem:bound-on-fixed-pt},
    Eq.~\eqref{eq:approx-fix-pt} and the derivative bound imply $\big|\rho_{g,\alpha} - \rho_{g,0}\big| < C \sqrt{\alpha}$.
    Because $\rho_{g,\alpha}^\perp = \sqrt{1-\rho_{g,\alpha}^2}$ and $\rho_{g,0}^\perp = \sqrt{1 - \rho_{g,0}^2}$, 
    and $|\rho_{g,\alpha}|,|\rho_{g,0}| < 1 -c$ by Lemma \ref{lem:bound-on-fixed-pt},
    we conclude also that $|\rho_{g,\alpha}^\perp - \rho_{g,0}^\perp| < C\sqrt{\alpha}$.

    We now construct the coupling between $(\bg_{1,0}^f,\bg_{2,0}^f)$ and $(\bg_{1,\alpha}^f,\bg_{2,\alpha}^f)$.
    Note that we may represent $\bg_{2,0}^f = (\tau_{g_2,0}\rho_{g,0}/\tau_{g_1,0}) \bg_{1,0}^f + \tau_{g_2,0}\rho_{g,0}^\perp \bxi_g$, where $\bxi_g\sim\normal(0,\id_p)$ independent of everything else.
    Using this representation, we construct the coupling $\bg_{1,\alpha}^f = \bg_{1,0}^f$ and $\bg_{2,\alpha}^f = (\tau_{g_2,\alpha}\rho_{g,\alpha}/\tau_{g_1,\alpha}) \bg_{1,0}^f + \tau_{g_2,\alpha}\rho_{g,\alpha}^\perp \bxi_g$.
    Note that 
    \begin{equation}
        \| \bg_{2,\alpha}^f - \bg_{2,*}^f \|_2
            \leq 
            \| (\tau_{g_2,\alpha}/\tau_{g_1,\alpha}) \bg_{1,0}^f \|_2 |\rho_{g,\alpha} - \rho_{g,0}|
            + 
            \| \tau_{g_2,\alpha} \bxi_g \|_2 |\rho_{g,\alpha}^\perp - \rho_{g,0}^\perp |.
    \end{equation}
    Using the bounds on $\big|\rho_{g,\alpha} - \rho_{g,0}\big|$ and $|\rho_{g,\alpha}^\perp - \rho_{g,0}^\perp|$ above and the bounds on $\tau_{g_2,\alpha}$ and $\tau_{g_2,0}$ from Lemma \ref{lem:bound-on-fixed-pt},
    we conclude that $\E[\| \bg_{2,\alpha}^f - \bg_{2,*}^f \|_2^2] \leq C \alpha$ for $\alpha < c'$.
    Because proximal operators are $1$-Lipschitz, 
    we then also conclude that $\| \hat \btheta_{2,\alpha}^f - \hat \btheta_{2,*}^f \|_{\bSigma}^2] \leq C\alpha$.
    We have already shown that $\E[\| \hat \btheta_{2,*}^f - \hat \btheta_{2,0}^f \|_{\bSigma}^2] < C\alpha$ and $\E[\| \bg_{2,*}^f - \bg_{2,0}^f \|_2^2] < C \alpha$.
    Thus,
    we have that $\E[\| \hat \btheta_{2,\alpha}^f - \hat \btheta_{2,0}^f \|_{\bSigma}^2] < C\alpha$ and $\E[\| \bg_{2,\alpha}^f - \bg_{2,0}^f \|_2^2] < C \alpha$ under this coupling.

    Using this coupling, we can derive the conclusion of the lemma.
    As argued in the proof of Lemma \ref{lem:phi-|1-conc} (see Section \ref{sec:phi|1-conc-proof}),
    $\phi_{\theta|1}^{(\alpha)}$ and $\phi_{\theta|1}^{(0)}$ can be viewed as a function of $\hat \bg_1$.
    Writing it as a function of this argument and computing expectations under this coupling, we get
    \begin{equation}
    \begin{aligned}
            \E\big[
                \big|
                \phi_{\theta|1}^{(\alpha)}(\bg_{1,0}^f)
                -
                \phi_{\theta|1}^{(0)}(\bg_{1,0}^f)
                \big|
            \big]
        &\leq 
        \E\big[
            \big|
                \E\big[\phi_\theta(\{ \hat \btheta_{k,\alpha}^f\},\{\hat \btheta_{k,\alpha}^{f,\de}\})\bigm| \bg_{1,0}^f\big]
                -
                \E\big[\phi_\theta(\{ \hat \btheta_{k,0}^f\},\{\hat \btheta_{k,0}^{f,\de}\})\bigm| \bg_{1,0}^f\big]
            \big|
        \big]
    \\
        &\leq 
        \E\big[
            \big|
                \phi_\theta(\{ \hat \btheta_{k,\alpha}^f\},\{\hat \btheta_{k,\alpha}^{f,\de}\})
                -
                \phi_\theta(\{ \hat \btheta_{k,0}^f\},\{\hat \btheta_{k,0}^{f,\de}\})
            \big|
        \big]
        \leq C M_2\sqrt{\alpha}.
    \end{aligned}
    \end{equation}
    As argued in the proof of Lemma \ref{lem:phi-|1-conc} (see Section \ref{sec:phi|1-conc-proof}),
    $\phi_{\theta|1}^{(\alpha)}$ and $\phi_{\theta|1}^{(0)}$ are $C(M_1 + M_2 \sqrt{n_1/n_2})$-Lipschitz in $\hat \bg_1$,
    whence $\big|
                \phi_{\theta|1}^{(\alpha)}(\cdot)
                -
                \phi_{\theta|1}^{(0)}(\cdot)
                \big|$ is $C(M_1 + M_2 \sqrt{n_1/n_2})$-Lipschitz.
    By Corollary \ref{cor:ext-to-hat-gh},
    we have for $\epsilon < c'$ with probability at least $1 - \sC(\epsilon) e^{-\sc(\epsilon)p}$
    that
    \begin{equation}
        \Big|
            \big|
                \phi_{\theta|1}^{(\alpha)}(\hat \bg_1)
                -
                \phi_{\theta|1}^{(0)}(\hat \bg_1)
            \big|
            -
            \E\big[
                \big|
                \phi_{\theta|1}^{(\alpha)}(\hat \bg_1)
                -
                \phi_{\theta|1}^{(0)}(\hat \bg_1)
                \big|
            \big]
        \Big| 
        < C (M_1 \sqrt{p/n_1} + M_2 \sqrt{p/n_2})\,\epsilon \leq C(M_1 \sqrt{p/n_1} + M_2)\epsilon.
    \end{equation}
    Taking $\epsilon = \sqrt{\alpha}$ gives Lemma \ref{lem:cond-alpha-approx} for $\alpha' = 0$.
    Lemma \ref{lem:cond-alpha-approx} for $\alpha' \leq \alpha$ follows by applying the previous display with $\alpha'$ in place of $\alpha$ and using the triangle inequality and $\sqrt{\alpha'} +\sqrt{\alpha} \leq 2 \sqrt{\alpha}$.
\end{proof}

\subsection{Gram-Schmidt approximation bound}

\begin{lemma}\label{lem:gs-approx}
    Consider $\bA \in \reals^{n \times k}$ such that $\big\|\bA^\top \bA - \sT\big\|_{\sF} < \epsilon$.
    Denoting the columns of $\bA$ by $\ba_1,\ldots,\ba_k$, assume also that  $\sT$ has singular values bounded below and above by $c$ and $C$ respectively.
    Consider applying the Gram-Schmidt procedure the $\ba_1,\ldots,\ba_k$ in that order, to produce $\ba_1^{\mathrm{gs}},\ldots,\ba_k^{\mathrm{gs}}$ which satisfy $\bA^{\mathrm{gs}\top}\bA^{\mathrm{gs}} = \sT$.
    Then there exists constants $C',c' >0 $ depending only on $c,C,k$ such that for $\epsilon < c'$
    \begin{equation}
        \max_{j \in [k]}\Big\{  \|\ba_j^{\mathrm{gs}} - \ba_j\|_2 \Big\} < C'\epsilon.
    \end{equation}
\end{lemma}

\begin{proof}[Proof of Lemma \ref{lem:gs-approx}]
    We denote the matrix in $\reals^{n \times j}$ formed by the first $j$ columns of $\bA$ by $\bA(j)$, and the submatrix of $\sT$ formed by the first $j$ rows and columns by $\sT(j)$.
    We prove the result inductively.
    Throughout the proof, $C',c'$ may change at each appearance, but only depend on $c,C,k$.

    For the base case, observe $\ba_1^{\mathrm{gs}} = \sqrt{\sT_{11}} \ba_1 / \| \ba_1 \|_2$. Because $\sT_{11} > c$, for $\epsilon < c'$ we have $\big| \sqrt{\sT_{11}}/\|\ba_1\|_2 - 1\big| < C'\epsilon$ and $\| \ba_1 \|_2 \leq C'$.
    Thus, $\| \ba_1^{\mathrm{gs}} - \ba_1 \|_2 < C'\epsilon$.

    Now assume the inductive hypothesis, that $\max_{j \leq i} \| \ba_j^{\mathrm{gs}} - \ba_j\|_2 < C'\epsilon$.
    We have
    \begin{equation}
        \ba^{\mathrm{gs}}_{i+1}
            =
            \bA^{\mathrm{gs}}(i) \sT(i)^{-1} \sT_{1:i,i+1}
            +
            \sqrt{T_{i+1,i+1}}
            \frac{\proj_{\bA^{\mathrm{gs}}(i)}^\perp \ba_{i+1}}{\|\proj_{\bA^{\mathrm{gs}}(i)}^\perp \ba_{i+1}\|_2},
    \end{equation}
    and 
    \begin{equation}
        \ba_{i+1}
            =
            \proj_{\bA^{\mathrm{gs}}(i)} \ba_{i+1}
            +
            \proj_{\bA^{\mathrm{gs}}(i)}^\perp \ba_{i+1}
            =
            \bA^{\mathrm{gs}}(i) \sT(i)^{-1} \bA^{\mathrm{gs}}(i)^\top \ba_{i+1}
            +
            \proj_{\bA^{\mathrm{gs}}(i)}^\perp \ba_{i+1}.
    \end{equation}
    Note
    \begin{equation}
        \big\| \bA^{\mathrm{gs}}(i)^\top \ba_{i+1} - \sT_{1:i,i+1}\big\|_2
            \leq 
            \big\| \bA^{\mathrm{gs}}(i)^\top \ba_{i+1} - \bA(i)^\top \ba_{i+1}\big\|_2
            +
            \big\| \bA(i)^\top \ba_{i+1} - \sT_{1:i,i+1}\big\|_2
            < C'\epsilon,
    \end{equation}
    whence because $\|\bA^{\mathrm{gs}}(i)\sT(i)^{-1}\|_{\mathsf{op}} < C'$,
    we conclude 
    \begin{equation}
        \big\| \bA^{\mathrm{gs}}(i) \sT(i)^{-1} \sT_{1:i,i+1} - \proj_{\bA^{\mathrm{gs}}(i)} \ba_{i+1}\big\|_2
            < C'\epsilon.
    \end{equation}
    This implies that
    \begin{equation}\label{eq:gs-par-bound}
        \Big|
            \| \proj_{\bA^{\mathrm{gs}}(i)} \ba_{i+1}^{\mathrm{gs}} \big\|_2 - \| \proj_{\bA^{\mathrm{gs}}(i)} \ba_{i+1}\big\|_2
        \Big|
        =
        \Big|
            \sT_{i+1,1:i} \sT(i)^{-1} \sT_{1:i,i+1} - \| \proj_{\bA^{\mathrm{gs}}(i)} \ba_{i+1}\big\|_2
        \Big| < C' \epsilon.
    \end{equation}
    Using the previous display and that $\big| \|\ba_{i+1}\|_2^2 - T_{i+1,i+1} \big| < \epsilon$,
    we have
    $\|\proj_{\bA^{\mathrm{gs}}(i)}^\perp \ba_{i+1}\|_2
    = \sqrt{\|\ba_{i+1}\|_2^2 - \| \proj_{\bA^{\mathrm{gs}}(i)} \ba_{i+1}^2 \| }$
    differs from $\sqrt{\sT_{i+1,i+1} - \sT_{i+1,1:i} \sT(i)^{-1} \sT_{1:i,i+1}}$ by at most $C'\epsilon$, 
    where we use that $\sT_{i+1,i+1} - \sT_{i+1,1:i} \sT(i)^{-1} \sT_{1:i,i+1} > c$ because $\sT$ has singular values bounded below by $c$.
    Thus,
    \begin{equation}
        \Big|
            \| \proj_{\bA^{\mathrm{gs}}(i)}^\perp \ba_{i+1}^\mathrm{gs} \|_2 
            -
            \| \proj_{\bA^{\mathrm{gs}}(i)}^\perp \ba_{i+1} \|_2            
        \Big|
        = 
        \Big|
            \sqrt{\sT_{i+1,i+1} - \sT_{i+1,1:i} \sT(i)^{-1} \sT_{1:i,i+1}}
            -
            \| \proj_{\bA^{\mathrm{gs}}(i)}^\perp \ba_{i+1} \|_2
        \Big| < C'\epsilon. 
    \end{equation}
    Combined with Eq.~\eqref{eq:gs-par-bound}, we conclude $\| \ba_{i+1}^{\mathrm{gs}} - \ba_{i+1} \|_2 < C'\epsilon$, and the inductive step is complete.
\end{proof}

\subsection{Comparison of cosines}

We prove the following lemma.
\begin{lemma}\label{lem:cos-bound}
    For a positive definite matrix $\bK$, define $\cos_{\bK}(\ba,\bb) = \ba^\top \bK \bb / (\| \ba \|_{\bK} \|\bb\|_{\bK})$,
    and define $\sin_{\bK}^2(\ba,\bb) = 1 - \cos_{\bK}^2(\ba,\bb)$.
    Let $\sigma_{\max}(\bK),\sigma_{\min}(\bK)$, and $\kappa_{\mathrm{cond}}(\bK)$ be the maximal and minimal singular values and the condition number of $\bK$ respectively.
    Then for any $\ba,\bb$
    \begin{equation}
        1 - \cos_{\bK}^2(\ba,\bb)
            \geq 
            \frac{1 - \cos_{\id}^2(\ba,\bb)}{4\kappa_{\mathrm{cond}}(\bK)}.
    \end{equation}
\end{lemma} 

\begin{proof}[Proof of Lemma \ref{lem:cos-bound}]
    Because the cosine is invariant to scaling $\ba$ and $\bb$, we may without loss of generality assume $\| \ba \|_2 = \| \bb \|_2 = 1$.
    Then
    \begin{equation}
    \begin{aligned}
        \cos_{\bK}^2(\ba,\bb)
            &= 
            \cos_{\id}^2\Big(
                \frac{\bK^{1/2}\ba}{\|\bK^{1/2}\ba\|_2},
                \frac{\bK^{1/2}\bb}{\|\bK^{1/2}\bb\|_2}
            \Big)
            = 1 - \frac14\Big\| \frac{\bK^{1/2}\ba}{\|\bK^{1/2}\ba\|_2} - \frac{\bK^{1/2}\bb}{\|\bK^{1/2}\bb\|_2}\Big\|_2^2
        \\
        &\leq 
            1 - \frac{\sigma_{\min}(\bK)}4\Big\| \frac{\ba}{\|\bK^{1/2}\ba\|_2} - \frac{\bb}{\|\bK^{1/2}\bb\|_2}\Big\|_2^2
            \leq 
            1 - \frac{\sigma_{\min}(\bK)}4\Big\| \frac{\ba}{\|\bK^{1/2}\ba\|_2} - \frac{\proj_{\bb}\ba }{\|\bK^{1/2}\bb\|_2}\Big\|_2^2
        \\
        &= 1 - \frac{\sigma_{\min}(\bK) \sin_{\id}(\ba,\bb)^2 }{4\| \bK^{1/2}\ba\|_2^2}
            \leq 1 - \frac{\sigma_{\min}(\bK) \sin_{\id}(\ba,\bb)^2 }{4\sigma_{\max}(\bK)}
            = 1 - \frac{1 - \cos_{\id}^2(\ba,\bb)}{4\kappa_{\mathrm{cond}}(\bK)}.
    \end{aligned}
    \end{equation}
    The proof is complete.
\end{proof}

\subsection{Change of normalization from \cite{celentano2020lasso}}\label{app:COV-lasso}

Our results rely on the lasso characterization provided by \cite{celentano2020lasso}, which uses a different normalization: in our case the covariates are $\bx_i \stackrel{\mathrm{iid}}\sim \normal(0,\bSigma)$, whereas in \cite{celentano2020lasso} the covariates are $\bx_i \stackrel{\mathrm{iid}}\sim \normal(0,\bSigma/n_{\cdot})$.
We here provide the change of variables to tranfer results between the two papers.

We will denote quantities which occur in their paper with a prime.
They consider a design matrix $\bX'$, observation $\by'$, true parameter $\btheta'$, noise standard deviation $\sigma'$, Lasso regularization $\lambda'$, and fixed-point parameter $\tau',\zeta'$.
These are related to the quantites in the present paper by 
\begin{equation}\label{eq:COV-lasso}
\begin{aligned}
    \bX' &= \frac1{\sqrt{n_{\sx}}} \bX_{\cI_k},\qquad& \by' &= \by_k,\qquad& \btheta' &= \sqrt{n_k}\, \btheta_k,\qquad& \bSigma' &= \bSigma, \qquad & \alpha' &= \alpha_k,\\
    \sigma'^2 &= \tau_{e_k}^2,\qquad& \lambda'&= \lambda_k,\qquad &\tau' &= \sqrt{n_k}\,\tau_{g_k},\qquad & \zeta'& = \zeta_k.
\end{aligned}
\end{equation}
With some algebra, one can check that the fixed point system equations (8) from \cite{celentano2020lasso} is equivalent to the outcome regression equations and precision regression equations in Eq.~\eqref{eq:fixed-pt-eqns} in the present paper.

\subsection{Reference on second moment identities}
\label{sec:identity-ref}

We write explicitly the expectations appearing in Eq.~\eqref{eq:cEgam}. They are
\begin{equation}\label{eq:T-marg-ref}
\begin{gathered}
    \sT_{\sN,k}
    =
    \E\Big[
        \bT\Big(\frac{\be_{k,\cI_k}}{\sqrt{n_k}},\frac{\bxi_{h,\cI_k}}{\sqrt{n_k}}\Big) 
        \Bigm| 
        \be_k
        \Big]   
        =
        \begin{pmatrix}
            \tau_{e_k}^{*2} & 0 \\[2pt]
            0 & 1
        \end{pmatrix},
    \\
    \sT_{\sP,k}
    =
    \E\Big[\bT\Big(\sqrt{\frac{n_k}{p}}\,\bv_k^f,\frac1{\sqrt{p}}\bxi_g\Big)\Big]
        =
        \begin{pmatrix}
            (n_k/p)\tau_{h_k}^2 & \sqrt{n_k}\,\tau_{g_k}\df_k/p \\[2pt]
            \sqrt{n_k}\,\tau_{g_k}\df_k/p & 1
        \end{pmatrix},
    \\
    \text{and $\sT_{\sP,k}$ has entries bounded above by $C$,}
\end{gathered}
\end{equation}
which we now check.
Because $\bxi_{h,\cI_k} \sim \normal(0,\id_{n_k})$ is independent of $\be_{k,\cI_k}$ and $\tau_{e_k}^* = \| \be_{k,\cI_k} \|_2/\sqrt{n_k}$ by definition,
the expression of $\sT_{\sN,k}$ is easily checked. 
By the definition of $\bv_k^f$ (see Eq.~\eqref{eq:vkf}) and the identity Eq.~\eqref{eq:tau-h-ref}, we have $\E[\|\bv_k^f\|_2^2] = \tau_{h_k}^2$. Because $\bxi_g \sim \normal(0,\id_p)$, we have $\E[\| \bxi_g \|_2^2] / p = 1$.
Thus, the diagonal entries of $\sT_{\sP,k}$ are as stated.
By the definition of $\bv_k^f$ (see Eqs.~\eqref{eq:fixed-design-est} and \eqref{eq:vkf}) and 
because proximal operators are 1-Lipschitz,
we see that $\bv_k^f$ is $1$-Lipschitz in $\bg_k^f$.
Thus, we may apply Gaussian integration by parts,
which gives $\E[\bg_k^{f\top} \bv_k^f] = \tau_{g_k}^2 \df_k $ by the definition of $\df_k$ (see Eq.~\eqref{eq:R-df}) and because $\bg_k^f \sim \normal(0,\tau_{g_k}^2 \id_p)$.
Thus, the off-diagonal entries of $\sT_{\sP,k}$ are as stated.
By Lemma \ref{lem:bound-on-fixed-pt}, $(n_k/p)\tau_{h_k}^2 \leq C$, so Eq.~\eqref{eq:T-marg-ref} follows.

We also rewrite explicitly the expectations appearing in Eq.~\eqref{eq:cEgam}.
They are
\begin{equation}\label{eq:T-cond-ref}
\begin{gathered}
\begin{aligned}
    \sT_{\sN,2|1}
    &=
    \E\Big[
        \bT\Big(
            \frac{\be_{2,\cI_2}}{\sqrt{n_2}},
            \frac{\bu_{1,\cI_2}^f}{\sqrt{n_1}},
            \frac{\bh_{1,\cI_2}^f}{\sqrt{n_2}\,\tau_{h_1}},
            \frac{\bh_{2,\cI_2}^f - (\tau_{h_2}\rho_h/\tau_{h_1})\bh_{1,\cI_2}^f}{\sqrt{n_2}\,\tau_{h_2}\rho_h^\perp}
        \Big) 
    \Big] 
    \\
        &=
        \begin{pmatrix}
            \tau_{e_2}^2 
                & \frac{n_{12}}{\sqrt{n_1n_2}}\tau_{e_1}\tau_{e_2} \rho_e \zeta_1
                & 0
                & 0  \\[4pt]
            \frac{n_{12}}{\sqrt{n_1n_2}}\tau_{e_1}\tau_{e_2} \rho_e \zeta_1
                & n_{12} \tau_{g_1}^2 \zeta_1^2
                & \frac{n_{12}}{\sqrt{n_1n_2}}\, \tau_{h_1} \zeta_1
                & 0 \\[4pt]
            0
                & \frac{n_{12}}{\sqrt{n_1n_2}}\, \tau_{h_1} \zeta_1
                & 1
                & 0 \\[4pt]
            0 
                & 0
                & 0
                & 1
        \end{pmatrix},
    \\[4pt]
    \sT_{\sP,2|1} 
        &= 
        \E\Big[
            \bT\Big(
                \sqrt{\frac{n_1}{p}}\,\bv_1^f,
                \sqrt{\frac{n_2}{p}}\,\bv_2^f,
                \frac{\bg_1^f}{\sqrt{p}\,\tau_{g_1}},
                \frac{\bg_2^f - (\tau_{g_2}\rho_g/\tau_{g_1})\bg_1^f}{\sqrt{p}\,\tau_{g_2}\rho_g^\perp}
            \Big) 
        \Big]
    \\
        &= 
        \begin{pmatrix}
            \frac{n_1}{p} \tau_{h_1}^2 
                & \frac{\sqrt{n_1n_2}}{p} \tau_{h_1}\tau_{h_2} \rho_h
                & \sqrt{n_1}\, \tau_{g_1}\frac{\df_1}{p}
                & 0  \\[4pt]
            \frac{\sqrt{n_1n_2}}{p} \tau_{h_1}\tau_{h_2} \rho_h
                & \frac{n_2}{p} \tau_{h_2}^2
                & \sqrt{n_2}\, \tau_{g_2}\rho_g \frac{\df_2}{p}
                & \sqrt{n_2}\, \tau_{g_2}\rho_g^\perp \frac{\df_2}{p} \\[4pt]
            \sqrt{n_1}\, \tau_{g_1}\frac{\df_1}{p}
                & \sqrt{n_2}\, \tau_{g_2}\rho_g \frac{\df_2}{p}
                & 1
                & 0 \\[4pt]
            0 
                & \sqrt{n_2}\, \tau_{g_2}\rho_g^\perp \frac{\df_2}{p}
                & 0
                & 1
        \end{pmatrix},
\end{aligned}\\[4pt]
\text{and $\sT_{\sN,2|1}$ and $\sT_{\sP,2|1}$ have entries bounded above by $C$,}
\end{gathered}
\end{equation}
which we now check.
By Eqs.~\eqref{eq:vkf} and \eqref{eq:fixed-design-err-est}, 
we have 
\begin{equation}
    \bu_{1,\cI_2}^f 
    = 
    \begin{pmatrix}
        \zeta_1(\be_{1,\cI_1\cap\cI_2} + \bh_{1,\cI_1 \cap \cI_2}^f)\\[4pt]
        \bzero_{n_2 - n_{12}}
    \end{pmatrix},
\end{equation}
where with some abuse of notation, we have written those indices in $\cI_1 \cap \cI_2$ first, and those in $\cI_2 \setminus \cI_1$ second.
This gives us that $\E[\| \be_{2,\cI_2}\|_2^2]/n_2 = \tau_{e_2}^2$, $\E[\< \be_{2,\cI_2} , \bu_{1,\cI_2}^f \> ] / \sqrt{n_1n_2} = \zeta_1\E[\< \be_{2,\cI_1 \cap \cI_2} , \be_{1,\cI_1\cap\cI_2} + \bh_{1,\cI_1 \cap \cI_2}^f \> ] / \sqrt{n_1n_2} = (n_{12}/\sqrt{n_1n_2})\tau_{e_1}\tau_{e_2}\rho_e \zeta_1$, and $\E[\| \bu_{1,\cI_2}^f \|_2^2]/n_1 = \zeta_1^2\E[\| \be_{1,\cI_1\cap\cI_2} + \bh_{1,\cI_1 \cap \cI_2}^f \|_2^2]/n_1 = \zeta_1^2 n_{12} (\tau_{e_1}^2 + \tau_{h_1}^2)/n_1 = n_{12}\tau_{g_1}^2\zeta_1^2 $, where we have applied the independence of $\be_2$ and $\bh_1$ and Eq.~\eqref{eq:res-ref}.
Note that $\tau_{e_2}^2 \leq C$ by \textsf{A1}4 and $n_{12}\tau_{g_1}^2 \zeta_1^2 \leq n_1 \tau_{g_1}^2 \leq C$ by Lemma \ref{lem:bound-on-fixed-pt}.
Thus, we have computed the top left $2\times2$ block of $\sT_{\sN,2|1}$.

By the definition of $\tau_{h_1},\tau_{h_2},\rho_h$ (see Section \ref{sec:non-matrix-fix-pt}),
we have $\bh_{1,\cI_2}^f/\tau_{h_1}$ and $(\bh_{2,\cI_2}^f - (\tau_{h_2}\rho_h/\tau_{h_1})\bh_{1,\cI_2}^f)/(\tau_{h_2}\rho_h^\perp$ are standard Gaussian and independent, both are independent of $\be_{2,\cI_2}$, and the latter is independent of $\bu_{1,\cI_2}^f$.
This gives the computation for the bottom right $2\times2$ block of $\sT_{\sN,2|1}$.
For the top right $2\times2$ block of $\sT_{\sN,2|1}$, the only non-zero entry comes from $\E[\<\bu_{1,\cI_2}^f,\bh_{1,\cI_2}^f\>]/(\sqrt{n_1n_2}\,\tau_{h_1})$, because the remaining entries are 0 by independence. 
We compute $\E[\<\bu_{1,\cI_2}^f,\bh_{1,\cI_2}^f\>]/(\sqrt{n_1n_2}\,\tau_{h_1}) = \zeta_1\E[\<\be_{1,\cI_1 \cap \cI_2} + \bh_{1,\cI_1 \cap \cI_2},\bh_{1,\cI_1 \cap \cI_2}^f\>]/(\sqrt{n_1n_2}\,\tau_{h_1}) = (n_{12}/\sqrt{n_1n_2})\tau_{h_1}\zeta_1$ as claimed.

The top left $2\times 2$ block of $\cT_{\sP,2|1}$ is computed using Eq.~\eqref{eq:tau-h-ref} and recalling that $\bv_k = \bSigma^{1/2}(\hat \btheta_k - \btheta_k)$ (see Eq.~\eqref{eq:vkf}).
By Lemma \ref{lem:bound-on-fixed-pt}, $(n_k/p)\tau_{h_k}^2 \leq C$.
By the definition of $\tau_{g_1},\tau_{g_2},\rho_g$ (see Section \ref{sec:non-matrix-fix-pt}),
we have $\bg_1^f/\tau_{g_1}$ and $(\bg_2^f - (\tau_{g_2}\rho_g/\tau_{g_1})\bg_1^f)/(\tau_{g_2}\rho_g^\perp)$ are standard Gaussian and independent, which gives the bottom right block. 
The top right block is computed using Gaussian integration by parts, recalling that because $\bv_k = \bSigma^{1/2}(\eta_k(\btheta_k + \bg_k^f;\zeta_k) - \btheta_k)$ (see Eqs.~\eqref{eq:fixed-design-est} and \eqref{eq:vkf}), the definition of $\df_k$ in Eq.~\eqref{eq:R-df} gives $\E[\< \bg_l^f , \bv_k^f \>] = \Cov(\bg_k^f,\bg_l^f) \df_k$.

\end{document}